%% file: A_PDE_model_for_bleb_formation_and_interaction_with_linker_proteins.tex
\documentclass{article} 

\usepackage{hyperref}
\usepackage{amsmath}
\usepackage{amssymb}
\usepackage{amsthm}
\usepackage{colonequals}
\usepackage{mathrsfs}
\usepackage{xifthen}
\usepackage{mathmacros}
\usepackage[backend=bibtex]{biblatex}
\usepackage{graphicx}
\usepackage{stmaryrd}
\usepackage{gnuplot-lua-tikz}
\usepackage{caption}
\usepackage{tikz}
\usepackage{pgfplots}
\usepackage{wrapfig}
\usepackage{longtable}
\usepackage{geometry}

\bibliography{A_PDE_model_for_bleb_formation_and_interaction_with_linker_proteins.bib}   

\newcommand{\revision}[2][]{%
  	\ifthenelse{\isempty{#1}}{%
	  	\textcolor{blue}{#2}%
    }{%
      	\ifthenelse{\equal{#1}{second}}{%
          	\textcolor{red}{#2}%
        }{%
          	\textcolor{green!50!black}{#2}%
        }%
    }%
}
\renewcommand{\revision}[2][]{#2}

\newcommand{\sndRevision}[2][]{%
  	\ifthenelse{\isempty{#1}}{%
	  	\textcolor{blue}{#2}%
    }{%
      	\ifthenelse{\equal{#1}{second}}{%
          	\textcolor{red}{#2}%
        }{%
          	\textcolor{green!50!black}{#2}%
        }%
    }%
}
\renewcommand{\sndRevision}[2][]{#2}

\usetikzlibrary{calc}

\definecolor{Green}{RGB}{0,255,100}

\newtheorem{mylemma}{Lemma}
\newtheorem{mytheorem}{Theorem}
\newtheorem{mycorollary}{Corollary}
\theoremstyle{definition}
\newtheorem{myproblem}{Problem}
\newtheorem{mycondition}{Condition}
\newtheorem{myremark}{Remark}
\newtheorem{commentEnv}{Comment}
\newtheorem{remark}{Remark}

\newcommand{\comment}[1]{%
	\begin{commentEnv} \textcolor{red}{#1} \end{commentEnv}
}
\renewcommand{\comment}[1]{}

\newcommand{\hide}[1]{#1}
\renewcommand{\hide}[1]{}

\newcommand{\rescExpConst}{\beta}

\newcommand{\sign}[1][]{\OPERATOR{\mathrm{sgn}}{#1}}
\newcommand{\difference}[1]{{#1}_{\Delta}}
\renewcommand{\identity}[1][]{I_{#1}}

\newcommand{\inverseFun}[2][]{%
  	\ifthenelse{\isempty{#1}}{%
  		{#2}^{-1}
    }{%
  		\left( #2 \right)^{-1}
    }
}
\DeclareMathOperator{\argmin}{\arg\min}
\newcommand{\landauSmallO}[1][]{\OPERATOR{o}{#1}}
\newcommand{\npPart}[1]{{#1}^-}
\newcommand{\nnPart}[1]{{#1}^+}
\newcommand{\jump}[1]{\left\llbracket #1 \right\rrbracket}

\newcommand{\symmGrad}[1]{\mathcal{J}\left(#1\right)}
\newcommand{\symmTangJacobian}[3][]{\DIFFOPERATOR{#1}{\mathcal{J}_{#2}}{#3}}
\renewcommand{\grad}[4][]{\DIFFOPERATOR{#1}{\nabla_{#2}^{#3}}{#4}}
\renewcommand{\diver}[3][]{\DIFFOPERATOR{#1}{\nabla_{#2}\cdot}{#3}}
\renewcommand{\laplacian}[4][]{\DIFFOPERATOR{#1}{\Delta_{#2}^{#3}}{#4}}

\renewcommand{\traceOp}[2][]{%
  	\ifthenelse{\isempty{#1}}{%
      	\OPERATOR{\gamma_0}{#2}
    }{%
      	\OPERATOR{\gamma_{#1}}{#2}
    }
}

\renewcommand{\tensorProd}{\otimes}

\newcommand{\shapeDeriv}[3]{d^{#2}\left( #3 ; #1 \right)}
\newcommand{\tDiff}[5][]{\DIFFOPERATOR{#1}{\partial_{#2,#3}^{#4}}{#5}}
\newcommand{\matDiff}[3][]{\DIFFOPERATOR{#1}{\partial^\circ_{#2}}{#3}}
\newcommand{\meanCurv}[1][]{\doubleSupScr{H}{#1}}
\newcommand{\spontMeanCurv}[1][]{\doubleSupScr{H_0}{#1}}
\newcommand{\gaussianCurv}[1][]{\doubleSupScr{K}{#1}}
\newcommand{\spaceParamDomain}[2]{\OPERATOR{\mathcal{G}}{#1_s;s\in #2}}
\newcommand{\weingartenMapping}[2][]{\OPERATOR{\doubleSupScr{\mathcal{H}}{#1}}{#2}}
\makeatletter
\newcommand{\@printSubst}[1]{r}
\newcommand{\perturbParam}{\@ifnextchar{'}{\@printSubst}{\delta}}
\makeatother
\newcommand{\perturbField}{v}
\newcommand{\localParam}[2][]{\OPERATOR{X^{#1}}{#2}}
\newcommand{\globalParam}[2][]{\OPERATOR{\doubleSupScr{\mathfrak{X}}{#1}}{#2}}
\newcommand{\liftingParametr}[2][]{\globalParam[#1]{#2}} 
\newcommand{\liftedFun}[2]{ \left[ #1 \right]_{#2} }
\newcommand{\frobScalarProd}{:}

\newcommand{\paramVel}[1][]{\mathcal{V}^{#1}}
\newcommand{\normalParamVel}[1][]{\mathcal{V}_\nu^{#1}}
\newcommand{\tangParamVel}[1][]{\mathcal{V}_\tau^{#1}}
\newcommand{\coeffNormalParamVel}[1][]{V_\nu^{#1}}

\newcommand{\orthProj}[2][]{\OPERATOR{\mathbb{P}_{#1}}{#2}}
\newcommand{\sobolevHSetMVF}[2]{\sobolevHSet[\text{mvf}]{#1}{#2}}
\newcommand{\sobolevTracesOfSolenoidals}[2]{%
  	\bochnerSobolevHSet[\traceOp{}\sigma]{#1}{#2}{\reals^3}%
}
\newcommand{\lebesgueTracesOfSolenoidals}[2]{%
  	\bochnerLebesgueSet[\traceOp{}\sigma]{#1}{#2}{\reals^3}%
}
\newcommand{\stokesDirDataSpace}[2]{%
  	\sobolevTracesOfSolenoidals{#1}{#2}%
}
\newcommand{\sobolevSolenoidals}[2]{%
	\sobolevHSet[\sigma]{#1}{#2}%
}
\newcommand{\bochnerLebesgueSetCMV}[3]{%
	\bochnerLebesgueSet[\text{cmv}]{#1}{#2}{#3}%
}
\newcommand{\bochnerDiffSetCMV}[3]{%
	\bochnerDiffSet[\text{cmv}]{#1}{#2}{#3}%
}

\renewcommand{\lebesgueM}[1]{\mathscr{L}^{#1}}
\renewcommand{\hausdorffM}[1]{\sigma^{#1}}

\newcommand{\eqAE}{\overset{a.\,e.}{=}}

\newcommand{\almostEvWh}{\text{a.\,e.}}

\newcommand{\physUnit}[1]{\mathsf{#1}}
\newcommand{\physDimOf}[1]{\left\langle #1 \right\rangle}
\newcommand{\dimLength}{\mathsf{L}}
\newcommand{\dimTime}{\mathsf{T}}
\newcommand{\dimMass}{\mathsf{M}}
\newcommand{\dimAmountOfSubst}{\mathsf{N}}
\newcommand{\rescaled}[1]{ \widehat{#1} }

\newcommand{\interiorLabel}{\text{int}}
\newcommand{\exteriorLabel}{\text{ext}}
\newcommand{\labeledObject}[2]{{#1}^{#2}}

\newcommand{\membraneEnergyFunc}[2]{ \OPERATOR{\mathcal{I}}{#1; #2} }
\newcommand{\membraneSurfEnergyFunc}[1][]{\OPERATOR{\mathcal{W}}{#1}}
\newcommand{\springFunc}[1][]{\OPERATOR{V}{#1}}
\makeatletter
\newcommand{\@perturbedScheme}[3]{{#1}_{#3}\left[ #2 \right]}
\def\perturbed#1#2#3{%
  	\ifthenelse{\equal{#3}{_}}{%
      	\@perturbedScheme{#1}{#2}
    }{%
      	\@perturbedScheme{#1}{#3}{#2}
    }
}
\makeatother

\newcommand{\heightLinkerSysHeightBF}[2]{a\left(#1,#2\right)}
\newcommand{\heightLinkerSysActiveLinkersBF}[2]{b_a\left(#1,#2\right)}
\newcommand{\heightLinkerSysInactiveLinkersBF}[2]{b_i\left(#1,#2\right)}

\newcommand{\generalDomain}{\mathcal{D}}
\newcommand{\finTime}{T}

\newcommand{\stokesDomain}[1][]{\labeledObject{\Omega}{#1}}
\newcommand{\stokesEulerDomain}[1][]{\stokesDomain[#1]_0}
\newcommand{\cauchyStressTensor}[3][]{\BINOPERATOR{#1{\mathbb{T}}}{#2}{#3}}
\newcommand{\dirToNeumOp}[2][]{\OPERATOR{#1{\mathcal{DN}}}{#2}}

\newcommand{\stokesOuterBoundary}{\Gamma}

\newcommand{\cell}{\stokesDomain[\interiorLabel]}

\newcommand{\cortex}{\mathcal{C}}
\newcommand{\eulerCortex}{\mathcal{C}_0}
\newcommand{\membrane}{\mathcal{M}}
\newcommand{\eulerMembrane}{\membrane_0}

\newcommand{\cellEvolMap}[3][]{\BINOPERATOR{#1{\Phi}}{#2}{#3}}
\newcommand{\cortexEvolMap}[2]{\BINOPERATOR{\Psi}{#1}{#2}}

\newcommand{\pressure}[1][]{\labeledObject{p}{#1}}
\newcommand{\velocity}[1][]{\labeledObject{u}{#1}}
\newcommand{\testFuncVelocity}{\varphi}
\newcommand{\testFuncPressure}{q}

\newcommand{\stokesDirichletToNeumannOp}[2][]{\OPERATOR{\mathcal{DN}_{#1}}{#2}}

\newcommand{\stokesNeumannData}{f}
\newcommand{\viscosity}[1][]{\labeledObject{\mu}{#1}}

\newcommand{\rootTimeDerivOperator}{S}
\newcommand{\timeDerivOperator}{L}
\newcommand{\cortexRadius}{R}
\newcommand{\membrStiffn}{\kappa}
\newcommand{\effectiveLengthParam}{\gamma}
\newcommand{\anotherMembrParam}{\lambda}
\newcommand{\repairRate}{k}
\newcommand{\activeLinkersDiffusiv}{\eta_a}
\newcommand{\inactiveLinkersDiffusiv}{\eta_i}
\newcommand{\rippingInterpol}{r}
\newcommand{\criticalHeight}{h^*}
\newcommand{\rippingLimitParam}{\vartheta}
\newcommand{\linkersSpringConst}{\xi}
\newcommand{\membraneDampingConst}{c}

\newcommand{\spaceMeshSize}{\Delta x}
\newcommand{\sphereTrafoTheta}{\theta}
\newcommand{\sphereTrafoPhi}{\phi}
\newcommand{\sphereTrafo}{S}
\newcommand{\twoSphere}[1][]{\omega^2_{#1}}

\newcommand{\membraneHeight}{h}
\newcommand{\activeLinkers}{\rho_a}
\newcommand{\inactiveLinkers}{\rho_i}
\newcommand{\testFuncMembraneHeight}{\varphi}
\newcommand{\testFuncActiveLinkers}{\sigma_a}
\newcommand{\testFuncInactiveLinkers}{\sigma_i}
\newcommand{\heightSolOp}[1]{\OPERATOR{H_\finTime}{#1}}
\newcommand{\linkersSolOp}[1]{\OPERATOR{G_\finTime}{#1}}

\newcommand{\totalLinkersMassDens}{\rho_0}
\newcommand{\totalLinkersMass}{m_0}

\newcommand{\LTwoNorm}[1]{\norm[\lebesgueSet{2}{\domain}]{#1}}
\newcommand{\HOneNorm}[1]{\norm[\sobolevHSet{1}{\domain}]{#1}}
\newcommand{\HTwoNorm}[1]{\norm[\sobolevHSet{2}{\domain}]{#1}}

\newcommand{\LTwoIP}[2]{\innerProd[\lebesgueSet{2}{\domain}]{#1}{#2}}

\begin{document}

\title{A Model for Bleb Formation and Interaction with Linker Proteins}

\author{%
    Philipp Werner
    \thanks{%
        Department Mathematik, Universit\"{a}t Erlangen-N\"{u}rnberg, Cauerstra\ss{}e 11,
        91058 Erlangen, Germany\newline
        \texttt{\{philipp.werner,martin.burger\}@fau.de}
    }
    \quad
    Martin Burger \footnotemark[2]
    \quad
    Jan-Frederik Pietschmann
    \thanks{%
        Fakult\"{a}t f\"{u}r Mathematik, Technische Universit\"{a}t Chemnitz, 
        Reichenhainer Stra\ss{}e 41, Chemnitz, Germany\newline
        \texttt{jfpietschmann@math.tu-chemnitz.de}
    }
}

\maketitle

\begin{abstract}
{%
	The aim of this paper is to further develop mathematical models for bleb formation in 
    cells, including cell-membrane interactions with linker proteins. This leads to 
    nonlinear reaction-diffusion equations on a surface coupled to fluid dynamics in 
    the bulk. We provide a detailed mathematical analysis and investigate some singular 
    limits of the model, connecting it to previous literature. Moreover, we provide 
    numerical simulations in different scenarios, confirming that the model can reproduce 
    experimental results on bleb initation.
}
{%
	Cell blebbing, Surface PDEs, Fluid-Structure Interaction, Free Boundary Problems.
}
\end{abstract}
	\section{Introduction}
    	Bleb formation or ``blebbing''
        is a biological process during which the cell membrane
        of an eucaryotic cell is disconnected from
        the cell cortex. The inflowing cytosol pushes out
        the free membrane part which builds a protrusion
        called a bleb. Due to regeneration mechanisms of the cell,
        the membrane connection to the cortex is restored and
        the protrusion is healed after some time. Despite
        these phases seem to be well established,
        the particular cause for the transition from one phase
        to the other, or the initialisation of the whole process
        are still subject to debate.
        Blebbing
        has been related to many interesting biological processes
        such as mitosis \cite{Boss1955},
        cell spreading \cite{Breiter-Hahn+1990},
        and apoptosis \cite{Robertson+1978}. It has also been
        noticed as migration mechanism, especially in embryonic 
        \cite{Hofreiter1943}, \cite{Concha+1998} and cancer cells. 
        Nevertheless, according to \cite{Charras+2008},
        the lamellipodia migration mechanism had received 
        a lot more interest and the authors found it necessary 
        to emphasise the
        importance of deeper investigations into bleb formation.
        Following their promotion, effort has been made to derive bio-physical
        models and develop an understanding of this phenomenon by numerical means 
        focusing on
        fluid-membrane interaction
        \cite{Strychalski+2013}, \cite{Strychalski+2016}, \cite{Young+2010}
        on membrane dynamics including linker influence \cite{Lim+2012},
        or on linker kinetics \cite{Alert+2016}.
        There has also been effort in enhanced mechanical modelling 
        \cite{Wolley+2014} and on 
        developing models of a full bleb life cycle in three dimensions
        \cite{Manakova+2016}. 
        
        In this work we propose a model for bleb formation in three dimensions taking into account the following aspects:\\
        \indent$\bullet$ Elastic properties of the membrane described via Canham-Helfrich energies.\\
        \indent$\bullet$ Forces of the actin cortex exerted on the membrane via linker proteins.\\
        \indent$\bullet$ Activation and deactivation of linker proteins as well as possible movement of the proteins.\\
        \indent$\bullet$ Intracellular fluid-dynamics and the corresponding fluid-structure interaction. 

		A key issue of the cell blebbing phenomenon is the fact that there are actually 
        two free boundaries, the membrane and the cortex that strongly interact via linker
        proteins.
        We describe the membrane by a height relative to the cell cortex,
        which is subject to forces exerted by a surrounding fluid as well as proteins
        that connect the cell membrane with the cell cortex and act like springs
        (cf. \cite{Alert+2016}). We take into account that these proteins
        may disconnect by introducing a ripping rate function whose steepness is 
        controlled by a small parameter $ \rippingLimitParam $. 
        This way we rediscover the model of \cite{Lim+2012}
        by ignoring movement of the linker proteins and passing to the limit 
        $ \rippingLimitParam \searrow 0 $.
				
		Besides the detailed mathematical modelling we provide a detailed analysis of the
        model in the case of small deformations of the membrane relative to the cortex, 
        where a linearization of the mechanics applies. The key nonlinearities we focus on
        are hence due to the presence of the linker proteins. Besides well-posedness of 
        the time-dependent model we prove the existence of stationary solution and show 
        that some critical pressure (e.g. arising from cortex contraction) is necessary 
        and sufficient to form a bleb, which we define as a deformation above a critical 
        height of the membrane relative to the cortex at which linkers are ripping off. 
        As mentioned above we study singular limits and show that the models of 
        \cite{Alert+2016} and \cite{Lim+2012} arise as special cases respectively scaling 
        limits of our model. Moreover, we provide a numerical study of the bleb 
        initiation by a critical pressure.

        The paper is organised as follows: After having fixed some basic notational
        conventions in \autoref{sec:preliminaries}, we derive a fourth order
        evolution system of equations in 
        \autoref{sec:modelling} starting from first principles. The main result of
        \autoref{sec:kinetic solutions} is a global-in-time existence theorem
        for solutions of this system. The following \autoref{sec:stationary solutions}
        is devoted to proving existence of stationary solutions and studying
        stability of a particular subclass of stationary solutions by means
        of nonlinear semigroups. The analytical part of this paper ends
        with \autoref{sec:singular limits}, in which we pass to the
        limit $ \rippingLimitParam \searrow 0 $ in the parameter controlling
        the disconnection rate.
        Finally, we illustrate some properties of our model numerically
        by presenting simulations of different biological situations
        in \autoref{sec:example scenarios} and
        conclude in \autoref{sec:conclusion}.

    \section{Preliminaries}
    	\label{sec:preliminaries}
    	\input{Preliminaries.tex}

    \section{Modelling}
    	\label{sec:modelling}
    	\input{Modelling.tex}

    \section{Time-dependent solutions}
    	\label{sec:kinetic solutions}
    	\input{Time-dependent_solutions.tex}

    \section{Stationary solutions}
    	\label{sec:stationary solutions}
    	\input{Stationary_solutions.tex}

    \section{Singular Limits}
    	\label{sec:singular limits}
    	\input{Singular_Limit_cases.tex}

    \section{Numerical examples}    
    	\label{sec:example scenarios}
        \input{Numerical_study.tex}
    \section{Conclusion}
    	\label{sec:conclusion}
        We derived a PDE model by balancing 
        the bending, stretching and linker forces 
        coming from the variational derivative of
        an extended Helfrich energy functional with
        the stress at the interface between the cytosol and
        the extracellular fluid. Based on the restriction of only
        small membrane displacement normal to the cortex, we could 
        derive a gradient flow
        describing the membrane height normal to the cell cortex.
        Additionally, linker kinetics were incorporated with 
        reaction-diffusion equations where we also introduced the concept of
        inactive linkers to include the phenomenon of cortex disruption.
        To our knowledge, this effect cannot be modelled with any
        other model.

        For the resulting system, we established global-in-time existence 
        and uniqueness of weak solutions 
        by applying the Banach fixed point theorem.
        The stationary case can also be treated
        with a fixed point argument,
        which uses the Schauder fixed point theorem, to establish existence.
        However, we do not have results about uniqueness or at least
        classification of stationary solutions.

        The a priori estimates used in the stationary solution existence 
        proof could be exploited for 
        passing to the limit in the rescaling parameter of the disconnection rate,
        $ \rippingLimitParam \searrow 0 $. We observed that the model of
        \cite{Lim+2012} could be rediscovered this way.
        The existence of a singular limit in the time-dependent case remains an 
        interesting open question for future research.
				
		Finally, let us mention that so far the model is purely mechanistic and 
        ignores any interaction with external and cell-internal 	
		signalling. In particular, the interaction of bleb formation with polarization of
        protein distributions influencing the 
		mechanical properties will become important in order to 
        fully understand cell migration by blebbing.
			
    \section*{Acknowledgements}
    	The work of MB and PW has been supported through the German Science Foundation DFG
        via Research Training Group GRK 2339 IntComSin. MB further acknowledges support 
        by ERC via Grant EU FP 7 - ERC Consolidator Grant 615216 LifeInverse. 
        The authors thank Helmut Abels, Harald Garcke (Regensburg),
        Erez Raz (WWU M\"{u}nster), and Florian Sonner (FAU Erlangen) for useful 
        suggestions and links to literature. 
        
        \printbibliography
    
    \newpage
    \appendix
        \section{Dirichlet-to-Neumann operator of the stationary Stokes problem}
        	\label{app:DtN Stokes}
        	\input{Appendix_Stokes_Dirichlet_to_Neumann.tex}
        \section{Existence and Uniqueness of solutions for the height equation}
        	\label{sec:existence and uniqueness of solutions for the height equation}
        	\input{Appendix_Galerkin_approximation_height_equation.tex}
    	\section{Taylor approximation of the Dirichlet-to-Neumann 
          operator}
        	\label{app:sec:taylor approximation of the Dirichlet-to-Neumann operator}
        	\input{Appendix_Taylor_approximation_DtN_operator.tex}
    	\section{Differential geometry}
        	\input{Appendix_Differential_geometry.tex}
        \section{Derivatives}
        	\input{Appendix_Derivatives.tex}

\end{document}

%% file: Preliminaries.tex
\hide{%
The set of natural numbers is denoted by $ \nats $ and the set of
real numbers $ \reals $. A column vector with $ n \in \nats $ entries
$ a_i $, $ i \in \{1,\dots,n\} $, will be written as 
$ \left( a_i \right)_{i=1,\dots,n} $, row vectors as 
$ \left[ a_i \right]_{i=1,\dots,n} $, and matrices with $ m \in \nats $
rows and $ n $ columns having entries $ m_{ij} $, $ i \in \{1,\dots,m\} $,
$ j \in \{1,\dots,n\} $ as $ \left( m_{ij} \right)_{i=1,\dots,m;j=1,\dots,n} $.
The transposed matrix is given by $ \transposed{\left( \left( m_{ij} 
\right)_{i=1,\dots,m;j=1,\dots,n} \right)} = $
$ \left( m_{ji} \right)_{j=1,\dots,m;i=1,\dots,n} $.

\paragraph{Integrable functions}
    The Borel sets of a topological space $ \left( X, \mathcal{O} \right) $
    are $ \borelSets{X} $.
    The $ d $-dimensional Lebesgue measure, $ d \in \nats $, is 
    $ \lebesgueM{d} $, and the Hausdorff measure of
    Hausdorff dimension $ d-1 $ is $ \hausdorffM{d-1} $.
    Whenever a set
    $ \Omega \subseteq \reals^d $ with Hausdorff dimension $ k \in \nats $
    stands in the place of a measure space, it represents the measure space
    $ \left( \Omega, \Omega \cap \borelSets{\reals^d}, \hausdorffM{k} \right) $.
    Lebesgues sets $ \lebesgueSet{p}{\Omega} $, $ p \in \nats $, contain
    $ \reals $-valued functions. We set
    $ \innerProd[\lebesgueSet{2}{\Omega}]{u}{v} = \integral{\Omega}{}{uv}{\lebesgueM{d}} $.
    The sets of matrix valued functions $ \funSig{u,v}{\Omega}{\reals^{(d,m)}} $
    whose components are in $ \lebesgueSet{p}{\Omega} $,
    are written
    $ \bochnerLebesgueSet{p}{\Omega}{\reals^{(d,m)}} $ and in case of $ p = 2 $, they may
    be equipped with the scalar product
    $ \innerProd[\bochnerLebesgueSet{2}{\Omega}{\reals^{(d,m)}}]{u}{v} = 
    \sum_{i=1,\dots,d;j=1,\dots,m} 
    \innerProd[\lebesgueSet{2}{\Omega}]{u_{ij}}{v_{ij}} $.

\paragraph{Sobolev functions}
    The sets of two-integrable Sobolev functions
    on $ \Omega $ which have weak derivatives up to order $ k\in\nats $ are denoted
    $ \sobolevHSet{k}{\Omega} $ for the real-valued and $ \bochnerSobolevHSet{k}{\Omega}{X} $
    for the $ X $-valued functions. If $ k \in (0,\infty) $, they are
    interpolation spaces of fractional order $ k $;
    the set of distributions on the 
    test function space $ \sobolevHSet[0]{k}{\Omega} $
    is $ \sobolevHSet{-k}{\Omega} $.

    (Weak) derivatives in the $ i $th component of 
    $ u \in \sobolevHSet{1}{\Omega} $ are denoted $ \pDiff{i}{}{}u $.
    Let $ \alpha = \left(\alpha',\alpha_\ell\right) = 
    \left(\alpha_1,\dots,\alpha_\ell \right) \in \nats^\ell $, $ \ell \leq d $,
    with $ \abs{\alpha} = \sum_{i=1}^\ell \alpha_i \leq k $,
    and $ u \in \sobolevHSet{k}{\Omega} $; 
    then $ \pDiff{\alpha}{}{} u = \pDiff{\alpha'}{}{} \pDiff{\alpha_\ell}{}{} u $.
    The gradient of a scalar-valued function is denoted by 
    \begin{equation*}
        \grad{}{}{}u = 
        \left(\pDiff{i}{}{}u\right)_{i=1,\dots,d}; 
    \end{equation*}
    this notation
    is adopted for the Jacobian matrix of 
    $ v \in \bochnerSobolevHSet{1}{\reals^d}{\reals^n} $, $ d\in\nats $:
    \begin{equation*}
        \grad{}{}{}v = \left( \pDiff{j}{}{}v_i \right)_{i=1,\dots,d; j=1,\dots,n}.
    \end{equation*}
    The divergence of matrix-valued functions $ v \in \sobolevHSet{1}{\reals^{\ell,d}} $
    is 
    \begin{equation*}
        \diver{}{}v = \left( \sum_{j=1}^d \pDiff{j}{}{}v_{ij} \right)_{i=1,\dots,\ell},
    \end{equation*}
    and the Laplacian
    $ \laplacian{}{}{}v = \diver[]{}{}\grad{}{}{}v $ for $ v \in \sobolevHSet{2}{\Omega} $.

    The closure of $ \diffSet[c]{\infty}{\Omega} $ 
    in the norm $ \norm[\sobolevHSet{k}{\Omega}]{u} = 
    \sqrt{\sum_{\alpha\in\nats^d, \abs{\alpha}\leq k } 
    \norm[\lebesgueSet{2}{\Omega}]{\pDiff{\alpha}{}{}u}^2 } $ is denoted by
    $ \sobolevHSet[0]{1}{\Omega} $.

    The set of all bounded linear operators from a Banach space
    $ X $ into a Banach space $ Y $ is $ \linearBoundedOps{X}{Y} $.
}
\paragraph{General}
By $ \identity[X] $ we denote the identity on the set $X $. We denote
the minimum of two values $ x $, $ y $ by $ x \wedge y $.
\revision[third]{Let $ \funSig{F}{X}{Y} $ for Banach spaces $ X $ and $ Y $; the
Gateaux derivative of $ F $ in direction $ v \in X $ at $ x \in X $, if it exists, 
written as $ \diff[x]{v}{}{F} $.} \sndRevision[second]{The space of Radon measures or regular,
countably additive measures on a measurable space $ \Omega \subseteq \reals^n $ 
that are absolutely continuous with respect to the Lebesgue measure is
$ \radonMeasures{\Omega} $.}

\paragraph{Differential operators and vectorial Sobolev spaces}
Gradients $ \grad{}{}{} u $ (in the weak and strong sense and independent of
whether they are on open sets or manifolds) of scalar functions $ u $ are column vectors.
The Jacobian $ \grad{}{}{} v $ of a vector-valued function $ v $ 
is the matrix whose lines are the transposed gradients of the components of $ v $.
We write $ \symmGrad{v} = 
\left( \grad{}{}{}v + 
\transposed{\left( \grad{}{}{}v \right)} \right) $
for the symmetrised gradient.
The set of $ k $-times, $ k \geq 0 $, weak differentiable, $ X $-valued functions,
where $ X $ is a vector space, on
some open set $ \Omega \subseteq \reals^3 $ is
$ \bochnerSobolevHSet{k}{\Omega}{X} $.
We will also employ the subspaces of mean-value-free functions:
\begin{equation*}
    \sobolevHSetMVF{k}{\Omega} = 
    \set{u\in\sobolevHSet{k}{\Omega}}{\integral{\Omega}{}{u}{x} = 0},
\end{equation*}
of functions with time-constant mean value:
\begin{equation*}
	\bochnerDiffSetCMV{}{[0,\finTime]}{X} 
    = 
    \set{
      	u\in\bochnerDiffSet{}{[0,\finTime]}{X}
    }{
      	\avgIntegral{\Omega}{}{u(t,x)}{x} 
        =
        \avgIntegral{\Omega}{}{u(0,x)}{x}
        \;
        \text{for\;a.\,e.\;}t\in[0,\finTime]
    },
\end{equation*}
$ \finTime \in [0,\infty) $,
of solenoidals functions:
\begin{equation*}
  	\sobolevSolenoidals{k}{\Omega}
    =
    \closure{
    \set{
      	u \in \diffSet{\infty}{\Omega}
    }{
      	\diver{}{}{}u = 0
    }
    }^{\sobolevHSet{k}{\Omega}},
\end{equation*}
and the traces of solenoidal functions:
\begin{equation*}
  	\sobolevTracesOfSolenoidals{k}{\mathcal{M}}
    =
    \set{
      	u \in \sobolevHSet{k}{\mathcal{M}}
    }{
      	\integral{\mathcal{M}}{}{
        	u\cdot\normal[\mathcal{M}]{}
        }{\hausdorffM{2}}
        =
        0
    },
\end{equation*}
where $ \mathcal{M} \subseteq \reals^3 $ is some two-dimensional manifold with outer
unit normal field $ \normal[\mathcal{M}]{} $ and $ \hausdorffM{2} $ the 
Hausdorff measure with Hausdorff dimension two.

\paragraph{Shape derivatives}
For denoting the shape derivative of a functional $ \mathcal{W} $
in directions $ \theta, \vartheta $, we following the
notation in \cite{Delfour+2011} using
$ \shapeDeriv{\theta}{}{\revision[second]{\mathcal{W}}} $ for the first and
$ \shapeDeriv{\theta,\vartheta}{2}{\revision[second]{\mathcal{W}}} $ for the second
shape derivative.
We denote by
\begin{equation*}
    \perturbed{M}{\perturbParam}{\perturbField} 
    =
    \set{ x + \sndRevision[second]{\perturbParam} \perturbField(x) }{ x \in M },
\end{equation*}
for a hypersurface $ M $, $ \perturbParam \in I \subseteq \reals $, $ I $ being an interval,
and $ \funSig{\perturbField}{M}{\reals^3} $,
the perturbation of $ M $ by $ \perturbField $.
\revision[third]{
In this particular case, the shape derivative of the functional $ \mathcal{W} $ being
defined on $ \perturbed{M}{\perturbParam}{\perturbField} $ 
is $ \shapeDeriv{\perturbField}{}{\mathcal{W}} =
\diff[\perturbParam=0]{\perturbParam}{}{
	\mathcal{W}\left(
      	\perturbed{M}{\perturbParam}{\perturbField}
    \right)
}
$}.

\paragraph{Differential geometry}
    The notation of this paragraph follows \cite{Amann+2006} and \cite{Barrett+2019}.
    With $ \diffeo{q}{X}{Y} $ we denote the set of all $ C^q $-diffeomorphisms mapping
    the Banach space $ X $ into the Banach space $ Y $.
    The set of functions 
    \begin{equation*}
        \set{ 
            \funSig{\varphi}{J\times X}{Y} 
        }{ 
            \forall s \in J \colon \varphi(s,\cdot) \in \diffeo{q}{X}{Y} 
            \wedge
            \forall x \in X \colon \varphi(\cdot,x) \in \bochnerDiffSet{k}{J}{Y}
        },
    \end{equation*}
    where $ J $ is a real interval, is denoted by $ \diffeo{k,q}{J \times X}{Y} $.
    We call an $ n $-dimensional
    manifold $ \Gamma \subseteq \reals^{n+k} $ a real submanifold.
    The tangent space of a real submanifold at a point $ p \in \Gamma $ is denoted by
    $ \tangentSet{p}{\Gamma} \subseteq \reals^{n+k} $.

    Let 
    $ \left( \Gamma_s \right)_{s \in J} $ be a family of
    real submanifolds in $ \reals^n $ with a mapping 
    $ \funSig{\liftingParametr{}}{J \times \Gamma}{\bigcup_{s\in J} \Gamma_s} $
    such that $ \liftingParametr[s]{} $ is a global parametrisation of $ \Gamma_s $
    on a reference manifold $ \Gamma $. We call the set $ \spaceParamDomain{\Gamma}{J}
    = \bigcup_{s\in J} \{ s \} \times \Gamma_s $ an \emph{evolving manifold}.
    If $ \Gamma_s $ are hypersurfaces, we use the term \emph{evolving hypersurface}.
    For functions
    $ \funSig{f^s}{\Gamma_s}{N} $, where $ N $ is a set, we define 
    the function 
    \begin{equation*}
        \fun{f}{\spaceParamDomain{\Gamma}{J}}{N}{\left(s, x\right)}{f^s(x)}.
    \end{equation*}

    An evolving manifold $ \spaceParamDomain{\Gamma}{J} $ is \emph{smooth} if 
    $ \tangentSet{p}{\spaceParamDomain{\Gamma}{J}} \neq \{ 0 \} \times \reals^n $
    for all $ p \in \spaceParamDomain{\Gamma}{J} $.
    For such a smooth evolving manifold, we denote the velocities
    $ \pDiff[r]{s}{}{\liftingParametr[s]{}} \concat 
    \left( \liftingParametr[r]{} \right)^{-1} $ by $ \paramVel[r] $.
    If the manifolds are orientable, i.\,e., there exist
    smooth outer unit normal fields $ \funSig{\normal[s]{}}{\Gamma_s}{\reals^3} $,
    the velocities can be
    decomposed in their normal $ \normalParamVel[s] = \coeffNormalParamVel[s] 
    \normal[s]{} $, $ \coeffNormalParamVel[s] = \normalParamVel[s] \cdot \normal[s]{} $,
    and tangential $ \tangParamVel[r] $ components. 
    We define a differential operator
    \begin{equation*}
        \matDiff[r]{s}{
            f
        }
        =
        \pDiff[r]{s}{}{
            f
            \concat
            (s,\theta)
            \mapsto
            \left(s, \liftingParametr[s]{\theta}\right)
        }
        \concat
        \left(
            \liftingParametr[r]{}
        \right)^{-1}
    \end{equation*}
    called the \emph{material derivative (of $ f $ with respect to $ \liftingParametr{} 
    $)}.---%
    The parametrisation mapping shall always be given by the context if not stated
    \sndRevision[first]{explicitly}.

    For any real submanifold $ \Gamma \subseteq \reals^n $, 
    the \emph{tangential gradient} $ \grad{\Gamma}{}{} $ of a 
    function $ \funSig{f}{\Gamma}{\reals} $ at $ p \in \Gamma $ is given by
    \begin{equation*}
        \grad[p]{\Gamma}{}{f} = \orthProj[\tangentSet{p}{\Gamma}]{\grad[p]{}{}{\bar{f}}},
    \end{equation*}
    where $ \bar{f} $ is any differentiable extension of $ f $ to an open 
    neighbourhood of $ p $ in $ \reals^n $ and
    $ \orthProj[\tangentSet{p}{\Gamma}]{} $ is the orthogonal projection
    onto the tangent space $ \tangentSet{p}{\Gamma} $ of $ \Gamma $ at $ p $.
    The projection matrix $ \identityMatrix - 
    \normal[\Gamma]{p}\tensorProd\normal[\Gamma]{p} $ and the projection
    $ \orthProj[\tangentSet{p}{\Gamma}]{} $ are identified.
    This way, we also acquire the \emph{partial tangential derivatives} 
    \begin{equation*}
        \tDiff[p]{\Gamma}{i}{}{f} = 
        \left(\grad[p]{\Gamma}{}{f}\right)_i, \quad i \in \{1,\dots,n\}.
    \end{equation*}

    For a differentiable, orientable real submanifold 
    $ \Gamma \subseteq \reals^n $ with normal field $ \normal{} $, we denote the
    \emph{Weingarten map} by $ \fun{\weingartenMapping[\Gamma]{}}{\Gamma}{
    \reals^{(n,n)}}{p}{
    \grad[p]{\Gamma}{}{\normal{}}} $.
    Then, the \emph{mean curvature of $ \Gamma $} is 
    $ \meanCurv[\Gamma] = \trace\left( \weingartenMapping[\Gamma]{} \right) = 
    \diver{\Gamma}{} \normal{} $ and the \emph{Gaussian curvature} is
    $ \gaussianCurv[\Gamma] = \det\left( \weingartenMapping[\Gamma]{} \right) $.

\paragraph{Physical dimension and units}
The sets $ \reals^n $, $ n \in \nats $, are identified with the
product sets $ \reals^n \times \mathcal{D} \times \mathcal{U} $,
where $ \mathcal{D} $ is the set of all physical dimensions
$ \mathcal{D} = \{ \dimTime, \dimLength, \dimMass, 
\dimAmountOfSubst, \dots \} $ (meaning time, length, mass, etc.) and
$ \mathcal{U} $ the set of all physical units.
For $ x \in \reals^n $, we denote by $ \physDimOf{ x } $
its second component (called the physical dimension
of $ x $).
When we write $ x $, we always refer to the
first component.


%% file: Modelling.tex
As mentioned in the introduction, many details of the process of bleb formation 
are still subject to research and not fully understood.
Therefore, we aim at a rather abstract model following the
general description of the process in \cite{Charras+2008}: 
The main parts of an eucaryotic cell that
are involed in bleb formation are the cell membrane, which is 
basically a bilayer of lipid molecules, the cell cortex,
which is a network of actin fibres,
and elastic proteins which connect the cell cortex 
to the cell membrane. These linker proteins are only stretchable to
a certain length above which they disconnect from the membrane.
Inside the cell there is a fluid that
is called the cytosol and the cell is itself swimming in an extracellular
fluid. Caused by mechanisms which have not completely been understood yet, 
a certain patch of the cell cortex contracts and raises the pressure on 
the membrane locally. 
This way, the corresponding membrane patch is pushed so far away 
from the cortex that most of the linker proteins 
disconnect. The cytosol that pushes against the free membrane patch
now causes the formation 
of a protrusion which is called a bleb. Over time, the protein linkers
are reconnected to the cell membrane causing the membrane patch
to be fixed to the cortex again and the bleb to vanish.
\subsection{The fluid system}
    \newlength{\squareLen}
    \setlength{\squareLen}{8cm}
    \newlength{\cortexRad}
    \setlength{\cortexRad}{2cm}
    \newlength{\membraneRad}
    \pgfmathsetlength{\membraneRad}{\cortexRad+1cm}
    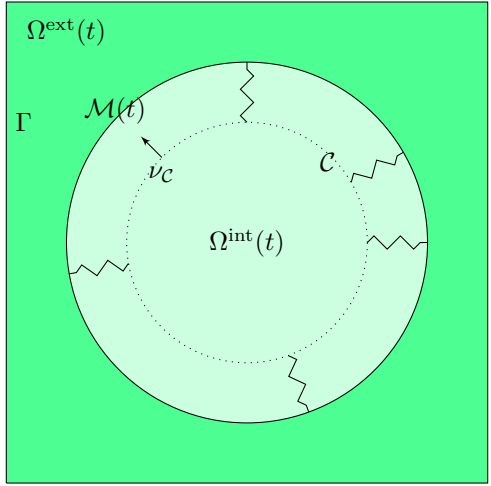
\begin{figure}
      	\centering
        \begin{tikzpicture}[tips,scale=0.8]
            \draw[fill=Green!70!white] (0,0) -- (\squareLen,0) -- 
                (\squareLen,\squareLen) -- (0,\squareLen) -- (0,0);

            \draw[fill=Green!20!white] (\squareLen/2,\squareLen/2) circle (\membraneRad);
            \pgfmathsetlengthmacro\ax{\squareLen/2 + cos(135)*(\membraneRad+0.1cm)}
            \pgfmathsetlengthmacro\ay{\squareLen/2 + sin(135)*(\membraneRad+0.1cm)}       
            \draw (\ax,\ay) node {$ \membrane(t) $};

            \draw[dotted] (\squareLen/2,\squareLen/2) circle (\cortexRad);
            \pgfmathsetlengthmacro\ax{\squareLen/2 + cos(45)*(\cortexRad-0.1cm)}
            \pgfmathsetlengthmacro\ay{\squareLen/2 + sin(45)*(\cortexRad-0.1cm)}       
            \draw (\ax,\ay) node {$ \cortex $};
            \pgfmathsetlengthmacro\ax{\squareLen/2 + cos(135)*(\cortexRad)}
            \pgfmathsetlengthmacro\ay{\squareLen/2 + sin(135)*(\cortexRad)}       
            \pgfmathsetlengthmacro\bx{\squareLen/2 + cos(135)*(\membraneRad-0.5cm)}
            \pgfmathsetlengthmacro\by{\squareLen/2 + sin(135)*(\membraneRad-0.5cm)}
            \draw[-latex'] (\ax,\ay) -- (\bx,\by);
            \node[anchor=north] at (\ax,\ay) {\revision[third]{$ \normal[\cortex]{} $}};

            \foreach \phi in {0,30,90,190,290}
                \pgfmathsetlengthmacro\ax{\squareLen/2 + cos(\phi)*\cortexRad}
                \pgfmathsetlengthmacro\ay{\squareLen/2 + sin(\phi)*\cortexRad}
                \pgfmathsetlengthmacro\bx{\squareLen/2 + cos(\phi)*\membraneRad}
                \pgfmathsetlengthmacro\by{\squareLen/2 + sin(\phi)*\membraneRad}
                \draw[decorate,decoration=zigzag]  (\ax,\ay) -- (\bx,\by);

            \draw ($(1cm,\squareLen-0.5cm)$) node {
                $ \stokesDomain[\exteriorLabel]\revision[three]{(t)} $
            };
            \draw ($(\squareLen/2,\squareLen/2)$) node {
              	$ \stokesDomain[\interiorLabel]\revision[three]{(t)} $
            };


            \draw (0,\squareLen*0.75) node[anchor=west] {
              	$ \stokesOuterBoundary $
            };
        \end{tikzpicture}
        \caption{Cell geometry}
        \label{fig:cell geometry}
    \end{figure}
    Let $ \generalDomain \subseteq \reals^3 $ be a bounded, connected, 
    and open set with
    sufficiently regular boundary. We require this set to be partitioned into
    the open connected set $ \stokesEulerDomain[\exteriorLabel] $, modelling a reference 
    region exterior 
    to the cell, the open connected set $ \stokesEulerDomain[\interiorLabel] $, 
    modelling a reference region interior to the cell, and
    the boundary $ \eulerMembrane $ of $ \stokesEulerDomain[\interiorLabel] $, which 
    shall be
    a two-dimensional orientable $ C^2 $-manifold,
    modelling the membrane in its
    initial state; its unit normal field is denoted by 
    $ \funSig{\normal[\eulerMembrane]{}}{\eulerMembrane}{\reals^3} $. 
    The region which is occupied by the cell at time $ t \in [0,\infty) $ 
    is given as the image $ \cell(t) = \cellEvolMap{t}{
    \stokesEulerDomain[\interiorLabel]} $ 
    of $ \cellEvolMap{}{} \in \diffeo{1,2}{[0,\infty) \times \generalDomain
    }{\generalDomain} $,
    where $ \cellEvolMap{0}{} = \identity[\generalDomain] $.
    Consequently, the exterior region is $ \stokesDomain[\exteriorLabel](t)
    = \cellEvolMap{t}{ \stokesEulerDomain[\exteriorLabel] } $ and
    the whole domain is denoted by
    $ \stokesDomain(t) = \stokesDomain[\exteriorLabel](t) \cup
    \stokesDomain[\interiorLabel](t) $ and $ \membrane(t) =
    \cellEvolMap{t}{ \eulerMembrane } $.
    Furthermore, the cell cortex 
    is denoted by $ \cortex \subseteq \stokesEulerDomain[\interiorLabel] $ 
    \revision[second]{and is modelled as a sphere which is fixed in time.}
    \autoref{fig:cell geometry} illustrates
    a typical geometry compatible with the previous description.

    Both the fluid in the inner region, representing the cytosol, and in the outer region
    wtih pressures
    $ \funSig{\pressure[i]}{\stokesDomain[i]_\finTime}{\reals} $ and
    velocities $ \funSig{\velocity[i]}{\stokesDomain[i]_\finTime}{\reals^3} $, $ i \in \{
    \interiorLabel, \exteriorLabel \} $,
    are described by incompressible stationary Stokes equations in Eulerian coordinates
    \begin{subequations}
        \label{equ:stokes equation}
        \begin{equation}
            \viscosity[i] \laplacian{}{}{}\velocity[i] 
            + 
            \grad{}{}{} \pressure[i]  = 0
        \end{equation}
        \begin{equation}
            \diver{}{} \velocity[i] = 0,
        \end{equation}
    \end{subequations}
    on $ \stokesDomain[i]_\finTime = \bigcup_{t\in[0,\finTime]} \{ t \} \times 
    \stokesDomain[i](t) $ with final time $ \finTime > 0 $
    (time dependency will be brought into the system by boundary conditions).
    \revision[third]{This is justified by small length scales as in \cite{Strychalski+2013}.}
    For any functions $ \funSig{ \labeledObject{g}{i} }{
    \stokesDomain[i]_\finTime }{X} $, $ i \in \{ \interiorLabel, \exteriorLabel \} $,
    into a vector space $ X $, we associate the function
    \begin{equation*}
        g(t,x) = 
        \begin{cases}
          	\labeledObject{g}{\exteriorLabel}(t,x) & x \in \stokesDomain[\exteriorLabel](t) \\
            \labeledObject{g}{\interiorLabel}(t,x) & x \in \stokesDomain[\interiorLabel](t).
        \end{cases}
    \end{equation*}
    We further pose a homogeneous Dirichlet boundary condition at the exterior boundary
    \begin{equation*}
        \sndRevision[second]{\stokesOuterBoundary}
        = 
        \boundary{\left( \stokesDomain(t) \cup \membrane(t) \right)} 
        = 
        \boundary{ \generalDomain },
    \end{equation*}
    which does not change over time,
    \begin{equation*}
        \shrinkFunc{
            \traceOp{\velocity[\exteriorLabel]}
        }{\sndRevision[second]{\stokesOuterBoundary}} = 0,
    \end{equation*}
    where $ \traceOp{} $ is the trace operator. \revision{For every $ t \in [0,\infty) $,
    let $ \stokesNeumannData_t \in \lebesgueTracesOfSolenoidals{2}{\membrane(t)} $,
    $ \physDimOf{\stokesNeumannData_{t,i}} = \frac{\dimMass \dimLength}{\dimTime^2} $, 
    $ i \in \{1,2,3\} $, be a force
    we will specify below.}
    The stress at the interface $ \membrane(t) $ is subject to a
    Neumann-type boundary condition: 
    \begin{equation}
      	\label{equ:stokes:Neumann-type boundary}
        \jump{ \cauchyStressTensor{t}{} }
        \normal[\membrane(t)]{} = 
        \revision{\stokesNeumannData_t},
    \end{equation}
    where $ \cauchyStressTensor{}{} = \viscosity\symmGrad{\velocity}
    - \pressure \identityMatrix $ is the Cauchy stress tensor of the interior and exterior
    fluid and we define
    $ \jump{g}(t,\cdot) = \shrinkFunc{ \traceOp{ 
    \shrinkFunc{g}{\stokesDomain[\exteriorLabel]_\finTime } (t,\cdot) } }{ \membrane(t) }
    - \traceOp{ \shrinkFunc{g}{\stokesDomain[\interiorLabel]_\finTime} (t,\cdot) } $ 
    for
    any function $ \funSig{g}{\stokesDomain[\exteriorLabel]_\finTime \cup 
    \stokesDomain[\interiorLabel]_\finTime}{X} $.
    To assure well-posedness of the problem, we further require
    \begin{equation*}
        \jump{ \velocity } = 0 \revision{\quad \text{on}\;\membrane(t)}.
    \end{equation*}

    Taking an energetic point of view,
    we consider a variational formulation of \revision[second]{the Stokes equations 
    \eqref{equ:stokes equation} with boundary conditions
    \eqref{equ:stokes:Neumann-type boundary} (the no jump condition is encoded in the
    solution space)}:
    \begin{myproblem}
        \label{problem:stokes:variational:static domains}
        Find
        $ 
        \velocity \in \bochnerSobolevHSet{1}{[0,\finTime]}{ 
          	\bochnerSobolevHSet{-1}{\generalDomain}{\reals^3}
        }
        \cap
        \bochnerLebesgueSet{2}{[0,\finTime]}{ 
          	\bochnerSobolevHSet[0]{1}{\generalDomain}{\reals^3}
        }        
        $
        and 
        $ 
        \pressure \in \bochnerLebesgueSet{2}{[0,\finTime]}{
          	\lebesgueSet{2}{\generalDomain}
        }
        $ 
        such that
        \begin{subequations}
            \begin{equation}
                \frac{\viscosity}{2}
                \innerProd[{\bochnerLebesgueSet{2}{\generalDomain}{\reals^{(3,3)}}}]{
                  	\symmGrad{\velocity}
                }{
                    \symmGrad{\testFuncVelocity}
                }
                - \innerProd[{\lebesgueSet{2}{\generalDomain}}]{
                    \pressure
                }{
                    \diver{}{} \testFuncVelocity
                }
                =
                \innerProd[\bochnerLebesgueSet{2}{\membrane(t)}{\reals^3}]{
                  	\stokesNeumannData_t
                }{
                    \testFuncVelocity
                }
            \end{equation}
            \begin{equation}
                \innerProd[\lebesgueSet{2}{\generalDomain}]{
                    \diver{}{} \velocity
                }{
                    \testFuncPressure
                } = 0
            \end{equation}
          \end{subequations}
        for all $ \testFuncVelocity \in 
        \bochnerSobolevHSet[0]{1}{\generalDomain}{\reals^3} $
        and $ \testFuncPressure \in \lebesgueSet{2}{\generalDomain} $.
    \end{myproblem}
    \begin{remark}
        \revision[third]{
        Another formulation of the Stokes equations
        includes the term 
        $ 
            \innerProd[\bochnerLebesgueSet{2}{\generalDomain}{\reals^{(3,3)}}]{
                \grad{}{}{}\velocity
            }{
                \grad{}{}{}\testFuncVelocity
            } 
        $
        instead of 
        $ 
            \frac{1}{2}
            \innerProd[\bochnerLebesgueSet{2}{\generalDomain}{\reals^{(3,3)}}]{
                \symmGrad{\velocity}
            }{
                \symmGrad{\testFuncVelocity}
            } 
        $. We observe that, because of $ \velocity $ being solenoidal,
        \begin{equation*}
                \innerProd[\bochnerLebesgueSet{2}{\generalDomain}{\reals^{(3,3)}}]{
                    \grad{}{}{}\velocity
                }{
                    \transposed{\left(\grad{}{}{}\testFuncVelocity\right)}
                } 
                =
                \innerProd[\bochnerLebesgueSet{2}{\generalDomain}{\reals^{(3,3)}}]{
                    \transposed{\left(\grad{}{}{}\velocity\right)}
                }{
                    \grad{}{}{}\testFuncVelocity
                } 
                =
                -
                \innerProd[\bochnerLebesgueSet{2}{\generalDomain}{\reals^3}]{
                    \diver{}{}{}
                    \transposed{\left(\grad{}{}{}\velocity\right)}
                }{
                    \testFuncVelocity
                } 
                = 0,
        \end{equation*}
        so both expressions are equal. The motivation to use the latter
        is related to the structure of the employed Neumann boundary conditions.
        }
    \end{remark}

    We introduce a function 
    $ \funSig{\membraneHeight}{
    \revision[second]{[0,\finTime] \times \cortex}}{\reals} $, \revision{$
    \physDimOf{\membraneHeight} = \dimLength $},
    which is intended to give the membrane's height relative to the cortex
    in normal direction $ \normal[\cortex]{} $, i.\,e.,
    \begin{equation}
        \label{equ:membrane height property}
        \membrane(t) = \set{ x + \membraneHeight(t,x)\normal[\cortex]{x} }{ 
        x \in \revision[second]{\cortex} }.
    \end{equation}
    \comment{%
        In order to guarantee this property, we could also first require 
        $ \eulerMembrane $
        to lie in a sufficiently small tubular neighbourhood 
        $ U_\eps\left(\eulerCortex\right) = 
        \set{ x + \eps \normal[\eulerCortex]{x} }{ x \in \eulerCortex } $, $ \eps > 0 $,
        of $ \eulerCortex $, i.\,e., $ \eulerMembrane \subseteq T_\eps $,
        such that
        the orthogonal projection 
        \begin{equation*}
            \fun{\orthProj[\eulerCortex]{}}{
                U_\eps\left(\eulerCortex\right)
            }{\eulerCortex}{y}{
                \argmin_{x\in\eulerCortex} \norm[2]{y-x} 
            }
        \end{equation*}
        exists and for all $ y_0 \in \eulerMembrane $, we may define 
        $ \membraneHeight\left(0, x_0\right) $ 
        \begin{equation*}
            y_0 = x_0 + \membraneHeight\left( 0, x_0 \right) \normal[\eulerCortex]{x_0},
        \end{equation*}
        where $ x_0 = \orthProj[\eulerCortex]{}y_0 $.
        Additionally, it shall hold
        \begin{equation}
            \label{equ:membrane velocity condition}
            \pDiff{t}{}{} \cellEvolMap{t}{y_0} = \pDiff{t}{}{}\cortexEvolMap{t}{x_0}
            + \diff{t}{}{\membraneHeight\left( t, \cortexEvolMap{t}{x_0} \right)
            \normal[\cortex_t]{\cortexEvolMap{t}{x_0}}}.
        \end{equation}
        Then, by integration, we obtain \eqref{equ:membrane height property}.
        This requires $ t \mapsto \membraneHeight\left(t,\cortexEvolMap{t}{x_0}\right)
        \normal[\cortex_t]{\cortexEvolMap{t}{x_0}} $ to be differentiable for all
        $ x_0 \in \eulerCortex $; therefore, the cortex deformation $ \cortexEvolMap{}{} $ 
        has to be appropriate. Throughout this work, we restrict to the 
        case of $ \cortexEvolMap{t}{} = \identity[\eulerCortex] $, where \eqref{equ:membrane %
        velocity condition} becomes the much simpler condition
        \begin{equation}
            \label{equ:membrane velocity condition simple}
            \pDiff{t}{}{}\cellEvolMap{t}{y_0} = \pDiff{t}{}{}\membraneHeight\left(t,x_0\right)
            \normal[\eulerCortex]{x_0}.
        \end{equation}
        As \autoref{problem:stokes:variational:static domains} is already well-posed, this is not
        an additional boundary condition for this problem, but a condition for 
        $ \membraneHeight $. This will become much clearer when we connect
        the fluid and the membrane-cortex models.
    }
    We further require 
    $ \membrane(t) $ 
    lying in a sufficiently small tubular neighbourhood of $ \revision[second]{\cortex} $
    (this approach is analogous to \cite{Elliott+2017}):
    \begin{mycondition}
      	\label{cond:small height}
        Let $ \revision[third]{U_{\perturbParam}}\left(\cortex\right) = 
        \set{ 
          	x + \perturbParam \normal[\cortex]{x} 
        }{ x \in \cortex } $, $ \perturbParam > 0 $.
        Then $ \membrane(t) \subseteq \revision[third]{U_{\perturbParam}}\left(\cortex\right) $,
        such that the orthogonal projection 
        \begin{equation*}
            \fun{
              	\orthProj[\cortex]{}
             }{
               	\revision[third]{U_{\perturbParam}}\left(\cortex\right)
             }{
               	\cortex(t)
             }{y}{
     			\argmin_{x\in\cortex} \norm[2]{y-x} 
             }
        \end{equation*}
        exists.
    \end{mycondition}
    \comment{%
        Note that the
        spring force is not to be included into the functional $ \mathcal{F} $ as
        for height $ \perturbParam \membraneHeight $ it has order $ \perturbParam^2 $!
    }
    Yet there is no guarantee that the mapping in \eqref{equ:membrane height property} 
    is bijective between $ \cortex $ and $ \membrane(t) $ 
    as multiple points in $ \membrane(t) $ may have 
    the same projection point rendering the existence of a function $ \membraneHeight $ 
    impossible. Therefore, we also pose an invertibility condition for the
    parametrisation of the membrane over the cortex:
    \begin{mycondition}
      	\label{cond:normal velocity}
        It shall hold,
        \revision[second]{
        \begin{equation}
            \label{equ:membrane velocity condition:simplified}
            \pDiff{t}{}{} \cellEvolMap{t}{y}  =
            \pDiff{t}{}{} \membraneHeight(t,x) \normal[\cortex]{x}
        \end{equation}}
        for $ y \in \eulerMembrane $, $ x = \orthProj[\cortex]{y} $.
    \end{mycondition}
    
    \paragraph{Potential energy}
    In view of \autoref{cond:small height}, we define a rescaled height
    $ \membraneHeight = \perturbParam \rescaled{\membraneHeight} $
    and we may regard $ \membrane(t) $ as small perturbation of $ \cortex $ by 
    $ \rescaled{\membraneHeight} \normal[\cortex]{} $ with order of magnitude 
    $ \perturbParam $. 
    In particular, we write 
    $ \membrane(t) = \perturbed{\cortex}{\perturbParam}{
    \rescaled{\membraneHeight}(t,\cdot) \normal[\cortex]{}} $.

    In order to derive a mathematical model, let us turn to the bio-physical 
    properties 
    of the membrane-cortex system just described: 
    The membrane shall consist of lipid molecules arranged in two layers.
    On the cortex, there are proteins connected to the membrane, therefore
    called linkers proteins. They
    are considered stretchable and shall obey Hooke's law for springs, so
    we can assign the potential energy density functional
    \begin{equation*}
        \springFunc[t, \membraneHeight] = 
        \frac{1}{2} \linkersSpringConst \activeLinkers
        \norm[\lebesgueSet{2}{\cortex}]{\revision{\membraneHeight}(t,\cdot)}^2,
    \end{equation*}
    where $ \linkersSpringConst \in
    \bochnerLebesgueSet{\infty}{\cortex}{[0,\infty)} $,
    \revision{$ \physDimOf{\linkersSpringConst} = \dimMass \dimTime^{-2} $,}
    is a \revision[third]{function playing the role of a spring constant in
    every spatial point} and
    $ \activeLinkers \in \lebesgueSet{2}{\cortex} $\revision{, $ \physDimOf{\activeLinkers}
    = \dimLength^{-2} $,} is the density of active linkers
    (further \revision{explanation} below, see \autoref{sec:linker kinetics}).
    There are several models (cf. \cite{Seifert1997}) for the 
    surface energy of membranes. A widely used example is the Helfrich energy
    (cf. \cite{Helfrich1973}, \cite{Zhong-can+1989}): 
    \begin{equation*}
        \membraneSurfEnergyFunc[\membrane] 
        = 
        \integral{\membrane}{}{
            \frac{\membrStiffn}{2}
            \left( \meanCurv[\membrane] + \spontMeanCurv[\membrane] \right)^2 
            + \membrStiffn_G \gaussianCurv[\membrane] 
        }{\hausdorffM{2}},
    \end{equation*}
    where $ \funSig{\meanCurv[\membrane]}{\membrane}{\reals} $, 
    $ \spontMeanCurv[\membrane] \in \reals $,
    and $ \funSig{\gaussianCurv[\membrane]}{\membrane}{\reals} $
    are the mean curvature, spontaneous mean curvature, and Gaussian curvature
    of $ \membrane $, 
    respectively, with bending rigidity $ \membrStiffn $ and Gaussian bending 
    rigidity $ \membrStiffn_G $.

    The second order expansion of the total energy density of the 
    membrane-cortex system at time 
    $ t \in [0,\infty) $ is
    \begin{align*}
        \springFunc[t, \perturbParam \rescaled{\membraneHeight}]
        +
        \membraneSurfEnergyFunc[ 
            \perturbed{\cortex}{\perturbParam}{ 
                \rescaled{\membraneHeight} \normal[\cortex]{} 
            } 
        ]
        &=
        \springFunc[t, 0]
        +
        \membraneSurfEnergyFunc\left( \eulerMembrane \right)
        +
        \perturbParam
        \left(
            \diff[0]{ \perturbParam }{}{ 
                \springFunc[ t, \perturbParam \rescaled{\membraneHeight} ] 
            } 
            \rescaled{\membraneHeight}
            +
            \shapeDeriv{\rescaled{\membraneHeight}\normal[\cortex]{}}{}{
                \membraneSurfEnergyFunc
            }
        \right)
        \\
        &+
        \perturbParam^2
        \left(
            \diff[0]{ \perturbParam }{2}{
                \springFunc[ t, \perturbParam \rescaled{\membraneHeight} ]
            } 
            (\rescaled{\membraneHeight}, \rescaled{\membraneHeight})
            +
            \shapeDeriv{
                \rescaled{\membraneHeight}
                \normal[\cortex]{},
                \rescaled{\membraneHeight}
                \normal[\cortex]{}
            }{2}{
                \membraneSurfEnergyFunc
            } 
         \right)
        +
        \landauSmallO(\perturbParam^3)
        \\
        &=
        \membraneSurfEnergyFunc\left( \eulerMembrane \right)
        +
        \perturbParam \shapeDeriv{
            \rescaled{\membraneHeight} \normal[\cortex]{}
        }{}{
            \membraneSurfEnergyFunc
        }
        +
        \perturbParam^2 
        \left(
            \frac{1}{2} \linkersSpringConst \activeLinkers
            \norm[\lebesgueSet{2}{\cortex}]{\rescaled{\membraneHeight}}^2
            +
            \shapeDeriv{
                \rescaled{\membraneHeight} \normal[\cortex]{},
                \rescaled{\membraneHeight} \normal[\cortex]{}
            }{2}{
                \membraneSurfEnergyFunc
            }                     
        \right)\\
        &+ 
        \landauSmallO(\perturbParam^3),
    \end{align*}
    where we have used that the derivatives of 
    $ \membraneSurfEnergyFunc[ 
        \perturbed{\cortex}{\perturbParam}{
            \rescaled{\membraneHeight}(t,\cdot) \normal[\cortex]{}
        } 
    ] $ 
    with respect to $ \perturbParam $ at $ 0 $ are equal to the shape 
    derivatives
    of $ \membraneSurfEnergyFunc $ in direction
    $ \hat{\membraneHeight}(t,\cdot) \normal[\cortex]{} $
    at zero.
    If the cortex was an equilibrium shape of Helfrich's energy, we would have
    $ \shapeDeriv{\rescaled{ \membraneHeight } \normal[\cortex]{} }{}{
    \membraneSurfEnergyFunc} = 0 $.
    But this may be especially not true if the cortex is contracted due to myosin 
    motor activity. 
    Therefore, we model
    $ \shapeDeriv{\rescaled{ \membraneHeight } \normal[\cortex]{}}{}{
    \membraneSurfEnergyFunc} = 
    \innerProd[\lebesgueSet{2}{\cortex}]{\perturbParam \pressure_0}{
    \rescaled{ \membraneHeight }} $ \revision[third]{for
    $ \pressure_0 \in \bochnerLebesgueSet{\infty}{[0,\finTime]}{\lebesgueSet{\infty}{\cortex}} $}
    interpreting $ \pressure_0 $ as a stress
    that is exerted on the membrane due to the cortex contraction
    and transmitted by the fluid.
    The energy functional up to second order therefore is
    \begin{equation*}
        \membraneEnergyFunc{t, \rescaled\membraneHeight}{
            \perturbParam, \cortex, \activeLinkers,
            \membrStiffn, \membrStiffn_G, \linkersSpringConst
        }
        =
        \membraneSurfEnergyFunc\left( \eulerMembrane \right)
        +
        \perturbParam^2
        \left(
        	\innerProd[\lebesgueSet{2}{\cortex}]{ \pressure_0 }{ \rescaled\membraneHeight }
        	+
            \frac{1}{2} \activeLinkers \linkersSpringConst 
            \norm[\lebesgueSet{2}{\cortex}]{\rescaled\membraneHeight}^2
            +
            \shapeDeriv{
                \rescaled\membraneHeight \normal[\cortex]{},
                \rescaled\membraneHeight \normal[\cortex]{}
            }{2}{
                \membraneSurfEnergyFunc
            }
        \right).
    \end{equation*}
    \begin{remark}
      	\revision[second]{%
          	Setting 
            $ \shapeDeriv{\rescaled{ \membraneHeight } \normal[\cortex]{}}{}{
            \membraneSurfEnergyFunc} = 
            \innerProd[\lebesgueSet{2}{\cortex}]{\perturbParam \pressure_0}{
            \rescaled{ \membraneHeight }} $
            introduces a mechanism by which an initially flat
            membrane in a resting fluid may be deformed after all:
            Not considering $ \shapeDeriv{\rescaled{ \membraneHeight } \normal[\cortex]{}}{}{
            \membraneSurfEnergyFunc} $ as a parameter, but instead taking the 
            terms that come out of a computation of this shape derivative
            would only change the coefficients of the $ \laplacian{\cortex}{}{}\membraneHeight $
            and $ \membraneHeight $ terms in the variational principle we will derive below.
            The resulting equation is homogeneous and therefore does not show any deforming
            behaviour in case the membrane is initially flat and the fluid velocity zero.
            The more physical but also rather complex
            approach for introducing this mechanism would be to 
            relate the pressure $ \pressure_0 $ to shape deformations
            of the cortex and then describe the influence of $ \pressure_0 $
            on the fluid introducing another surface-bulk coupling this way.
        }
    \end{remark}

\subsection{Connecting the fluid and the membrane model}
	\label{sec:connecting the fluid and the membrane model}
    The fluid system is not closed, but subject to external forces 
    \revision{$ \stokesNeumannData_t $}. This is exactly 
    where the membrane-cortex system comes into
    play: The potential energy of this system is considered to be the source of
    forces acting on the fluid and therefore being transformed into kinetic energy
    of the fluid. We also take a damping effect due to friction
    between the fluid particles and the cortex into account \revision{with a linear
    friction model with friction constant $ \membraneDampingConst $ having 
    dimension
    $ \physDimOf{\membraneDampingConst} = \dimMass \dimTime^{-1} \dimLength^{-2} $.}
    In order to enforce \autoref{cond:normal velocity}, these forces
    are all directed normally to the cortex, so
    the tangential part of \revision{$ \stokesNeumannData_t $} is zero:
    \revision{
    \begin{equation}
        \label{equ:energy balance membrane-cortex and fluid}
        \begin{split}
            \stokesNeumannData_t
            =
            -\diff[0]{\rescaled\membraneHeight}{}{ 
                \membraneEnergyFunc{t,\rescaled\membraneHeight}{
                    \perturbParam, \cortex, \membrStiffn, \membrStiffn_G,
                    \linkersSpringConst
                }
            } \varphi 
            -
            \perturbParam^2
            \innerProd[\bochnerLebesgueSet{2}{\cortex}{\reals^3}]{
                \membraneDampingConst \pDiff{t}{}{} \rescaled\membraneHeight
                \normal[\cortex]{}
            }{
                \varphi \normal[\cortex]{}
            }
        \end{split}
    \end{equation}
    }
    Recalling \eqref{equ:membrane velocity condition:simplified}, we have
    the fluid particles at the membrane moving in the direction of the cortex normal. 
    Taking the length of the velocity vector to be the
    change of the membrane's height in time, we have specified the
    Dirichlet boundary of $ \velocity $ at $ \membrane(t) $, and we therefore
    may express
    \begin{equation*}
        \jump{ \cauchyStressTensor[\rescaled]{t}{} }\normal[\rescaled{\membrane}(t)]{} 
         = 
         \stokesDirichletToNeumannOp[t]{
            \liftedFun{
                \pDiff{t}{}{ \perturbParam \rescaled\membraneHeight(t,\cdot) }
                \normal[\cortex]{} 
            }{\mathfrak{X}}
        },
    \end{equation*}
    where $ \liftedFun{\varphi}{\mathfrak{X}} = \varphi \concat 
    \mathfrak{X}(t,\cdot)^{-1} $
    for $ \mathfrak{X}(t,\cdot) \colonequals
    \identity[\cortex] + 
    \membraneHeight(t,\cdot) \normal[\cortex]{} $ and any function 
    $ \varphi $ with domain $ \cortex $
    and 
    \begin{equation*}
        \funSig{\stokesDirichletToNeumannOp[t]{}}{
        \sobolevTracesOfSolenoidals{1}{\rescaled{\membrane}(t)}}{
        \lebesgueTracesOfSolenoidals{2}{\rescaled{\membrane}(t)}}
    \end{equation*}
    is the Dirichlet-to-Neumann operator of the Stokes problem 
    \autoref{problem:stokes:variational:static domains}.
    \revision[third]{A definition and references to important properties of
    this operator is given in Appendix~\ref{app:DtN Stokes}.}
    By combination of both descriptions of the Neumann data, we
    are going to derive a PDE model:

\subsection{PDE description of the height function}
	\label{sec:PDE description of the height function}
    \paragraph{Approximation of the Dirichlet-to-Neumann operator}
        For sufficiently regular Stokes flow velocity, we can make use of
        the small height condition \autoref{cond:small height} to
        approximate the time-dependent Dirichlet-to-Neumann operator
        (see \autoref{app:sec:taylor approximation of the Dirichlet-to-Neumann operator})
        with its stationary version on $ \cortex $.
        This way, we arrive at a gradient-flow structure 
        \begin{equation}
            \label{equ:PDE description of height function:gradient flow}
            \begin{split}
            -\diff[0]{h}{}{ 
                \membraneEnergyFunc{t,\rescaled\membraneHeight}{
                    \perturbParam, \cortex, \activeLinkers, \membrStiffn, \membrStiffn_G,
                    \linkersSpringConst
                }
            } \varphi 
            &= 
            \perturbParam^2
            \innerProd[\bochnerLebesgueSet{2}{\cortex}{\reals^3}]{
                \membraneDampingConst \pDiff{t}{}{}\rescaled\membraneHeight(t,\cdot)
                \normal[\cortex]{}
            }{
                \varphi \normal[\cortex]{}
            }
            \\
            &+
            \perturbParam^2
            \innerProd[
                \bochnerLebesgueSet{2}{\cortex}{\reals^3}
            ]{
                \stokesDirichletToNeumannOp[0]{
                    \pDiff{t}{}{}\rescaled\membraneHeight
                    \normal[\cortex]{}
                }
            }{
                \varphi \normal[\cortex]{}
            }
            + \landauSmallO[ \perturbParam^3 ]
            \\
            &=
            \perturbParam^2
            \innerProd[
                \bochnerLebesgueSet{2}{\cortex}{\reals^3}
            ]{
                \left( 
                    \membraneDampingConst 
                    \identity[\bochnerSobolevHSet{1}{\cortex}{\reals^3}]
                    +
                    \stokesDirichletToNeumannOp[0]{}
                \right)
                \left(
                    \pDiff{t}{}{}\rescaled\membraneHeight
                    \normal[\cortex]{}
                \right)
            }{
                \varphi \normal[\cortex]{}
            } + \landauSmallO[ \perturbParam^3 ]
            \\
            &=
            \perturbParam^2
            \dualProd[
                \bochnerSobolevHSet{-\frac{1}{2}}{\cortex}{\reals^3}
            ]{
                \timeDerivOperator
                \left(
                    \pDiff{t}{}{}\rescaled\membraneHeight
                    \normal[\cortex]{}
                \right)
            }{
                \varphi \normal[\cortex]{}
            } + \landauSmallO[ \perturbParam^3 ]            
            \end{split}
        \end{equation} 
        with $ \timeDerivOperator = \membraneDampingConst 
        \identity[\bochnerSobolevHSet{1}{\cortex}{\reals^3}]
        +
        \stokesDirichletToNeumannOp[0]{} $ and $ \varphi \in \sobolevHSet{2}{\cortex} $.

    \paragraph{Calculating the variation of the potential energy}
    To derive a full PDE description for $ \membraneHeight $, we have
    to calculate the variation of the potential energy functional.
    So first, we calculate the first and second shape derivatives
    of $ \membraneSurfEnergyFunc $. 
    Recall,
    \begin{equation*}
        \membraneSurfEnergyFunc[
            \perturbed{ \cortex }{ \perturbParam }{ 
                \hat{\membraneHeight} \normal[\cortex]{}
             }
        ] = 
        \frac{\membrStiffn}{2}
        \integral{
            \perturbed{ \cortex }{ \perturbParam }{ 
                \hat{\membraneHeight} \normal[\cortex]{}
             }
        }{}{
            \meanCurv[\perturbParam]^2
        }{\hausdorffM{2}}
        +
        \membrStiffn
        \integral{
            \perturbed{ \cortex }{ \perturbParam }{ 
                \hat{\membraneHeight} \normal[\cortex]{}
             }
         }{}{
            \meanCurv[\perturbParam] \spontMeanCurv
        }{\hausdorffM{2}}
        +
        \integral{
            \perturbed{ \cortex }{ \perturbParam }{ 
                \hat{\membraneHeight} \normal[\cortex]{}
             }
         }{}{
            \frac{\membrStiffn}{2}
            \spontMeanCurv[2]%
            + 
            \membrStiffn_G \gaussianCurv[\perturbParam]
        }{\hausdorffM{2}},
    \end{equation*}
    where $ \meanCurv[\perturbParam] $ is the mean curvature of 
    $ \perturbed{ \cortex }{ \perturbParam }{ 
    \hat{\membraneHeight} \normal[\cortex]{} } $

    Observe that the integral over the Gaussian curvature
    is constant in $ \perturbParam $ 
    due to the Gauss-Bonnet theorem 
    ($ \perturbed{ \cortex }{ \perturbParam }{ 
    \hat{\membraneHeight} \normal[\eulerCortex]{} } $
    is homeomorphic to $ \cortex $) and therefore vanishes when differentiated 
    in $ \perturbParam $.
    The neccessary caclulations have been carried out before
    for the Willmore energy, e.\,g. in 
    \cite[p.\,7]{Elliott+2017}:
    \begin{equation*}
        \begin{split}
            \revision{\frac{1}{2}}
            \shapeDeriv{\hat{\membraneHeight}\normal[\cortex]{}}{2}{
                \integral{
                    \cortex
                }{}{
                    \meanCurv^2
                }{\hausdorffM{2}}
            }
            &=
            \integral{\cortex}{}{
                \left(
                  \laplacian{\cortex}{}{}\rescaled\membraneHeight
                  +
                  \abs{ \weingartenMapping{} }^2
                  \rescaled\membraneHeight
                \right)^2
                +
                2 \meanCurv
                \weingartenMapping{}
                \frobScalarProd
                2 \rescaled\membraneHeight
                \grad{\cortex}{2}{}\rescaled\membraneHeight
                \\
                &+
                2 \meanCurv
                \grad{\cortex}{}{}\rescaled\membraneHeight
                \cdot
                \grad{\cortex}{}{}\rescaled\membraneHeight
                +
                \meanCurv
                \membraneHeight
                \grad{\cortex}{}{}\rescaled\membraneHeight
                \cdot
                \grad{\cortex}{}{}\meanCurv
                \\
                &-
                \meanCurv^2
                \grad{\cortex}{}{}\rescaled\membraneHeight
                \cdot
                \grad{\cortex}{}{}\rescaled\membraneHeight
                -
                \frac{5}{2}
                \meanCurv[2]
                \rescaled\membraneHeight
                \laplacian{\cortex}{}{}\rescaled\membraneHeight
                \\
                &+
                {\rescaled\membraneHeight}^2 \left(
                  2 \meanCurv
                  \trace\left( 
                    \weingartenMapping{}^3
                  \right)
                  -
                  \frac{5}{2}
                  \meanCurv[2]
                  \abs{ \weingartenMapping{} }^2
                  +
                  \frac{1}{2}
                  \meanCurv[4]
                \right)
            }{\hausdorffM{2}}.
        \end{split}
    \end{equation*}
    The additional calculations use the same techniques 
    and the
    interested reader may consult the appendix (\autoref{corollary:second %
    derivative int mean curv normal vel})
    and
    \autoref{corollary:second derivative surface area normal vel}) for details:
    \begin{equation*}
        \begin{split}
            \shapeDeriv{\hat{\membraneHeight}\normal[\eulerCortex]{}}{2}{
                \integral{
                    \cortex
                }{}{
                    \meanCurv
                }{\hausdorffM{2}}
            } =
                \int_{\cortex}
                2 \rescaled\membraneHeight
                \trace\left(
                    \weingartenMapping{}
                    \grad{\cortex}{2}{} \rescaled\membraneHeight
                    +
                    \rescaled\membraneHeight
                    \weingartenMapping{}^3
                \right)
                -
                \rescaled\membraneHeight \meanCurv 
                \laplacian{\cortex}{}{} \rescaled \membraneHeight
                - 3 {\rescaled\membraneHeight}^2 \meanCurv \abs{ \weingartenMapping{} }^2
                +
                \meanCurv
                \grad{\cortex}{}{} \rescaled\membraneHeight
                \cdot
                \grad{\cortex}{}{} \rescaled\membraneHeight
                +
                {\rescaled\membraneHeight}^2 \meanCurv^3
                \;
                \text{d}\,\hausdorffM{2}
        \end{split}
    \end{equation*}
    and 
    \begin{equation*}
        \begin{split}
            \shapeDeriv{ \hat{\membraneHeight} \normal[\eulerCortex]{} }{2}{
                \integral{
                    \cortex
                }{}{
                    1
                }{\hausdorffM{2}}
            }
            =
            \integral{\cortex}{}{
                \grad{\cortex}{}{} \rescaled\membraneHeight \cdot \grad{\cortex}{}{} 
                \rescaled\membraneHeight
                -
                h^2 \abs{ \weingartenMapping{} }^2
                + h^2 \meanCurv^2
            }{\hausdorffM{2}}.
        \end{split}
    \end{equation*}

    \paragraph{Spherical cortex shape}
    Significant simplification of these terms is achieved by 
    considering $ \cortex $ to be a sphere with radius $ \cortexRadius $:
    \begin{equation*}
        \revision{\frac{1}{2}}
        \shapeDeriv{ \hat{\membraneHeight} \normal[\cortex]{} }{2}{
            \integral{\cortex}{}{
                \meanCurv^2
            }{\hausdorffM{2}}
        } 
        =
        \integral{\cortex}{}{
            \left(
                \laplacian{\cortex}{}{} \membraneHeight
            \right)^2
            \revision{
            -
            \frac{2}{\cortexRadius^2}}
            \grad{\cortex}{}{}\membraneHeight
            \cdot
            \grad{\cortex}{}{}\membraneHeight
        }{\hausdorffM{2}}
    \end{equation*}
    (cf. \cite{Elliott+2017}) and
    \begin{equation*}
        \shapeDeriv{ \hat{\membraneHeight} \normal[\eulerCortex]{} }{2}{
            \integral{\cortex}{}{
                \meanCurv
            }{\hausdorffM{2}}
        } 
        =
        \integral{\cortex}{}{
            \frac{2}{\cortexRadius}
            \grad{\cortex}{}{}\membraneHeight
            \cdot
            \grad{\cortex}{}{}\membraneHeight
        }{\hausdorffM{2}}
    \end{equation*}
    (cf. \autoref{corollary:second derivative int mean curv sphere normal vel})
    and
    \begin{equation*}
        \begin{split}
            \shapeDeriv{ \hat{\membraneHeight} \normal[\cortex]{} }{2}{
                \integral{\cortex}{}{
                    1
                }{\hausdorffM{2}}
            }
            =
            \integral{\cortex}{}{
                \grad{\cortex}{}{} h \cdot \grad{\cortex}{}{} h
                + \frac{\revision[third]{2}}{\cortexRadius^2} h^2
            }{\hausdorffM{2}}.
        \end{split}
    \end{equation*}
\paragraph{Force density equation}
	\label{par:force density equation}
    All together, we arrive at the following expression for 
    \eqref{equ:PDE description of height function:gradient flow}:
    \begin{equation}
        \label{equ:PDE description of height function:gradient flow:full}
        \begin{split}
            -
            \innerProd[ \lebesgueSet{2}{\cortex} ]{
                \activeLinkers \linkersSpringConst
                \membraneHeight
            }{
                \varphi
            }
            \revision[third]{
            -\innerProd[\lebesgueSet{2}{\cortex}]{
              	\pressure_0
            }{
              	\varphi
            }
            }
            -
            \heightLinkerSysHeightBF{\membraneHeight}{\varphi}
            = 
            \dualProd[
                \bochnerSobolevHSet{-\frac{1}{2}}{\cortex}{\reals^3}
            ]{
                \timeDerivOperator
                \left(
                    \pDiff{t}{}{}\membraneHeight
                    \normal[\cortex]{}
                \right)
            }{
                \varphi \normal[\cortex]{}
            },
        \end{split}
    \end{equation} 
    where 
    \begin{equation*}
      	\heightLinkerSysHeightBF{\membraneHeight}{\varphi}
        =
        \membrStiffn
        \innerProd[ \lebesgueSet{2}{\cortex} ]{
            \laplacian{\cortex}{}{} \membraneHeight
        }{
            \laplacian{\cortex}{}{} \varphi
        }
        +
        \effectiveLengthParam
        \innerProd[ \bochnerLebesgueSet{2}{\cortex}{\reals^3} ]{
            \grad{\cortex}{}{} \membraneHeight
        }{
            \grad{\cortex}{}{} \varphi
        }
        +
        \anotherMembrParam
        \innerProd[ \lebesgueSet{2}{\cortex} ]{
            \membraneHeight
        }{
            \varphi
        }        
    \end{equation*}
    with $ \effectiveLengthParam = \frac{\membrStiffn}{2}\left(
        -\frac{4}{\cortexRadius^2} + 
        \frac{2}{\cortexRadius} \spontMeanCurv + \frac{1}{2} \spontMeanCurv^2
    \right) $ and $ \anotherMembrParam = \membrStiffn
    \frac{\revision[third]{2}}{\cortexRadius^2} \spontMeanCurv^2 $.

    \begin{remark}
      	\label{remark:Poincare inequality}
        \revision{%
        The form $  \heightLinkerSysHeightBF{\cdot}{\cdot} $
        may not be coercive on $ \sobolevHSet{2}{\cortex} $
        nor may it be non-negative. In order to assure at least
        non-negativity, we make the following considerations:
        }

        \revision{%
        In case $ \spontMeanCurv \geq 0 $, we follow \cite{Elliott+2017}
        and derive a Poincar\'{e}-type 
        inequality from Courant's min-max principle
        \begin{equation*}
          	\integral{\cortex}{}{\membraneHeight^2}{\hausdorffM{2}}
            \leq
            \frac{\cortexRadius^2}{2}
            \integral{\cortex}{}{\abs{\grad{\cortex}{}{}\membraneHeight}^2}{
            \hausdorffM{2}}
            \leq
            \frac{\cortexRadius^4}{4}
            \integral{\cortex}{}{\left( \laplacian{\cortex}{}{}\membraneHeight \right)^2}{
            \hausdorffM{2}},
        \end{equation*}
        where $ \frac{2}{\cortexRadius^2} $ is the second eigenvalue of the Laplace-Beltrami
        operator,
        on $ \text{span}\,\{ 1 \}^\bot $, i.\,e. for all functions with mean zero.
        As 
        $ \sobolevHSet{2}{\cortex} = \text{span}\{1\} \oplus \text{span}\{1\}^\bot $,
        for every $ u \in \sobolevHSet{2}{\cortex} $ there is a constant $ m $ ($u$'s mean value) 
        and $ u_0 $ (being mean-value-free) such that $ u = m + u_0 $.
        We observe 
        \begin{align*}
	        a(u,u) &= a(m,m + u_0) + a(u_0,m + u_0) = a(m,m) + 2a(m,u_0) + a(u_0,u_0) 
            \\
            &\geq \anotherMembrParam \norm[\lebesgueSet{2}{\cortex}]{m}^2
            +
            2 \anotherMembrParam \innerProd[\lebesgueSet{2}{\cortex}]{m}{u_0}
            +
            \anotherMembrParam \norm[\lebesgueSet{2}{\cortex}]{u_0}^2
            \\
            &\geq 0.
        \end{align*}
        }

        \revision{%
        In case $ \spontMeanCurv < 0 $, we need a compatibility condition on
        $ \spontMeanCurv $ and $ \cortexRadius $. We require
        $ \frac{2}{\cortexRadius} \spontMeanCurv + \frac{1}{2} \spontMeanCurv^2 \geq 0 $.
        This leads to
        $ \frac{2}{\cortexRadius} + \frac{1}{2} \spontMeanCurv \leq 0 $, and further
        $ \spontMeanCurv \leq - \frac{4}{\cortexRadius} $.
        }
        
        \sndRevision[second]{%
        For coercivity, we shall therefore require $ \spontMeanCurv \in (0,\infty) \cup 
        (-\infty,-\frac{4}{\cortexRadius}) $.%
        }
    \end{remark}


\subsection{Protein linkers}
	\label{sec:linker kinetics}
    The quantity $ \activeLinkers $ has been mentioned before in modelling
    the potential energy of the membrane-cortex system.
    It models the density of linkers that are connected to the membrane.
    We also take linkers into account that are disconnected and whose density
    is denoted $ \inactiveLinkers $.
    Both active and inactive linkers are considered to be mobile species diffusing on the
    cortex. Moreover, they are transformed into each other
    as result of overstretching above a critical height 
    $ \criticalHeight \in \bochnerDiffSet{}{\cortex}{[0,\infty)} $, 
    which causes active linkers to disconnect, or
    regeneration mechanisms connecting inactive linkers to the membrane again.
    A reaction-diffusion-kind of system may be used to model these processes:
    \begin{subequations}
      	\label{equ:PDE description of linker kinetics}
        \begin{equation}
            \label{equ:PDE description of active linkers}
            \pDiff{t}{}{} \activeLinkers 
            - \activeLinkersDiffusiv \laplacian{\cortex}{}{} \activeLinkers
            =
            \repairRate \inactiveLinkers
            -
            \rippingInterpol
            \left(
                \frac{\membraneHeight - \criticalHeight}{\rippingLimitParam}
            \right)
            \activeLinkers
        \end{equation}
        \begin{equation}
          \label{equ:PDE description of inactive linkers}
          \pDiff{t}{}{} \inactiveLinkers
          - \inactiveLinkersDiffusiv \laplacian{\cortex}{}{} \inactiveLinkers
          =
          - \repairRate \inactiveLinkers
          +
          \rippingInterpol
          \left(
                \frac{\membraneHeight - \criticalHeight}{\rippingLimitParam}
          \right)
          \activeLinkers,
        \end{equation}
    \end{subequations}
    where $ \activeLinkersDiffusiv, \inactiveLinkersDiffusiv \in [0,\infty) $ are
    the active and inactive linker diffusivities,
    $ \repairRate \in [0,\infty) $ a regeneration rate, and
    $ \funSig{\rippingInterpol}{\reals}{[0,\infty)} $ a disconnection rate
    being Lipschitz continuous and
    $ \shrinkFunc{\rippingInterpol}{(-\infty,0)} = 0 $ (a typical example
    is the non-negative part). The disconnection of linkers from the membrane
    is considered to be a fast process. To account for this, a (small) parameter
    $ \rippingLimitParam \in (0,\infty) $ is used for rescaling
    the argument of $ \rippingInterpol $. 
    (In \autoref{sec:singular limits},
    we analyse the solution's behaviour when $ \rippingLimitParam \searrow 0 $.)
    For the sake of readability,
    we may use the abbreviation 
    \begin{equation*}
      	\rippingInterpol_{\rippingLimitParam}\left(
        \membraneHeight \right) 
        =
        \rippingInterpol
        \left(
              \frac{\membraneHeight - \criticalHeight}{\rippingLimitParam}
        \right)
    \end{equation*}
    in the following.

    \begin{remark}
        Approaching the linker movement by a reaction-diffusion model 
        is motivated by the work of \cite{Alert+2016}. 
        They consider the following 
        equation for the membrane height $ \membraneHeight $ 
        and linker density $ \activeLinkers $
        (we adopt the notation of this work for their parameters
        and quanitites): 
        \begin{subequations}
            \label{equ:system of Alert and Casademunt}
            \begin{equation}
                \membraneDampingConst \pDiff{t}{}{} \membraneHeight 
                = 
                \pressure_0 
                - 
                \linkersSpringConst \membraneHeight \activeLinkers
            \end{equation}
            \begin{equation}
                \pDiff{t}{}{} \activeLinkers
                = 
                \repairRate
                \left( \rho_0 - \activeLinkers \right)
                -
                k_{\mathrm{off}}\left( \membraneHeight \right)
                \activeLinkers
            \end{equation}
        \end{subequations}
        with a maximal linker density $ \rho_0 $ and
        a disconnection rate $ k_{\mathrm{off}} $.
        But there is an important new aspect to the model
        presented here:
        The concept of inactive linkers is not present in \eqref{equ:system of %
        Alert and Casademunt}, but 
        a gauge protein density $ \rho_0 $ is assumed of which a part is
        connected $ \rho_a $ and $ \rho_0 - \rho_a $ is disconnected.
        As consequence of this condition, their approach is limited to scenarios
        where the cortex is intact. Nevertheless, it has also been observed
        \cite{Charras+2008} that bleb formation may be triggered
        by cortex disruption (leading to a hole in the cortex). 
        This case is contained in our active-inactive linker
        setting with
        $ \activeLinkers(0,x) \equiv \inactiveLinkers(0,x) = 0 $
        for $ x \in D $, where $ D $ is the area of the hole in
        the cortex (cf. \autoref{sec:Scenarios}).

        Indeed, \eqref{equ:system of Alert and Casademunt}
        is a specialisation of our model: Set 
        $ \activeLinkersDiffusiv = \inactiveLinkersDiffusiv = \eta $
        and $ \activeLinkers(0,\cdot) +
        \inactiveLinkers(0,\cdot) \equiv \rho_0 $. Adding 
        \eqref{equ:PDE description of active linkers}
        and
        \eqref{equ:PDE description of inactive linkers},
        we get
        \begin{equation*}
            \pDiff{t}{}{\activeLinkers + \inactiveLinkers}
            +
            \eta 
            \laplacian{\cortex}{}{\activeLinkers + \inactiveLinkers}
            =
            0,
        \end{equation*}
        which is solved by $ \activeLinkers + \inactiveLinkers \equiv \rho_0 $.
        This way, we can express $ \inactiveLinkers = \rho_0 - \activeLinkers $
        giving \eqref{equ:system of Alert and Casademunt}.
    \end{remark}

After having derived a PDE model for the blebbing phenomenon, we will
deal analytically with the following issues in the next sections:
\begin{enumerate}
  	\item global-in-time existence of weak solutions
    (\autoref{sec:kinetic solutions}),
    \item existence of stationary solutions and their stability
    (\autoref{sec:stationary solutions}),
    \item convergence of stationary solutions to a singular limit
      	when $ \rippingLimitParam \searrow 0 $,
    \item and rediscovering the model for bleb formation proposed in \cite{Lim+2012}
        (\autoref{sec:singular limits}).
\end{enumerate}


%% file: Time-dependent_solutions.tex
\renewcommand{\domain}{\cortex}
\newcommand{\poincConst}{\Pi}
\newcommand{\posDefConst}{\Xi}
\newcommand{\embeddingConst}{E}
\newcommand{\solSpaceHeight}{X}
\newcommand{\solSpaceActiveLinkers}{Y}
\newcommand{\solSpaceInactiveLinkers}{\solSpaceActiveLinkers}
\newcommand{\powerOfTwoConst}{Z}
\newcounter{constCnt}
\newcommand{\fixedPointOp}[1]{\OPERATOR{F}{#1}}
\newcommand{\membraneHeightMV}{\hslash}
\newcommand{\membraneHeightInitMV}{\hslash_0}

In the following three chapters, we analyse 
a variatonal formulation of
\eqref{equ:PDE description of height function:gradient flow:full},
\eqref{equ:PDE description of active linkers}, and
\eqref{equ:PDE description of inactive linkers}
having the following
strong equivalent for sufficiently regular $ \membraneHeight $, 
$ \activeLinkers $, $ \inactiveLinkers $:
\begin{subequations}
  	\label{equ:height-linker system:strong}
  	\begin{equation}
      	\timeDerivOperator\left( \pDiff{t}{}{}\membraneHeight \normal[\cortex]{} \right)
        \cdot \normal[\cortex]{}
        +
        \membrStiffn 
        \laplacian{\cortex}{2}{} \membraneHeight
        -
        \effectiveLengthParam
        \laplacian{\cortex}{}{} \membraneHeight
        +
        \anotherMembrParam
        \membraneHeight
        =
        -\linkersSpringConst \activeLinkers \membraneHeight
        +
        \pressure_0
    \end{equation}
    \begin{equation}
      	\label{equ:height-linker system:strong:active linkers}
        \pDiff{t}{}{} \activeLinkers 
        - \activeLinkersDiffusiv \laplacian{\cortex}{}{} \activeLinkers
        =
        \repairRate \inactiveLinkers
        -
        \rippingInterpol_\rippingLimitParam\left( \membraneHeight \right)
        \activeLinkers
    \end{equation}
    \begin{equation}
      	\label{equ:height-linker system:strong:inactive linkers}
         \pDiff{t}{}{} \inactiveLinkers
         - \inactiveLinkersDiffusiv \laplacian{\cortex}{}{} \inactiveLinkers
         =
         - \repairRate \inactiveLinkers
         +
         \rippingInterpol_\rippingLimitParam\left( \membraneHeight \right)
         \activeLinkers,
    \end{equation}    
\end{subequations}
where we write $ \rippingInterpol_\rippingLimitParam\left( \membraneHeight \right) = 
\rippingInterpol
\left(
   \frac{\membraneHeight - \criticalHeight}{\rippingLimitParam}
\right) $. \revision[third]{Let us summarise the properties of the parameters involved:
\begin{itemize}
	\item $ \membrStiffn $ and $ \anotherMembrParam $ are non-negative constants, whereas
	$ \effectiveLengthParam $ is also a constant but not necessarily non-negative. 
    \item The operator
    $ \timeDerivOperator $ in front of the time-derivative is the sum of the identity
    and the Dirichlet-to-Neumann operator of the Stokes problem. 
    \item The function $ \linkersSpringConst $
    is in $ \bochnerLebesgueSet{\infty}{[0,\finTime]}{\lebesgueSet{\infty}{\cortex}} $ as well
    as the pressure $ \pressure_0 $. Also, $ \linkersSpringConst $ is assumed to be non-negative a.\,e.
    \item The repairing rate $ \repairRate $ is a non-negative constant.
    \item \sndRevision[second]{The diffusivities $ \activeLinkersDiffusiv $, 
        $ \inactiveLinkersDiffusiv $
        are taken to be positive.}
    \item The disconnection rate $ \rippingInterpol $
    is assumed to be non-negative and Lipschitz; the corresponding steepness parameter $ \rippingLimitParam $
    shall be non-negative as well.
    \item The critical height $ \criticalHeight $ is a non-negative 
    $ \bochnerLebesgueSet{\infty}{[0,\finTime]}{\lebesgueSet{\infty}{\cortex}} $ function.
\end{itemize}
}

For better readability, we introduce the following forms:
$
    \heightLinkerSysActiveLinkersBF{\activeLinkers}{\testFuncActiveLinkers}
    =
    \activeLinkersDiffusiv
    \innerProd[ \bochnerLebesgueSet{2}{\cortex}{\reals^3} ]{
        \grad{\cortex}{}{} \activeLinkers
    }{
        \grad{\cortex}{}{} \testFuncActiveLinkers
    }
$
and
$
    \heightLinkerSysInactiveLinkersBF{\inactiveLinkers}{\testFuncInactiveLinkers}
    =
    \inactiveLinkersDiffusiv
    \innerProd[ \bochnerLebesgueSet{2}{\cortex}{\reals^3} ]{
        \grad{\cortex}{}{} \inactiveLinkers
    }{
        \grad{\cortex}{}{} \testFuncInactiveLinkers
    }.
$ 

\begin{myproblem}
    \label{problem:height-linker system:variational}
    Find
    \begin{equation*}
        \begin{split}
            \membraneHeight \in 
            \bochnerLebesgueSetCMV{2}{[0,\finTime]}{
            	\sobolevHSet{2}{\cortex}
            } 
            \cap
            \bochnerSobolevHSet{1}{[0,\finTime]}{
              	\sobolevHSetMVF{1}{\cortex}
            },
        \end{split}
    \end{equation*}
    \begin{equation*}
      	\begin{split}
            \activeLinkers,\inactiveLinkers \in \bochnerLebesgueSet{2}{[0,\finTime]}{
            \sobolevHSet{1}{\cortex}} \cap 
            \bochnerSobolevHSet{1}{[0,\finTime]}{
              \sobolevHSet{-1}{\cortex}} 
        \end{split}
    \end{equation*}
    such that
    \begin{subequations}
        \begin{equation}
            \label{equ:height-linker system:variational:height}
            \begin{split}
                \innerProd[\lebesgueSet{2}{\cortex}]{
                    \timeDerivOperator
                    \left(
                        \pDiff{t}{}{}\membraneHeight
                        \normal[\cortex]{}
                    \right)
                }{
                    \varphi \normal[\cortex]{}
                }
                +
                \heightLinkerSysHeightBF{\membraneHeight}{\testFuncMembraneHeight}
                =
                -
                \innerProd[ \lebesgueSet{2}{\cortex} ]{
                    \activeLinkers \linkersSpringConst
                    \membraneHeight
                }{
                    \varphi
                }
                +
                \innerProd[ \lebesgueSet{2}{\cortex} ]{
                    \pressure_0
                }{
                    \varphi
                }
            \end{split}
        \end{equation}
        \begin{equation}
            \label{equ:height-linker system:variational:active linkers}
            \dualProd[ \sobolevHSet{-1}{\cortex} ]{
                \pDiff{t}{}{} \activeLinkers
            }{
                \testFuncActiveLinkers
            }
            +
            \heightLinkerSysActiveLinkersBF{\activeLinkers}{\testFuncActiveLinkers}
            =
            \repairRate
            \innerProd[ \lebesgueSet{2}{\cortex} ]{
                \inactiveLinkers
            }{
                \testFuncActiveLinkers
            }
            -
            \innerProd[ \lebesgueSet{2}{\cortex} ]{
                \rippingInterpol_\rippingLimitParam
                \left(
                    \membraneHeight
                \right)
                \activeLinkers
            }{
                \testFuncActiveLinkers
            }
        \end{equation}
        \begin{equation}
            \label{equ:height-linker system:variational:inactive linkers}
            \dualProd[ \sobolevHSet{-1}{\cortex} ]{
                \pDiff{t}{}{} \inactiveLinkers
            }{
                \testFuncInactiveLinkers
            }
            +
            \heightLinkerSysInactiveLinkersBF{\inactiveLinkers}{\testFuncInactiveLinkers}
            =
            -
            \repairRate
            \innerProd[ \lebesgueSet{2}{\cortex} ]{
                \inactiveLinkers
            }{
                \testFuncInactiveLinkers
            }
            +
            \innerProd[ \lebesgueSet{2}{\cortex} ]{
                \rippingInterpol_\rippingLimitParam
                \left(
                    \membraneHeight
                \right)
                \activeLinkers
            }{
                \testFuncInactiveLinkers
            }
        \end{equation}
    \end{subequations}
    for all 
    $ \testFuncMembraneHeight \in \sobolevHSet{2}{\cortex} $, 
    $ \testFuncActiveLinkers \in \sobolevHSet{1}{\cortex} $,
    and $ \testFuncInactiveLinkers \in \sobolevHSet{1}{\cortex} $
    and initial values $ \membraneHeight(0,\cdot) \in \lebesgueSet{2}{\cortex} $, 
    $ \activeLinkers(0,\cdot) \in
    \bochnerLebesgueSet{2}{\cortex}{[0,\infty)} $, $ \inactiveLinkers(0,\cdot) \in 
    \bochnerLebesgueSet{2}{\cortex}{[0,\infty)} $.
\end{myproblem}

The antisymmetric structure of the linker equations allows for a simple conclusion:
\begin{mylemma}[Mass Conservation]
    \label{lemma:mass conservation}
    Let $ \activeLinkers, \inactiveLinkers $
    be parts of a solution to \autoref{problem:height-linker system:variational}
    in the strong sense. Then, there exists
    $ \totalLinkersMass \in [0,\infty) $ such that for almost all $ t \in [0,\infty) $
    \begin{equation*}
        \integral{\domain}{}{\activeLinkers(t,x) + \inactiveLinkers(t,x)}{x} = \totalLinkersMass.
    \end{equation*}
\end{mylemma}
\begin{proof}
    Add \eqref{equ:height-linker system:strong:active linkers}
    and \eqref{equ:height-linker system:strong:inactive linkers}, 
    integrated
    over $ \domain $, to achieve
    \begin{equation*}
        \integral{\domain}{}{
            \pDiff{t}{}{\activeLinkers + \inactiveLinkers}
        }{x} =
        \integral{\domain}{}{
            \activeLinkersDiffusiv 
            \laplacian{\domain}{}{}\activeLinkers + 
            \inactiveLinkersDiffusiv 
            \laplacian{\domain}{}{}\inactiveLinkers
        }{x}.
    \end{equation*}
    Then, with the Divergence Theorem on closed manifolds, we get
    \begin{equation*}
        \integral{\domain}{}{%
            \pDiff{t}{}{\activeLinkers + \inactiveLinkers}
        }{x} = 0.
    \end{equation*}
    The integral and the weak differential operator commute, so
    $ \integral{\domain}{}{\activeLinkers + \inactiveLinkers}{x} $ is constant
    almost everywhere in 
    time. Considering the non-negativity
    of the initial values, the claim follows.
\end{proof}

\subsection{Global-in-time existence}
    In the following, we refer to 
    \begin{itemize}
      	\item the initial values $ \activeLinkers^0 = \activeLinkers(0,\cdot) $,
            $ \inactiveLinkers^0 = \inactiveLinkers(0,\cdot) $, 
            $ \membraneHeight^0 = \membraneHeight(0,\cdot) $,
    	\item the coefficients $ \membrStiffn $,
            $ \effectiveLengthParam $, $ \anotherMembrParam $, the pressure $ \pressure_0 $,
            the critical height $ \criticalHeight $, 
            the linkers spring constant
            $ \linkersSpringConst  $,
            the function $ \rippingInterpol $ with its Lipschitz
            constant $ L_\rippingInterpol $, the parameter $ \rippingLimitParam $,
            the repairing rate $ \repairRate $,
        \item the positive definite, self-adjoint operator 
            $ \timeDerivOperator = \rootTimeDerivOperator^2 $, and
            the constant $ \posDefConst > 0 $ such that
            $ \posDefConst \norm[ \sobolevHSet{1}{\cortex} ]{ u }^2
            \leq \innerProd[\lebesgueSet{2}{\cortex} ]{\timeDerivOperator u}{u} $
    \end{itemize}
    as the \emph{data of Problem~\ref{problem:height-linker system:variational}}.

    We set 
    \begin{equation*}
      	\solSpaceHeight_\finTime = \bochnerLebesgueSet{\infty}{[0,\finTime]}{\sobolevHSet{1}{\cortex}},
        \quad
        \solSpaceActiveLinkers_\finTime =
        \bochnerLebesgueSet{\infty}{[0,\finTime]}{\lebesgueSet{2}{\cortex}}
    \end{equation*}
    and define the operator 
    \begin{equation*}
        \funSig{F_\finTime}{
            \solSpaceHeight_\finTime \times 
            \solSpaceActiveLinkers_\finTime \times 
            \solSpaceInactiveLinkers_\finTime
        }{
            \solSpaceHeight_\finTime \times 
            \solSpaceActiveLinkers_\finTime \times 
            \solSpaceInactiveLinkers_\finTime
        }
    \end{equation*}
    such that
    a triple 
    $ \left( \bar{\membraneHeight}, \bar{\activeLinkers}, \bar{\inactiveLinkers} \right) \in
    	\revision[third]{
            \solSpaceHeight_\finTime \times 
            \solSpaceActiveLinkers_\finTime \times 
            \solSpaceInactiveLinkers_\finTime
        }
    $
    is mapped to 
    \revision{
    $ \left( \heightSolOp{\bar{\activeLinkers}}, \linkersSolOp{\bar{\membraneHeight}} \right) $
    with $ \funSig{\heightSolOp{}}{\solSpaceActiveLinkers_\finTime}{\solSpaceHeight_\finTime} $
    such that $ \heightSolOp{ \bar{\activeLinkers} } $ solves
    \begin{equation}
        \label{equ:existence proof:time-dependent:operator equations:height}
        \begin{split}
            \innerProd[
            	\bochnerLebesgueSet{2}{\cortex}{\reals^3}
            ]{
                \timeDerivOperator
                \left(
                    \pDiff{t}{}{}\membraneHeight
                    \normal[\cortex]{}
                \right)
            }{
                \varphi \normal[\cortex]{}
            }
            &+
            \heightLinkerSysHeightBF{\membraneHeight}{\testFuncMembraneHeight}
            =
            -
            \innerProd[ \lebesgueSet{2}{\cortex} ]{
                \bar{\activeLinkers}
                \linkersSpringConst
                \membraneHeight
            }{
                \varphi
            }
            +
            \innerProd[ \lebesgueSet{2}{\cortex} ]{
                \pressure_0
            }{
                \varphi
            }
        \end{split}
    \end{equation}
    and $ \funSig{\linkersSolOp{}}{\solSpaceHeight_\finTime}{\solSpaceActiveLinkers_\finTime \times
    \solSpaceActiveLinkers_\finTime} $ such that 
    $ \linkersSolOp{\bar{\membraneHeight}} $ solves
    \begin{subequations}
        \label{equ:existence proof:time-dependent:operator equations:linkers}
        \begin{equation}
            \label{equ:existence proof:time-dependent:operator equations:active linkers}
            \dualProd[ \sobolevHSet{-1}{\cortex} ]{
                \pDiff{t}{}{} \activeLinkers
            }{
                \testFuncActiveLinkers
            }
            +
            \heightLinkerSysActiveLinkersBF{\activeLinkers}{\testFuncActiveLinkers}
            =
            \repairRate
            \innerProd[ \lebesgueSet{2}{\cortex} ]{
                \inactiveLinkers
            }{
                \testFuncActiveLinkers
            }
            -
            \innerProd[ \lebesgueSet{2}{\cortex} ]{
                \rippingInterpol_\rippingLimitParam
                \left(
                    \bar{\membraneHeight}
                \right)
                \activeLinkers
            }{
                \testFuncActiveLinkers
            }
        \end{equation}
        \begin{equation}
            \label{equ:existence proof:time-dependent:operator equations:inactive linkers}
            \dualProd[ \sobolevHSet{-1}{\cortex} ]{
                \pDiff{t}{}{} \inactiveLinkers
            }{
                \testFuncInactiveLinkers
            }
            +
            \heightLinkerSysInactiveLinkersBF{\inactiveLinkers}{\testFuncInactiveLinkers}
            =
            -
            \repairRate
            \innerProd[ \lebesgueSet{2}{\cortex} ]{
                \inactiveLinkers
            }{
                \testFuncInactiveLinkers
            }
            +
            \innerProd[ \lebesgueSet{2}{\cortex} ]{
                \rippingInterpol_\rippingLimitParam
                \left(
                    \bar{\membraneHeight}
                \right)
                \activeLinkers
            }{
                \testFuncInactiveLinkers
            }
        \end{equation}
    \end{subequations}
    for all $ \testFuncMembraneHeight \in \sobolevHSet{2}{\cortex} $, and
    $ \testFuncActiveLinkers, \testFuncInactiveLinkers \in \sobolevHSet{1}{\cortex} $
    with initial data $ \membraneHeight^0 $, $ \activeLinkers^0 $,
    $ \inactiveLinkers^0 $. 
    \begin{remark}[Well-defined]
        The unique existence of a solution to
        \eqref{equ:existence proof:time-dependent:operator equations:linkers}
        follows by standard parabolic PDE theory.
        Existence and uniqueness of solutions to \eqref{equ:existence proof:time-dependent:%
        operator equations:height} may be shown by employing a Petrov-Galerkin-type
        approximation argument. (We refer the interested reader to 
        Appendix~\ref{sec:existence and uniqueness of solutions for the height %
        equation}.) Hence,
        $ F_\finTime $ is well-defined.
    \end{remark}
    }

    We aim at a Banach-type fixed point argument.
    To this end, the following a priori estimates are derived, which will give us
    Lipschitz continuity of the map $ F_\finTime $. 
    We will then show that $ F_\finTime $ is contractive
    in a rescaled topology of $ \solSpaceHeight_\finTime $ and
    $ \solSpaceActiveLinkers_\finTime $ by introducing the rescaled $ L^\infty $ norms
    $$ 
    \altNorm{u} = \text{ess\,sup}_{0\leq s \leq \finTime} \left( e^{-\rescExpConst s} \norm{u}(s) \right),
    $$
    $ \norm{\cdot} \in \{ \norm[\lebesgueSet{2}{\cortex}]{\cdot},
    \norm[\sobolevHSet{1}{\cortex}]{\cdot} \} $,
    for a constant $ \rescExpConst > 0 $ that we may choose with respect to the 
    problem data and $ \finTime $. In the following, $ \solSpaceHeight_\finTime $
    and $ \solSpaceActiveLinkers_\finTime $ are equipped with their corresponding
    rescaled norms.
    We will not state the rescaling constant explicitely since,
    from now on, the rescaling factor of every rescaled norm we deal with shall be
    $ e^{-\rescExpConst t} $ if not stated otherwise.

    \begin{remark}
        \label{remark:interpolation}
      	We will make use of the following interpolation embedding,
        which is a specialisation of \cite{Amann2000}, Theorem~3.1 for $ \theta \in [0,1] $
        and $ m \in [0,\infty) $:
        \begin{equation*}
          	\bochnerLebesgueSet{2}{[0,\finTime]}{\sobolevHSet{m}{\cortex}}
            \cap
            \bochnerSobolevHSet{1}{[0,\finTime]}{\sobolevHSet{-m}{\cortex}}
            \cong
            \bochnerSobolevHSet{\theta}{[0,\finTime]}{
              	\left( \sobolevHSet{m}{\cortex}, \sobolevHSet{-m}{\cortex} \right)_{\theta,2}
            }
        \end{equation*}
        isometrically. Choosing $ \theta = \frac{1}{4} $ and $ m = 1 $, we have
        \begin{equation*}
          	\bochnerSobolevHSet{\frac{1}{4}}{[0,\finTime]}{X} 
            \imbeddedR 
        	\bochnerLebesgueSet{4}{[0,\finTime]}{X},
        \end{equation*}
        \begin{equation*}
          	\left( \sobolevHSet{1}{\cortex}, \sobolevHSet{-1}{\cortex} \right)_{\frac{1}{4},2}
        	\cong 
            \sobolevHSet{\frac{1}{2}}{\cortex} 
            \imbeddedR 
            \lebesgueSet{3}{\cortex};
        \end{equation*}
        hence
        \begin{equation}
          	\label{equ:time-space interpolation}
          	\begin{split}
                \bochnerLebesgueSet{2}{[0,\finTime]}{\sobolevHSet{1}{\cortex}}
                \cap
                \bochnerSobolevHSet{1}{[0,\finTime]}{\sobolevHSet{-1}{\cortex}}
                &\imbeddedR
                \bochnerLebesgueSet{4}{[0,\finTime]}{\lebesgueSet{3}{\cortex}},
            \end{split}
        \end{equation}
    \end{remark}
    
    \paragraph{A priori estimates}
    There will be a lot of constants in the following estimates, so for
    ease of notation and in the sake of readability, we will slightly abuse
    notation and write $ C(\linkersSpringConst,\posDefConst,\dots) $ or similar
    for an expression that only depends on the problem data; it does not necessarily denote the same
    expression in every occurance. For convenience, we abbreviate
    $ \membraneHeightMV = \avgIntegral{\domain}{}{\membraneHeight}{x} $.
    \begin{mylemma}
      	\label{lemma:a priori bound of membrane height}
        Let $ \bar{\activeLinkers} \in 
        \solSpaceActiveLinkers_\finTime $, \revision[second]{%
        $ \bar{\activeLinkers} \geq 0 $ a.\,e.}
        and $ \finTime > 0 $.
        The following bound holds
        for
        \revision{
        $ \membraneHeight = \heightSolOp{\bar{\activeLinkers}} $}:
        \begin{equation}
          	\label{equ:a priori bound height LTwoHTwo}
            \begin{split}
            \norm[
            	\bochnerLebesgueSet{\infty}{[0,\finTime]}{
                  	\sobolevHSet{1}{\cortex}
                }
            ]{ \membraneHeight }^2
            \leq
            C(\finTime,\posDefConst,\linkersSpringConst,\abs{\cortex})
            \bigg(
            \norm[ \lebesgueSet{2}{\cortex} ]{
                \rootTimeDerivOperator(
                (\membraneHeight^0 - \membraneHeightInitMV) \normal[\cortex]{}
                )
            }^2
            +
            \norm[
                \bochnerLebesgueSet{\infty}{[0,\finTime]}{\lebesgueSet{2}{\cortex}}
            ]{ \pressure_0 }^2
            +
            \membraneHeightInitMV^2
            \bigg).
            \end{split}
        \end{equation}
    \end{mylemma}
    \begin{proof}
        (i) We observe that because the mean value of $ \membraneHeight $ is constant
        in time, we have $ \membraneHeightMV = \membraneHeightInitMV $ and
        $ \timeDerivOperator(\pDiff{t}{}{}\membraneHeight \normal[\cortex]{}) =
        \timeDerivOperator(\pDiff{t}{}{\membraneHeight - \membraneHeightInitMV}\normal[\cortex]{}) $.
        As 
        $ 
        	\integral{\cortex}{}{
              	\timeDerivOperator(
                	\pDiff{t}{}{}\membraneHeight \normal[\cortex]{}
                ) 
                \cdot
                \normal[\cortex]{}
            }{x}
            = 0
        $, we have
        $ 
        	\integral{\cortex}{}{
              	\timeDerivOperator(
                	\pDiff{t}{}{}\membraneHeight \normal[\cortex]{}
                ) 
                \cdot
                \membraneHeightInitMV
                \normal[\cortex]{}
            }{x}
            = 0
        $, so by testing \eqref{equ:existence proof:time-dependent:operator equations:height} 
        with $ \testFuncMembraneHeight =
        \membraneHeight $, we achieve
        \begin{equation*}
            \begin{split}
                \frac{1}{2}
                \pDiff{t}{}{}
                \norm[
                    \bochnerLebesgueSet{2}{\cortex}{\reals^3}
                ]{
                    \rootTimeDerivOperator
                    \left(
                        (
                        \membraneHeight
                        - \membraneHeightInitMV
                        )
                        \normal[\cortex]{}
                    \right)
                }^2
                &+
                \membrStiffn
                \norm[ \lebesgueSet{2}{\cortex} ]{
                    \laplacian{\cortex}{}{} \membraneHeight
                }^2
                +
                \effectiveLengthParam
                \norm[ \bochnerLebesgueSet{2}{\cortex}{\reals^3} ]{
                    \grad{\cortex}{}{} \membraneHeight 
                }^2
                +
                \anotherMembrParam
                \norm[ \lebesgueSet{2}{\cortex} ]{
                    \membraneHeight
                }^2
                \\
                &=
                -
                \innerProd[ \lebesgueSet{2}{\cortex} ]{
                    \bar{\activeLinkers}
                    \linkersSpringConst
                    \membraneHeight
                }{
                    \membraneHeight
                }
                +
                \innerProd[ \lebesgueSet{2}{\cortex} ]{
                    \pressure_0
                }{
                    \membraneHeight
                }
                \\
                &\leq
                \innerProd[ \lebesgueSet{2}{\cortex} ]{
                    \pressure_0
                }{
                    \membraneHeight
                },
            \end{split}
        \end{equation*}                
        where we dropped the term
        $  -\innerProd[ \lebesgueSet{2}{\cortex} ]{
        \bar{\activeLinkers}
            \linkersSpringConst
            \membraneHeight
        }{
            \membraneHeight
        } $ because of its non-positivity.
        We note that
        \begin{align*}
          	\posDefConst \norm[\lebesgueSet{2}{\cortex}]{u}^2
            \leq \posDefConst \norm[\sobolevHSet{1}{\cortex}]{u}^2
            \leq \posDefConst \norm[\bochnerSobolevHSet{1}{\cortex}{\reals^3}]{u\normal[\cortex]{}}^2
            \leq \innerProd[\bochnerLebesgueSet{2}{\cortex}{\reals^3}]{
              \timeDerivOperator(u \normal[\cortex]{})}{
              u \normal[\cortex]{}}
            \leq \norm[\bochnerLebesgueSet{2}{\cortex}{\reals^3}]{
              \rootTimeDerivOperator(u\normal[\cortex]{})}^2 
        \end{align*}
        for $ u \in \sobolevHSetMVF{1}{\cortex} $); therefore,
        \begin{align*}
          	\sqrt{\posDefConst} \norm[\sobolevHSet{1}{\cortex}]{\membraneHeight - \membraneHeightInitMV}
            \leq
            \norm[ \lebesgueSet{2}{\cortex} ]{
          		\rootTimeDerivOperator(
          		(\membraneHeight - \membraneHeightInitMV)\normal[\cortex]{}
                )
            },
        \end{align*}
        leading to
        \begin{align*}
        	\norm[\sobolevHSet{1}{\cortex}]{\membraneHeight}^2
            \leq
        	2\norm[\sobolevHSet{1}{\cortex}]{\membraneHeightInitMV}^2
            +
            2\posDefConst^{-1}
            \norm[\lebesgueSet{2}{\cortex}]{
              	\rootTimeDerivOperator(
                (
                	\membraneHeight - \membraneHeightInitMV
                )\normal[\cortex]{}
                )
            }^2.
        \end{align*}
        With this simple observation,
        the remaining right hand side term is bounded
        by employing the Cauchy-Schwartz and then the Young inequality:
        \begin{equation}
          	\label{equ:proof:a priori bound of membrane height}
            \begin{split}
                \frac{1}{2}
                \pDiff{t}{}{}
                \norm[
                    \bochnerLebesgueSet{2}{\cortex}{\reals^3}
                ]{
                    \rootTimeDerivOperator
                    \left(
                        (
                        \membraneHeight
                        -
                        \membraneHeightInitMV
                        )
                        \normal[\cortex]{}
                    \right)
                }^2
                &+
                \membrStiffn
                \norm[ \lebesgueSet{2}{\cortex} ]{
                    \laplacian{\cortex}{}{}\membraneHeight
                }^2
                +
                \effectiveLengthParam
                \norm[ \bochnerLebesgueSet{2}{\cortex}{\reals^3} ]{
                    \grad{\cortex}{}{} \membraneHeight
                }^2
                +
                \anotherMembrParam
                \norm[ \lebesgueSet{2}{\cortex} ]{
                    \membraneHeight
                }^2
                \\
                &\leq
                \frac{1}{2}
                \norm[ \lebesgueSet{2}{\cortex} ]{
                    \pressure_0
                }^2
                +
                \abs{\cortex}
                \membraneHeightInitMV^2
                +
                \frac{1}{\posDefConst}
                \norm[ \lebesgueSet{2}{\cortex} ]{                        
                    \rootTimeDerivOperator( 
                    	(\membraneHeight - \membraneHeightInitMV)\normal[\cortex]{}
                    )
                }^2
            \end{split}
        \end{equation}
        The Gr\"{o}nwall inequality then gives us the bound:
        \begin{equation}
          	\label{equ:proof:a priori bound of membrane height:bound on time derivative op}
            \begin{split}
                \norm[
                    \bochnerLebesgueSet{2}{\cortex}{\reals^3}
                ]{
                    \rootTimeDerivOperator
                    \left(
                        (
                        \membraneHeight
                        -
                        \membraneHeightInitMV
                        )
                        \normal[\cortex]{}
                    \right)
                }^2(t)
                &\leq
                \norm[ \lebesgueSet{2}{\cortex} ]{
                  	\rootTimeDerivOperator(
                    (\membraneHeight^0 - \membraneHeightInitMV) \normal[\cortex]{}
                    )
                }^2
                e^{t 2\posDefConst^{-1}}
                +
                \integral{0}{t}{
                  	e^{(t -s)2\posDefConst^{-1}}
                    \left(
                      	\norm[\lebesgueSet{2}{\cortex}]{\pressure_0}^2
                        +
                        2 \abs{\cortex} \membraneHeightInitMV^2
                    \right)
                }{s}
                \\
                &\leq
                C(\posDefConst,\linkersSpringConst,\abs{\cortex})
                \bigg(
                \norm[ \lebesgueSet{2}{\cortex} ]{
                  	\rootTimeDerivOperator(
                    (\membraneHeight^0 - \membraneHeightInitMV) \normal[\cortex]{}
                    )
                }^2
                +
                \finTime(
                \norm[
                	\bochnerLebesgueSet{\infty}{[0,\finTime]}{\lebesgueSet{2}{\cortex}}
                ]{ \pressure_0 }^2
                +
                \membraneHeightInitMV^2
                )
                \bigg).
            \end{split}
        \end{equation}

        With the positive definiteness of $ \timeDerivOperator $, we even have
        \begin{equation*}
            \begin{split}
                \norm[
                    \sobolevHSet{1}{\cortex}
                ]{
                    \membraneHeight
                }^2(t)
                \leq
                C(\posDefConst,\linkersSpringConst,\abs{\cortex})
                \bigg(
                \norm[ \lebesgueSet{2}{\cortex} ]{
                  	\rootTimeDerivOperator(
                    (\membraneHeight^0 - \membraneHeightInitMV) \normal[\cortex]{}
                    )
                }^2
                +
                \finTime(
                \norm[
                	\bochnerLebesgueSet{\infty}{[0,\finTime]}{\lebesgueSet{2}{\cortex}}
                ]{ \pressure_0 }^2
                +
                \membraneHeightInitMV^2
                )
                +
                \membraneHeightInitMV^2
                \bigg).
            \end{split}
        \end{equation*}
        This finishes the proof.
    \end{proof}

    \newcommand{\activeLinkersInitComp}{\delta_a}
    \newcommand{\inactiveLinkersInitComp}{\delta_i}
    \begin{mylemma}
      	\label{lemma:a priori bound of active linkers}
        Let $ \finTime > 0 $ and 
        $ \bar{\membraneHeight} \in \solSpaceHeight_\finTime $.
        The following a priori bounds hold for
        \revision{$ \left( \activeLinkers, \inactiveLinkers \right) =
        G_\finTime(\bar{\membraneHeight}) $}:
        \begin{itemize}
        	\item[(i)]
            $
                \norm[ 
                    \bochnerLebesgueSet{\infty}{[0,\finTime]}{
                        \lebesgueSet{2}{\cortex} 
                    }
                ]{ \activeLinkers  }^2
                +
                \norm[ 
                    \bochnerLebesgueSet{\infty}{[0,\finTime]}{
                        \lebesgueSet{2}{\cortex} 
                    }
                ]{ \inactiveLinkers  }^2
                \leq
                \left(
                \norm[ \lebesgueSet{2}{\cortex} ]{
                    \activeLinkers^0
                }^2
                +
                \norm[ \lebesgueSet{2}{\cortex} ]{
                    \inactiveLinkers^0
                }^2
                \right)
                e^{
                    \finTime(
                        \repairRate 
                        + 
                        \norm[
                        	\bochnerLebesgueSet{1}{[0,\finTime]}{
                              	\lebesgueSet{\infty}{\cortex}
                            }
                        ]{
                            \rippingInterpol_\rippingLimitParam( \bar\membraneHeight )
                        }
                    )
                },
            $
            \item[(ii)]
            \begin{equation*}
                \begin{split}
                \norm[
                    \bochnerLebesgueSet{2}{[0,\finTime]}{
                        \sobolevHSet{1}{\cortex}
                    }
                ]{ \activeLinkers }^2
                &+
                \norm[
                	\bochnerLebesgueSet{2}{[0,\finTime]}{
                      	\sobolevHSet{-1}{\cortex}
                    }
                ]{
                  	\pDiff{t}{}{}\activeLinkers
                }^2
                \leq
                C(\repairRate,\activeLinkersDiffusiv)
                \left(
                    \norm[ \lebesgueSet{2}{\cortex} ]{
                        \activeLinkers^0
                    }^2
                    +
                    \norm[ \lebesgueSet{2}{\cortex} ]{
                        \inactiveLinkers^0
                    }^2
                \right)
                \cdot
                \\
                &\cdot
                \left(
                    1
                    +
                    \finTime \repairRate
                    +
                    \norm[ 
                        \bochnerLebesgueSet{1}{[0,\finTime]}{
                            \lebesgueSet{\infty}{\cortex} 
                        }
                    ]{
                        \rippingInterpol_\rippingLimitParam
                        \left(
                            \bar{\membraneHeight}
                        \right)
                    }                
                    +
                    \norm[ 
                        \bochnerLebesgueSet{2}{[0,\finTime]}{
                            \lebesgueSet{4}{\cortex} 
                        }
                    ]{
                        \rippingInterpol_\rippingLimitParam
                        \left(
                            \bar{\membraneHeight}
                        \right)
                    }^2
                \right)
                e^{
                    \finTime(
                        \repairRate 
                        + 
                        \norm[
                        	\bochnerLebesgueSet{1}{[0,\finTime]}{
                              	\lebesgueSet{\infty}{\cortex}
                            }
                        ]{
                            \rippingInterpol_\rippingLimitParam( \bar\membraneHeight )
                        }
                    )
                }.
                \end{split}
            \end{equation*}
        \end{itemize}
    \end{mylemma}
    \begin{proof}
        (i) 
        We test \eqref{equ:existence proof:time-dependent:operator equations:%
        active linkers} by
        $ \testFuncActiveLinkers = \activeLinkers $:
        \begin{equation*}
            \frac{1}{2}
            \pDiff{t}{}{}
            \norm[ \lebesgueSet{2}{\cortex} ]{
                \activeLinkers
            }^2
            +
            \activeLinkersDiffusiv
            \norm[ \bochnerLebesgueSet{2}{\cortex}{\reals^3} ]{
                \grad{\cortex}{}{} \activeLinkers
            }^2
            =
            \repairRate
            \innerProd[ \lebesgueSet{2}{\cortex} ]{
              	\inactiveLinkers
            }{
                \activeLinkers
            }
            -
            \innerProd[ \lebesgueSet{2}{\cortex} ]{
                \rippingInterpol_\rippingLimitParam
                \left(
                    \bar{\membraneHeight}
                \right)
                \activeLinkers
            }{
                \activeLinkers
            }
        \end{equation*}                
        and \eqref{equ:existence proof:time-dependent:operator equations:%
        inactive linkers} by
        $ \testFuncInactiveLinkers = \inactiveLinkers $:
        \begin{equation*}
            \frac{1}{2}
            \pDiff{t}{}{}
            \norm[ \lebesgueSet{2}{\cortex} ]{
                \inactiveLinkers
            }^2
            +
            \inactiveLinkersDiffusiv
            \norm[ \bochnerLebesgueSet{2}{\cortex}{\reals^3} ]{
                \grad{\cortex}{}{} \inactiveLinkers
            }^2
            =
            -
            \repairRate
            \innerProd[ \lebesgueSet{2}{\cortex} ]{
                \inactiveLinkers
            }{
                \inactiveLinkers
            }
            +
            \innerProd[ \lebesgueSet{2}{\cortex} ]{
                \rippingInterpol_\rippingLimitParam
                \left(
                    \bar{\membraneHeight}
                \right)
                \activeLinkers
            }{
                \inactiveLinkers
            }
        \end{equation*}        
        and add the equations leaving out the non-positive terms on 
        the right hand sides ($ \rippingInterpol_\rippingLimitParam \geq 0 $
        by assumption):
        \begin{equation}
          	\label{equ:existence proof:time-dependent:a priori bound active linkers}          
            \begin{split}
                \frac{1}{2}
                \pDiff{t}{}{
                \norm[ \lebesgueSet{2}{\cortex} ]{
                    \activeLinkers
                }^2
                +
                \norm[ \lebesgueSet{2}{\cortex} ]{
                    \inactiveLinkers
                }^2
                }
                +
                \activeLinkersDiffusiv
                \norm[ \bochnerLebesgueSet{2}{\cortex}{\reals^3} ]{
                    \grad{\cortex}{}{} \activeLinkers
                }^2
                &+
                \inactiveLinkersDiffusiv
                \norm[ \bochnerLebesgueSet{2}{\cortex}{\reals^3} ]{
                    \grad{\cortex}{}{} \inactiveLinkers
                }^2
                \\
                &\leq
                \repairRate
                \innerProd[ \lebesgueSet{2}{\cortex} ]{
                    \inactiveLinkers
                }{
                    \activeLinkers
                }
                +
                \innerProd[ \lebesgueSet{2}{\cortex} ]{
                    \rippingInterpol_\rippingLimitParam
                    \left(
                        \bar{\membraneHeight}
                    \right)
                    \activeLinkers
                }{
                    \inactiveLinkers
                }
                \\
                &\leq
                \frac{\repairRate}{2}
                \left(
                \norm[ \lebesgueSet{2}{\cortex} ]{
                    \activeLinkers
                }^2
                +
                \norm[ \lebesgueSet{2}{\cortex} ]{
                  	\inactiveLinkers
                }^2
                \right) +
                \\
                &+
                \frac{\norm[ \lebesgueSet{\infty}{\cortex} ]{
                    \rippingInterpol_\rippingLimitParam
                    \left(
                        \bar{\membraneHeight}
                    \right)
                }}{2}
                \left(
                \norm[ \lebesgueSet{2}{\cortex} ]{
                    \activeLinkers
                }^2
                +
                \norm[ \lebesgueSet{2}{\cortex} ]{                
                    \inactiveLinkers
                }^2     
                \right).
            \end{split}
        \end{equation}  
        The Gr\"{o}nwall inequality now implies
        \begin{align*}
            \norm[ \lebesgueSet{2}{\cortex} ]{
                \activeLinkers
            }^2(t)
            +
            \norm[ \lebesgueSet{2}{\cortex} ]{
                \inactiveLinkers
            }^2(t)
            &\leq
            \left(
            \norm[ \lebesgueSet{2}{\cortex} ]{
                \activeLinkers^0
            }^2
            +
            \norm[ \lebesgueSet{2}{\cortex} ]{
                \inactiveLinkers^0
            }^2
            \right)
            e^{
              	t(
              		\repairRate 
                    + 
              		\norm[
              			\bochnerLebesgueSet{1}{[0,\finTime]}{
              				\lebesgueSet{\infty}{\cortex}
                        }
                    ]{
              			\rippingInterpol_\rippingLimitParam( \bar\membraneHeight )
                    }
                )
            }
        \end{align*}
        thus giving the claimed $ L^\infty - L^2 $ bound.

        (ii) We start again with \eqref{equ:existence proof:time-dependent:%
        a priori bound active linkers}, drop
        $ \inactiveLinkersDiffusiv \norm[\lebesgueSet{2}{\cortex}]{\grad{\cortex}{}{}\inactiveLinkers}^2 $
        on the left, integrate in time, and then also drop
        $ 
        	\norm[\lebesgueSet{2}{\cortex}]{\activeLinkers}^2(t) 
            +
        	\norm[\lebesgueSet{2}{\cortex}]{\inactiveLinkers}^2(t) 
        $ 
        on the left:
        \begin{equation*}
            \begin{split}
                2\activeLinkersDiffusiv
                \norm[ 
                	\bochnerLebesgueSet{2}{[0,\finTime]}{
                      	\bochnerLebesgueSet{2}{\cortex}{\reals^3} 
                    }
                ]{
                    \grad{\cortex}{}{} \activeLinkers
                }^2
                &\leq
                \left(
                \finTime \repairRate
                +
                \norm[ 
                	\bochnerLebesgueSet{1}{[0,\finTime]}{
                      	\lebesgueSet{\infty}{\cortex} 
                    }
                ]{
                    \rippingInterpol_\rippingLimitParam
                    \left(
                        \bar{\membraneHeight}
                    \right)
                }
                \right)
                \left(
                \norm[ 
                	\bochnerLebesgueSet{\infty}{[0,\finTime]}{
                      	\lebesgueSet{2}{\cortex} 
                    }
                ]{
                    \activeLinkers
                }^2
                +
                \norm[ 
                	\bochnerLebesgueSet{\infty}{[0,\finTime]}{
                      	\lebesgueSet{2}{\cortex} 
                    }
                ]{
                  	\inactiveLinkers
                }^2
                \right).
            \end{split}
        \end{equation*}
        Together with the previous estimate, we obtain the claimed bound.

        (iii) We test \eqref{equ:existence proof:time-dependent:operator equations:%
        active linkers} by
        $ \testFuncActiveLinkers \in \sobolevHSet{1}{\cortex} $ with
        $ \norm[ \sobolevHSet{1}{\cortex} ]{ \testFuncActiveLinkers } \leq 1 $ 
        such that 
        $ \dualProd[ \sobolevHSet{-1}{\cortex} ]{ \activeLinkers }{ 
        \testFuncActiveLinkers } \geq 0 $ w.\,l.\,o.\,g., 
        shift the gradient term to the right,
        and use the H\"{o}lder inequality on
        the right hand side terms:
        \begin{equation*}
          	\begin{split}
              	\dualProd[ \sobolevHSet{-1}{\cortex} ]{
                  	\pDiff{t}{}{} \activeLinkers
                }{
                  	\testFuncActiveLinkers
                }
                &\leq
                \activeLinkersDiffusiv
                \norm[\lebesgueSet{2}{\cortex}]{
                  	\grad{\cortex}{}{} \activeLinkers
                }
                +
                \repairRate
                \norm[\lebesgueSet{2}{\cortex}]{
                  	\inactiveLinkers
                }
                +
                \norm[\lebesgueSet{4}{\cortex}]{	
                    \rippingInterpol_\rippingLimitParam(
                    	\bar\membraneHeight
                    )
                }
                \norm[\lebesgueSet{2}{\cortex}]{
                  	\activeLinkers
                }.
            \end{split}
        \end{equation*}
        Squaring the inequality and integrating in time, we obtain
        \begin{equation*}
          	\begin{split}
                \integral{0}{\finTime}{
                    \dualProd[ 
                        \sobolevHSet{-1}{\cortex} 
                    ]{
                        \pDiff{s}{}{} \activeLinkers
                    }{
                      	\testFuncActiveLinkers
                    }^2
                }{s}
                &\leq
                C(\repairRate,\activeLinkersDiffusiv)
                \bigg(
                    \norm[
                    	\bochnerLebesgueSet{2}{[0,\finTime]}{
                          	\lebesgueSet{2}{\cortex}
                        }
                    ]{
                        \grad{\cortex}{}{} \activeLinkers
                    }^2
                    +
                    \norm[
                    	\bochnerLebesgueSet{2}{[0,\finTime]}{
                          	\lebesgueSet{2}{\cortex}
                        }
                    ]{
                        \inactiveLinkers
                    }^2
                    \\
                    &+
                    \integral{0}{\finTime}{
                        \norm[\lebesgueSet{4}{\cortex}]{	
                            \rippingInterpol_\rippingLimitParam(
                                \bar\membraneHeight
                            )
                        }^2(s)
                        \norm[\lebesgueSet{2}{\cortex}]{
                            \activeLinkers
                        }^2(s)
                    }{s}
                \bigg)
                \\
                &\leq
                C(\repairRate,\activeLinkersDiffusiv)
                \bigg(
                    \norm[
                    	\bochnerLebesgueSet{2}{[0,\finTime]}{
                          	\lebesgueSet{2}{\cortex}
                        }
                    ]{
                        \grad{\cortex}{}{} \activeLinkers
                    }^2
                    +
                    \norm[
                    	\bochnerLebesgueSet{2}{[0,\finTime]}{
                          	\lebesgueSet{2}{\cortex}
                        }
                    ]{
                        \inactiveLinkers
                    }^2
                    \\
                    &+
                    \norm[
                    	\bochnerLebesgueSet{\infty}{[0,\finTime]}{
                      		\lebesgueSet{2}{\cortex}
                        }
                    ]{
                        \activeLinkers
                    }^2
                    \norm[
                    	\bochnerLebesgueSet{2}{[0,\finTime]}{
                          	\lebesgueSet{4}{\cortex}
                        }
                    ]{	
                        \rippingInterpol_\rippingLimitParam(
                            \bar\membraneHeight
                        )
                    }^2
                \bigg)
            \end{split}
        \end{equation*}
    \end{proof}

    \newcommand{\apBoundHeightH}{M_1}
    \newcommand{\apBoundHeightL}{N_1}
    \newcommand{\apBoundActiveLinkersL}{M_2}
    \newcommand{\apBoundInactiveLinkersL}{M_3}
    \begin{mylemma}
      	\label{lemma:first Banach argument step}
        Given a time $ \finTime > 0 $, there exist constants $ \apBoundHeightH > 0$,
        $ \apBoundActiveLinkersL > 0 $,
        depending on the data of
        Problem~\ref{problem:height-linker system:variational}
        such that
        the operator $ F_\finTime $ maps the set
        \begin{align*}
            K_\finTime = \big\{
                \left( u, v, w \right) \in 
                \solSpaceHeight_\finTime  \times \solSpaceActiveLinkers_\finTime \times 
                \solSpaceInactiveLinkers_\finTime
            \big\vert
                &\norm[ 
                  	\bochnerLebesgueSet{\infty}{[0,\finTime]}{ 
                  		\sobolevHSet{1}{\cortex} 
                    } 
                ]{u}^2 
                \leq \apBoundHeightH,
                \\
                &\norm[ 
                  	\bochnerLebesgueSet{\infty}{[0,\finTime]}{ 
                  		\lebesgueSet{2}{\cortex} 
                    } 
                ]{v}^2
                +
                \norm[ 
                  	\bochnerLebesgueSet{\infty}{[0,\finTime]}{ 
                  		\lebesgueSet{2}{\cortex} 
                    } 
                ]{w}^2
                \leq \apBoundActiveLinkersL,
                \\
                &v, w \geq 0 \; \almostEvWh
            \big\},
        \end{align*}
        into itself
        for initial data
        $ \membraneHeight^0 $,
        $ \activeLinkers^0 \geq 0 $,
        $ \inactiveLinkers^0 \geq 0 $ a.\,e.
    \end{mylemma}
    \begin{proof}
        We choose $ \apBoundHeightH $ as the expression on the right hand
        side of the equation in
        \autoref{lemma:a priori bound of membrane height}.
        The bound $ \apBoundActiveLinkersL $ then directly follows
        with \autoref{lemma:a priori bound of active linkers}
        by setting 
        $
        	\apBoundActiveLinkersL
            = 
            N
            \left(
              	\norm[\lebesgueSet{2}{\cortex}]{\activeLinkers^0}^2
                +
                \norm[\lebesgueSet{2}{\cortex}]{\inactiveLinkers^0}^2
            \right)
            e^{
              	\finTime
                \left(
                  	\repairRate 
                    + 
                    \frac{\apBoundHeightH L_\rippingInterpol}{\rippingLimitParam} 
            	\right)
            }
        $
        for some sufficiently large $ N \in \nats $.
        
        We may consider \eqref{equ:existence proof:time-dependent:operator equations:%
        active linkers}, \eqref{equ:existence proof:time-dependent:operator equations:%
        inactive linkers} as reaction-diffusion system with right hand side
        \begin{equation*}
          	f(\activeLinkers,\inactiveLinkers)
            =
            \begin{pmatrix}
              	f_1(\activeLinkers,\inactiveLinkers)
                \\
                f_2(\activeLinkers,\inactiveLinkers)
            \end{pmatrix}
            =
            \begin{pmatrix}
              	\repairRate \inactiveLinkers 
                -
                \rippingInterpol_\rippingLimitParam(\bar\membraneHeight) \activeLinkers
                \\
                -\repairRate \inactiveLinkers
                + 
                \rippingInterpol_\rippingLimitParam(\bar\membraneHeight) \activeLinkers
            \end{pmatrix}.
        \end{equation*}
        According to \cite{Pierre2010}, Lemma~1.1, quasi-positivity of $ f $ is sufficient
        to guarantee preservation of non-negativity.
        We see immediately that for $r_1,r_2\geq 0 $, $ f_1(0,r_2) = 
        \repairRate r_2 \geq 0 $ and $ f_2(r_1,0) = 
        \rippingInterpol_\rippingLimitParam(\bar\membraneHeight) r_1 \geq 0 $,
        so
        $ \activeLinkers $ and $ \inactiveLinkers $ are non-negative
        since the initial data are non-negative.
        Therefore, $ F_\finTime(K_\finTime) \subseteq K_\finTime $.
    \end{proof}
    
    To apply Banach's fixed point theorem, we have to show that there is 
    a $ \rescExpConst > 0 $ such that $ F_\finTime $
    is a contraction for an arbitrarily chosen $ \finTime > 0 $.
    To this purpose, we choose an arbitrary
    pair of arguments
    $ u = \left( \bar{\membraneHeight}^1, \bar{\activeLinkers}^1, 
    \bar{\inactiveLinkers}^1 \right) $,
    $ v = \left( \bar{\membraneHeight}^2, \bar{\activeLinkers}^2, 
    \bar{\inactiveLinkers}^2 \right) $
    and derive a bound on the difference of 
    $ \left( \membraneHeight^1, \activeLinkers^1, \inactiveLinkers^1 \right) =
    F_\finTime(u) $ and 
    $ \left( \membraneHeight^2, \activeLinkers^2, \inactiveLinkers^2 \right) =
    F_\finTime(v) $ in the rescaled norm:
    \begin{equation*}
        \altNorm[ 
            \solSpaceHeight_\finTime \times \solSpaceActiveLinkers_\finTime \times 
            \solSpaceInactiveLinkers_\finTime 
        ]{ 
            F_\finTime(u) - F_\finTime(v) 
        } 
        \leq 
        L
        \altNorm[ 
            \solSpaceHeight_\finTime \times \solSpaceActiveLinkers_\finTime \times 
            \solSpaceInactiveLinkers_\finTime 
        ]{ 
            u - v 
        },
    \end{equation*}
    where $ L \in [0,1) $.
    For ease of notation, we abbreviate $ \difference{\membraneHeight} = 
    \membraneHeight^1 - \membraneHeight^2 $, $ \difference{\activeLinkers}
    = \activeLinkers^1 - \activeLinkers^2 $, and 
    $ \difference{\inactiveLinkers} = \inactiveLinkers^1 - \inactiveLinkers^2 $.
    \begin{mylemma}
       \label{lemma:existence proof:time-dependent:contraction estimate membrane height}
       There is a $ \rescExpConst $ (depending only on problem data)
       such that the following estimate holds:
        \begin{equation}
          	\label{equ:existence proof:time-dependent:contraction estimate height}
            \begin{split}
                \altNorm[ 
                	\bochnerLebesgueSet{\infty}{[0,T]}{ \sobolevHSet{1}{\cortex} } 
                ]{
                    \difference{\membraneHeight}
                }
                &\leq
                L_\membraneHeight
                \altNorm[ 
                	\bochnerLebesgueSet{\infty}{[0,T]}{ \lebesgueSet{2}{\cortex} } 
                ]{
                    \difference{\bar{\activeLinkers}}
                }
            \end{split}
        \end{equation}
        for some $ L_\membraneHeight \in [0,1) $.
    \end{mylemma}
    \begin{proof}
        Subtracting \eqref{equ:existence proof:time-dependent:operator equations:%
        height} for the two different 
        arguments one achieves (using the lineariy of $ \timeDerivOperator $):
        \begin{equation}
          	\label{equ:existence proof:time-dependent:contraction estimate height:%
            subtraction equ}
            \begin{split}
                &\innerProd[
                    \bochnerLebesgueSet{2}{\cortex}{\reals^3}
                ]{
                    \timeDerivOperator
                    \left(
                        \left(
                        \pDiff{t}{}{} \difference{\membraneHeight}
                        \right)
                        \normal[\cortex]{}
                    \right)
                }{
                   	\testFuncMembraneHeight \normal[\cortex]{}
                }
                +
                \membrStiffn
                \innerProd[ \lebesgueSet{2}{\cortex} ]{
                    \laplacian{\cortex}{}{} \difference{\membraneHeight}
                }{
                    \laplacian{\cortex}{}{} \testFuncMembraneHeight
                }
                +
                \effectiveLengthParam
                \innerProd[ \bochnerLebesgueSet{2}{\cortex}{\reals^3} ]{
                    \grad{\cortex}{}{} \difference{\membraneHeight}
                }{
                    \grad{\cortex}{}{} \testFuncMembraneHeight
                }
                \\
                &\qquad+
                \anotherMembrParam
                \innerProd[ \lebesgueSet{2}{\cortex} ]{
                  	\difference{\membraneHeight}
                }{
                    \varphi
                }
				=
                -
                \innerProd[ \lebesgueSet{2}{\cortex} ]{
                    \bar{\activeLinkers}^1
                    \linkersSpringConst
                    \membraneHeight^1
                }{
                    \varphi
                }
                +
                \innerProd[ \lebesgueSet{2}{\cortex} ]{
                    \bar{\activeLinkers}^2
                    \linkersSpringConst
                    \membraneHeight^2
                }{
                    \varphi
                }.
            \end{split}
        \end{equation}
        (i) We then test by $ \testFuncMembraneHeight = \difference{\membraneHeight} $
        and insert a suitable zero expression on the right hand side.
        (Note that $ \difference{\membraneHeight} $ is mean-value-free as 
        $\membraneHeightMV^1 = \membraneHeightInitMV = \membraneHeightMV^2$.)
        \begin{equation}
          	\label{equ:contraction of height:initial estimate}
            \begin{split}
                \frac{1}{2}\pDiff{t}{}{}
                \norm[
                    \bochnerLebesgueSet{2}{\cortex}{\reals^3}
                ]{
                    \rootTimeDerivOperator
                    \left(
                        \difference{\membraneHeight}
                        \normal[\cortex]{}
                    \right)
                }^2
                &+
                \membrStiffn
                \norm[ \lebesgueSet{2}{\cortex} ]{
                    \laplacian{\cortex}{}{}
                    \difference{\membraneHeight}
                }^2
                +
                \effectiveLengthParam
                \norm[ \bochnerLebesgueSet{2}{\cortex}{\reals^3} ]{
                    \grad{\cortex}{}{}
                    \difference{\membraneHeight}
                }^2
                +
                \anotherMembrParam
                \norm[ \lebesgueSet{2}{\cortex} ]{
                  	\difference{\membraneHeight}
                }^2
                \\
                &=
                -
                \innerProd[ \lebesgueSet{2}{\cortex} ]{
                   	\difference{\bar{\activeLinkers}}
                    \linkersSpringConst
                    \membraneHeight^1
                }{
                  	\difference{\membraneHeight}
                }
                -
                \innerProd[ \lebesgueSet{2}{\cortex} ]{
                    \bar{\activeLinkers}^2
                    \linkersSpringConst
                    \difference{\membraneHeight}
                }{
                  	\difference{\membraneHeight}
                }.
            \end{split}
        \end{equation} 
        The last term on the right hand side is non-positive, so we drop it. 
        We apply the Hölder and then Young inequality, and use
        the positive definiteness of $ \timeDerivOperator $:
        \begin{equation*}
            \begin{split}
                \frac{1}{2}\pDiff{t}{}{}
                \norm[
                    \bochnerLebesgueSet{2}{\cortex}{\reals^3}
                ]{
                    \rootTimeDerivOperator
                    \left(
                        \difference{\membraneHeight}
                        \normal[\cortex]{}
                    \right)
                }^2
                &+
                \membrStiffn
                \norm[ \lebesgueSet{2}{\cortex} ]{
                    \laplacian{\cortex}{}{}
                    \difference{\membraneHeight}
                }^2
                +
                \effectiveLengthParam
                \norm[ \bochnerLebesgueSet{2}{\cortex}{\reals^3} ]{
                    \grad{\cortex}{}{}
                    \difference{\membraneHeight}
                }^2
                +
                \anotherMembrParam
                \norm[ \lebesgueSet{2}{\cortex} ]{
                  	\difference{\membraneHeight}
                }^2
                \\
                &\leq
                \frac{1}{2}
                \norm[ \lebesgueSet{2}{\cortex} ]{
                   	\difference{\bar{\activeLinkers}}
                }^2
                \norm[ \lebesgueSet{4}{\cortex} ]{
                    \linkersSpringConst
                    \membraneHeight^1
                }^2
                +
                \frac{1}{2}
                \norm[\lebesgueSet{4}{\cortex}]{
                  	\difference{\membraneHeight}
                }^2
                \\
                &\leq
                \frac{1}{2}
                C(\posDefConst,\cortex,\linkersSpringConst)
                \left(
                \norm[ \lebesgueSet{2}{\cortex} ]{
                   	\difference{\bar{\activeLinkers}}
                }^2
                \norm[ \lebesgueSet{4}{\cortex} ]{
                    \membraneHeight^1
                }^2
                +
                \norm[\lebesgueSet{2}{\cortex}]{
                  	\rootTimeDerivOperator(\difference{\membraneHeight}\normal[\cortex]{})
                }^2
                \right).
            \end{split}
        \end{equation*} 
        With the Gr\"{o}nwall inequality, we obtain
        \begin{equation*}
            \begin{split}
                \norm[
                    \bochnerLebesgueSet{2}{\cortex}{\reals^3}
                ]{
                    \rootTimeDerivOperator
                    \left(
                       	\difference{\membraneHeight}
                        \normal[\cortex]{}
                    \right)
                }^2(t)
                &\leq
                C(\posDefConst,\cortex,\linkersSpringConst)
                e^{\finTime C(\posDefConst,\cortex,\linkersSpringConst)}
                \integral{0}{t}{
                    \norm[ \lebesgueSet{2}{\cortex} ]{
                        \difference{\bar{\activeLinkers}}
                    }^2(s)
                    \norm[ \lebesgueSet{4}{\cortex} ]{
                        \membraneHeight^1
                    }^2(s)
                }{s}
            \end{split}
        \end{equation*} 
        By multiplying the inequality with $ e^{-\rescExpConst t} $
        and inserting $ e^{-\rescExpConst s} e^{\rescExpConst s} $ under the right
        hand side time integral, we arrive at:
        \begin{equation*}
            \begin{split}
                e^{-\rescExpConst t}
                \norm[
                    \bochnerLebesgueSet{2}{\cortex}{\reals^3}
                ]{
                    \rootTimeDerivOperator
                    \left(
                       	\difference{\membraneHeight}
                        \normal[\cortex]{}
                    \right)
                }^2(t)
                &\leq
                C(\finTime,\posDefConst,\cortex,\linkersSpringConst)
                e^{-\rescExpConst t}
                \integral{0}{t}{
                    e^{-\rescExpConst s}
                    e^{\rescExpConst s}
                    \norm[ \lebesgueSet{2}{\cortex} ]{
                        \difference{\bar{\activeLinkers}}
                    }^2(s)
                    \norm[ \lebesgueSet{4}{\cortex} ]{
                        \membraneHeight^1
                    }^2(s)
                }{s}
                \\
                &\leq
                C(\finTime,\posDefConst,\cortex,\linkersSpringConst)
                \altNorm[ 
                	\bochnerLebesgueSet{\infty}{[0,\finTime]}{
                      	\lebesgueSet{2}{\cortex}
                    }
                ]{
                    \difference{\bar{\activeLinkers}}
                }^2
                \integral{0}{t}{
                    e^{\rescExpConst(s-t)}
                    \norm[ \lebesgueSet{4}{\cortex} ]{
                        \membraneHeight^1
                    }^2(s)
                }{s}.
            \end{split}
        \end{equation*} 
        Observe, 
        \begin{equation*}
            \integral{0}{t}{
                e^{\rescExpConst(s-t)}
                \norm[ \lebesgueSet{4}{\cortex} ]{
                    \membraneHeight^1
                }^2(s)
            }{s}
            \leq
            \norm[ 
            	\bochnerLebesgueSet{\infty}{[0,\finTime]}{\lebesgueSet{4}{\cortex}} 
            ]{
                \membraneHeight^1
            }^2
            \integral{0}{t}{
                e^{\rescExpConst(s-t)}
            }{s}            
            =
            \norm[ 
            	\bochnerLebesgueSet{\infty}{[0,\finTime]}{\lebesgueSet{4}{\cortex}} 
            ]{
                \membraneHeight^1
            }^2
            \rescExpConst^{-1}
            \left(
            1-
            e^{-\rescExpConst t}
            \right),
        \end{equation*}
        so
        by choosing $ \rescExpConst $ appropriately large,
        the claimed contraction estimate follows.
    \end{proof}
    \begin{mylemma}
       \label{lemma:existence proof:time-dependent:contraction estimate linkers}
        The following estimates hold:
        \begin{equation}
          	\label{equ:existence proof:time-dependent:%
            contraction estimate inactive linkers}
            \begin{split}
                \altNorm[ 
                    \bochnerLebesgueSet{\infty}{[0,\finTime]}{ 
                        \lebesgueSet{2}{\cortex} 
                    } 
                ]{
                    \difference{\inactiveLinkers}
                }
                +
                \altNorm[ 
                    \bochnerLebesgueSet{\infty}{[0,\finTime]}{ 
                        \lebesgueSet{2}{\cortex} 
                    } 
                ]{
                    \difference{\activeLinkers}
                }
                \leq
                L_\ell
                \altNorm[ 
                    \bochnerLebesgueSet{\infty}{[0,\finTime]}{
                        \lebesgueSet{6}{\cortex}
                    }
                ]{
                    \difference{\bar{\membraneHeight}}
                }
            \end{split}
        \end{equation}
        for $ L_\ell \in (0,1) $.
    \end{mylemma}
    \begin{proof}
        (i)
        Subtracting \eqref{equ:existence proof:time-dependent:operator equations:%
        inactive linkers} for the
        two different arguments and inserting a suitable zero expression on the 
        right hand side, we obtain
        \begin{equation*}
            \begin{split}
                &\dualProd[\sobolevHSet{-1}{\cortex}]{
                    \pDiff{t}{}{} \difference{\inactiveLinkers}
                }{
                    \testFuncInactiveLinkers
                }
                +
                \inactiveLinkersDiffusiv
                \innerProd[ \bochnerLebesgueSet{2}{\cortex}{\reals^3} ]{
                    \grad{\cortex}{}{}
                    \difference{\inactiveLinkers}
                }{
                    \grad{\cortex}{}{} \testFuncInactiveLinkers
                }
                \\ & \qquad =
                -
                \repairRate
                \innerProd[ \lebesgueSet{2}{\cortex} ]{
                  	\difference{\inactiveLinkers}
                }{
                    \testFuncInactiveLinkers
                }
                             +
                \innerProd[ \lebesgueSet{2}{\cortex} ]{
                    \left(
                        \rippingInterpol_{\rippingLimitParam}
                        \left(
                            \bar{\membraneHeight}^1
                        \right)
                        -
                        \rippingInterpol_{\rippingLimitParam}
                        \left(
                            \bar{\membraneHeight}^2
                        \right)
                    \right)
                    \activeLinkers^1
                }{
                    \testFuncInactiveLinkers
                }
                                +
                \innerProd[ \lebesgueSet{2}{\cortex} ]{                     
                    \rippingInterpol_{\rippingLimitParam}
                    \left(
                        \bar{\membraneHeight}^2
                    \right)
                    \difference{\activeLinkers}
                }{
                    \testFuncInactiveLinkers
                }.
            \end{split}
        \end{equation*}          
        By
        choosing $ \testFuncInactiveLinkers = \difference{\inactiveLinkers} $,
        we find
        \begin{equation*}
            \begin{split}
            & \frac{1}{2}
                \pDiff{t}{}{}
                \norm[ \lebesgueSet{2}{\cortex} ]{
                    \difference{\inactiveLinkers}
                }^2
                +
                \inactiveLinkersDiffusiv
                \norm[ \bochnerLebesgueSet{2}{\cortex}{\reals^3} ]{
                    \grad{\cortex}{}{}
                    \difference{\inactiveLinkers}
                }^2 +
                \repairRate
                \norm[ \lebesgueSet{2}{\cortex} ]{
                    \difference{\inactiveLinkers}
                }^2             
                \\
                & \qquad \qquad \leq
                \innerProd[ \lebesgueSet{2}{\cortex} ]{
                    \left(
                        \rippingInterpol_{\rippingLimitParam}
                        \left(
                            \bar{\membraneHeight}^1
                        \right)
                        -
                        \rippingInterpol_{\rippingLimitParam}
                        \left(
                            \bar{\membraneHeight}^2
                        \right)
                    \right)
                    \activeLinkers^1
                }{
                    \difference{\inactiveLinkers}
                }
                +
                \innerProd[ \lebesgueSet{2}{\cortex} ]{                     
                    \rippingInterpol_{\rippingLimitParam}
                    \left(
                        \bar{\membraneHeight}^2
                    \right)
                    \difference{\activeLinkers}
                }{
                    \difference{\inactiveLinkers}
                }.
            \end{split}
        \end{equation*} 
        We apply the H\"{o}lder inequality on the right hand side
        and then use Young's inequality with a parameter $ \eps $ so
        small that 
        $ \eps \norm[\lebesgueSet{4}{\cortex}]{\difference{\inactiveLinkers}}^2 $
        can be absorbed by the $ H^1 $ norm on the left:
        \begin{equation}
          	\label{equ:contraction of active and inactive linkers:difference %
            estimate inactive linkers}
          	\begin{split}
                \frac{1}{2}
                \pDiff{t}{}{}
                \norm[ \lebesgueSet{2}{\cortex} ]{
                    \difference{\inactiveLinkers}
                }^2
                +
                \alpha 
                \norm[ \sobolevHSet{1}{\cortex} ]{
                    \difference{\inactiveLinkers}
                }^2             
                &\leq
                \frac{1}{4\eps}
                \norm[ \lebesgueSet{4}{\cortex} ]{
                    \left(
                        \rippingInterpol_{\rippingLimitParam}
                        \left(
                            \bar{\membraneHeight}^1
                        \right)
                        -
                        \rippingInterpol_{\rippingLimitParam}
                        \left(
                            \bar{\membraneHeight}^2
                        \right)
                    \right)
                }^2
                \norm[ \lebesgueSet{2}{\cortex} ]{
                    \activeLinkers^1
                }^2
                \\
                &+
                \frac{1}{4\eps}
                \norm[ \lebesgueSet{4}{\cortex} ]{                     
                    \rippingInterpol_{\rippingLimitParam}
                    \left(
                        \bar{\membraneHeight}^2
                    \right)
                }^2
                \norm[ \lebesgueSet{2}{\cortex} ]{                     
                    \difference{\activeLinkers}
                }^2,      
            \end{split}
        \end{equation} 
        where $ \alpha $ is some positive constant.

        Next, we subtract \eqref{equ:existence proof:time-dependent:operator equations:%
        active linkers}
        for the two different arguments and insert a suitable zero expression
        on the right hand side to obtain
        \begin{equation*}
          	\begin{split}
             &   \dualProd[ \sobolevHSet{-1}{\cortex} ]{
	                \pDiff{t}{}{} \difference{\activeLinkers} 
                }{
                    \testFuncActiveLinkers
                }
                +
                \activeLinkersDiffusiv
                \innerProd[ \bochnerLebesgueSet{2}{\cortex}{\reals^3} ]{
                    \grad{\cortex}{}{} \difference{\activeLinkers}
                }{
                  	\grad{\cortex}{}{} \testFuncActiveLinkers
                }
                \\ &  \qquad =
                \repairRate
                \innerProd[ \lebesgueSet{2}{\cortex} ]{
                    \difference{\inactiveLinkers}
                }{
                	\testFuncActiveLinkers
                }
                +
                \innerProd[ \lebesgueSet{2}{\cortex} ]{
                    \left(
                    \rippingInterpol_{\rippingLimitParam}
                    \left(
                        \bar{\membraneHeight}^2
                    \right)
                    -
                    \rippingInterpol_{\rippingLimitParam}
                    \left(
                        \bar{\membraneHeight}^1
                    \right)
                    \right)
                    \activeLinkers^1
                }{
                  	\testFuncActiveLinkers
                }
                -
                \innerProd[ \lebesgueSet{2}{\cortex} ]{
                    \rippingInterpol_{\rippingLimitParam}
                    \left(
                        \bar{\membraneHeight}^2
                    \right)
                    \difference{\activeLinkers}
                }{
                    \testFuncActiveLinkers
                }.
            \end{split}
        \end{equation*}
        Testing 
        with $ \testFuncActiveLinkers = \difference{\activeLinkers} $ 
        leads to
        \begin{equation*}
          	\begin{split}
              &  \frac{1}{2}
                \pDiff{t}{}{}
                \norm[ \lebesgueSet{2}{\cortex} ]{
                    \difference{\activeLinkers}
                }^2
                +
                \activeLinkersDiffusiv
                \norm[ \bochnerLebesgueSet{2}{\cortex}{\reals^3} ]{
                    \grad{\cortex}{}{} \difference{\activeLinkers} 
                }^2
                \\  
                &\qquad =
                \repairRate
                \innerProd[ \lebesgueSet{2}{\cortex} ]{
                    \difference{\inactiveLinkers}
                }{
                    \difference{\activeLinkers}                        
                }
                +
                \innerProd[ \lebesgueSet{2}{\cortex} ]{
                    \left(
                    \rippingInterpol_{\rippingLimitParam}
                    \left(
                        \bar{\membraneHeight}^2
                    \right)
                    -
                    \rippingInterpol_{\rippingLimitParam}
                    \left(
                        \bar{\membraneHeight}^1
                    \right)
                    \right)
                    \activeLinkers^1
                }{
                    \difference{\activeLinkers}
                }
               -
                \innerProd[ \lebesgueSet{2}{\cortex} ]{
                    \rippingInterpol_{\rippingLimitParam}
                    \left(
                        \bar{\membraneHeight}^2
                    \right)
                    \difference{\activeLinkers}
                }{
                    \difference{\activeLinkers}
                }
                \\
                & \qquad \leq 
                \repairRate
                \innerProd[ \lebesgueSet{2}{\cortex} ]{
                    \difference{\inactiveLinkers}
                }{
                    \difference{\activeLinkers}                        
                }
                 +
                \innerProd[ \lebesgueSet{2}{\cortex} ]{
                    \left(
                    \rippingInterpol_{\rippingLimitParam}
                    \left(
                        \bar{\membraneHeight}^2
                    \right)
                    -
                    \rippingInterpol_{\rippingLimitParam}
                    \left(
                        \bar{\membraneHeight}^1
                    \right)
                    \right)
                    \activeLinkers^1
                }{
                    \difference{\activeLinkers}
                }.
            \end{split}
        \end{equation*}
        Applying the H\"{o}lder and then the Young inequality on the right hand
        side terms, we obtain
        \begin{equation}
          	\label{equ:contraction of active and inactive linkers:difference %
            estimate active linkers}
          	\begin{split}
                \frac{1}{2}
                \pDiff{t}{}{}
                \norm[ \lebesgueSet{2}{\cortex} ]{
                    \difference{\activeLinkers}
                }^2
                +
                \activeLinkersDiffusiv
                \norm[ \bochnerLebesgueSet{2}{\cortex}{\reals^3} ]{
                    \grad{\cortex}{}{} \difference{\activeLinkers} 
                }^2
                &\leq 
                \frac{\repairRate}{2}
                \left(
                \norm[ \lebesgueSet{2}{\cortex} ]{
                    \difference{\inactiveLinkers}
                }^2
                +
                \norm[ \lebesgueSet{2}{\cortex} ]{
                    \difference{\activeLinkers}                        
                }^2
                \right)
                \\
                &+
                \frac{1}{2}
                \norm[ \lebesgueSet{6}{\cortex} ]{
                    \left(
                    \rippingInterpol_{\rippingLimitParam}
                    \left(
                        \bar{\membraneHeight}^2
                    \right)
                    -
                    \rippingInterpol_{\rippingLimitParam}
                    \left(
                        \bar{\membraneHeight}^1
                    \right)
                    \right)
                }^2
                \norm[ \lebesgueSet{3}{\cortex} ]{
                    \activeLinkers^1
                }^2
                +
                \frac{1}{2}
                \norm[ \lebesgueSet{2}{\cortex} ]{
                    \difference{\activeLinkers}
                }^2.
            \end{split}
        \end{equation}
        We add \eqref{equ:contraction of active and inactive linkers:difference %
        estimate active linkers} and
        \eqref{equ:contraction of active and inactive linkers:difference %
        estimate inactive linkers} leaving out all non-negative terms on the
        left hand sides:
        \begin{equation*}
          	\begin{split}
                \frac{1}{2}
                \left(
                    \pDiff{t}{}{}
                    \norm[ \lebesgueSet{2}{\cortex} ]{
                        \difference{\activeLinkers}
                    }^2
                    +
                    \pDiff{t}{}{}
                    \norm[ \lebesgueSet{2}{\cortex} ]{
                        \difference{\inactiveLinkers}
                    }^2
                \right)
                &\leq
                \frac{1}{2}
                \norm[ \lebesgueSet{6}{\cortex} ]{
                    \left(
                    \rippingInterpol_{\rippingLimitParam}
                    \left(
                        \bar{\membraneHeight}^2
                    \right)
                    -
                    \rippingInterpol_{\rippingLimitParam}
                    \left(
                        \bar{\membraneHeight}^1
                    \right)
                    \right)
                }^2
                \norm[ \lebesgueSet{3}{\cortex} ]{
                    \activeLinkers^1
                }^2
                \\
                &+
                \frac{1}{4\eps}
                \norm[ \lebesgueSet{4}{\cortex} ]{
                    \left(
                        \rippingInterpol_{\rippingLimitParam}
                        \left(
                            \bar{\membraneHeight}^1
                        \right)
                        -
                        \rippingInterpol_{\rippingLimitParam}
                        \left(
                            \bar{\membraneHeight}^2
                        \right)
                    \right)
                }^2
                \norm[ \lebesgueSet{2}{\cortex} ]{
                    \activeLinkers^1
                }^2
                \\
                &+
                \frac{\repairRate}{2}
                \left(
                \norm[ \lebesgueSet{2}{\cortex} ]{
                    \difference{\inactiveLinkers}
                }^2
                +
                \norm[ \lebesgueSet{2}{\cortex} ]{
                    \difference{\activeLinkers}                        
                }^2
                \right)
                +
                \frac{1}{4\eps}
                \norm[ \lebesgueSet{4}{\cortex} ]{                     
                    \rippingInterpol_{\rippingLimitParam}
                    \left(
                        \bar{\membraneHeight}^2
                    \right)
                }^2
                \norm[ \lebesgueSet{2}{\cortex} ]{                     
                    \difference{\activeLinkers}
                }^2
                \\
                &+
                \frac{1}{2}
                \norm[ \lebesgueSet{2}{\cortex} ]{
                    \difference{\activeLinkers}
                }^2.
            \end{split}
        \end{equation*}
        The Gr\"{o}nwall inequality now implies
        \begin{equation*}
          	\norm[\lebesgueSet{2}{\cortex}]{
              	\difference{\activeLinkers}
            }^2(t)
            +
          	\norm[\lebesgueSet{2}{\cortex}]{
              	\difference{\inactiveLinkers}
            }^2(t)
            \leq
            C(\inactiveLinkersDiffusiv,\cortex,\rippingLimitParam,L_\rippingInterpol)
            \integral{0}{t}{
              	e^{
                  	(t-s)
                    C(\repairRate,\cortex,\inactiveLinkersDiffusiv)
                    \norm[ \lebesgueSet{4}{\cortex} ]{
                        \rippingInterpol_\rippingLimitParam(
                            \bar\membraneHeight^2
                        )
                    }^2
                }
                \norm[\lebesgueSet{6}{\cortex}]{
                  	\difference{\bar\membraneHeight}
                }^2
                \norm[\lebesgueSet{3}{\cortex}]{
                  	\activeLinkers^1
                }^2
            }{s}.
        \end{equation*}
        We multiply the inequality by $ e^{-\rescExpConst t} $ and
        insert $ e^{-\rescExpConst s} e^{\rescExpConst s} $ under the right hand
        side time integral:
        \begin{equation*}
          	\begin{split}
                e^{-\rescExpConst t}
                \left(
                    \norm[\lebesgueSet{2}{\cortex}]{
                        \difference{\activeLinkers}
                    }^2(t)
                    +
                    \norm[\lebesgueSet{2}{\cortex}]{
                        \difference{\inactiveLinkers}
                    }^2(t)
                \right)
                \leq
                &C(\inactiveLinkersDiffusiv,\cortex,\rippingLimitParam,L_\rippingInterpol)
                e^{
                    \finTime
                    C(\repairRate,\cortex,\inactiveLinkersDiffusiv)
                    \norm[ 
                        \bochnerLebesgueSet{\infty}{[0,\finTime]}{
                            \lebesgueSet{4}{\cortex} 
                        }
                    ]{
                        \rippingInterpol_\rippingLimitParam(
                            \bar\membraneHeight^2
                        )
                    }^2
                }
                \cdot
                \\
                &\cdot e^{-\rescExpConst t}
                \integral{0}{t}{
                    e^{\rescExpConst s}
                    e^{-\rescExpConst s}
                    \norm[\lebesgueSet{6}{\cortex}]{
                        \difference{\bar\membraneHeight}
                    }^2
                    \norm[\lebesgueSet{3}{\cortex}]{
                        \activeLinkers^1
                    }^2
                }{s}
                \\
                \leq
                &C(\inactiveLinkersDiffusiv,\cortex,\rippingLimitParam,L_\rippingInterpol)
                e^{
                    \finTime
                    C(\repairRate,\cortex,\inactiveLinkersDiffusiv)
                    \norm[ 
                        \bochnerLebesgueSet{\infty}{[0,\finTime]}{
                            \lebesgueSet{4}{\cortex} 
                        }
                    ]{
                        \rippingInterpol_\rippingLimitParam(
                            \bar\membraneHeight^2
                        )
                    }^2
                }
                \cdot
                \\
                &\cdot 
                \altNorm[
                	\bochnerLebesgueSet{\infty}{[0,\finTime]}{
                      	\lebesgueSet{6}{\cortex}
                    }
                ]{
                    \difference{\bar\membraneHeight}
                }^2
                \integral{0}{t}{
                    e^{\rescExpConst(s-t)}
                    \norm[
                        \lebesgueSet{3}{\cortex}
                    ]{
                        \activeLinkers^1
                    }^2
                }{s}                
                \\
                \leq
                &C(\inactiveLinkersDiffusiv,\cortex,\rippingLimitParam,L_\rippingInterpol)
                e^{
                    \finTime
                    C(\repairRate,\cortex,\inactiveLinkersDiffusiv)
                    \norm[ 
                        \bochnerLebesgueSet{\infty}{[0,\finTime]}{
                            \lebesgueSet{4}{\cortex} 
                        }
                    ]{
                        \rippingInterpol_\rippingLimitParam(
                            \bar\membraneHeight^2
                        )
                    }^2
                }
                \cdot
                \\
                &\cdot 
                \altNorm[
                	\bochnerLebesgueSet{\infty}{[0,\finTime]}{
                      	\lebesgueSet{6}{\cortex}
                    }
                ]{
                    \difference{\bar\membraneHeight}
                }^2
                \left(
                    \integral{0}{t}{
                        e^{2\rescExpConst(s-t)}
                    }{s}
                \right)^{\frac{1}{2}}
                \left(
                    \integral{0}{t}{
                        \norm[
                            \lebesgueSet{3}{\cortex}
                        ]{
                            \activeLinkers^1
                        }^4
                    }{s}
                \right)^{\frac{1}{2}}
            \end{split}
        \end{equation*}
        Due to the a priori bounds on $ \activeLinkers^1 $ and
        the interpolation mentioned in Remark~\ref{remark:interpolation}, 
        we know that
        \begin{equation*}
           \left(
            \integral{0}{t}{
                \norm[
                    \lebesgueSet{3}{\cortex}
                ]{
                    \activeLinkers^1
                }^4
            }{s}
            \right)^{\frac{1}{2}}
            =
            \norm[
            	\bochnerLebesgueSet{4}{[0,\finTime]}{
                  	\lebesgueSet{3}{\cortex}
                }
            ]{
                \activeLinkers^1
            }^2
        \end{equation*}
        is also bounded.
        Hence, choosing $ \rescExpConst $ large enough (depending only
        on problem data and $ \finTime $), we obtain the claimed
        contractive estimate.
     \end{proof}
    \begin{mytheorem}
      	(i) Given any $ \finTime > 0 $, Problem~\ref{problem:height-linker system:variational}
        has a unique weak solution in $ K_\finTime $. 
        
        \revision[third]{%
        (ii) Therefore, given the fact that the bounds in $ K_\finTime $ include
        only problem data, Problem~\ref{problem:height-linker system:%
        variational} is well-posed in terms of the definition of Hadamard.
        }
    \end{mytheorem}
    \begin{proof}
        Let $ \finTime > 0 $ arbitrarily chosen and define $ K_\finTime $ as above.
        It is clear that $ K_\finTime $ is closed in $ \solSpaceHeight_\finTime \times
        \solSpaceActiveLinkers_\finTime \times \solSpaceActiveLinkers_\finTime $.
    	The fixed point operator $ F_\finTime $ maps $ K_\finTime $ into itself
        according to \autoref{lemma:first Banach argument step}.
        Furthermore, $ F_\finTime $ is a contraction
        due to Lemma~\ref{lemma:existence proof:time-dependent:contraction estimate %
        membrane height} and 
        Lemma~\ref{lemma:existence proof:time-dependent:contraction estimate linkers}, so
        the Banach fixed point theorem applies on $ F_\finTime $ and guarantees
        the existence of $ (\membraneHeight,\activeLinkers,\inactiveLinkers) $ 
        such that $ F_\finTime(\membraneHeight,\activeLinkers,\inactiveLinkers) =
        (\membraneHeight,\activeLinkers,\inactiveLinkers) $ what immediately implies
        that $ (\membraneHeight,\activeLinkers,\inactiveLinkers) $ is a
        weak solution of Problem~\ref{problem:height-linker system:variational}.
        Banach's fixed point theorem also guarantees
        uniqueness of such a fixed point, so the weak solution is unique.
    \end{proof}


%% file: Stationary_solutions.tex
\renewcommand{\domain}{\cortex}

In this section, we deal with solutions $ \left(\membraneHeight, \activeLinkers,
\inactiveLinkers \right) $ of \autoref{problem:height-linker system:variational}
with $ \pDiff{t}{}{}\membraneHeight \eqAE
\pDiff{t}{}{}\activeLinkers \eqAE \pDiff{t}{}{} \inactiveLinkers \eqAE 0 $,
which we call \emph{stationary solutions}.
From now on, we only consider the case where $ \heightLinkerSysHeightBF{\cdot}{\cdot} $
is coercive (cf. Remark~\ref{remark:Poincare inequality}) and
$ \linkersSpringConst $, $ \criticalHeight $, and $ \pressure_0 $ shall be time-independent.
\sndRevision[second]{%
As pointed out in Remark~\ref{remark:Poincare inequality}, the coercivity requirement
puts restrictions on the spontaneous mean curvature, i.\,e., $ \spontMeanCurv \in 
(-\infty,-\frac{4}{\cortexRadius}) \cup (0,\infty) $. Physically speaking, this means that
we eiher consider a membrane whose natural tendency is to form a concave shape (negative curvature)
with a curvature of an absolute value of at least $ \frac{4}{\cortexRadius} $, or
a strictly convex shape (positive curvature).
}%
\revision[third]{%
Apart from this, all assumptions on the parameters are the same as in the previous section.%
}
The associated stationary problem reads:
\begin{myproblem}[Stationary variational problem]
    \label{problem:height-linker system:variational:stationary}
    Find $ \membraneHeight \in \sobolevHSet{2}{\cortex}, 
    \activeLinkers, \inactiveLinkers \in 
    \bochnerSobolevHSet{1}{\cortex}{[0,\infty)} $ 
    such that it holds
    \begin{subequations}
        \label{equ:height-linker system:variational:stationary}
        \begin{equation}
            \label{equ:height-linker system:variational:stationary:height}
            \membrStiffn 
            \innerProd[
                \lebesgueSet{2}{\cortex}
            ]{
                \laplacian{\cortex}{}{} \membraneHeight
            }{
                \laplacian{\cortex}{}{} \testFuncMembraneHeight
            } 
            + 
            \effectiveLengthParam 
            \innerProd[
                \bochnerLebesgueSet{2}{\cortex}{\reals^3}
            ]{
                \grad{\cortex}{}{} \membraneHeight
            }{
                \grad{\cortex}{}{} \testFuncMembraneHeight
            } 
            + \anotherMembrParam
            \innerProd[
                \lebesgueSet{2}{\cortex}
            ]{
                \membraneHeight
            }{
                \testFuncMembraneHeight
            }
            =
            - 
            \innerProd[
                \lebesgueSet{2}{\cortex}
            ]{
                \linkersSpringConst
                \activeLinkers \membraneHeight
            }{ 
                \testFuncMembraneHeight
            } 
            +
            \innerProd[
                \lebesgueSet{2}{\cortex}
            ]{\pressure_0}{\testFuncMembraneHeight}
        \end{equation}
        \begin{equation}
            \label{equ:height-linker system:variational:stationary:active linkers}
            \activeLinkersDiffusiv 
            \innerProd[
                \bochnerLebesgueSet{2}{\cortex}{\reals^3}
            ]{
                \grad{\cortex}{}{}\activeLinkers
            }{
                \grad{\cortex}{}{}\testFuncActiveLinkers
            } 
            =
            \repairRate 
            \innerProd[\lebesgueSet{2}{\cortex}]{
                \inactiveLinkers
            }{
                \testFuncActiveLinkers
            } 
            -
            \innerProd[\lebesgueSet{2}{\cortex}]{
                \rippingInterpol_{\rippingLimitParam}
                \left(
                    \membraneHeight
                \right) 
                \activeLinkers
            }{\testFuncActiveLinkers} 
        \end{equation}
        \begin{equation}
            \label{equ:height-linker system:variational:stationary:inactive linkers}
            \inactiveLinkersDiffusiv 
            \innerProd[
                \bochnerLebesgueSet{2}{\cortex}{\reals^3}
            ]{
                \grad{\cortex}{}{}\inactiveLinkers
            }{
                \grad{\cortex}{}{}\testFuncInactiveLinkers
            } 
            =
            -\repairRate 
            \innerProd[\lebesgueSet{2}{\cortex}]{
                \inactiveLinkers
            }{
                \testFuncInactiveLinkers
            } 
            +
            \innerProd[\lebesgueSet{2}{\cortex}]{
                \rippingInterpol_{\rippingLimitParam}
                \left(
                    \membraneHeight
                \right) 
                \activeLinkers
            }{\testFuncInactiveLinkers}
        \end{equation}
    \end{subequations} 
    for all $ \testFuncMembraneHeight \in \sobolevHSet{2}{\cortex},
    \testFuncActiveLinkers, \testFuncInactiveLinkers \in \sobolevHSet{1}{\cortex} $.
\end{myproblem}

\subsection{Basic properties}
\begin{mylemma}
    \label{lemma:boundedness of membrane height}
    Let $ \membraneHeight $ and $ \activeLinkers $
    be parts of a solution to \autoref{problem:height-linker system:variational:%
    stationary}. 
    If $ \activeLinkers \geq 0 \; \almostEvWh $, 
    $ \HTwoNorm{\membraneHeight} $ is bounded by a constant depending only on
    $ \membrStiffn $, $ \effectiveLengthParam $, $ \anotherMembrParam $, $ \pressure_0 $, 
    and the domain $ \domain $.
\end{mylemma}
\begin{proof}
    Testing \eqref{equ:height-linker system:variational:stationary:height} with 
    $ \membraneHeight $, we obtain
    \begin{equation*}
        \membrStiffn 
        \LTwoNorm{\laplacian{\domain}{}{}\membraneHeight}^2 
        + 
        \effectiveLengthParam 
        \norm[ \bochnerLebesgueSet{2}{\domain}{\reals^3} ]{
          	\grad{\domain}{}{}\membraneHeight
        }^2
        + 
        \anotherMembrParam 
        \LTwoNorm{ \membraneHeight }^2
        + 
        \LTwoIP{
          	\linkersSpringConst \membraneHeight \activeLinkers
        }{\membraneHeight}
        = \LTwoIP{\pressure_0}{\membraneHeight}.
    \end{equation*}
    \sndRevision[second]{%
    Due to the coercivity assumption on $ \heightLinkerSysHeightBF{\cdot}{\cdot} $
    we stated at the beginning of the section, we further have%
    }%
    \begin{equation*}
        \alpha
        \norm[\sobolevHSet{2}{\cortex}]{\membraneHeight}^2 
        +
        \innerProd[\lebesgueSet{2}{\cortex}]{\linkersSpringConst \activeLinkers \membraneHeight}{
        \membraneHeight}
        \leq
        \frac{1}{4\eps}
        \LTwoNorm{ \pressure_0 }^2
        +
        \eps \LTwoNorm{ \membraneHeight }^2
    \end{equation*}
    for some $ \alpha > 0 $.
    Choosing $ \eps $ small enough (depending on $ \cortex $ and $ \membrStiffn $,
    $ \effectiveLengthParam$, and $ \anotherMembrParam $) 
    such that $ \eps \LTwoNorm{ \membraneHeight }^2 $
    may be absorbed on the left hand side, 
    we derive
    \begin{equation*}
        \tilde\alpha
        \norm[\sobolevHSet{2}{\cortex}]{\membraneHeight}^2 
        \leq
        \frac{1}{4\eps}
        \LTwoNorm{ \pressure_0 }^2
    \end{equation*}
    with $ \tilde\alpha > 0 $. Since $ \alpha $ also depends only on $ \cortex $, 
    $ \membrStiffn $, $ \effectiveLengthParam $, and $ \anotherMembrParam $,
    the claim now directly follows.
\end{proof}

The special structure of \eqref{equ:height-linker system:variational:stationary:%
active linkers} 
and \eqref{equ:height-linker system:variational:stationary:%
inactive linkers} also gives:
\begin{mylemma}
    \label{lemma:weighted sum of active and inactive linkers is constant}
    Let $ \activeLinkers, \inactiveLinkers \in \sobolevHSet{1}{\eulerCortex} $ 
    be parts of a solution to
    \autoref{problem:height-linker system:variational:stationary}.
    Then 
    \begin{equation*}
        \activeLinkersDiffusiv \activeLinkers + 
        \inactiveLinkersDiffusiv \inactiveLinkers = \totalLinkersMassDens
    \end{equation*}
    for \revision{$ \totalLinkersMassDens \in [0,\infty) $}.
\end{mylemma}
\begin{proof}
    Testing \eqref{equ:height-linker system:variational:stationary:active linkers} 
    and \eqref{equ:height-linker system:variational:stationary:inactive linkers} with
    the same 
    $ \sigma \in \diffSet[c]{\infty}{\domain}{} $ and adding both, we get
    \begin{align*}
        \innerProd[
            \bochnerLebesgueSet{2}{\domain}{\reals^3}
        ]{
            \activeLinkersDiffusiv
            \grad{\domain}{}{}\activeLinkers
            +
            \inactiveLinkersDiffusiv
            \grad{\domain}{}{}\inactiveLinkers
        }{
            \grad{\domain}{}{}\sigma
        } = 0.
    \end{align*}
    On closed Riemannian manifolds, all weak solutions of this problem are smooth and
    only differ up to a constant. 
    Therefore 
    \begin{equation*}
        \activeLinkersDiffusiv \activeLinkers + 
        \inactiveLinkersDiffusiv \inactiveLinkers = \totalLinkersMassDens \; \almostEvWh
    \end{equation*}
    for a constant \revision{$ \totalLinkersMassDens \in [0,\infty) $ since $ \activeLinkers,
    \inactiveLinkers \geq 0 $ a.\,e.}
\end{proof}

With the previous results, we are in the position to state the following observation:
\begin{mylemma}
  	\label{lemma:not everywhere below the critical height}
    Let $ \membraneHeight $ be part 
    of a solution to
    \autoref{problem:height-linker system:variational:stationary}
    with non-negative linker densities.
    If $ \pressure_0 $ is pointwise $ \almostEvWh $ large enough, there exists
    a set $ M $ with two-dimensional Hausdorff measure non-zero such that
    $ \shrinkFunc{\left(\membraneHeight(x) - \criticalHeight\right)}{M} > 0 $
    for a.\,e. $ x \in \cortex $.
\end{mylemma}
\begin{proof}
    Choose an arbitrary function $ \pressure_0 \in 
    \bochnerLebesgueSet{2}{\domain}{[0,\infty)} $
    and consider the problem
    of finding $ \membraneHeight^s \in 
    \sobolevHSet{2}{\domain} $,
    for $ s \in [0,\infty) $, such that
    \begin{equation}
        \label{equ:proof:reaching above critical height:%
        aux problem}
        \membrStiffn 
        \innerProd[
            \lebesgueSet{2}{\domain}
        ]{
            \laplacian{\cortex}{}{}\membraneHeight^s 
        }{
            \laplacian{\cortex}{}{}\testFuncMembraneHeight
        }
        +
        \effectiveLengthParam 
        \innerProd[
            \bochnerLebesgueSet{2}{\domain}{\reals^3}
        ]{
            \grad{\cortex}{}{}\membraneHeight^s 
        }{
            \grad{\cortex}{}{}\testFuncMembraneHeight
        }
        +
        \anotherMembrParam
        \innerProd[
            \lebesgueSet{2}{\domain}
        ]{
            \membraneHeight^s 
        }{
            \testFuncMembraneHeight
        }            
        + 
        \frac{\totalLinkersMassDens}{\activeLinkersDiffusiv} 
        \innerProd[\lebesgueSet{2}{\cortex}]{
            \revision{\linkersSpringConst}
        	\membraneHeight^s
        }{
          	\testFuncMembraneHeight
        }
        =
        s 
        \innerProd[\lebesgueSet{2}{\cortex}]{
          	\pressure_0
        }{
          	\testFuncMembraneHeight
        }
    \end{equation}
    for all $ \testFuncMembraneHeight \in \sobolevHSet{2}{\cortex} $,
    where $ \totalLinkersMassDens $ is the constant of 
    \autoref{lemma:weighted sum of active and inactive linkers is constant}.
    We note, $ \membraneHeight^s \eqAE s \membraneHeight^1 $, which directly implies
    \begin{equation*}
    	\sup_{x\in\domain} \membraneHeight^s(x) = 
    	s \sup_{x\in\domain} \membraneHeight^1(x).
    \end{equation*}
    Assume $ \membraneHeight^1 \leq 0 $. Testing
    \eqref{equ:proof:reaching above critical height:aux problem}
    with $ \membraneHeight^1 $, we find $ \HTwoNorm{\membraneHeight^1} \leq 0 $,
    so $ \membraneHeight^1 \eqAE 0 $, which is no solution to
    \eqref{equ:proof:reaching above critical height:aux problem}. Therefore, there exists
    $ x \in \cortex $ such that $ \membraneHeight^1(x) > 0 $, so 
    $ \sup_{x\in\domain} \membraneHeight^1 > 0 $.

    Since $ \membraneHeight^s $ is continuous, 
    there exists $ x^* \in \domain $ 
    such that $ \membraneHeight^s(x^*) = 
    \sup_{x\in\domain} \membraneHeight^s(x) $. 
    Choose $ s^* $ large enough such that $ \membraneHeight^{s^*}(x^*) > 
    \criticalHeight $; then,
    we have a ball $ B_\delta(x^{*}) \subseteq \domain $ (in the induced subtopology of 
    $ \cortex $), $ \delta > 0 $, where \revision{
    $ \membraneHeight^{s^*}(x) > \criticalHeight $, 
    $ x \in B_\delta(x^{*}) $},
    which is a set of \revision{non-zero two-dimensional Hausdorff measure}.

    Now assume $ \membraneHeight $ being part of a stationary solution to 
    \autoref{problem:height-linker system:variational:stationary}
    with pressure $ s^* \pressure_0 $ and being 
    lower or equal $ \criticalHeight $ almost everywhere.
    Then
    \eqref{equ:height-linker system:variational:stationary:inactive linkers}
    becomes
    \begin{equation*}
      	\inactiveLinkersDiffusiv
      	\innerProd[
			\bochnerLebesgueSet{2}{\domain}{\reals^3}
        ]{
          	\grad{\cortex}{}{} \inactiveLinkers
        }{
          	\grad{\cortex}{}{} \testFuncInactiveLinkers
        }
        +
        \repairRate
        \innerProd[
        	\lebesgueSet{2}{\cortex}
        ]{
          	\inactiveLinkers
        }{
          	\testFuncInactiveLinkers
        }
        = 
        0,
    \end{equation*}
    so $ \inactiveLinkers \eqAE 0 $. Consequently,
    $ \activeLinkersDiffusiv \activeLinkers \eqAE \totalLinkersMassDens $
    for some $ \totalLinkersMassDens \geq 0 $
    (cf. \autoref{lemma:weighted sum of %
    active and inactive linkers is constant}).
    Therefore, $ \membraneHeight $ fulfils
    \eqref{equ:proof:reaching above critical height:aux problem}
    for $ s^* $. But this contradicts the observation
    $ \shrinkFunc{\membraneHeight}{B_\delta(x^{*})} > \criticalHeight $ made above
    and the claim must be true.
\end{proof}

\paragraph{Elimination and reconstruction of the inactive linker density}
Motivated by Lemma~\ref{lemma:weighted sum of active and inactive linkers is constant}, 
we introduce two auxiliary variational problems\revision[third]{, which are
parametrised by $ \totalLinkersMass $ (recall Lemma~\ref{lemma:mass conservation})
or $ \totalLinkersMassDens $, respectively,}
such that for every stationary solutions of 
\autoref{problem:height-linker system:variational} there is an auxiliary
problem that is fulfilled by it
and whose solutions allow for construction of a solution of
\autoref{problem:height-linker system:variational:stationary}.
In case $ \activeLinkersDiffusiv \geq \inactiveLinkersDiffusiv $, we will employ
\begin{myproblem}[Auxiliary problem]
    \label{problem:height-linker system:variational:stationary:simplified one}
    Let $ \totalLinkersMass \in [0,\infty) $.
    Find $ \membraneHeight \in \sobolevHSet{2}{\domain} $ and $ \activeLinkers \in
    \sobolevHSet{1}{\domain} $ such that
    \begin{subequations}
        \begin{equation}
            \heightLinkerSysHeightBF{\membraneHeight}{\testFuncMembraneHeight}
            + \LTwoIP{
                \linkersSpringConst \membraneHeight \activeLinkers
            }{\testFuncMembraneHeight}
            = 
            \LTwoIP{\pressure_0}{\testFuncMembraneHeight}
        \end{equation}
        \begin{equation}
            \label{equ:height-linker system:variational:stationary:simplified one:%
            active linkers}
            \begin{split}
                \heightLinkerSysActiveLinkersBF{\activeLinkers}{\sigma}
                + 
                \LTwoIP{%
                    \left( 
                        \frac{
                            \repairRate 
                            \activeLinkersDiffusiv
                        }{
                            \inactiveLinkersDiffusiv
                        }
                        + 
                        \rippingInterpol_{\rippingLimitParam}
                        \left(
                            \membraneHeight
                        \right)
                    \right)
                    \activeLinkers
                }{%
                    \sigma
                } =
                &\frac{k}{\abs{\domain}} 
                \LTwoIP{
                    \left(
                        \left(
                            \frac{\activeLinkersDiffusiv}{\inactiveLinkersDiffusiv} 
                            - 1
                        \right)
                        \integral{\cortex}{}{\activeLinkers}{x}
                        +
                        \totalLinkersMass
                    \right)
                }{ \sigma}
            \end{split}
        \end{equation}
    \end{subequations}
    for all $ \testFuncMembraneHeight \in \sobolevHSet{2}{\domain} $ and
    $ \sigma \in \sobolevHSet{1}{\domain} $.
\end{myproblem}
In case $ \inactiveLinkersDiffusiv > \activeLinkersDiffusiv $, we will use
\begin{myproblem}[Auxiliary problem]
    \label{problem:height-linker system:variational:stationary:simplified}
    Let $ \totalLinkersMassDens \in \reals $.
    Find $ \membraneHeight \in \sobolevHSet{2}{\domain} $ and $ \activeLinkers \in
    \sobolevHSet{1}{\domain} $ such that
    \begin{subequations}
        \begin{equation}
            \heightLinkerSysHeightBF{\membraneHeight}{\testFuncMembraneHeight}
            + \LTwoIP{
                \linkersSpringConst \membraneHeight \activeLinkers
            }{\testFuncMembraneHeight}
            = 
            \LTwoIP{\pressure_0}{\testFuncMembraneHeight}
        \end{equation}
        \begin{equation}
            \label{equ:height-linker system:variational:stationary:simplified:%
            active linkers}
            \begin{split}
                \heightLinkerSysActiveLinkersBF{\activeLinkers}{\sigma}
                + 
                \LTwoIP{%
                    \left( 
                        \frac{
                            \repairRate 
                            \activeLinkersDiffusiv
                        }{
                            \inactiveLinkersDiffusiv
                        }
                        + 
                        \rippingInterpol_{\rippingLimitParam}
                        \left(
                            \membraneHeight
                        \right)
                    \right)
                    \activeLinkers
                }{%
                    \sigma
                } =
                &\frac{\repairRate}{\inactiveLinkersDiffusiv}
                \LTwoIP{
                  	\totalLinkersMassDens
                }{ \sigma}
            \end{split}
        \end{equation}
    \end{subequations}
    for all $ \testFuncMembraneHeight \in \sobolevHSet{2}{\domain} $ and
    $ \sigma \in \sobolevHSet{1}{\domain} $.
\end{myproblem}
\begin{mylemma}
    \label{lemma:meaning of simplified variational formulation}
    (i) 
    For all stationary solutions $ (\membraneHeight,\activeLinkers,\inactiveLinkers)$ of 
    \autoref{problem:height-linker system:variational}
    with total linker mass $ \totalLinkersMass $
    \autoref{problem:height-linker system:variational:stationary:simplified} is fulfilled
    with $ \totalLinkersMassDens = 
    \frac{\inactiveLinkersDiffusiv}{\abs{\domain}}
    \left(
        \left(
            \frac{\activeLinkersDiffusiv}{\inactiveLinkersDiffusiv} 
            - 1
        \right)
        \integral{\cortex}{}{\activeLinkers}{x}
        +
        \totalLinkersMass
    \right) $ (and therefore also
    \autoref{problem:height-linker system:variational:stationary:simplified one}).

    (ii) In case $ \inactiveLinkersDiffusiv \leq \activeLinkersDiffusiv $,
    all solutions $ (\membraneHeight, \activeLinkers) $
    of \autoref{problem:height-linker system:variational:stationary:simplified one} 
    (with parameter $ \totalLinkersMass $)
    can be extended to a solution $ (\membraneHeight, \activeLinkers, \inactiveLinkers) $ of
    \autoref{problem:height-linker system:variational:stationary}
    such that $ \activeLinkersDiffusiv \activeLinkers + \inactiveLinkersDiffusiv \inactiveLinkers
    \equiv \text{const} $
    and
    $ \integral{\cortex}{}{\activeLinkers+\inactiveLinkers}{x} = \totalLinkersMass $.
    If $ \activeLinkers \geq 0 $ a.\,e., $ \inactiveLinkers \geq 0 $ a.\,e.

    (iii) In case $ \inactiveLinkersDiffusiv > \activeLinkersDiffusiv $,
    all solutions $ (\membraneHeight, \activeLinkers) $
    of \autoref{problem:height-linker system:variational:stationary:simplified} (with
    parameter $ \totalLinkersMassDens $)
    can be extended to a solution $ (\membraneHeight, \activeLinkers, \inactiveLinkers) $ of
    \autoref{problem:height-linker system:variational:stationary}
    such that $  \activeLinkersDiffusiv \activeLinkers + \inactiveLinkersDiffusiv \inactiveLinkers \equiv 
    \totalLinkersMassDens $
    If $ \activeLinkers \geq 0 $ a.\,e., $ \inactiveLinkers \geq 0 $ a.\,e.
\end{mylemma}
\begin{proof}
    (i)
    Let $ \membraneHeight $, $ \activeLinkers $, and $ \inactiveLinkers $
    be stationary solutions of 
    \autoref{problem:height-linker system:variational}.
    According to \autoref{lemma:mass conservation}, it holds
    \begin{equation}
        \label{equ:proof:meaning of simplified variational %
        formulation for homogeneous boundary:mass conservation}
        \totalLinkersMass = \integral{\domain}{}{\activeLinkers + \inactiveLinkers}{x}.
    \end{equation}
    Additionally, due to \autoref{lemma:weighted sum of active and inactive linkers 
    is constant}, we have
    \begin{equation}
        \label{equ:proof:meaning of simplified variational %
        formulation for homogeneous boundary:constant sum}
        \totalLinkersMassDens  = \activeLinkersDiffusiv \activeLinkers + 
        \inactiveLinkersDiffusiv \inactiveLinkers.
    \end{equation}
    Inserting
    \eqref{equ:proof:meaning of simplified variational formulation for 
    homogeneous boundary:mass conservation}
    into the integrated version 
    of \eqref{equ:proof:meaning of simplified variational %
    formulation for homogeneous boundary:constant sum} 
    gives an equation for $ \totalLinkersMassDens $:
    \begin{align*}
        \totalLinkersMassDens &= \frac{1}{\abs{\domain}} 
        \left( 
            \activeLinkersDiffusiv \integral{\domain}{}{\activeLinkers}{x}  
            +
            \inactiveLinkersDiffusiv \integral{\domain}{}{\inactiveLinkers}{x}
        \right)  = \frac{1}{\abs{\domain}}
        \left(
            \activeLinkersDiffusiv \integral{\domain}{}{\activeLinkers}{x}  
            +
            \inactiveLinkersDiffusiv
            \left(
                \totalLinkersMass
                -
                \integral{\domain}{}{\activeLinkers}{x}
            \right)
        \right) \\
        &= \frac{\inactiveLinkersDiffusiv}{\abs{\domain}}
        \left(
            \frac{\activeLinkersDiffusiv}{\inactiveLinkersDiffusiv}
            \integral{\domain}{}{\activeLinkers}{x}
            +
            \totalLinkersMass
            -
            \integral{\domain}{}{\activeLinkers}{x}
          \right)  = \frac{\inactiveLinkersDiffusiv}{\abs{\domain}}
        \left(
            \left( \frac{\activeLinkersDiffusiv}{\inactiveLinkersDiffusiv} - 1 \right)
            \integral{\domain}{}{\activeLinkers}{x} + \totalLinkersMass
        \right).
    \end{align*}
    Inserting this expression and
        $\inactiveLinkers = \frac{1}{\inactiveLinkersDiffusiv} 
      ( 
            \totalLinkersMassDens
            -
            \activeLinkersDiffusiv \activeLinkers$
        )
    into 
    \eqref{equ:height-linker system:variational:stationary:active linkers},
    we get 
    \begin{equation*}
        \activeLinkersDiffusiv \LTwoIP{%
            \grad{\domain}{}{}\activeLinkers
        }{%
            \grad{\domain}{}{}\testFuncActiveLinkers
        }
        + 
        \LTwoIP{%
            \left( 
                \frac{\repairRate \activeLinkersDiffusiv}{\inactiveLinkersDiffusiv}
                + 
                \rippingInterpol_\rippingLimitParam\left(
                    \membraneHeight
                \right)
            \right)
            \activeLinkers
        }{%
            \testFuncActiveLinkers
        } =
        \frac{k}{\abs{\domain}} 
        \left(
            \left(
                \frac{\activeLinkersDiffusiv}{\inactiveLinkersDiffusiv} - 1
            \right)
            \integral{\domain}{}{\activeLinkers}{x}
            +
            \totalLinkersMass
        \right)
        \integral{\domain}{}{\testFuncActiveLinkers}{x}.
    \end{equation*}            

    (ii) Now let $ \left( \membraneHeight, \activeLinkers \right) $ be a solution
    of \autoref{problem:height-linker system:variational:stationary:simplified one}.
    Choose 
    \begin{equation*}
        \totalLinkersMassDens = 
        \frac{\inactiveLinkersDiffusiv}{\abs{\domain}} 
        \left(
            \left(
                \frac{\activeLinkersDiffusiv}{\inactiveLinkersDiffusiv} - 1
            \right)
            \integral{\cortex}{}{\activeLinkers}{x}
            +
            \totalLinkersMass
        \right)
    \end{equation*}
    and set 
    \begin{equation*}
        \inactiveLinkers = \frac{1}{\inactiveLinkersDiffusiv}
        \left( 
        \totalLinkersMassDens - \activeLinkersDiffusiv \activeLinkers
        \right).
    \end{equation*}
    We know
    \begin{equation*}
        \begin{split}
        \activeLinkersDiffusiv \LTwoIP{%
            \grad{\domain}{}{}\activeLinkers
        }{%
            \grad{\domain}{}{}\testFuncActiveLinkers
        }
        + 
        \LTwoIP{%
            \left( 
                \frac{\repairRate \activeLinkersDiffusiv}{\inactiveLinkersDiffusiv}
                + 
                \rippingInterpol_\rippingLimitParam\left(
                    \membraneHeight
                \right)
            \right)
            \activeLinkers
        }{%
            \testFuncActiveLinkers
        } 
        &=
        \frac{\repairRate}{\inactiveLinkersDiffusiv}
        \totalLinkersMassDens
        \integral{\domain}{}{\testFuncActiveLinkers}{x},
        \end{split}
    \end{equation*}
    so
    \begin{equation*}
        \begin{split}
        \activeLinkersDiffusiv \LTwoIP{%
            \grad{\domain}{}{}\activeLinkers
        }{%
            \grad{\domain}{}{}\testFuncActiveLinkers
        }
        + 
        \LTwoIP{%
            \rippingInterpol_\rippingLimitParam\left(
                \membraneHeight
            \right)
            \activeLinkers
        }{%
            \testFuncActiveLinkers
        } 
        &=
        \frac{
          	\repairRate
        }{
          	\inactiveLinkersDiffusiv
        } 
        \left(
            \innerProd[\lebesgueSet{2}{\domain}]{
                \totalLinkersMassDens
                -
                \activeLinkersDiffusiv        
                \activeLinkers
            }{
                \testFuncActiveLinkers
            }     
        \right)
        \\
        &=
        \repairRate
        \innerProd[\lebesgueSet{2}{\domain}]{
            \inactiveLinkers
        }{
            \testFuncActiveLinkers
        }     
        \end{split}
    \end{equation*}            
    and 
    $ \activeLinkers $ satisfies
    \eqref{equ:height-linker system:variational:%
      stationary:active linkers}.
    We further calculate:
    \begin{equation*}
      	\grad{\cortex}{}{}\inactiveLinkers =
        \grad{\cortex}{}{\frac{1}{\inactiveLinkersDiffusiv}\left(\totalLinkersMassDens-
            \activeLinkersDiffusiv\activeLinkers\right)}
        =
        -\grad{\cortex}{}{\frac{\activeLinkersDiffusiv}{\inactiveLinkersDiffusiv}\activeLinkers},
    \end{equation*}
    so
    \begin{equation*}
      	\inactiveLinkersDiffusiv \grad{\cortex}{}{}\inactiveLinkers
        = 
        -
        \activeLinkersDiffusiv \grad{\cortex}{}{}\activeLinkers.
    \end{equation*}
    and
    $ \inactiveLinkers $ fulfills \eqref{equ:height-linker
      system:variational:%
      stationary:inactive linkers}.
    
    A small computation shows
    \begin{equation*}
      	\integral{\cortex}{}{\activeLinkers + \inactiveLinkers}{x}
        =
        \integral{\cortex}{}{
          	\frac{1}{\inactiveLinkersDiffusiv}
            \left(
              	\totalLinkersMassDens
                -
                \activeLinkersDiffusiv \activeLinkers
            \right)
            +
            \activeLinkers
        }{x}
        =
        \frac{\abs{\cortex}}{\inactiveLinkersDiffusiv}
        \totalLinkersMassDens
        +
        \left(
        	1
            -
            \frac{\activeLinkersDiffusiv}{\inactiveLinkersDiffusiv}
        \right)
        \integral{\cortex}{}{
          	\activeLinkers
        }{x}
        =
        \totalLinkersMass.
    \end{equation*}

    Furthermore, if $ \activeLinkers \geq 0 $ a.\,e., a standard maximum principle
    guarantees non-negativity of $ \inactiveLinkers $.

    (iii) The reconstruction of $ \inactiveLinkers $ is just the same as in (ii) with
    $ \totalLinkersMassDens $ being directly given. 
\end{proof}
\begin{remark}
  	In the case $ \inactiveLinkersDiffusiv > \activeLinkersDiffusiv $, 
    the total linkers' mass is given by
    \begin{equation*}
        \totalLinkersMass
        =
      	\integral{\domain}{}{
          	\activeLinkers
            +
            \frac{1}{\inactiveLinkersDiffusiv}
            \left(
              	\totalLinkersMassDens
                -
                \activeLinkersDiffusiv
                \activeLinkers
            \right)
        }{x}
        =
      	\integral{\domain}{}{
            \left(
              	1
                -
                \frac{\activeLinkersDiffusiv}{\inactiveLinkersDiffusiv}
            \right)
          	\activeLinkers
            +
            \frac{1}{\inactiveLinkersDiffusiv}
           	\totalLinkersMassDens
        }{x}.        
    \end{equation*}
    It is not hard to see that
    $$
    	\frac{\abs{\cortex}}{\inactiveLinkersDiffusiv}\totalLinkersMassDens
        \leq
        \totalLinkersMass
        \leq
        \frac{\abs{\cortex}}{\activeLinkersDiffusiv}\totalLinkersMassDens
    $$
    (the latter inequality being due to $ \integral{\cortex}{}{\activeLinkers}{x} 
    \leq \frac{1}{\activeLinkersDiffusiv}\integral{\cortex}{}{\totalLinkersMassDens}{x} $)
    implying that there are stationary solutions with arbitrary small (but non-negative)
    or arbitrary large mass.
    We conjecture existence of stationary solutions
    for \emph{all} non-negative masses $ \totalLinkersMass $. However, it is not
    clear how a surjective map from $ \totalLinkersMassDens $ to $ \totalLinkersMass $
    can be defined---not even a continuous map, so the mean value theorem is not directly applicable---;
    hence, a rigorous argument is still missing.
\end{remark}

\newcommand{\fpOp}{F}
\subsection{Fixed point argument}
    \begin{mytheorem}
      	\label{theorem:existence of stationary solutions}
        There exists a solution $(\membraneHeight,\activeLinkers,\inactiveLinkers)$
        to 
        \autoref{problem:height-linker system:variational:stationary}
        with $ \activeLinkers, \inactiveLinkers \geq 0 $ a.\,e.
    \end{mytheorem}
    \begin{proof}
        1.) Let $ \activeLinkersDiffusiv \geq \inactiveLinkersDiffusiv $.
        According to Lemma~\ref{lemma:meaning of simplified variational formulation},
        it is sufficient to prove existence of solutions $ (\membraneHeight,\activeLinkers) $
        of Problem~\ref{problem:height-linker system:variational:stationary:simplified one}
        where $ \activeLinkers \geq 0 $ a.\,e.

        Fix $ \totalLinkersMass \geq 0 $ and let
        \begin{equation*}
            \begin{split}
            K = \bigg\{
                \left(\membraneHeight,\activeLinkers\right) \in
                \lebesgueSet{\infty}{\domain} \times \lebesgueSet{1}{\domain}
            &\bigg\vert
                \norm[ \lebesgueSet{\infty}{\domain} ]{ \membraneHeight }
                \leq C \LTwoNorm{ \pressure_0 },
                \\
                &0 \leq \activeLinkers \; \almostEvWh, 
                \integral{\domain}{}{\activeLinkers}{x} 
                \leq 
                \totalLinkersMass
            \bigg\},
            \end{split}
        \end{equation*}
        where $ C $ is a constant depending on the constant in \autoref{lemma:%
        boundedness of membrane height} and the embedding constant of
        $ \sobolevHSet{2}{\domain} \imbeddedR[cont] \lebesgueSet{\infty}{\domain} $.
        $ K $ is clearly convex and bounded.
        $ K $ is also closed: Let 
        $ \left( \membraneHeight^n, \activeLinkers^n \right)_{n\in\nats} $ be
        a sequence in $ K $ that converges in 
        $ \lebesgueSet{\infty}{\domain} \times \lebesgueSet{1}{\domain} $.
        Hence, for all $ n \in \nats $, 
        \begin{equation*}
            \integral{\domain}{}{\activeLinkers^n}{x} 
            \leq 
            \totalLinkersMass,
        \end{equation*}
        and, as $ \lebesgueSet{1}{\domain} $-convergence implies convergence of the
        integrals $ \integral{\domain}{}{\activeLinkers^n}{x}
        \specConv{n\to\infty} \integral{\domain}{}{\activeLinkers}{x} $,
        the inequality it preserved in the limit.
        Furthermore, for all $ x \in \domain $ and all $ \eps > 0 $,
        \begin{equation*}
            \integral{\ball{\eps}{x}}{}{\activeLinkers^n}{\xi} \geq 0,
        \end{equation*}
        so, due to $ \lebesgueSet{1}{\domain} $-convergence,
        \begin{equation*}
            \integral{\ball{\eps}{x}}{}{\activeLinkers}{\xi} \geq 0
        \end{equation*}                
        which finally gives (using the Lebesgue point property)
        \begin{equation*}
            \activeLinkers \geq 0 \; \almostEvWh
        \end{equation*}

        We now  define 
        \begin{equation*}
            \funSig{\fpOp}{
                K
            }{
                \sobolevHSet{2}{\domain} \times \sobolevHSet{1}{\domain}
                \subseteq \lebesgueSet{\infty}{\domain} \times \lebesgueSet{1}{\domain}
            }
        \end{equation*}
        (based on \autoref{problem:height-linker system:variational:stationary:simplified one})
        with $ \left( \membraneHeight, \activeLinkers \right) =
        \fpOp\left(\bar{\membraneHeight}, \bar{\activeLinkers}\right) $
        being the functions satisfying
        \begin{subequations}
            \label{equ:proof:existence:problem:variational formulation:%
            homogeneous boundary:stationary:simplified:operator equations}
            \begin{equation}
                \heightLinkerSysHeightBF{\membraneHeight}{\testFuncMembraneHeight}
                + 
                \LTwoIP{
                    \linkersSpringConst \membraneHeight \bar{\activeLinkers}
                }{
                    \testFuncMembraneHeight
                }
                = 
                \LTwoIP{\pressure_0}{\testFuncMembraneHeight}
            \end{equation}
            \begin{equation}
                \label{equ:proof:existence:problem:variational formulation:%
                homogeneous boundary:stationary:simplified:operator equations:%
                active linkers}
                \heightLinkerSysActiveLinkersBF{\activeLinkers}{\sigma}
                + 
                \LTwoIP{%
                    \left( 
                        \frac{\repairRate \activeLinkersDiffusiv}{\inactiveLinkersDiffusiv}
                        + 
                        \rippingInterpol_\rippingLimitParam\left(
                            \bar{\membraneHeight}
                        \right)
                    \right)
                    \activeLinkers
                }{%
                    \sigma
                } =
                \frac{k}{\abs{\cortex}} 
                \left(
                    \left(
                        \frac{\activeLinkersDiffusiv}{\inactiveLinkersDiffusiv} - 1
                    \right)
                    \integral{\cortex}{}{\bar\activeLinkers}{x}
                    +
                    \totalLinkersMass
                \right)
                \integral{\domain}{}{\sigma}{x}
            \end{equation}
        \end{subequations}
        for all $ \testFuncMembraneHeight \in \sobolevHSet{2}{\domain} $
        and all $ \sigma \in \sobolevHSet{1}{\domain} $. 

        It holds $ \fpOp(K) \subseteq K $:
        In order to show non-negativity of $ \activeLinkers $, 
        we test \eqref{equ:proof:existence:%
        problem:variational formulation:homogeneous boundary:stationary:%
        simplified:operator equations:active linkers} with 
        $ \npPart{\left(\activeLinkers\right)} $:
        \begin{align*} 
            \activeLinkersDiffusiv 
            \LTwoNorm{
                \grad{\domain}{}{}\npPart{\left( \activeLinkers \right)}
            }^2
            + 
            \frac{\repairRate \activeLinkersDiffusiv}{\inactiveLinkersDiffusiv}
            \LTwoNorm{%
                \npPart{\left( \activeLinkers \right)}
            }^2 
            \leq
            \revision{-}
            \frac{k}{\abs{\domain}} 
            \left(
                \left(
                    \frac{\activeLinkersDiffusiv}{\inactiveLinkersDiffusiv} - 1
                \right)
                \integral{\cortex}{}{\bar\activeLinkers}{x}
                +
                \totalLinkersMass
            \right)
            \integral{\domain}{}{\npPart{\left( \activeLinkers \right)}}{x}.
        \end{align*}
        By assumption $ \activeLinkersDiffusiv \geq \inactiveLinkersDiffusiv $ and
        the right hand side is always $ \leq 0 $, which implies
        $ \npPart{\left( \activeLinkers \right)} = 0 \; \almostEvWh $
        The appropriate a priori bound of $ \activeLinkers $
        is obtained by testing \eqref{equ:proof:existence:%
        problem:variational formulation:homogeneous boundary:stationary:simplified:%
        operator equations:active linkers} with $ 1 $ 
        and leaving out the gradient term on the left hand side:
        \begin{equation*}
           \frac{\repairRate \activeLinkersDiffusiv}{\inactiveLinkersDiffusiv}
            \integral{\cortex}{}{
                \activeLinkers
            }{x}
            \leq 
            \repairRate
            \left(
                \left(
                    \frac{\activeLinkersDiffusiv}{\inactiveLinkersDiffusiv} - 1
                \right)
                \integral{\cortex}{}{\bar\activeLinkers}{x}
                +
                \totalLinkersMass
            \right)
            \leq
            \repairRate
            \left(
                \left(
                    \frac{\activeLinkersDiffusiv}{\inactiveLinkersDiffusiv} - 1
                \right)
                \totalLinkersMass
                +
                \totalLinkersMass
            \right).
        \end{equation*}
        Dividing both sides by 
        $ \repairRate\frac{\activeLinkersDiffusiv}{\inactiveLinkersDiffusiv} $
        leads to the claimed bound. The bound of $ \membraneHeight $ follows
        directly with the non-negativity of $\bar\activeLinkers $ and
        \autoref{lemma:boundedness of %
        membrane height}.

        $ \fpOp $ is continuous: Let 
        $ \left( \bar{\membraneHeight}^n, \bar{\activeLinkers}^n \right)_{n\in\nats} $ be
        a sequence that converges in $ K $.
        Take an arbitrary subsequence $ \fpOp\left( \bar{\membraneHeight}^{n_k}, 
        \bar{\activeLinkers}^{n_k} \right) = \left( \membraneHeight^{n_k},
        \activeLinkers^{n_k} \right) $.
        We have the a priori bound
        \begin{equation*}
            \HTwoNorm{\membraneHeight^{n_k}} \leq 
            \left(
                D\left(\domain, \membrStiffn, \effectiveLengthParam\right) + 1
            \right) \LTwoNorm{\pressure_0}
        \end{equation*}
        immediately by
        \autoref{lemma:boundedness of %
        membrane height}.
        We also observe
        \begin{align*}
            \frac{\repairRate}{\abs{\cortex}} 
            \left(
                \left(
                    \frac{\activeLinkersDiffusiv}{\inactiveLinkersDiffusiv} - 1
                \right)
                \integral{\cortex}{}{\bar\activeLinkers}{x}
                +
                \totalLinkersMass
            \right)
            \integral{\domain}{}{\activeLinkers^{n_k}}{x}
            &\leq
            \frac{k^2\activeLinkersDiffusiv^2}{4\eps\inactiveLinkersDiffusiv^2}  
            \totalLinkersMass^2
            +
            \eps 
            \left( 
                \avgIntegral{\domain}{}{\activeLinkers^{n_k}}{x}
            \right)^2
            \leq
            \frac{k^2\activeLinkersDiffusiv^2}{4\eps\inactiveLinkersDiffusiv^2} 
            \totalLinkersMass^2
            +
            \eps 
            \avgIntegral{\domain}{}{\left( \activeLinkers^{n_k} \right)^2}{x}.
        \end{align*}
        We can choose $ \eps $ small enough such that $ \frac{\eps}{\abs{\domain}} <
        \frac{\repairRate\activeLinkersDiffusiv}{\inactiveLinkersDiffusiv} $, so
        $ \eps \avgIntegral{\domain}{}{%
            \left( \activeLinkers^{n_k} \right)^2
        }{x} $
        can be absorbed on the left hand side of \eqref{equ:proof:existence:problem:%
        variational formulation:homogeneous boundary:stationary:simplified:%
        operator equations:active linkers}, which gives a uniform bound
        of $ \HOneNorm{\activeLinkers^{n_k}} $.
        So by weak compactness, there are subsequences
        $ \membraneHeight^{n_{k_\ell}} $, $ \activeLinkers^{n_{k_\ell}} $,
        $ \bar\membraneHeight^{n_{k_\ell}} $, $ \bar\activeLinkers^{n_{k_\ell}} $
        such that the integrals in \eqref{equ:proof:existence:problem:%
        variational formulation:homogeneous boundary:stationary:simplified:%
        operator equations} converge and we have for the limits
        \begin{equation}
            \label{equ:proof:existence:problem:variational formulation:%
            homogeneous boundary:stationary:simplified:limits}
            \fpOp\left( \bar{\membraneHeight}, \bar{\activeLinkers} \right) = 
            \left( \membraneHeight, \activeLinkers \right).
        \end{equation}
        Due to unique solvability of
        \eqref{equ:proof:existence:problem:%
        variational formulation:homogeneous boundary:stationary:simplified:%
        operator equations},
        $ \bar{\membraneHeight} $ and $ \bar{\activeLinkers} $ are the same limits for
        all subsequences $ \bar{\membraneHeight}^{n_{k_l}} $, $ \bar{\activeLinkers}^{n_{k_l}} $,
        respectively, so $ \membraneHeight^n, \activeLinkers^n $ converge to
        $ \fpOp(\bar\membraneHeight,\bar\activeLinkers) $.
        But this implies that $ \fpOp $ is continuous.

        $ \fpOp(K) $ is relatively compact: Choose a sequence 
        $ \left(\membraneHeight^n,
        \activeLinkers^n \right) = \fpOp\left(\bar{\membraneHeight}^n, 
        \bar{\activeLinkers}^n \right) $. From the previous calculations, we have
        a bound of $ \HOneNorm{\activeLinkers^n} $ and know
        that $ \HTwoNorm{\membraneHeight^n} $ is bounded a priori.
        The compact impeddings $ \sobolevHSet{1}{\domain} \imbeddedR[comp] 
        \lebesgueSet{1}{\domain} $
        and $ \sobolevHSet{2}{\domain} \imbeddedR[comp] \lebesgueSet{\infty}{\domain} $
        directly give us a convergent subsequence $ \left( \membraneHeight^{n_k}, 
        \activeLinkers^{n_k} \right) $ in $ \lebesgueSet{\infty}{\domain}\times
        \lebesgueSet{1}{\domain} $.

        With these properites of $ K $ and $ \fpOp $, Schauder's 
        fixed point theorem
        applies, so $ \fpOp $ has a fixed point in $ K $.
        This fixed point $ (\membraneHeight, \activeLinkers) $ is a solution 
        of \autoref{problem:height-linker system:variational:stationary:simplified one}, 
        which then may be extended to a solution $ (\membraneHeight, \activeLinkers, \inactiveLinkers) $
        of \autoref{problem:height-linker system:variational:stationary} due to
        \autoref{lemma:meaning of simplified variational formulation}.

        2.) In case $ \activeLinkersDiffusiv < \inactiveLinkersDiffusiv $, the proof
        is similar: We choose $ \totalLinkersMassDens >0 $ and set 
        \begin{equation*}
            \begin{split}
            K = \bigg\{
                \left(\membraneHeight,\activeLinkers\right) \in
                \lebesgueSet{\infty}{\domain} \times \lebesgueSet{1}{\domain}
            &\bigg\vert
                \norm[ \lebesgueSet{\infty}{\domain} ]{ \membraneHeight }
                \leq C \LTwoNorm{ \pressure_0 },
                \\
                &0 \leq \activeLinkers \; \almostEvWh, 
                \integral{\domain}{}{\activeLinkers}{x} 
                \leq 
                \frac{\abs{\cortex}}{\activeLinkersDiffusiv}
                \totalLinkersMassDens
            \bigg\},
            \end{split}
        \end{equation*}   
        and define 
        \begin{equation*}
            \funSig{\fpOp}{
                K
            }{
                \sobolevHSet{2}{\domain} \times \sobolevHSet{1}{\domain}
                \subseteq \lebesgueSet{\infty}{\domain} \times \lebesgueSet{1}{\domain}
            }
        \end{equation*}
        (based on \autoref{problem:height-linker system:variational:stationary:simplified})
        with $ \left( \membraneHeight, \activeLinkers \right) =
        \fpOp\left(\bar{\membraneHeight}, \bar{\activeLinkers}\right) $
        being the functions satisfying
        \begin{subequations}
            \begin{equation*}
                \heightLinkerSysHeightBF{\membraneHeight}{\testFuncMembraneHeight}
                + 
                \LTwoIP{
                    \linkersSpringConst \membraneHeight \bar{\activeLinkers}
                }{
                    \testFuncMembraneHeight
                }
                = 
                \LTwoIP{\pressure_0}{\testFuncMembraneHeight}
            \end{equation*}
            \begin{equation*}
                \heightLinkerSysActiveLinkersBF{\activeLinkers}{\sigma}
                + 
                \LTwoIP{%
                    \left( 
                        \frac{\repairRate \activeLinkersDiffusiv}{\inactiveLinkersDiffusiv}
                        + 
                        \rippingInterpol_\rippingLimitParam\left(
                            \bar{\membraneHeight}
                        \right)
                    \right)
                    \activeLinkers
                }{%
                    \sigma
                } =
                \frac{k}{\inactiveLinkersDiffusiv} 
                \totalLinkersMassDens
                \integral{\domain}{}{\sigma}{x}
            \end{equation*}
        \end{subequations}
        for all $ \testFuncMembraneHeight \in \sobolevHSet{2}{\domain} $
        and all $ \sigma \in \sobolevHSet{1}{\domain} $. 
    \end{proof}	

\subsection{Local exponential stability of non-critical stationary solutions}
	\label{sec:local exponential stability}
    \newcommand{\statMembraneHeight}{\overline{\membraneHeight}}
    \newcommand{\statActiveLinkers}{\overline{\activeLinkers}}
    \newcommand{\statInactiveLinkers}{\overline{\inactiveLinkers}}
    \newcommand{\ellipticOp}{\mathcal{A}}
    \newcommand{\nonlinearity}{F}
    \newcommand{\inhomogeneity}{G}
    \newcommand{\membraneHeightSet}{\sobolevHSet{6}{\domain}}
    \newcommand{\activeLinkersSet}{\sobolevHSet{2}{\domain}}
    \newcommand{\inactiveLinkersSet}{\activeLinkersSet}
    \newcommand{\diffMembraneHeight}{{\membraneHeight}_\Delta}
    \newcommand{\diffActiveLinkers}{{\activeLinkers}_\Delta}
    \newcommand{\diffInactiveLinkers}{{\inactiveLinkers}_\Delta}
    \newcommand{\strongTimeDerivOperator}{\mathcal{L}}
    
    In this section, we will be concerned with the local stability of stationary solutions
    for a specific ripping interpolation function 
    $ \rippingInterpol_{\rippingLimitParam}(\membraneHeight) = 
    \nnPart{ \left(\frac{\membraneHeight - \criticalHeight}{\rippingLimitParam}\right) } $
    for
    solutions below the critical height $ \criticalHeight $.
    We are going to make a linearised stability argument using nonlinear semigroup
    theory (for an introduction see, e.\,g., \cite{Barbu1976}).
    Note that, due to \autoref{lemma:boundedness of membrane height},
    there are stationary solutions $ \statMembraneHeight < \criticalHeight \;
    \almostEvWh $ for sufficiently small $ \pressure_0 $.
    Without loss of generality, we set $ \linkersSpringConst \eqAE 1 $.
    Observe that the operator 
    \revision[third]{%
    \begin{equation*}
      	\begin{split}
            \strongTimeDerivOperator \colon&
                \sobolevHSetMVF{1}{\cortex}
            \rightarrow
                \lebesgueSet[\text{mvf}]{2}{\cortex}
            \\
            x &\mapsto
            \timeDerivOperator
            \left( 
            	x \normal[\domain]{} 
            \right) \cdot \normal[\domain]{}
        \end{split}
    \end{equation*}
    }
    is invertible as it is linear and its kernel is zero-dimensional 
    (to see this, one may test with $ x $ and use the positive-definiteness
    of $ \timeDerivOperator $).
    
    In the following, we will be concered with a strong version of 
    \autoref{problem:height-linker system:variational}:
    \begin{myproblem}
      	\label{problem:height-linker system:strong}
      	Find $ \membraneHeight \in 
        \bochnerSobolevHSet{1}{[0,\finTime]}{ \sobolevHSet{2}{\domain} } \cap
        \bochnerLebesgueSet{2}{[0,\finTime]}{ \membraneHeightSet } $
        and
        $ \activeLinkers $, $ \inactiveLinkers \in 
        \bochnerSobolevHSet{1}{[0,\finTime]}{ \lebesgueSet{2}{\domain} } \cap $\\
        $ \bochnerLebesgueSet{2}{[0,\finTime]}{ \activeLinkersSet } $ such that
        \begin{equation}
            \label{equ:dynamic problem:homogeneous boundary:strong formulation:ODE %
            notation}
            \pDiff{t}{}{}
            \begin{pmatrix}
                \membraneHeight
                \\
                \activeLinkers \\
                \inactiveLinkers
            \end{pmatrix}
           =
           \ellipticOp \left( \membraneHeight, \activeLinkers, \inactiveLinkers \right)
           +
           \nonlinearity\left( \membraneHeight, \activeLinkers, \inactiveLinkers \right)
           + 
           \inhomogeneity
        \end{equation}
        with the operator
        \begin{align*}
            \ellipticOp 
            \colon
                &\membraneHeightSet
                \times
                \activeLinkersSet
                \times
                \inactiveLinkersSet
           \rightarrow
                \sobolevHSet{2}{\domain}
                \times
                \lebesgueSet{2}{\domain}
                \times
                \lebesgueSet{2}{\domain}
           \\
                &\left( \membraneHeight, \activeLinkers, \inactiveLinkers \right)
           \mapsto
                \begin{pmatrix} 
                    \strongTimeDerivOperator^{-1}
                    \left(
                      	-\membrStiffn \laplacian{\domain}{2}{}\membraneHeight
                        +
                        \effectiveLengthParam \laplacian{\domain}{}{}\membraneHeight 
                        +
                        \anotherMembrParam \membraneHeight
                    \right)
                    \\
                    \activeLinkersDiffusiv \laplacian{\domain}{}{}\activeLinkers \\
                    \inactiveLinkersDiffusiv \laplacian{\domain}{}{}\inactiveLinkers
                \end{pmatrix},
        \end{align*}
        which is densely defined and closed,
	    the nonlinearity
        \begin{align*}
            \nonlinearity \colon 
            & \membraneHeightSet \times
            \activeLinkersSet \times 
            \inactiveLinkersSet
            \rightarrow
            \membraneHeightSet \times
            \activeLinkersSet \times 
            \inactiveLinkersSet \\
            &\left( \membraneHeight, \activeLinkers, \inactiveLinkers \right)
            \mapsto
            \begin{pmatrix}
                -\strongTimeDerivOperator^{-1}
                \left( \membraneHeight \activeLinkers \right)
                \\
                \repairRate \inactiveLinkers - \rippingInterpol_{\rippingLimitParam}
                \left(
                    \membraneHeight
                \right) \activeLinkers \\
                -\repairRate \inactiveLinkers + \rippingInterpol_{\rippingLimitParam}
                \left(
                    \membraneHeight
                \right) \activeLinkers
            \end{pmatrix},
        \end{align*}
        and inhomogeneity
        \begin{align*}
            \fun{\inhomogeneity}{\domain}{[0,\infty)}{%
                x
            }{%
                \begin{pmatrix}
                    \strongTimeDerivOperator^{-1}\pressure_0(x) \\
                    0 \\
                    0
                \end{pmatrix}
            }.
        \end{align*}
        for initial data $ \activeLinkers(0,\cdot), \inactiveLinkers(0,\cdot) \in 
        \bochnerLebesgueSet{2}{\domain}{[0,\infty)} $,
        and $ \membraneHeight(0,\cdot) \in \sobolevHSet{2}{\domain} $, and
        $ \pressure_0 \in \sobolevHSet{1}{\domain} $.
    \end{myproblem}
    \begin{myremark}
        Well-posedness of this problem follows with sufficient regularity
        of initial data and the inhomogeneity almost by standard techniques 
        as presented, e.g.\, in \cite[Chapter~7]{Evans2002}:
        Using a Galerkin approach, we take existing solutions
        $ \membraneHeight, \activeLinkers, \inactiveLinkers $
        of \autoref{problem:height-linker system:variational} 
        and express them with eigenfunctions $ \left( w_k \right)_{k\in\nats} $
        of $ \laplacian{\cortex}{}{} $ constituting a Schauder
        basis
        of $ \lebesgueSet{2}{\domain} $. We obtain
        the projections $ \membraneHeight^m(t,\cdot) = \sum_{k=0}^m a_k(t) w_k $,
        $ \activeLinkers^m(t,\cdot) = \sum_{k=0}^m b_k(t) w_k $, and 
        $ \inactiveLinkers^m(t,\cdot) = \sum_{k=0}^m c_k(t) w_k $. 
        The projected height equation is
        \begin{equation*}
          	\LTwoIP{ 
              	\revision{\timeDerivOperator}(
              	\pDiff{t}{}{} \membraneHeight^m
                \normal[\cortex]{}
                )
            }{ 
                \testFuncMembraneHeight
                \normal[\cortex]{}
            }
            + 
            \heightLinkerSysHeightBF{ \membraneHeight^m }{ \testFuncMembraneHeight  }
            =
            -
            \LTwoIP{ \pi_m(\membraneHeight \activeLinkers) }{ \testFuncMembraneHeight }
            +
            \LTwoIP{ \pressure[m]_0 }{ \testFuncMembraneHeight }
        \end{equation*}
        for $ \testFuncMembraneHeight \in \revision{\text{span}\{w_1,\dots,w_m\}} $,
        where $ \pi_m $ is the projection into $ \text{span}\{w_1,\dots,w_m\} $.
        We differentiate in time and test by $ \pDiff{t}{}{}\membraneHeight^m $.
        Integrating
        in time then gives an energy estimate of the 
        $ \bochnerLebesgueSet{2}{[0,\finTime]}{ \sobolevHSet{2}{\domain} } $ norm
        of $ \pDiff{t}{}{} \membraneHeight_m $ (the nonlinearity on the right hand side
        is uniformly controlled due to the already established a priori bounds
        of the solutions).
        $ \lebesgueSet{2}{\domain} $
        regularity for $ \pDiff{t}{}{}\activeLinkers $, $ \pDiff{t}{}{}\inactiveLinkers $,
        is established by testing the projected equations with $ \pDiff{t}{}{}\activeLinkers^m $,
        $ \pDiff{t}{}{}\inactiveLinkers^m $, respectively.
        Control of the nonlinearities also helps with establishing increased
        spatial regularity.

        This result makes us confident that existence theory of
        \eqref{equ:height-linker system:strong}
        might as well be developed using nonlinear semigroups.
    \end{myremark}
    \begin{mylemma}
      	\label{lemma:linearised stability}
      	A stationary solution to \autoref{problem:height-linker system:strong} is locally 
        exponentially stable iff it is an exponentially stable stationary solution to 
        the corresponding linearised system (see below).
    \end{mylemma}
    \begin{proof}
        Due to \cite{Desch+1986}, Theorem\,2.1, the claim follows if
        we can show differentiability in 
        $ \sobolevHSet{2}{\domain} \times \lebesgueSet{2}{\domain}^2 $
        of the nonlinear $ C_0 $ semigroup \revision{generated} by $ \ellipticOp + \nonlinearity $.
        Applying Theorem 3.3 in \cite{Jamal+2014}, it is sufficient to show
        Fr\'{e}chet differentiability of $ \nonlinearity $ in
        $ \sobolevHSet{2}{\domain}\times\lebesgueSet{2}{\domain} \times 
        \lebesgueSet{2}{\domain} $
        on a sufficiently small
        ball around a stationary solution $ \left( \statMembraneHeight, 
        \statActiveLinkers, \statInactiveLinkers \right) $ 
        with
        $ \statMembraneHeight < \criticalHeight \; \almostEvWh $
        and the Lipschitz continuity of the derivative therein.
        Indeed, its Fr\'{e}chet derivative is
        \begin{equation*}
            \diff{
                \left( 
                \membraneHeight, \activeLinkers, \inactiveLinkers
                \right)
            }{}{} \nonlinearity
            \left( 
                \diffMembraneHeight, \diffActiveLinkers, \diffInactiveLinkers
            \right)
            =
            \begin{pmatrix}
                \strongTimeDerivOperator^{-1}\left(
                    -\activeLinkers \diffMembraneHeight 
                    -\membraneHeight \diffActiveLinkers
                \right)
                \\	
                \repairRate \diffInactiveLinkers \\
                -\repairRate \diffInactiveLinkers
            \end{pmatrix}
        \end{equation*}
        for
        $
            \left( \membraneHeight, \activeLinkers, \inactiveLinkers \right)
            \in \ball{r}{%
                    \statMembraneHeight, \statActiveLinkers, \statInactiveLinkers
            }, r > 0,
        $
        \revision{where $ r $ is chosen sufficiently small such that $\membraneHeight <
        \criticalHeight $ a.\,e. (This is possible because 
        $ \sobolevHSet{2}{\domain} \imbeddedR[cont]
        \lebesgueSet{\infty}{\domain} $.)}

        Note, as $ \strongTimeDerivOperator^{-1} $ is linear and bounded, we
        may drop it in the following calculations without loss of generality.
        Consider for $ \left( \diffMembraneHeight, \diffActiveLinkers,
        \diffInactiveLinkers \right) $ the difference quotient
        \begin{align*}
            \frac{%
                \norm[{\sobolevHSet{2}{\domain}}]{%
                    \left( \membraneHeight + \diffMembraneHeight \right)
                    \left( \activeLinkers + \diffActiveLinkers \right)
                    - \membraneHeight \activeLinkers
                    -\activeLinkers \diffMembraneHeight
                    -\membraneHeight \diffActiveLinkers
                }
            }{
                \norm[{%
                    \sobolevHSet{2}{\domain} \times 
                    \lebesgueSet{2}{\domain} \times
                    \lebesgueSet{2}{\domain}
                }]{h}
            } + \\
            \frac{%
                \norm[\lebesgueSet{2}{\domain}]{%
                    \repairRate \left( \inactiveLinkers + \diffInactiveLinkers\right)
                    -
                    \rippingInterpol_{\rippingLimitParam}\left(
                        \membraneHeight + \diffMembraneHeight
                    \right) \left(\activeLinkers + \diffActiveLinkers \right)
                    -\repairRate \inactiveLinkers
                    +
                    \rippingInterpol_{\rippingLimitParam}\left(
                        \membraneHeight
                    \right) \left(\activeLinkers \right)
                    - \repairRate \diffInactiveLinkers
                }  
            }{%
                \norm[{%
                    \sobolevHSet{2}{\domain} \times 
                    \lebesgueSet{2}{\domain} \times
                    \lebesgueSet{2}{\domain}
                }]{h}
            } + \\
            \frac{%
                \norm[\lebesgueSet{2}{\domain}]{%
                    -\repairRate 
                    \left( \inactiveLinkers + \diffInactiveLinkers\right)
                    +
                    \rippingInterpol_{\rippingLimitParam}
          			\left(
                        \membraneHeight - \diffMembraneHeight
                    \right) \left( \activeLinkers + \diffActiveLinkers \right)
                    + \repairRate \inactiveLinkers
                    -
                    \rippingInterpol_{\rippingLimitParam}
          			\left(
                        \membraneHeight
                    \right) \activeLinkers
                    + \repairRate \diffInactiveLinkers
                }                             
            }{%
                \norm[{%
                    \sobolevHSet{2}{\domain} \times 
                    \lebesgueSet{2}{\domain} \times
                    \lebesgueSet{2}{\domain}
                }]{h}
            }.	
        \end{align*}
        After straightforward simplifications, we obtain
        \begin{align*}
            \frac{%
                \norm[{\sobolevHSet{2}{\domain}}]{%
                    \diffMembraneHeight
                    \diffActiveLinkers
                }
            }{
                \norm[{%
                    \sobolevHSet{2}{\domain} \times 
                    \lebesgueSet{2}{\domain} \times
                    \lebesgueSet{2}{\domain}
                }]{h}
            } + 
            2\frac{%
                \norm[\lebesgueSet{2}{\domain}]{%
                    \rippingInterpol_{\rippingLimitParam}
          			\left(
                        \membraneHeight + \diffMembraneHeight
                    \right) \left(\activeLinkers + \diffActiveLinkers \right)
                }  
            }{%
                \norm[{%
                    \sobolevHSet{2}{\domain} \times 
                    \lebesgueSet{2}{\domain} \times
                    \lebesgueSet{2}{\domain}
                }]{h}
            }.
        \end{align*}                    
        (Since $ \membraneHeight < \criticalHeight \; \almostEvWh $, 
        $ \rippingInterpol_{\rippingLimitParam}\left( \membraneHeight \right) $ 
        vanishes $ \almostEvWh $)
        Again, we use the imbedding $ \sobolevHSet{2}{\domain}
        \imbeddedR[cont] \lebesgueSet{\infty}{\domain} $ to conclude
        that if $ \diffMembraneHeight $ is small enough, 
        $ \membraneHeight + \diffMembraneHeight < \criticalHeight \; \almostEvWh $,
        \revision{so the last term vanishes and all together the difference
        quotient goes to zero as $ (\diffMembraneHeight,\diffActiveLinkers,
        \diffInactiveLinkers) \specConv{\sobolevHSet{2}{\cortex}} 0 $;
        hence, $ \nonlinearity $ is Fr\'{e}chet differentiable.}
    \end{proof}


    For a stationary solution $ \left( \statMembraneHeight, \statActiveLinkers, 
    \statInactiveLinkers \right) $
    of \autoref{problem:height-linker system:strong} with $ \statActiveLinkers,
    \statInactiveLinkers \geq 0 $ a.\,e. and
    with mass $ \integral{\domain}{}{ \statActiveLinkers + \statInactiveLinkers }{x} 
    = \totalLinkersMass $ \revision[third]{and $ \bar\membraneHeight < \criticalHeight $ 
    a.\,e.}, the
    linearised system of \autoref{problem:height-linker system:strong} is declared as
    \begin{myproblem}
      	\label{problem:height-linker system:strong:linearised}
      	Find $ \diffMembraneHeight \in 
        \bochnerSobolevHSet{1}{[0,\finTime]}{ \sobolevHSet{2}{\domain} } \cap
        \bochnerLebesgueSet{2}{[0,\finTime]}{ \membraneHeightSet } $
        and
        $ \diffActiveLinkers, \diffInactiveLinkers \in \bochnerSobolevHSet{1}{[0,\finTime]}{ \lebesgueSet{2}{\domain} } 
        \cap \bochnerLebesgueSet{2}{[0,\finTime]}{ \activeLinkersSet } $ such that
        \begin{equation*}
            \pDiff{t}{}{}
            \begin{pmatrix}
                \diffMembraneHeight \\
                \diffActiveLinkers \\
                \diffInactiveLinkers
            \end{pmatrix}
            =
            \ellipticOp 
            \left(
                \diffMembraneHeight, \diffActiveLinkers, \diffInactiveLinkers
            \right)
            +
            \diff[
            	\statMembraneHeight, \statActiveLinkers, \statInactiveLinkers
            ]{}{}{\nonlinearity}
            \left( 
                \diffMembraneHeight, \diffActiveLinkers, \diffInactiveLinkers
            \right)
        \end{equation*}
        or, equivalently,
        \begin{subequations}
            \label{equ:linearised dynamic problem:homogeneous boundary:%
            strong formulation}
            \begin{equation}
                \label{equ:linearised dynamic problem:homogeneous boundary:%
                strong formulation:membrane height}
                \pDiff{t}{}{ \diffMembraneHeight }
                +
                \strongTimeDerivOperator^{-1}
                \left(
                    \membrStiffn 
                    \laplacian{\domain}{2}{}\diffMembraneHeight
                    - 
                    \effectiveLengthParam \laplacian{\domain}{}{}\diffMembraneHeight 
                    +
                    \anotherMembrParam \diffMembraneHeight
                \right)
                =
                -
                \strongTimeDerivOperator^{-1}
                \left(
                    \statActiveLinkers \diffMembraneHeight
                    + \statMembraneHeight \diffActiveLinkers
                \right)
            \end{equation}
            \begin{equation}
                \label{equ:linearised dynamic problem:homogeneous boundary:%
                strong formulation:active linkers}
                \pDiff{t}{}{ \diffActiveLinkers }
                - \activeLinkersDiffusiv \laplacian{\domain}{}{}\diffActiveLinkers
                =
                \repairRate \diffInactiveLinkers
            \end{equation}
            \begin{equation}
                \label{equ:linearised dynamic problem:homogeneous boundary:%
                strong formulation:inactive linkers}
                \pDiff{t}{}{ \diffInactiveLinkers }
                - \inactiveLinkersDiffusiv \laplacian{\domain}{}{}\diffInactiveLinkers
                =
                - \repairRate \diffInactiveLinkers
            \end{equation}
        \end{subequations}
        with initial values
        \begin{equation}
            \label{equ:linearised dynamic problem:homogeneous boundary:%
            strong formulation:initial conditions}
            \begin{split}
                \diffMembraneHeight (0) = \membraneHeight^0 - \statMembraneHeight, \;
                \diffActiveLinkers (0) = \activeLinkers^0 - \statActiveLinkers, \;
                \diffInactiveLinkers (0) = \inactiveLinkers^0 - \statInactiveLinkers,
          \end{split}
        \end{equation}
        such that
        $ \activeLinkers^0, \inactiveLinkers^0 \geq 0 $.
    \end{myproblem}
    We start by showing that the inactive linkers decay exponentially:
    \begin{mylemma}
        \label{lemma:stability of stationary solutions:%
        exponential decay of inactive linkers}
        Let $ \diffInactiveLinkers \in \inactiveLinkersSet $ be part
        of a \revision{triple} solving
        \autoref{problem:height-linker system:strong:linearised}.
        Then, 
        \begin{equation*}
            \LTwoNorm{\diffInactiveLinkers(t)} \leq \exp\left(-\omega t\right)
            \LTwoNorm{\diffInactiveLinkers(0)}
        \end{equation*}
        for $ \omega \in (0,\infty) $.
    \end{mylemma}
    \begin{proof}
        The operator 
        \begin{align*}
            A\colon
            \inactiveLinkersSet \longrightarrow \lebesgueSet{2}{\domain},
            u \mapsto 
            - \inactiveLinkersDiffusiv \laplacian{\domain}{}{}u 
            + \repairRate u
        \end{align*}
        is self-adjoint and its resolvent $ R_\lambda(A) $ is bounded by
        $ \frac{1}{\lambda + \repairRate\revision{\wedge\inactiveLinkersDiffusiv}} $ 
        for all $ \lambda > -\repairRate\revision{\wedge\inactiveLinkersDiffusiv} $:
        We have the energy estimate $ \revision{\repairRate\wedge\inactiveLinkersDiffusiv}
        \norm[\sobolevHSet{1}{\domain}]{ u }^2
        \leq \inactiveLinkersDiffusiv
        \LTwoIP{\grad{\domain}{}{}u}{\grad{\domain}{}{}u}
        + \repairRate \LTwoIP{u}{u} = c_i(u,u) $, so
        the bilinear form $ c_i + \lambda \LTwoIP{u}{u} $ is 
        coercive for all $ \lambda > - \repairRate\revision{\wedge\inactiveLinkersDiffusiv} $, 
        so
        $ c_i(u,\varphi) + \lambda \LTwoIP{u}{\varphi} = \LTwoIP{f}{\varphi} $
        has unique solution for all $ f \in \lebesgueSet{2}{\domain} $.
        As $ u = R_\lambda(A)f $, we have the estimate $ \LTwoNorm{
        R_\lambda(A)f } \leq \frac{1}{\lambda + 
        \repairRate\revision{\wedge\inactiveLinkersDiffusiv}} \LTwoNorm{f} $.

        Therefore, $ \spectrum{}{A} \subseteq (-\infty, -\repairRate] $.
        So $ A $ generates an analytic semigroup $ T $
        (\cite[p.\,105, Corollary\,3.7]{EngelAndNagel2000}).
        This also implies the growth bound of the solution
        $ \diffInactiveLinkers(t) = T(t) \diffInactiveLinkers(0) $
        (\cite[p.\,416, Theorem\,12.33]{RenardyAndRogers2004}).
    \end{proof}

    Despite equations 
    \eqref{equ:linearised dynamic problem:homogeneous boundary:%
    strong formulation:active linkers} and    
    \eqref{equ:linearised dynamic problem:homogeneous boundary:%
    strong formulation:inactive linkers}
    looking very symmetric, their decaying behaviour is not. \revision[third]{To show
    exponential decay for the acctive linkers, we require
    an additional condition: The initial values}
    \eqref{equ:linearised dynamic problem:homogeneous boundary:%
    strong formulation:initial conditions} to be chosen such that
    $ \integral{\domain}{}{ \activeLinkers^0 + \inactiveLinkers^0 }{x} = 
    \totalLinkersMass $, so
    $ \integral{\domain}{}{ \diffActiveLinkers + \diffInactiveLinkers }{x} = 0 $,
    which we call the \emph{mass conservation property}.
    \begin{mylemma}
        \label{lemma:stability of stationary solutions:%
        exponential decay of integral of active linkers difference}
        Let $ \diffActiveLinkers $ be part of a \revision{triple} solving
        \autoref{problem:height-linker system:strong:linearised}.
        Then, 
        \begin{equation*}
            \abs{\integral{\domain}{}{\diffActiveLinkers}{x}} \leq 
            D \exp\left(-\omega t \right)
        \end{equation*}
        where $ \omega > 0 $ is the same as in \autoref{lemma:stability of %
        stationary solutions:exponential decay of inactive linkers}
        and $ D \in (0,\infty) $.
    \end{mylemma}
    \begin{proof}
        From the mass conservation property, we have 
        \begin{equation*}
            \integral{\domain}{}{%
                \diffActiveLinkers + \diffInactiveLinkers
            }{x} = 0,
        \end{equation*}
        so, using \autoref{lemma:stability of stationary solutions:%
        exponential decay of inactive linkers}, we find
        \begin{align*}
            \abs{\integral{\domain}{}{\diffActiveLinkers}{x}} \leq
            \integral{\domain}{}{\abs{\inactiveLinkers}}{x} \leq
            D' \LTwoNorm{\inactiveLinkers} \leq
            D \exp\left( -\omega t \right)
        \end{align*}
        for $ D', D \in (0,\infty) $.
    \end{proof}

    For $ \LTwoNorm{\diffActiveLinkers(t,\cdot)} $, we have at least a time-uniform bound:
    \begin{mylemma} 
        \label{lemma:stability of stationary solutions:%
        boundedness of active linkers difference}
        Let $ \diffActiveLinkers $ be part of a \revision{triple} solving
        \autoref{problem:height-linker system:strong:linearised}.
        Then $ \diffActiveLinkers $ is bounded in $ \LTwoNorm{\cdot} $
        uniformly for all $ t \in [0,\infty) $.
    \end{mylemma}
    \begin{proof}
        The operator
        \begin{equation*}
            \fun{B}{%
                \activeLinkersSet
            }{%
                \lebesgueSet{2}{\domain}
            }{%
                \diffActiveLinkers
            }{%
                -\activeLinkersDiffusiv \laplacian{\domain}{}{}\diffActiveLinkers
            }
        \end{equation*}
        is self-adjoint and it's resolvent $ R_\lambda(B) $ is bounded 
        by $ \frac{1}{\lambda} $ for $ \lambda > 0 $. Therefore, 
        $ \spectrum{}{B} \subseteq (-\infty, 0] $ and so it generates a
        strongly continuous (even analytic) semigroup
        bounded by $ 1 $.
        As $ \diffActiveLinkers $ solves \eqref{equ:linearised dynamic problem:%
        homogeneous boundary:strong formulation:active linkers},
        it has representation as 
        \begin{equation*}
            \diffActiveLinkers(t) = 
            T(t) \diffActiveLinkers(0) 
            +
            \revision{\repairRate}
            \integral{0}{t}{T(s-t)\diffInactiveLinkers(s)}{s},
        \end{equation*}
        where $ T $ is the semigroup generated by $ B $.
        Consequently,
        \begin{equation*}
            \LTwoNorm{\diffActiveLinkers(t)} 
            \leq
            \LTwoNorm{\diffActiveLinkers(0)} 
            +
            \revision{\repairRate}
            \integral{0}{t}{\exp(-\omega s) \LTwoNorm{\inactiveLinkers(0)}}{s},
        \end{equation*}
        so we finally have
        \begin{equation*}
            \LTwoNorm{\diffActiveLinkers(t)} 
            \leq 
            \LTwoNorm{\diffActiveLinkers(0)} 
            - 
            \revision{\repairRate}
            \frac{\LTwoNorm{\inactiveLinkers(0)}}{\omega} 
            \left( \exp(-\omega t) - 1 \right).
        \end{equation*}
    \end{proof}

    Combining these lemmas, we even find exponential decay for the active linkers:
    \begin{mylemma}
        \label{lemma:stability of stationary solutions:exponential decay of active linkers}
        Let $ \diffActiveLinkers $ be part of a \revision{triple} solving
        \autoref{problem:height-linker system:strong:linearised}.
        Then
        \begin{equation*}
            \LTwoNorm{\diffActiveLinkers} \leq E_a \exp\left( -\alpha_a t \right)
        \end{equation*}
        for $ \alpha_a, E_a \in (0, \infty) $.
    \end{mylemma}
    \begin{proof}
        Test \eqref{equ:linearised dynamic problem:homogeneous boundary:%
        strong formulation:active linkers} with $ \diffActiveLinkers $ to get
        \begin{equation*}
            \frac{1}{2} \pDiff{t}{}{}\LTwoNorm{\diffActiveLinkers}^2 +
            \activeLinkersDiffusiv \LTwoNorm{\grad{\domain}{}{}\diffActiveLinkers}^2
            = 
            \repairRate \LTwoIP{\diffInactiveLinkers}{\diffActiveLinkers}.
        \end{equation*}
        Next, we apply Poincar\'{e}'s inequality (Neumann type) and Young's inequality (the modulus $ \eps $ 
        is specified below):
        \begin{equation}
            \label{equ:stability of stationary solutions:%
            proof:corollary:exponential decay of active linkers difference:%
            bound for time derivative}
            \frac{1}{2} \pDiff{t}{}{}\LTwoNorm{\diffActiveLinkers}^2
            \leq 
            \frac{\repairRate^2}{4\eps} \LTwoNorm{\diffInactiveLinkers}^2 +
            \eps \LTwoNorm{\diffActiveLinkers}^2
            - \frac{\activeLinkersDiffusiv}{\poincConst} 
            \LTwoNorm{\diffActiveLinkers
            - \avgIntegral{\domain}{}{\diffActiveLinkers}{x}}^2,
        \end{equation}
        where $ \poincConst $ is the appropriate Poincar\'{e} constant.
        Observe, due to \autoref{lemma:stability of stationary solutions:%
        exponential decay of integral of active linkers difference},
        that 
        \begin{equation*}
            \abs{ \LTwoNorm{%
                \diffActiveLinkers - \avgIntegral{\domain}{}{%
                    \diffActiveLinkers
                }{x}
            } 
            - \LTwoNorm{%
                \diffActiveLinkers
            } }
            \leq 
            \LTwoNorm{ \avgIntegral{\domain}{}{\diffActiveLinkers}{x} }
            \leq 
            C_1 \exp\left( -\omega t \right)
        \end{equation*}
        for a constant $ C_1 \geq 0 $.
        Therefore, 
        \begin{equation*}
            \LTwoNorm{%
                \diffActiveLinkers 
                - \avgIntegral{\domain}{}{%
                    \diffActiveLinkers
                }{x}
            } 
            =
            \LTwoNorm{%
                \diffActiveLinkers 
            }
            +
            \delta(t)
        \end{equation*}
        with $ \abs{\delta(t)} \leq C_2 \exp\left( - \omega t \right) $
        for some $ C_2 \geq 0 $.
        Squaring the terms, it follows
        \begin{equation*}
            \LTwoNorm{%
                \diffActiveLinkers 
                - \avgIntegral{\domain}{}{%
                    \diffActiveLinkers
                }{x}
            }^2
            =
            \LTwoNorm{%
                \diffActiveLinkers
            }^2
            + 2 \delta(t) 
            \LTwoNorm{\diffActiveLinkers} 
            + \delta(t)^2,
        \end{equation*}
        so, as $ \LTwoNorm{\diffActiveLinkers} $ is bounded 
        (cf. \autoref{lemma:stability of stationary solutions:%
        boundedness of active linkers difference}), 
        we have
        \begin{equation*}
            \LTwoNorm{%
                \diffActiveLinkers 
                - \avgIntegral{\domain}{}{%
                    \diffActiveLinkers
                }{x}
            }^2
            =
            \LTwoNorm{%
                \diffActiveLinkers
            }^2
            + \zeta(t)
        \end{equation*}
        with $ \abs{\zeta(t)} \leq C_3 \exp\left( - \omega t \right) $
        for some $ C_3 \geq 0 $.

        Substitution into \eqref{equ:stability of stationary solutions:%
        proof:corollary:exponential decay of active linkers difference:%
        bound for time derivative} gives
        \begin{align*}
            \frac{1}{2} \pDiff{t}{}{}\LTwoNorm{\diffActiveLinkers}^2
            &\leq 
            \frac{\repairRate^2}{4\eps} \LTwoNorm{\diffInactiveLinkers}^2 +
            \left( \eps - \frac{\activeLinkersDiffusiv}{\poincConst} \right)
            \LTwoNorm{\diffActiveLinkers}^2
            - \frac{\activeLinkersDiffusiv}{\poincConst} \zeta(t) \\
            &\leq
            C_4
            \left( 
                \frac{\repairRate^2}{4\eps} \LTwoNorm{\diffInactiveLinkers(0)}^2
                + \frac{\activeLinkersDiffusiv}{\poincConst}
            \right) 
            \exp\left(- \omega t \right)
            + \left( \eps - \frac{\activeLinkersDiffusiv}{\poincConst} \right)
            \LTwoNorm{\diffActiveLinkers}^2,
        \end{align*}
        where $ C_4 \geq 0 $.

        Set $ E(\eps) = 
        2 C_4 \left( \frac{\repairRate^2}{4\eps} \LTwoNorm{\diffInactiveLinkers(0)}^2
        + \frac{\activeLinkersDiffusiv}{\poincConst} \right )$, 
        $ \beta(\eps) = 2 \left( \eps - \frac{\activeLinkersDiffusiv}{\poincConst} \right) $
        and apply Gronwall's inequality:
        \begin{align*}
            \LTwoNorm{\diffActiveLinkers}^2 &\leq
            \LTwoNorm{\diffActiveLinkers(0)}^2 
            \exp\left( \beta(\eps)t \right)
            +
            \integral{0}{t}{%
                E(\eps) \exp\left(-\omega s \right)
                \exp\left( 
                    \beta(\eps) (t-s)
                \right)
            }{s} \\
            &= 
            \LTwoNorm{\diffActiveLinkers(0)}^2 
            \exp\left( \beta(\eps)t \right)
            +
            \exp\left( \beta(\eps) t \right)
            \integral{0}{t}{%
                E(\eps) \exp\left( - \left( \omega + \beta(\eps) \right) s \right)
            }{s}.
        \end{align*}
        The parameter $ \eps $ is still free, and we choose it 
        such that $ -\omega < \beta(\eps) < 0 $.
        With $ E_a = \LTwoNorm{\diffActiveLinkers(0)}^2 +
        \integral{0}{t}{%
            E(\eps) \exp\left( - \left( \omega + \beta(\eps) \right) s \right)
        }{s} $ and $ \alpha_a = \revision{-}\beta(\eps) $, the claim follows.
    \end{proof}

    The last step to showing local exponential stability of \autoref{problem:height-linker system:%
    strong:linearised} is showing exponential decay of the height difference:
    \begin{mylemma}
        \label{lemma:stability of stationary solutions:exponential decay of membrane height}
        Let $ \diffMembraneHeight $ be part of a \revision{triple} solving
        \autoref{problem:height-linker system:strong:linearised}.
        Then
        \begin{equation*}
            \HTwoNorm{\diffMembraneHeight} \leq 
            E_{\membraneHeight} \exp\left( - \alpha_{\membraneHeight} t \right)
        \end{equation*}
        for $ E_{\membraneHeight}, \alpha_{\membraneHeight} \in (0,\infty) $.
    \end{mylemma}
    \begin{proof}
        First, we observe that for stationary solutions with 
        $ \statMembraneHeight < \criticalHeight \; \almostEvWh $,
        there are no inactive linkers, i.\,e., $ \statInactiveLinkers \eqAE 0 $. 
        \revision[third]{(This can can be seen in the last paragraph of the proof to
        Lemma~\ref{lemma:not everywhere below the critical height}: Since
        $ \rippingInterpol_\rippingLimitParam(\statMembraneHeight) = 0 $, 
        the right hand side of the inactive linkers equation is zero
        and therefore testing with $ \statInactiveLinkers $ shows that
        $ \norm[\sobolevHSet{1}{\cortex}]{\statInactiveLinkers}^2 = 0 $.)}
        This has the critical implication that $ \statActiveLinkers $ is constant 
        in space (see
        \autoref{lemma:weighted sum of active and inactive linkers is constant}).
        Second, recall that the operator $ \timeDerivOperator $ is positive and 
        self-adjoint, so we have positive square root $ \rootTimeDerivOperator^2 = 
        \timeDerivOperator $.
        We apply $ \strongTimeDerivOperator $ on both
        sides of \eqref{equ:linearised dynamic problem:homogeneous boundary:%
        strong formulation:membrane height} and test 
        with $ \diffMembraneHeight $. Then, we make use of Young's inequality
        and leave out the terms of the bilinear form on the left hand side
        \begin{align*}
            \frac{1}{2} 
            \pDiff{t}{}{}
            \LTwoNorm{\revision{\rootTimeDerivOperator}\diffMembraneHeight}^2
            &\leq 
            \left(-\statActiveLinkers + \eps \right) 
            \LTwoNorm{\diffMembraneHeight}^2
            + 
            \frac{\revision{1}}{4\eps} 
            \norm[\lebesgueSet{\infty}{\domain}]{\statMembraneHeight}^2
            \LTwoNorm{\diffActiveLinkers}^2 
            \\
            &\leq
            \beta(\eps) 
            \LTwoNorm{\revision{\rootTimeDerivOperator}\diffMembraneHeight}^2
            + E(\eps) \exp\left( - 2 \alpha_a t \right),
        \end{align*}
        where $ E(\eps) = \frac{ E_a^2 
        \norm[\lebesgueSet{\infty}{\domain}]{\statMembraneHeight}^2 }{4\eps} $ 
        and 
        $ \beta(\eps) = \revision{\posDefConst^{-1}}(\eps - \statActiveLinkers) $. 
       
        Application of Gr\"{o}nwall's inequality leads to
        \begin{align*}
            \LTwoNorm{\revision{\rootTimeDerivOperator}\diffMembraneHeight}^2(t)
            &\leq
            2 \LTwoNorm{\revision{\rootTimeDerivOperator}\left(\diffMembraneHeight(0)\right)}^2
            \exp\left( \beta(\eps) t \right)
            + 
            \integral{0}{t}{%
                E(\eps)\exp(- 2 \alpha_a s)
                \exp( \beta(\eps)(t-s) )
            }{s} \\
            &=
            2 \LTwoNorm{\revision{\rootTimeDerivOperator}\left(\diffMembraneHeight(0)\right)}^2
            \exp\left( \beta(\eps) t \right)
            + 
            \exp(\beta(\eps) t)
            \integral{0}{t}{%
                E(\eps)\exp(- (2 \alpha_a + \beta(\eps)) s)
            }{s}                    
        \end{align*}
        As in the proof of \autoref{lemma:stability of stationary solutions:%
        exponential decay of active linkers}, we choose 
        $ \eps $ such that $ -2\alpha_a < \beta(\eps) < 0 $.
        Going back to \eqref{equ:linearised dynamic problem:homogeneous boundary:%
        strong formulation:membrane height}, we find
        \begin{equation*}
            \sndRevision[first]{\heightLinkerSysHeightBF{\diffMembraneHeight}{\diffMembraneHeight}}
            \leq
            -\frac{1}{2}\pDiff{t}{}{}
            \LTwoNorm{
              	\revision{\rootTimeDerivOperator}
              	\diffMembraneHeight
            }^2
            - \statActiveLinkers \LTwoNorm{\diffMembraneHeight}^2(t)
            + \norm[\lebesgueSet{\infty}{\domain}]{\statMembraneHeight}(t)
            \LTwoNorm{\diffActiveLinkers}(t)
            \LTwoNorm{\diffMembraneHeight}(t).
        \end{equation*}
        We see that the right hand side consists only of exponentially decaying terms,
        which gives the decay rate for $ \diffMembraneHeight $ in 
        $ \HTwoNorm{\cdot} $.
    \end{proof}

    \begin{mytheorem}
        Every stationary solution $ \left( \statMembraneHeight, \statActiveLinkers,
        \statInactiveLinkers \right) $
        of \autoref{problem:height-linker system:strong} with 
        $ \statMembraneHeight < \criticalHeight \; \almostEvWh $
        is locally exponentially stable under disturbance that fulfills
        the mass equality condition, i.\,e., 
        a solution $ \left( \membraneHeight, \activeLinkers, \inactiveLinkers \right) $
        of \autoref{problem:height-linker system:strong} 
        in a sufficiently small neighbourhood of $ \left( \statMembraneHeight, \statActiveLinkers,
        \statInactiveLinkers \right) $
        with
        $ \integral{\domain}{}{\activeLinkers^0 + \inactiveLinkers^0}{x} =
        \integral{\domain}{}{\statActiveLinkers + \statInactiveLinkers}{x} $
        converges exponentially fast in time to $ \left( \statMembraneHeight, \statActiveLinkers,
        \statInactiveLinkers \right) $.
    \end{mytheorem}
    \begin{proof}
      	The linearised problem of \autoref{problem:height-linker system:strong} 
        in a sufficiently small neighbourhood around
        $ \left( \statMembraneHeight, \statActiveLinkers,
        \statInactiveLinkers \right) $ is \autoref{problem:height-linker system:strong:linearised}.
        Due to \autoref{lemma:linearised stability}, we only need to show 
        exponential stability of \autoref{problem:height-linker system:strong:linearised} 
        in zero.
        But this is follows from the results in \autoref{lemma:stability of stationary solutions:%
        exponential decay of inactive linkers},
        \autoref{lemma:stability of stationary solutions:exponential decay of active linkers}, 
        \autoref{lemma:stability of stationary solutions:exponential decay of membrane %
        height} and we are finished.
    \end{proof}


%% file: Singular_Limit_cases.tex
\newcommand{\rippingInterpolLimit}{\rippingInterpol_0}
We are going to have a closer look at stationary solutions
in the limit $ \rippingLimitParam \searrow 0 $. This way, we also
rediscover the model of
\cite{Lim+2012} as specialisation of our model.
Henceforth, we restrict to the special disconnection rate
$ \rippingInterpol_\rippingLimitParam(\membraneHeight) =
\nnPart{ \left( \frac{\membraneHeight - \criticalHeight}{\rippingLimitParam} \right) } $.
\subsection{Singular limit of the stationary system}
    \begin{mytheorem}
        Let $ \left( \membraneHeight^{\rippingLimitParam}, 
        \activeLinkers^{\rippingLimitParam}, \inactiveLinkers^{\rippingLimitParam} \right) $ 
        be a solution to \autoref{problem:height-linker system:variational:stationary} 
        for the parameter $ \rippingLimitParam $.
        For every $ \series{\rippingLimitParam} $ with
        $ \limit{n} \rippingLimitParam_n = 0 $, there
        exists a subsequence $ \left( \rippingLimitParam_{n_k} \right)_{k\in\nats} $
        with $ \membraneHeight^{\rippingLimitParam_{n_k}} \specConv[w]{%
        \sobolevHSet{2}{\domain}} \membraneHeight^0 $,
        $ \activeLinkers^{\rippingLimitParam_{n_k}} \specConv[w]{%
            \sobolevHSet{1}{\domain}} \activeLinkers^0 $,
        and $ \inactiveLinkers^{\rippingLimitParam_{n_k}} \specConv[w]{%
            \sobolevHSet{1}{\domain}} \inactiveLinkers^0 $
        such that 
        \begin{subequations}
            \begin{equation}
                \label{equ:variational formulation:homogeneous boundary:%
                stationary:singular limit:membrane height}
                \heightLinkerSysHeightBF{ \membraneHeight^0 }{ \testFuncMembraneHeight }
                =
                -\LTwoIP{
                  	\linkersSpringConst
                  	\membraneHeight^0 
                    \activeLinkers^0
                }{
                  	\testFuncMembraneHeight
                }
                +
                \LTwoIP{
                  	\pressure_0
                }{
                  	\testFuncMembraneHeight
                }
            \end{equation}
            \begin{equation}
                \label{equ:height-linker system:variational:stationary:singular limit:%
                active linkers}
                \heightLinkerSysActiveLinkersBF{
                  	\activeLinkers^0
                }{
                  	\testFuncActiveLinkers
                }
                =
                \repairRate 
                \innerProd[\lebesgueSet{2}{\cortex}]{
                    \inactiveLinkers^0
                }{
                    \testFuncActiveLinkers
                } 
                -
                \innerProd[
                	\lebesgueSet{2}{\domain}
                ]{
                	\rippingInterpolLimit
                }{
                  	\testFuncActiveLinkers
                }
            \end{equation}
            \begin{equation}
                \label{equ:height-linker system:variational:stationary:singular limit:%
                inactive linkers}
                \heightLinkerSysInactiveLinkersBF{
                  	\inactiveLinkers^0
                }{
                  	\testFuncInactiveLinkers
                }
                =
                -\repairRate 
                \innerProd[\lebesgueSet{2}{\cortex}]{
                    \inactiveLinkers^0
                }{
                    \testFuncInactiveLinkers
                } 
                +
                \innerProd[
                	\lebesgueSet{2}{\domain}
                ]{
                	\rippingInterpolLimit
                }{
                  	\testFuncInactiveLinkers
                }
            \end{equation}
        \end{subequations}
        for all $ \testFuncMembraneHeight \in \sobolevHSet{2}{\domain} $,
        \sndRevision[second]{%
        $ \testFuncActiveLinkers,\testFuncInactiveLinkers \in \diffSet[c]{}{\domain} $,
        }%
        where $ \rippingInterpolLimit \in \revision{\radonMeasures{\domain}} $. 
        \sndRevision[second]{%
          	(We identify a Radon measure and its density function w.\,r.\,t. the Lebesgue
            measure in the following.)%
        }
        Moreover, 
        \begin{equation}
        	\nnPart{ \left( \criticalHeight - \membraneHeight^0 \right) }  
            \rippingInterpolLimit = 0
        \end{equation}
        \almostEvWh
    \end{mytheorem}
    \begin{proof}
      	Let $ \left( \rippingLimitParam_n \right)_{n\in\nats} $ be a zero sequence.
      	Lemma~\ref{lemma:weighted sum of active and inactive linkers is constant} and
        the non-negativity of solutions of \autoref{problem:height-linker system:%
        variational:stationary} (cf. \autoref{theorem:existence of stationary solutions})
        implies boundedness of 
        $ \norm[\lebesgueSet{\infty}{\domain}]{ \activeLinkers^{\rippingLimitParam_n} } $
        and 
        $ \norm[\lebesgueSet{\infty}{\domain}]{ 
        \inactiveLinkers^{\rippingLimitParam_n} } $ 
        independent of $ \rippingLimitParam_n $. \revision{By testing
        \eqref{equ:height-linker system:variational:stationary:active linkers} 
        with $ \activeLinkers^{\rippingLimitParam_n} $, we derive an
        $ \sobolevHSet{1}{\cortex} $ bound on $ \activeLinkers^{\rippingLimitParam_n} $ 
        uniformly w.\,r.\,t. $ \rippingLimitParam_n $
        (the critical term including $ \rippingInterpol_{\rippingLimitParam_n} $
        can be dropped due to its non-positivity).
        We employ Lemma~\ref{lemma:weighted sum of active and inactive linkers is %
        constant} again; this time to see that 
        $ \activeLinkersDiffusiv \grad{\cortex}{}{}\activeLinkers^{\rippingLimitParam_n} =
        -\inactiveLinkersDiffusiv\grad{\cortex}{}{}\inactiveLinkers^{\rippingLimitParam_n} $ 
        and conclude that the $ \sobolevHSet{1}{\cortex} $ norm of
        $ \inactiveLinkers^{\rippingLimitParam_n} $ is as well
        bounded uniformly w.\,r.\,t. $ \rippingLimitParam_n $.}
        \autoref{lemma:boundedness of membrane height} assures the boundedness
        of $ \HTwoNorm{ \membraneHeight^{\rippingLimitParam_n} } $ 
        independently of $ \rippingLimitParam_n $.

        In $ \sobolevHSet{1}{\domain}^2 $, the theorem of Banach-Alaoglu
        gives us a weakly convergent subsequence
        $ \left( \activeLinkers^{\alpha_n}, \inactiveLinkers^{\alpha_n} 
        \right)_{n\in\nats} $ to
        $ \left( \activeLinkers^0, \inactiveLinkers^0 \right) 
        \in \sobolevHSet{1}{\domain}^2 $
        as well as going to another subsequence 
        $ \left( \membraneHeight^{\beta_n} \right)_{n\in\nats} $ of 
        $ \left( \membraneHeight^{\alpha_n} \right)_{n\in\nats} $ does
        in $ \sobolevHSet{2}{\domain} $.
        We use the Rellich-Kondrachov theorem to single out another
        subsequence $ \left( \activeLinkers^{\gamma_n} \right)_{n\in\nats} $
        converging in $ \lebesgueSet{2}{\domain} $ to $ \activeLinkers^0 $
        (since weak convergence in $ \sobolevHSet{1}{\domain} $ implies
        weak convergence in $ \lebesgueSet{2}{\domain} $ to the same
        limit). Analogously, we have convergence of a subsequence
        $ \left( \membraneHeight^{\delta_n} \right)_{n\in\nats} $ to
        $ \membraneHeight^0 $ in $ \lebesgueSet{2}{\domain} $, which includes
        another subsequence $ \left( \membraneHeight^{\zeta_n} \right)_{n\in\nats} $
        converging pointwise to $ \membraneHeight^0 $.

        Due to weak convergence, 
        \begin{equation*}
            \heightLinkerSysHeightBF{ \membraneHeight^{\zeta_n} }{ 
            \testFuncMembraneHeight }
            \specConv{n\to\infty}
            \heightLinkerSysHeightBF{ \membraneHeight^0 }{ 
            \testFuncMembraneHeight }
        \end{equation*}
        for any $ \testFuncMembraneHeight \in \sobolevHSet{2}{\domain} $.
        Moreover,
        \begin{align*}
            \abs{\LTwoIP{%
                \membraneHeight^0 \activeLinkers^0 - 
                \membraneHeight^{\zeta_n} \activeLinkers^{\zeta_n}
            }{\testFuncMembraneHeight}} 
            &\leq
            \abs{\LTwoIP{%
                \membraneHeight^0
                \left( \activeLinkers^0 - \activeLinkers^{\zeta_n} \right)
            }{\testFuncMembraneHeight}}
            +
            \abs{\LTwoIP{%
                \activeLinkers^{\zeta_n}\left( \membraneHeight^0 - 
                \membraneHeight^{\zeta_n} \right)
            }{\testFuncMembraneHeight}} \\       
            &\leq 
            \norm[\lebesgueSet{\infty}{\domain}]{\membraneHeight^0}
            \LTwoNorm{\activeLinkers^0 - \activeLinkers^{\zeta_n}}
            \LTwoNorm{\testFuncMembraneHeight}
            +
            \LTwoNorm{\membraneHeight^0 - \membraneHeight^{\zeta_n}}
            \LTwoNorm{\activeLinkers^{\zeta_n}\testFuncMembraneHeight},
        \end{align*}
        so we have
        \begin{equation*}
            \LTwoIP{\membraneHeight^{\zeta_n}\activeLinkers^{\zeta_n}}{
            \testFuncMembraneHeight} 
            \specConv{n\to\infty} 
            \LTwoIP{\membraneHeight^0 \activeLinkers^0}{
            \testFuncMembraneHeight}
        \end{equation*}
        and \eqref{equ:variational formulation:homogeneous boundary:%
        stationary:singular limit:membrane height} holds.

        To retrieve \eqref{equ:height-linker system:variational:stationary:%
        singular limit:active linkers}
        and
        \eqref{equ:height-linker system:variational:stationary:%
        singular limit:inactive linkers}, test
        \eqref{equ:height-linker system:variational:stationary:active linkers}
        with a smoothed signum $ S_\eps \in \sobolevHSet{1}{\domain} $
        of $ \activeLinkers $ with
        $ S_\eps \specConv{\sobolevHSet{1}{\domain}} \sign \concat 
        \sndRevision[second]{\activeLinkers^{\zeta_n}} $.
        We then have
        \begin{equation*}
          	\LTwoIP{
              	\grad{\domain}{}{}\activeLinkers^{\zeta_n}
            }{
              	\grad{\domain}{}{}S_\eps
            }
            +
            \LTwoIP{
              	\rippingInterpol_{\zeta_n}
                \left(
                  	\membraneHeight^{\zeta_n}
                \right)
                \activeLinkers^{\zeta_n}
            }{
              	S_\eps
            }
            =
            \repairRate
            \LTwoIP{
              	\inactiveLinkers^{\zeta_n}
            }{
              	S_\eps
            }.
        \end{equation*} 
        Therefore, 
        \begin{equation*}
            \LTwoIP{
              	\rippingInterpol_{\zeta_n}
                \left(
                  	\membraneHeight^{\zeta_n}
                \right)
                \activeLinkers^{\zeta_n}
            }{
              	S_\eps
            }
            \leq
            C_1
            \left(
                \sndRevision[second]{\HOneNorm{\activeLinkers^{\zeta_n}}}
                \HOneNorm{S_\eps}
                +
                \LTwoNorm{
                    \inactiveLinkers^{\zeta_n}
                }
                \LTwoNorm{
                    S_\eps
                }
            \right)
        \end{equation*}         
        for a constant $ C_1 > 0 $.
        So in the limit $ \eps \to 0 $, we get the estimate
        \begin{equation*}
          	\norm[ 
            	\lebesgueSet{1}{\domain}
            ]{
              	\rippingInterpol_{\zeta_n}
                \left(
                  	\membraneHeight^{\zeta_n}
                \right)
                \activeLinkers^{\zeta_n}
            }
            \leq
            C_2
        \end{equation*}
        for a constant $ C_2 > 0 $ due to the a priori bound
        on $ \activeLinkers $ and $ \inactiveLinkers $.
        This implies weak-$\star$-convergence \revision{in $ \radonMeasures{\cortex} $} 
        of a subsequence 
        $ \rippingInterpol_{\zeta_{n_k}} 
        \left( \membraneHeight^{\zeta_{n_k}} \right) \activeLinkers^{\zeta_{n_k}} $
        to a Radon measure $ \rippingInterpol_0 \in \radonMeasures{\domain} $.
        As $ \membraneHeight^{\zeta_{n_k}} $ and $ \criticalHeight $ are in
        \revision{$ \lebesgueSet{\infty}{\cortex} $},
        $ \nnPart{ \left( \criticalHeight - \membraneHeight^{\zeta_{n_k}} \right) }
        \rippingInterpol_{\zeta_{n_k}} 
        \left( \membraneHeight^{\zeta_{n_k}} \right) \activeLinkers^{\zeta_{n_k}}
        \in \radonMeasures{\domain} $
        and
        $ \nnPart{ \left( \criticalHeight - \membraneHeight^{\zeta_{n_k}} \right) } 
            \rippingInterpol_{\zeta_{n_k}}
            \left(
                \membraneHeight^{\zeta_{n_k}}
            \right)
            \activeLinkers^{\zeta_{n_k}}
        $
        weak-$\star$-converges to
        $ \nnPart{ \left( \criticalHeight - \membraneHeight^0 \right) } 
            \rippingInterpol_0
        $ \revision{in $\radonMeasures{\cortex}$}.
        We observe that 
        \begin{equation*}
            \supp\left( \nnPart{ \left(
            \criticalHeight - \membraneHeight^{\zeta_{n_k}} \right) } \right)
            \cap \supp\left(             \rippingInterpol_{\zeta_{n_k}}
                \left(
                    \membraneHeight^{\zeta_{n_k}}
                \right) \right) = \emptyset, 
        \end{equation*}
        so
        \begin{equation*}
            \nnPart{ \left( \criticalHeight - \membraneHeight^{\zeta_{n_k}} \right) } 
                \rippingInterpol_{\zeta_{n_k}}
                \left(
                    \membraneHeight^{\zeta_{n_k}}
                \right)
                \activeLinkers^{\zeta_{n_k}} = 0
        \end{equation*}
        and 
        $ \nnPart{ \left( \criticalHeight - \membraneHeight^0 \right) }
        \rippingInterpol_0 = 0 $.
    \end{proof}

\newcommand{\linkerMembrDep}{g}
\newcommand{\antiDerivLinkerMembrDep}{G}
\newcommand{\energyAntiDerivLinkerMembrDep}{\mathcal{G}}
\newcommand{\energyFunc}{\mathcal{J}}
\subsection{Model without diffusion}
    \renewcommand{\totalLinkersMassDens}{\rho}
    \sndRevision[second]{We now turn to the system 
    \eqref{equ:height-linker system:variational:stationary} with}
    $ \activeLinkersDiffusiv = \inactiveLinkersDiffusiv = 0 $,
    which will lead us 
    to a model of \cite{Lim+2012} in the $\Gamma$-limit $ \rippingLimitParam
    \to 0 $. \revision[second]{Recall, that a sequence of functionals $ \funSig{F_n}{X}{\reals} $,
    $ n \in \nats $, defined on a topological space
    $\Gamma$-converges to a functional $ \funSig{F}{X}{\reals} $ iff 
    \begin{itemize}
     	\item[(i)] For all $ x \in X $ and $ x_n \specConv{n\to\infty} x $,
        $ \liminf_{n\to\infty} F_n(x_n) \geq F(x) $  and
        \item[(ii)] For all $ x \in X $ there exists $ x_n \specConv{n\to\infty} x $
        such that $ \limsup_{n\to\infty} F_n(x_) \leq F(x) $.
    \end{itemize}}

    \sndRevision[second]{%
    Vanishing diffusivities have not been treated in our previous existence proofs
    for the time-dependent or stationary case. We attack this issue by 
    reducing the system
    \eqref{equ:height-linker system:variational:stationary}
    to minimising an energy functional. In this order we choose
    a function $ \totalLinkersMassDens \in \sobolevHSet{1}{\domain}, 
    \totalLinkersMassDens \geq 0 \; a.\,e., $ and formally
    substitute
    $
        \inactiveLinkers = \totalLinkersMassDens - \activeLinkers
    $
    into \eqref{equ:height-linker system:variational:stationary:active linkers}.
    This way, we obtain:
    }
    \begin{equation}
      	\label{equ:active linkers dependency on membrane height}
        \rippingInterpol\left(
            \frac{%
                \membraneHeight - \criticalHeight
            }{\rippingLimitParam}
        \right) \activeLinkers
        =
        \repairRate
        \left(
            \totalLinkersMassDens - \activeLinkers
        \right)
        \iff
        \activeLinkers = 
        \frac{%
            \repairRate \totalLinkersMassDens
        }{%
            \repairRate
            + 
            \rippingInterpol
            \left(
                \frac{%
                    \membraneHeight - \criticalHeight
                }{\rippingLimitParam}
            \right)
        }
        = \linkerMembrDep_{\rippingLimitParam} \concat \membraneHeight,
    \end{equation}
    where 
    $ 
    \fun{\linkerMembrDep_{\rippingLimitParam}}{\reals}{\reals}{x}{
        \frac{%
            \repairRate \totalLinkersMassDens
        }{%
            \repairRate
            + 
            \rippingInterpol\left(
                \frac{%
                    x - \criticalHeight
                }{\rippingLimitParam}
            \right)
        }
    }.
    $
    Inserting into \eqref{equ:height-linker system:variational:stationary:height}, we have 
    only one equation left:
    \begin{equation}
      	\label{equ:height-linkers sytem:variational:stationary:reduced}
      	\heightLinkerSysHeightBF{ \membraneHeight }{ \testFuncMembraneHeight }
        +
        \LTwoIP{%
            \membraneHeight \linkerMembrDep_{\rippingLimitParam} \concat\membraneHeight
        }{%
            \testFuncMembraneHeight
        }
        = 
        \LTwoIP{\pressure_0}{\testFuncMembraneHeight},
    \end{equation}
    whose solutions are the critical points of the following energy functional
    \begin{align*}
        \energyFunc_{\rippingLimitParam} 
        &\colon 
        \sobolevHSet{2}{\domain} 
        \rightarrow 
        \reals
        \\
        \energyFunc_{\rippingLimitParam} \left( \membraneHeight \right)
        &=
        \frac{1}{2}
        \heightLinkerSysHeightBF{ \membraneHeight }{ \membraneHeight }
        + 
        \integral{\domain}{}{%
            \integral{0}{\membraneHeight(x)}{%
                s \linkerMembrDep_{\rippingLimitParam}(s)
            }{s}
        }{x}
        - \integral{\domain}{}{\pressure_0(x) \membraneHeight(x)}{x}.
    \end{align*}

    \begin{mylemma}
      	\label{lemma:existence of minimisers}
        There exists a minimiser of $ \energyFunc_{\rippingLimitParam} $.
    \end{mylemma}
    \begin{proof}
        This functional is coercive in the $ \sobolevHSet{2}{\domain} $ norm:
        \begin{align*}
            \frac{1}{2} \heightLinkerSysHeightBF{ \membraneHeight }{ \membraneHeight }
            - \integral{\domain}{}{\pressure_0(x) \membraneHeight(x)}{x}
            \geq
            C_1 \norm[ \sobolevHSet{2}{\domain} ]{ \membraneHeight }^2
            - \frac{1}{4\eps} \LTwoNorm{\pressure_0}^2 
            - \eps \LTwoNorm{\membraneHeight}^2,
        \end{align*}
        $ C_1 > 0 $, 
        and we choose $ \eps $ smaller enough for 
        $ \eps \LTwoNorm{ \membraneHeight }^2 $
        to be absorbed
        by 
        $  C_1 \norm[ \sobolevHSet{2}{\domain} ]{ \membraneHeight }^2 $.

        Set $ A(h) = \frac{1}{2}\heightLinkerSysHeightBF{ \membraneHeight }{ \membraneHeight } $
        and $ B(h) =         \integral{\domain}{}{%
            \integral{0}{\membraneHeight(x)}{%
                s \linkerMembrDep_{\rippingLimitParam}(s)
            }{s}
        }{x}
        + F(\membraneHeight) $,
        where 
        \begin{align*}
            F \colon \sobolevHSet{2}{\domain}
            \rightarrow \reals, \quad
            \membraneHeight
            \mapsto
            -
            \integral{\domain}{}{\pressure_0(x) \membraneHeight(x)}{x}.
        \end{align*}
        Note that $ A $ is weakly lower semicontinuous in $ \sobolevHSet{2}{\domain} $.
        We show that $ B $ is weakly continuous in $ \sobolevHSet{2}{\domain} $, so 
        $ A + B $ is weakly lower 
        semicontinuous in $ \sobolevHSet{2}{\domain} $.

        $ B $ is continuous in $ \diffSet[b]{}{\domain} $: Take a sequence
        $ \left( \membraneHeight_n \right)_{n \in \nats} $ converging
        in $ \diffSet[b]{}{\domain} $ \revision{to $ \membraneHeight $}. 
        Now observe that for all $ \delta > 0 $, there exists an $ N \in \nats $ such 
        that for all $ n \geq N $, it holds
        \begin{align*}
            \abs{%
                \integral{\domain}{}{%
                    \integral{0}{\membraneHeight(x)}{%
                        s \linkerMembrDep_{\rippingLimitParam}(s)
                    }{s}
                }{x} 
                - 
                \integral{\domain}{}{%
                    \integral{0}{\membraneHeight_{n}(x)}{%
                        s \linkerMembrDep_{\rippingLimitParam}(s)
                    }{s}
                }{x}
            }
            &\leq
            \abs{\domain} 
            \sup_{x\in\domain}
            \abs{%
                \integral{\membraneHeight_{n}(x)}{\membraneHeight(x)}{%
                    s \linkerMembrDep_{\rippingLimitParam}(s)
                }{s}
            } \\
            &\leq 
            \abs{\domain}
            \sup_{x\in\domain}
            \abs{%
                \integral{\membraneHeight(x)\pm\delta}{\membraneHeight(x)}{%
                    s \linkerMembrDep_{\rippingLimitParam}(s)
                }{s}
            },
        \end{align*}
        so
        \begin{equation*}
            \integral{\domain}{}{%
                \integral{0}{\membraneHeight_{n}(x)}{%
                    s \linkerMembrDep_{\rippingLimitParam}(s)
                }{s}
            }{x} 
            \specConv{n \to \infty}
            \integral{\domain}{}{%
                \integral{0}{\membraneHeight(x)}{%
                    s \linkerMembrDep_{\rippingLimitParam}(s)
                }{s}
            }{x},
        \end{equation*}
        Conclusively, $ B(\membraneHeight_{n}) \to B(\membraneHeight) $
        for $ n \to \infty $.

        Since $ \sobolevHSet{2}{\domain} \imbeddedR[comp] \diffSet[b]{}{\domain} $,
        for every sequence $ \left( \membraneHeight_n \right)_{n\in\nats} $ converging 
        weakly in $ \sobolevHSet{2}{\domain} $ to $ \membraneHeight $, we may take
        from every subsequence
        a subsubsequence $ \left( \membraneHeight_{n_k} \right)_{k\in\nats} $
        that converges in $ \diffSet[b]{}{\domain} $ 
        to $ \membraneHeight' $. It holds $ \membraneHeight' \eqAE \membraneHeight $ 
        as weak convergence in $ \sobolevHSet{2}{\domain} $ implies weak
        convergence in $ \lebesgueSet{2}{\domain} $ and
        convergence in $ \diffSet[b]{}{\domain} $ implies (strong) convergence 
        in $ \lebesgueSet{2}{\domain} $. So with the previous result, 
        $ B(\membraneHeight_{n}) \specConv{n\to\infty} B(\membraneHeight) $.
    \end{proof}

    \sndRevision[second]{%
      	With this lemma we have proven that for every $ \totalLinkersMassDens \in \sobolevHSet{1}{\domain},
        \totalLinkersMassDens \geq 0 \;a.\,e. $, we find a minimiser $ \membraneHeight $ 
        of $ \energyFunc_\rippingLimitParam $
        with which we may then define 
        $ \activeLinkers = g_\rippingLimitParam \concat \membraneHeight $, and
        further $ \inactiveLinkers = \totalLinkersMassDens - \activeLinkers $ such that
        $ (\membraneHeight,\activeLinkers,\inactiveLinkers) $ solves
        \eqref{equ:height-linker system:variational:stationary} with
        vanishing diffusivities. It is easily checked that the constructed linker
        densities are non-negative.
    }%

    \begin{mylemma}
        \label{lemma:gamma convergence}
        For $ \rippingLimitParam \searrow 0 $, the energy functional 
        $ \energyFunc_{\rippingLimitParam} $ $ \Gamma $-converges to
        \begin{equation*}
            \energyFunc_0\left( \membraneHeight \right)
            =
            \frac{1}{2}
            \heightLinkerSysHeightBF{ \membraneHeight }{ \membraneHeight }
            + 
            \integral{\domain}{}{%
            	\revision{g_0(x)}
            }{x}
            - \integral{\domain}{}{\pressure_0(x) \membraneHeight(x)}{x},
        \end{equation*}
        \revision{%
        where
        $ g_0(x) = \frac{\totalLinkersMassDens}{2}\min\{\membraneHeight(x)^2,
        (\criticalHeight)^2\} $.}
    \end{mylemma}
    \begin{proof}
        As the other terms in $ \energyFunc_{\rippingLimitParam} $ are independent
        of $ \rippingLimitParam $,
        we only need to consider the functional
        \begin{align*}
            \energyAntiDerivLinkerMembrDep_{\rippingLimitParam} \colon
            \sobolevHSet{2}{\domain} \rightarrow \reals, \quad
            \membraneHeight \mapsto
            \integral{\domain}{}{%
                \integral{0}{\membraneHeight(x)}{%
                    s \linkerMembrDep_{\rippingLimitParam}(s)
                }{s}
            }{x}.
        \end{align*}

        It is first shown that from any 
        $ \series{ \rippingLimitParam } \to 0 $ and any
        $ \series{ \membraneHeight } $ that
        converges in $ \sobolevHSet{2}{\domain} $ to 
        $ \membraneHeight $, a 
        subsequence such that 
        $ \energyAntiDerivLinkerMembrDep_n\left(
        \membraneHeight_n \right) \to \energyAntiDerivLinkerMembrDep_0
        \left( \membraneHeight \right) $ can be singled out.
        (For the sake of notational simplicity, we abbreviate
        $ \energyAntiDerivLinkerMembrDep_{\rippingLimitParam_n} =
        \energyAntiDerivLinkerMembrDep_n $ and
        $ \linkerMembrDep_{\rippingLimitParam_n} = \linkerMembrDep_n $.)

        In this order, rewrite
        \begin{align*}
            \integral{\domain}{}{
                \integral{0}{\membraneHeight_n(x)}{
                    s \linkerMembrDep_n(s)
                }{s}
            }{x}
            -
            \integral{\domain}{}{
                \integral{0}{\membraneHeight(x)}{
                    s \linkerMembrDep_0(s)
                }{s}
            }{x}
            &=
            \left(
                \integral{\domain}{}{
                    \integral{0}{\membraneHeight_n(x)}{
                        s \linkerMembrDep_n(s)
                    }{s}
                }{x}
                -
                \integral{\domain}{}{
                    \integral{0}{\membraneHeight(x)}{
                        s \linkerMembrDep_n(s)
                    }{s}
                }{x}
            \right)
            \\
            &+
            \left(
                \integral{\domain}{}{
                    \integral{0}{\membraneHeight(x)}{
                        s \linkerMembrDep_n(s)
                    }{s}
                }{x}
                -
                \integral{\domain}{}{
                    \integral{0}{\membraneHeight(x)}{
                        s \linkerMembrDep_0(s)
                    }{s}
                }{x}        
            \right)
        \end{align*}
        We take a subsequence such that $ \membraneHeight_{n_k} $ converges in 
        $ \diffSet[b]{}{\domain} $ to $ \membraneHeight $.
        For any $ \delta > 0 $ consider a sufficiently large $ k $ such that
        \begin{align*}
            \abs{%
                \integral{\domain}{}{%
                    \integral{0}{\membraneHeight_{n_k}(x)}{%
                        s \linkerMembrDep_{n_k}(s)          
                    }{s}
                }{x}
                -
                \integral{\domain}{}{%
                    \integral{0}{\membraneHeight(x)}{%
                        s \linkerMembrDep_{n_k}(s)
                    }{s}
                }{x}
            }
            &\leq
            \integral{\domain}{}{%
                \abs{%
                    \integral{\membraneHeight(x)\pm\delta}{\membraneHeight(x)}{%
                        s \linkerMembrDep_{n_k}(s)
                    }{s}
                }
            }{x} \\
            &\leq
            \integral{\domain}{}{%
                \integral{\membraneHeight(x)-\delta}{\membraneHeight(x)}{%
                    s \linkerMembrDep_{n_k}(s)
                }{s}
            }{x}.
        \end{align*}
        This term converges to zero for $ \delta \to 0 $ since
        $ \norm[ \lebesgueSet{\infty}{\domain} ]{ \linkerMembrDep_{n_k} } $ 
        is uniformly (in \revision{$ k $}) bounded.
        Now consider
        \begin{align*}
            \integral{\domain}{}{%
                \integral{0}{\membraneHeight(x)}{%
                    s \linkerMembrDep_n(s)
                }{s}
            }{x}
            =
            \integral{%
                \set{ x \in \domain }{ \membraneHeight(x) \leq \criticalHeight }
            }{}{%
                \integral{0}{\membraneHeight(x)}{%
                    \totalLinkersMassDens s	
                }{s}
            }{x}
            &+
            \integral{%
                \set{ x \in \domain }{ \membraneHeight(x) > \criticalHeight }
            }{}{%
                \integral{0}{\criticalHeight}{%
                    \totalLinkersMassDens s	
                }{s}
            }{x} \\
            &+
            \integral{%
                \set{ x \in \domain }{ \membraneHeight(x) > \criticalHeight }
            }{}{%
                \integral{\criticalHeight}{\membraneHeight(x)}{%
                    s \linkerMembrDep_n(s)
                }{s}
            }{x}
        \end{align*}
        Due to the monotonicity of $ \rippingInterpol $,
        $ \linkerMembrDep_n(s) $ monotonically decreases to zero
        and $ \linkerMembrDep_n \leq 
        \linkerMembrDep_1 \in \lebesgueSet{1}{\domain} $, 
        the monotone convergence theorem applies, so
        the last term converges to zero for $ n \to \infty $.
        Therefore, 
        \begin{align*}
            \integral{\domain}{}{%
                \integral{0}{\membraneHeight(x)}{%
                    s \linkerMembrDep_n(s)
                }{s}
            }{x}
            &\specConv{n\to\infty}
            \integral{%
                \set{ x \in \domain }{ \membraneHeight(x) \leq \criticalHeight }
            }{}{%
                \frac{\totalLinkersMassDens}{2} \membraneHeight(x)^2
            }{x}
            +
            \integral{%
                \set{ x \in \domain }{ \membraneHeight(x) > \criticalHeight }
            }{}{%
                \frac{\totalLinkersMassDens}{2} \left( \criticalHeight \right)^2
            }{x} \\
            &= \energyAntiDerivLinkerMembrDep_0\left( \membraneHeight \right).
        \end{align*}

        Take any clustering point of
        $ \energyAntiDerivLinkerMembrDep_n\left(
        \membraneHeight_n \right) $
        and a subsequence
        $ \energyAntiDerivLinkerMembrDep_{n_k}\left(
        \membraneHeight_{n_k} \right) $ converging to it.
        Using the previous result, we single out another
        subsequence that converges to $ \energyAntiDerivLinkerMembrDep_0\left(
        \membraneHeight \right) $, which has to be the clustering point,
        so $ \energyAntiDerivLinkerMembrDep_0\left( \membraneHeight \right)
        \leq \liminf_{n\to\infty}  \energyAntiDerivLinkerMembrDep_{\rippingLimitParam_n}
        \left(
        \membraneHeight_n \right) $. 

        \revision{Choosing an arbitrary $ \membraneHeight $ and taking 
        the sequence $ \membraneHeight_n = \membraneHeight $ converging to it,
        the previous considerations also lead to}
        $ \energyAntiDerivLinkerMembrDep_0\left( \membraneHeight \right) \geq
        \limsup_{n\to\infty} \energyAntiDerivLinkerMembrDep_n
        \left( \membraneHeight_n \right)
        $ and the claimed $ \Gamma $-limit is shown.
    \end{proof}

    \begin{mylemma}
        \label{lemma:Euler-Lagrange equation}
        The Euler-Lagrange equation of $ \energyFunc_0 $ is given by
        \begin{equation*}
            \membrStiffn 
            \laplacian{\domain}{2}{}\membraneHeight
            - 
            \effectiveLengthParam 
            \laplacian{\domain}{}{}\membraneHeight
            +
            \anotherMembrParam 
            \membraneHeight
            + 
            \totalLinkersMassDens
            \membraneHeight
            H\left( 1 - \frac{\membraneHeight}{\criticalHeight} \right)
            =
            \pressure_0
        \end{equation*}
        for all $ \membraneHeight \in \sobolevHSet{4}{\domain} $
        with $ H(x) = \begin{cases} 0 & x \leq 0 \\ 1 & x > 0 \end{cases} $
        being the Heaviside function.
    \end{mylemma}
    \begin{proof}
        Let $ \membraneHeight $ be a stationary point of $ \energyFunc_0 $, i.\,e.,
        $
            \diff[0]{\eps}{}{%
                \energyFunc_0\left( \membraneHeight + \eps v \right)
            } = 0
        $
        for all $ v \in \sobolevHSet{2}{\domain} $. Take an 
        arbitrary $ v \in \sobolevHSet{2}{\domain} $ and consider $ \eps > 0 $
        small enough such that $ \membraneHeight < \criticalHeight \; \almostEvWh $ implies
        $ \membraneHeight + \eps v \leq \criticalHeight \; \almostEvWh $
        Then calculate
        \begin{align*}
           &\qquad \diff[0]{\eps}{}{%
                \integral{\domain}{}{%
                    \min\{%
                        \left( \membraneHeight(x) + \eps v(x) \right)^2, 
                        \left( \criticalHeight \right)^2
                    \}
                }{x}
            }
           \\ &\qquad = 
            \diff[0]{\eps}{}{%
                \integral{\set{x\in\domain}{\membraneHeight(x) < \criticalHeight}}{}{%
                    \left( \membraneHeight(x) + \eps v(x) \right)^2
                }{x}
            } 
						+ 
            \diff[0]{\eps}{}{%
                \integral{\set{x\in\domain}{\membraneHeight(x) \geq \criticalHeight}}{}{%
                    \left( \criticalHeight  \right)^2
                }{x}
            }
            \\
            &\qquad = 
            \integral{
                \set{x\in\domain}{\membraneHeight(x) \leq \criticalHeight}
            }{}{%
                \diff[0]{\eps}{}{%
                    \left( \membraneHeight(x) + \eps v(x) \right)^2
                }
            }{x} 
           = 
            2 \integral{
                \set{x\in\domain}{\membraneHeight(x) \leq \criticalHeight}
            }{}{%
                \membraneHeight(x) v(x)
            }{x}
            \\
            &\qquad= 
            2 \integral{
                \domain
            }{}{%
                \membraneHeight(x) 
                H\left( 1 - \frac{\membraneHeight(x)}{\criticalHeight} \right) v(x)
            }{x}.
        \end{align*}
        The other derivatives are standard and the claim follows.
    \end{proof}

    We may summarise the results of this section as follows:
    \begin{mytheorem}
        i) For a total mass density $ \totalLinkersMassDens \revision{\geq 0} $, and
        a ripping parameter $ \rippingLimitParam \revision{>0} $,
        any minimiser $ \membraneHeight $ of 
        $ \energyFunc_{\rippingLimitParam} $ 
        constitutes a (variational) solution to
        \autoref{problem:height-linker system:variational:stationary} 
        for $ \activeLinkersDiffusiv = \inactiveLinkersDiffusiv = 0 $ 
        together with some 
        $ \activeLinkers, \inactiveLinkers \in \sndRevision[second]{\sobolevHSet{1}{\domain}} $.

        ii) Sending $ \rippingLimitParam $ to zero, minimisers of
        $ \energyFunc_{\rippingLimitParam} $ converge to variational solutions
        of the model of \cite{Lim+2012}, cf. p.\,2,\;Equation~(2) therein.
    \end{mytheorem}
    \begin{proof}
      	i) Take any $ \totalLinkersMassDens $ and $ \rippingLimitParam $. Existence
        of minimisers of $ \energyFunc_{\rippingLimitParam} $ is guaranteed
        by \autoref{lemma:existence of minimisers}. As critical points,
        the minimisers solve \eqref{equ:height-linkers sytem:variational:stationary:%
        reduced}. Defining $ \activeLinkers $ according to 
        \eqref{equ:active linkers dependency on membrane height}
        and $ \inactiveLinkers = \totalLinkersMassDens - \activeLinkers $,
        we get solutions of \autoref{problem:height-linker system:variational:stationary}
        for $ \activeLinkersDiffusiv = \inactiveLinkersDiffusiv = 0 $.
        
        ii) follows directly from \autoref{lemma:gamma convergence} and
        \autoref{lemma:Euler-Lagrange equation}.
    \end{proof}


%% file: Numerical_study.tex
\renewcommand{\domain}{D}
\newcommand{\massMatrix}{M}
\newcommand{\stiffnMatrix}{A}
\newcommand{\firstRHS}{b}
\newcommand{\secondRHS}{q}
\newcommand{\nbIntNodes}{N}
\newcommand{\timeStepSize}{\tau}

\newcommand{\disc}[2]{%
  	\ifthenelse{\isempty{#2}}{%
    	{#1}^{\spaceMeshSize}
    }{%
    	{#1}^{\spaceMeshSize,#2}
    }
}

In this section, we discuss results from numerical simulations
of the parabolic PDE system we analysed in the previous sections.
The predicted bleb size after a typical bleb formation
time is compared against heights observed by biologists.
Furthermore, we investigate the role of the critical pressure
defined in \cite{Lim+2012} for static systems in the case of 
our time-dependent PDE system.

\subsection{Parameters and initial conditions}
	\label{sec:Parameters}
    \begin{table}[h]
        \centering
        \begin{tabular}{l || c c c c}
            Parameter & Symbol & Value & Unit & Source
            \\
            \hline
            \hline
            Damping constant & $ \membraneDampingConst $ & $ 5\cdot 10^{-3} $ & 
            $ \physUnit{Pa}\,\physUnit{s}\,\physUnit{m}^{-1} $ & \parbox{3cm}{%
            \cite[p.\,1879]{Alert+2015}} 
            \\
            Membrane bending rigidity & $ \membrStiffn $ & $ 2\cdot 10^{-19} $ & 
            $ \physUnit{J} $ & \cite{Dai+1999}
            \\
            Surface tension & $ \effectiveLengthParam $ & $ -4\cdot 10^{-9} $ &
            $ \physUnit{N} \, \physUnit{m}^{-1} $ & see text
            \\
            Linker spring constant & $ \linkersSpringConst $ & $ 10^{-4} $ & 
            $ \physUnit{N} \, \physUnit{m}^{-1} $ & \cite{Alert+2015}
            \\
            Linker diffusivities & 
            $ \activeLinkersDiffusiv = \inactiveLinkersDiffusiv $ & $ 10^{-6} $ &
            $ \physUnit{m}^2 \, \physUnit{s}^{-1} $ & \cite{Jacobson+2019}
            \\
            Reconnection rate & $ \repairRate $ & $ 10^4 $ & $ \physUnit{s}^{-1} $
            & \cite{Rognoni+2012}
            \\
            Critical linkers length & $ \criticalHeight $ & $ 10^{-9} $ &
            $ \physUnit{m} $ & \cite{Lim+2012}
            \\
            Cortex radius & $ R $ & $ 10^{-5} $ & $ \physUnit{m} $ &
            \cite{CharrasII+2008}
        \end{tabular}
        \caption{Parameter configuration}
        \label{fig:parameter basic choice}
    \end{table}
    As the tabular \autoref{fig:parameter basic choice} suggests, quite a lot
    parameters of our model have already been assessed and discussed in the literature.
    The interested reader may consult the given sources and the references therein as
    a thorough treatment of these parameters is not in the scope of this work.
    Nevertheless, the choice of $ \effectiveLengthParam $ requires a comment:
    We decided to follow \cite{Lim+2012},\cite{Alert+2015} and set $ \spontMeanCurv = 0 $. 
    Despite their choice of $ \membrStiffn = 2\cdot10^{-19}\,\physUnit{J} $ and
    $ \effectiveLengthParam = -2\cdot10^{-6} $, we stay consistent with
    our computations in \autoref{par:force density equation}, p.\,\pageref{par:%
    force density equation}
    and take $ \effectiveLengthParam \sim -\frac{\membrStiffn}{\cortexRadius^2} $.

    In correspondence to \cite{Alert+2015}, we choose the linker density
    $ \activeLinkers(0,\cdot) $ to be $ 10^{14}\,\physUnit{m}^{-2} $
    and start with $ \inactiveLinkers(0,\cdot) = 0 $, accordingly.
    
\subsection{Discretisation}
    The computations were done on a rectangle $ \domain = [-0.49\pi,0.49\pi] \times [0,2\pi] $
    being an approximation of the parameter space of the transformation
    \begin{equation*}
        \begin{split}
            \sphereTrafo &\colon \left[-\frac{\pi}{2},\frac{\pi}{2}\right] \times [0,2\pi)
            \rightarrow 
            \twoSphere[\cortexRadius]
            \\
            ( \sphereTrafoTheta, \sphereTrafoPhi ) &\mapsto \left(
                \cortexRadius \cos(\sphereTrafoTheta) \cos(\sphereTrafoPhi), 
                \cortexRadius \cos(\sphereTrafoTheta) \sin(\sphereTrafoPhi),
                \cortexRadius \sin(\sphereTrafoTheta) 
            \right)^T
        \end{split}
    \end{equation*}
    onto the $2$-sphere $ \twoSphere[\cortexRadius] \subseteq \reals^3 $ in three dimensions
    with radius $ \cortexRadius $.
    The boundary $ [-0.49,0.49] \times \{ 0 \} \cup [-0.49, 0.49] \times \{ 2 \pi \} $
    is chosen to be periodic and at $ \{ -0.49 \} \times [0,2\pi] \cup
    \{ 0.49 \} \times [0,2\pi] $ we employ 
    homogeneous Neumann boundary conditions. 
    This domain approximation is a simple approach towards
    the singularities at the poles of $ \twoSphere[\cortexRadius] $ and there are other
    more elaborate methods like surface finite elements 
    (cf. \cite{Dzuik+2013}) which can handle such problems. However, we will show
    that we can in fact recover established results and conclude that this approach is 
    sufficient for studying the bleb height near the equator.
    The fluid influence is neglected for simplicity setting
    $ \timeDerivOperator = \membraneDampingConst\identity $.

    In order to avoid a $ \sobolevHSet{2}{\domain} $ trial space,
    but to stick with conformal Galerkin methods,
    we use a standard splitting of the $ \laplacian{}{2}{} $ operator with 
    boundary conditions for $\membraneHeight$ and $\laplacian{}{}{}\membraneHeight $ by 
    introducing 
    $ w = - \laplacian{}{}{}\membraneHeight $.
    This leads to the following system:
    \begin{subequations}
        \label{equ:operator splitting}
        \begin{equation*}
            \membraneDampingConst
            \pDiff{t}{}{} \membraneHeight
            -
            \membrStiffn 
            \laplacian{}{}{} w
            +
            \effectiveLengthParam
            w
            =
            -\linkersSpringConst
            \activeLinkers
            \membraneHeight
            +
            \pressure_0
        \end{equation*}
        \begin{equation*}
            w + \laplacian{}{}{}\membraneHeight = 0
        \end{equation*}
        \begin{equation*}
            \pDiff{t}{}{} \activeLinkers 
            - \activeLinkersDiffusiv \laplacian{}{}{} \activeLinkers
            =
            \repairRate \inactiveLinkers
            -
            \rippingInterpol_{\rippingLimitParam}
            \left(
                \membraneHeight
            \right)
            \activeLinkers
        \end{equation*}
        \begin{equation*}
          \pDiff{t}{}{} \inactiveLinkers
          - \inactiveLinkersDiffusiv \laplacian{}{}{} \inactiveLinkers
          =
          - \repairRate \inactiveLinkers
          +
          \rippingInterpol_{\rippingLimitParam}
          \left(
                \membraneHeight
          \right)
          \activeLinkers.
        \end{equation*}
    \end{subequations}

    In the Galerkin approximation, we employ
    Lagrangian $ \mathbb{P}_1 $ finite elements conforming with the trial space
    $ \sobolevHSet{1}{\domain} $
    for $ \membraneHeight $, $ w $, $ \activeLinkers $, and $ \inactiveLinkers $. 
    The mesh width is denoted by $ \spaceMeshSize $.
    In this configuration, we obtain a semidiscretisation of the form
    \begin{subequations}
        \begin{equation*}
            \membraneDampingConst
            \massMatrix
            \pDiff{t}{}{} \disc{\membraneHeight}{}
            +
            \membrStiffn 
            \stiffnMatrix
            \disc{w}{}
            +
            \effectiveLengthParam
            \massMatrix
            \disc{w}{}
            =
            \firstRHS\left(\disc{\activeLinkers}{}, \disc{\membraneHeight}{}\right)
            +
            \massMatrix
            \disc{\pressure_0}{}
        \end{equation*}
        \begin{equation*}
            \massMatrix \disc{w}{} - \stiffnMatrix\disc{\membraneHeight}{} = 0
        \end{equation*}
        \begin{equation*}
            \massMatrix
            \pDiff{t}{}{} \disc{\activeLinkers}{}
            + \activeLinkersDiffusiv \stiffnMatrix 
            \disc{\activeLinkers}{}
            =
            \repairRate
            \massMatrix
            \disc{\inactiveLinkers}{}
            -
            \secondRHS_{\rippingLimitParam}
            \left(
                \disc{\membraneHeight}{},
                \disc{\activeLinkers}{}
            \right)
        \end{equation*}
        \begin{equation*}
            \massMatrix
            \pDiff{t}{}{} \disc{\inactiveLinkers}{}
            + \inactiveLinkersDiffusiv 
            \stiffnMatrix
            \disc{\inactiveLinkers}{}
            =
            - \repairRate 
            \massMatrix
            \disc{\inactiveLinkers}{}
            +
            \secondRHS_{\rippingLimitParam}
            \left(
                \disc{\membraneHeight}{},
                \disc{\activeLinkers}{}          
            \right)
        \end{equation*}
    \end{subequations}    
    with a matrices $ \massMatrix,\stiffnMatrix \in \reals^{(\nbIntNodes,\nbIntNodes)} $, where 
    $ \nbIntNodes $ is the number of interior nodes and the corresponding
    basis functions $ \set{\varphi_i}{i\in\{1,\dots \nbIntNodes \}} $ defined there
    and
    \begin{equation*}
        \firstRHS\left(\disc{\activeLinkers}{}, \disc{\membraneHeight}{}\right)
        =
        \left( 
            \innerProd[\lebesgueSet{2}{\domain}]{
                \linkersSpringConst
                \disc{\activeLinkers}{}
                \disc{\membraneHeight}{}         	
            }{ 
               \varphi_i
            }
        \right)_{i\in\{1,\dots,N\}},
    \end{equation*}
    \begin{equation*}
        \secondRHS_{\rippingLimitParam}
        \left(
            \disc{\membraneHeight}{},
            \disc{\activeLinkers}{}
        \right)
        =
        \left( 
            \innerProd[\lebesgueSet{2}{\domain}]{
                \disc{\activeLinkers}{}
                \rippingInterpol_\rippingLimitParam(\disc{\membraneHeight}{})
            }{
                \varphi_i
            }
        \right)_{i\in\{1,\dots,N\}}.
    \end{equation*}
    For implementation of this spatial discretisation, we used
    the FEM solver \textit{Netgen/NGSolve} 
    (\url{https://ngsolve.org}, \cite{schoberl2014c++}).

    Discretisation in time is achieved by applying a semi-implicit Euler scheme
    with time step size $ \timeStepSize > 0 $ and time points
    $ 0 = t_1 < t_2 < \dots < t_k < t_{k+1} < \dots < \finTime $, $ n \in \nats $:
    \begin{subequations}
        \begin{equation}
            \membraneDampingConst
            \timeStepSize^{-1}
            \massMatrix
            \left(
                \disc{\membraneHeight}{k+1}
                -
                \disc{\membraneHeight}{k}            
            \right)
            +
            \membrStiffn 
            \stiffnMatrix
            \disc{w}{k+1}
            +
            \effectiveLengthParam
            \massMatrix
            \disc{w}{k+1}
            =
            \firstRHS\left(\disc{\activeLinkers}{k+1}, \disc{\membraneHeight}{k+1}\right)
            +
            \massMatrix
            \disc{\pressure_0}{}
        \end{equation}
        \begin{equation}
            \massMatrix \disc{w}{k+1} - \stiffnMatrix\disc{\membraneHeight}{k+1} = 0
        \end{equation}
        \begin{equation}
            \label{equ:semi-implicit euler scheme:active linkers}
            \timeStepSize^{-1}
            \massMatrix
            \left(
                \disc{\activeLinkers}{k+1}
                -
                \disc{\activeLinkers}{k}
            \right)
            + 
            \activeLinkersDiffusiv 
            \stiffnMatrix 
            \disc{\activeLinkers}{k+1}
            =
            \repairRate
            \massMatrix
            \disc{\inactiveLinkers}{k+1}
            -
            \secondRHS_{\rippingLimitParam}
            \left(
                \disc{\membraneHeight}{k},
                \disc{\activeLinkers}{k+1}
            \right)
        \end{equation}
        \begin{equation}
            \label{equ:semi-implicit euler scheme:inactive linkers}
            \timeStepSize^{-1}
            \massMatrix
            \left(
                \disc{\inactiveLinkers}{k+1}
                -
                \disc{\inactiveLinkers}{k}
            \right)
            + 
            \inactiveLinkersDiffusiv 
            \stiffnMatrix
            \disc{\inactiveLinkers}{k+1}
            =
            - \repairRate 
            \massMatrix
            \disc{\inactiveLinkers}{k+1}
            +
            \secondRHS_{\rippingLimitParam}
            \left(
                \disc{\membraneHeight}{k},
                \disc{\activeLinkers}{k+1}
            \right).
        \end{equation}
    \end{subequations} 
    To cope with the implicit terms and the nonlinearities, we use Newton's
    fixed point iteration.

\subsection{Scenarios}
	\label{sec:Scenarios}
    For all the simulations we will discuss in the following, the parameters
    of the PDE are as in \autoref{fig:parameter basic choice}.
    The typical expansion time for a bleb to nucleate is about $ 30\,\physUnit{s} $
    \cite{CharrasII+2008}, so we normalised the simulation time with respect to
    this reference quantity.
    The cortex is modelled (as before in the analytic part) as a sphere 
    with radius $ R = 10^{-5}\,\physUnit{m} $ \cite{CharrasII+2008}.
    \paragraph{Nucleation of a bleb}
    	\label{par:nucleation of a bleb}
		We are interested in the (maximal) height of the
        bleb that is nucleated after this time. 
        It has been observed that the typical height of 
        blebs is
        about $ 2\,\mu\,\physUnit{m} $, cf.~\cite{CharrasII+2008}.
        
        We prescribe a pressure as the function
        $ x \mapsto 10^3\,e^{-\frac{d(x,m)}{2r^2}}\,\physUnit{Pa}$,
        where $ d $ is the geodesic distance between the argument 
        $ x \in S^2 $ and a midpoint $ m \in S^2 $ and $ r $ controls
        the width of the pressure pulse. The pressure is constant in time.
        The scaling $ \rippingLimitParam $ of the 
        disconnection rate $ \rippingInterpol_\rippingLimitParam(\membraneHeight) $
        is still free.
        A parameter study shows that with $ \rippingLimitParam = 1.1\cdot10^{-12} $, we
        may achieve a bleb height of $ 1.57\cdot 10^{-7} $, which is about
        one tenth of the experimentally observed bleb height. With a more
        elaborate view on the protein distribution at the membrane and especially
        their behaviour after the critical height $ \criticalHeight $ is
        passed, one may achieve better results.
        For some time points, we plotted the membrane height as a heat map
        on the cortex, see \autoref{fig:bleb formation}.
        \begin{figure}
            {%
            \centering
            \includegraphics[scale=0.9]{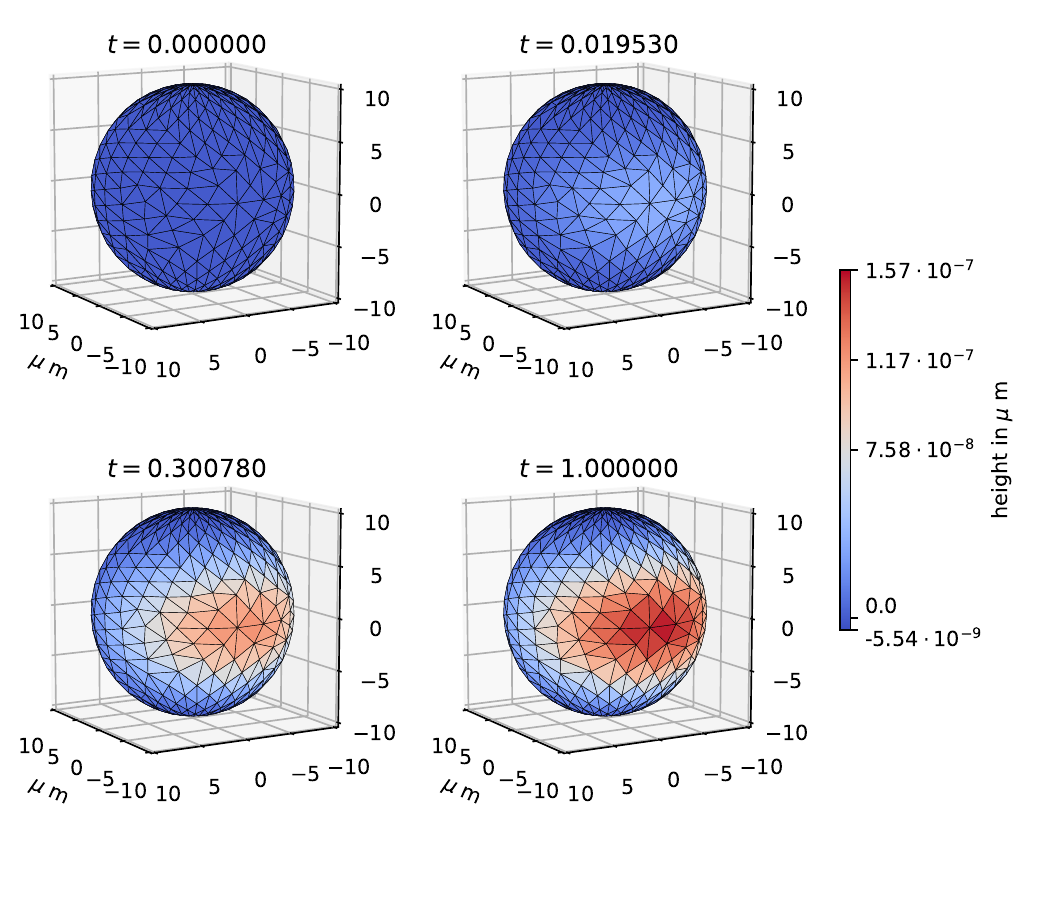}
            \caption{Bleb formation over time}
            \label{fig:bleb formation}
            }
            \label{fig:bleb formation over time}
            The development of a bleb as a height function with
            respect to the cortex (which is modeled by a sphere 
            of radius $ 10^{-5}\,\physUnit{m} $)
            is shown at several points
            in time (normalised with respect to $30\,\physUnit{s}$)
            as reaction to an applied pressure 
            that is  constant in time.
            (For details see p.\,\pageref{par:nucleation of a bleb}.)
        \end{figure}
        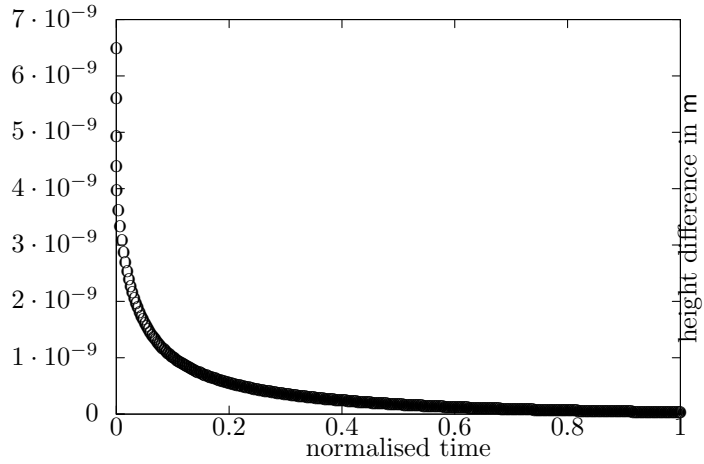
\begin{figure}
            {\centering
          	\input{maxDiff.tex}
            \caption{Maximal height difference}
            \label{fig:maximal height difference between time steps}}
            The maximum of the height function difference of two 
            preceeding time points is plotted against time
            (normalised with respect to $30\,\physUnit{s}$).
            (For details see \autoref{remark:numerics:exponential stability}.)
        \end{figure}
        \begin{remark}
          	\label{remark:numerics:exponential stability}
          	Considering the difference between the maximal height of every
            time step (see \autoref{fig:maximal height difference between time steps}),
            there seems to be numerical indication of a 
            stability property like that rigorously shown in 
            \autoref{sec:local exponential stability}, which
            does not apply in this case as $ \spontMeanCurv = 0$
            and $ \heightLinkerSysHeightBF{\cdot}{\cdot} $ is not coercive.
        \end{remark}
    \paragraph{Critical pressure}
    	\label{par:critical pressure}
    	\cite{Lim+2012} consider a bleb to form
        when the the membrane height reaches above the critical height $ \criticalHeight $
        (on a certain interval $ I \subseteq \reals $)
        and linker bonds are broken in response.
        According to this notion of a bleb in their static model, they
        define the \emph{critical pressure}
        to be the greatest pressure below which the membrane
        height is beneath the critical height everywhere.
        We adopt this notion to our (dynamic) model in the sense
        that the critical pressure $ \pressure[*]_0 $ is the greatest value
        below which the maximal height of the membrane is beneath the critical height
        after the nucleation phase of $ 30\,\physUnit{s} $
        which is driven by the pressure function
        \begin{equation*}
          	 \pressure_0(x) = \pressure[*]_0 
             \cdot e^{-\frac{d(x,m)}{2r^2}}\,\physUnit{Pa}.
        \end{equation*}
        \begin{figure}
            {\centering
           	\input{maxHeightValues.tex}
            \caption{Maximal heights after $ 30\,\physUnit{s} $ against the applied
            	pressure
            }
            \label{fig:max heights}}
            The maximal height that is reached after $ t = 30\,\physUnit{s} $ 
            (black circles)
            is plotted against a time-constant pressure (abscissa) that has been applied.
            The green line is a linear function with parameters $ a \sim 6.38 \cdot 10^{-11} $,
            $ b \sim 0 $ that goes through
            the first two data samples exactly. The blue labels are the critical
            pressure $ \pressure[*]_0 = 16\,\physUnit{Pa} $ and the critical
            height $ \criticalHeight = 10^{-9}\,\physUnit{m} $.
            (For details see p.\,\pageref{par:critical pressure}).
        \end{figure}
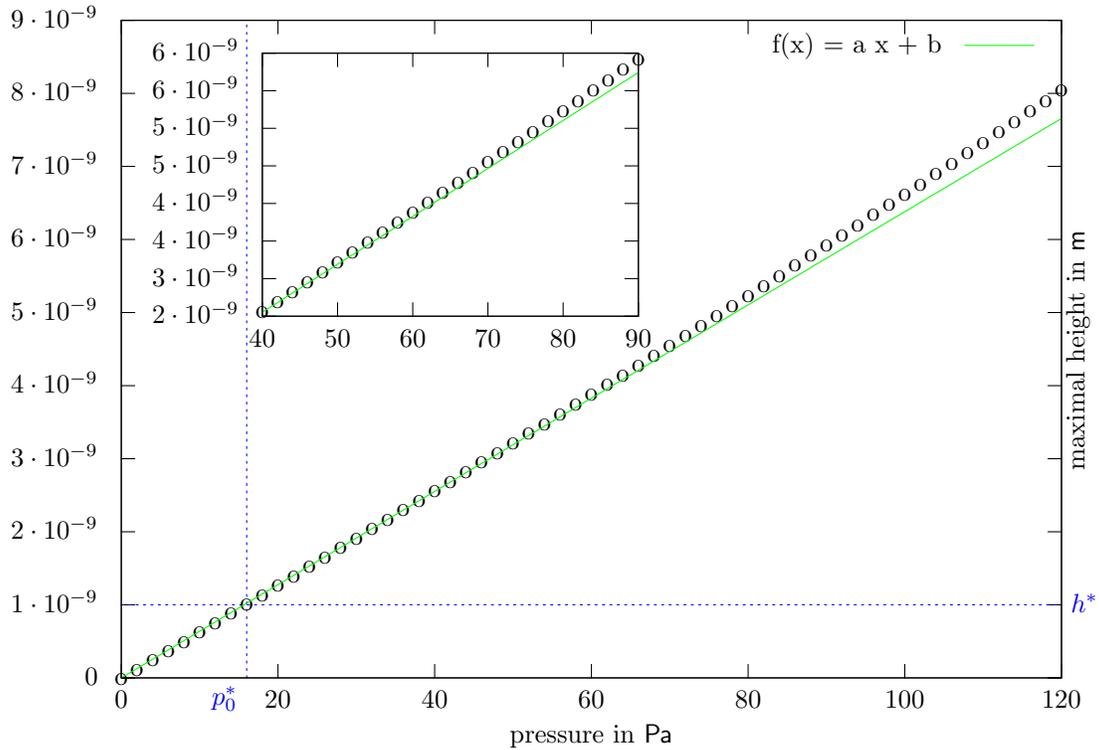
        Passing the critical height triggers the linker disconnection process.
        \begin{table}[h]
            \centering
            \begin{tabular}{c | c}
                Pressure (in $\physUnit{Pa}$) & RMS
                \\
                \hline\hline
                $ 2 $ & $ 0 $\\
                $ 4 $ & $ 4.47\cdot 10^{-17} $ \\
                $ 6 $ & $ 3.27\cdot 10^{-17} $ \\
                $ 8 $ & $ 3.50\cdot 10^{-17} $ \\
                $ 10 $ & $ 3.09\cdot 10^{-17} $ \\
                $ 12 $ & $ 3.28\cdot 10^{-17} $ \\
                $ 14 $ & $ 3.16\cdot 10^{-17} $ \\
                \textcolor{blue}{$ 16 $} & \textcolor{blue}{$ 2.99\cdot 10^{-17} $} \\
                \textcolor{blue}{$ 18 $} & \textcolor{blue}{$ 2.85\cdot 10^{-14} $} \\
                $ 20 $ & $ 1.11\cdot 10^{-13} $ \\
                $ 22 $ & $ 2.40\cdot 10^{-13} $ \\
                $ 24 $ & $ 4.18\cdot 10^{-13} $ \\
                $ 26 $ & $ 6.45\cdot 10^{-13} $ \\
                $ 28 $ & $ 9.17\cdot 10^{-13} $ \\
                $ 30 $ & $ 1.23\cdot 10^{-12} $ \\
                $ \vdots $ & $ \vdots  $ \\
                $ 120 $ & $ 5.98\cdot 10^{-11} $
            \end{tabular}
            \caption{Fitting errors (RMS)}
            \label{fig:fitting errors}
            The tabular shows the root mean squared residual errors
            of the Marquardt-Levenberg algorithm for fitting a linear function into the
            datasets of the intervals $ [0,x] $, where $ x $ is the pressure value
            of the corresponding line.
            The blue marked lines exhibit a significant jump in the RMS. As the critical 
            pressure
            was determined before to be $ 16\,\physUnit{Pa} $ this consolidates the hypothesis
            that above the critical pressure there is a qualitative change in the bleb 
            development. (For details see p.\,\pageref{par:critical pressure}).
        \end{table}
        By applying the pressure as previously described and increasing 
        the pressure in steps of $ 2\,\physUnit{Pa} $, we plotted the maximal height
        after the nucleation phase against the pressure, see \autoref{fig:max %
        heights}. Considering a linear function through the first two data samples
        gives evidence that the
        data samples do not grow linearly at least not on the whole pressure range;
        up to a pressure of about $ 60\,\physUnit{Pa} $ linear growth seems to be a 
        good description. We hypothesise that there is 
        a change in the growing behaviour and that this change occurs at the 
        critical pressure.
        To give some evidence for this, we successively fitted linear functions 
        (with the gnuplot implementation of the Marquardt-Levenberg algorithm) to 
        the datasets of the pressure intervals $ [0:q] $, $ q \in \{2,\dots,120\} $,
        and assessed the root mean square residuum (RMS), see \autoref{fig:fitting %
        errors}. 
        We notice that until
        $ 16\,\physUnit{Pa} $ there is very little change in the RMS. From
        $ 16\,\physUnit{Pa} $ to $ 18\,\physUnit{Pa} $ there is a specifically 
        large increase of the RMS (three orders of magnitude). Afterwards,
        the RMS increases readily. Taking a look at, 
        \autoref{fig:max heights}, we see that $ 16\,\physUnit{Pa} $ happens to be the
        critical pressure which substantiates our hypothesis that reaching the critical 
        pressure triggers a major change in the bleb growth behaviour.

%% file: maxDiff.tex
\begin{tikzpicture}[gnuplot,scale=0.6]
\path (0.000,0.000) rectangle (12.500,8.750);
\gpcolor{color=gp lt color border}
\gpsetlinetype{gp lt border}
\gpsetdashtype{gp dt solid}
\gpsetlinewidth{1.00}
\draw[gp path] (0.000,0.000)--(0.180,0.000);
\draw[gp path] (12.499,0.000)--(12.319,0.000);
\node[gp node right] at (-0.184,0.000) {0};
\draw[gp path] (0.000,0.000)--(0.180,0.000);
\draw[gp path] (12.499,0.000)--(12.319,0.000);
\draw[gp path] (0.000,1.250)--(0.180,1.250);
\draw[gp path] (12.499,1.250)--(12.319,1.250);
\node[gp node right] at (-0.184,1.250) {$1 \cdot 10^{-9}$};
\draw[gp path] (0.000,2.500)--(0.180,2.500);
\draw[gp path] (12.499,2.500)--(12.319,2.500);
\node[gp node right] at (-0.184,2.500) {$2 \cdot 10^{-9}$};
\draw[gp path] (0.000,3.750)--(0.180,3.750);
\draw[gp path] (12.499,3.750)--(12.319,3.750);
\node[gp node right] at (-0.184,3.750) {$3 \cdot 10^{-9}$};
\draw[gp path] (0.000,4.999)--(0.180,4.999);
\draw[gp path] (12.499,4.999)--(12.319,4.999);
\node[gp node right] at (-0.184,4.999) {$4 \cdot 10^{-9}$};
\draw[gp path] (0.000,6.249)--(0.180,6.249);
\draw[gp path] (12.499,6.249)--(12.319,6.249);
\node[gp node right] at (-0.184,6.249) {$5 \cdot 10^{-9}$};
\draw[gp path] (0.000,7.499)--(0.180,7.499);
\draw[gp path] (12.499,7.499)--(12.319,7.499);
\node[gp node right] at (-0.184,7.499) {$6 \cdot 10^{-9}$};
\draw[gp path] (0.000,8.749)--(0.180,8.749);
\draw[gp path] (12.499,8.749)--(12.319,8.749);
\node[gp node right] at (-0.184,8.749) {$7 \cdot 10^{-9}$};
\draw[gp path] (0.000,0.000)--(0.000,0.180);
\draw[gp path] (0.000,8.749)--(0.000,8.569);
\node[gp node center] at (0.000,-0.308) {$0$};
\draw[gp path] (2.500,0.000)--(2.500,0.180);
\draw[gp path] (2.500,8.749)--(2.500,8.569);
\node[gp node center] at (2.500,-0.308) {$0.2$};
\draw[gp path] (5.000,0.000)--(5.000,0.180);
\draw[gp path] (5.000,8.749)--(5.000,8.569);
\node[gp node center] at (5.000,-0.308) {$0.4$};
\draw[gp path] (7.499,0.000)--(7.499,0.180);
\draw[gp path] (7.499,8.749)--(7.499,8.569);
\node[gp node center] at (7.499,-0.308) {$0.6$};
\draw[gp path] (9.999,0.000)--(9.999,0.180);
\draw[gp path] (9.999,8.749)--(9.999,8.569);
\node[gp node center] at (9.999,-0.308) {$0.8$};
\draw[gp path] (12.499,0.000)--(12.499,0.180);
\draw[gp path] (12.499,8.749)--(12.499,8.569);
\node[gp node center] at (12.499,-0.308) {$1$};
\draw[gp path] (0.000,8.749)--(0.000,0.000)--(12.499,0.000)--(12.499,8.749)--cycle;
\node[gp node center,rotate=-270] at (12.730,4.374) {height difference in $\physUnit{m}$};
\node[gp node center] at (6.249,-0.769) {normalised time};
\node[gp node center] at (0.000,8.132) {o};
\node[gp node center] at (0.000,7.029) {o};
\node[gp node center] at (0.000,6.190) {o};
\node[gp node center] at (0.002,5.529) {o};
\node[gp node center] at (0.010,4.996) {o};
\node[gp node center] at (0.049,4.560) {o};
\node[gp node center] at (0.088,4.196) {o};
\node[gp node center] at (0.127,3.890) {o};
\node[gp node center] at (0.166,3.628) {o};
\node[gp node center] at (0.205,3.402) {o};
\node[gp node center] at (0.244,3.204) {o};
\node[gp node center] at (0.283,3.030) {o};
\node[gp node center] at (0.322,2.875) {o};
\node[gp node center] at (0.361,2.735) {o};
\node[gp node center] at (0.400,2.609) {o};
\node[gp node center] at (0.439,2.495) {o};
\node[gp node center] at (0.478,2.390) {o};
\node[gp node center] at (0.518,2.294) {o};
\node[gp node center] at (0.557,2.205) {o};
\node[gp node center] at (0.596,2.123) {o};
\node[gp node center] at (0.635,2.046) {o};
\node[gp node center] at (0.674,1.975) {o};
\node[gp node center] at (0.713,1.908) {o};
\node[gp node center] at (0.752,1.846) {o};
\node[gp node center] at (0.791,1.788) {o};
\node[gp node center] at (0.830,1.733) {o};
\node[gp node center] at (0.869,1.681) {o};
\node[gp node center] at (0.908,1.632) {o};
\node[gp node center] at (0.947,1.586) {o};
\node[gp node center] at (0.986,1.542) {o};
\node[gp node center] at (1.025,1.500) {o};
\node[gp node center] at (1.064,1.461) {o};
\node[gp node center] at (1.103,1.424) {o};
\node[gp node center] at (1.142,1.388) {o};
\node[gp node center] at (1.182,1.354) {o};
\node[gp node center] at (1.221,1.321) {o};
\node[gp node center] at (1.260,1.290) {o};
\node[gp node center] at (1.299,1.261) {o};
\node[gp node center] at (1.338,1.232) {o};
\node[gp node center] at (1.377,1.205) {o};
\node[gp node center] at (1.416,1.179) {o};
\node[gp node center] at (1.455,1.154) {o};
\node[gp node center] at (1.494,1.130) {o};
\node[gp node center] at (1.533,1.107) {o};
\node[gp node center] at (1.572,1.085) {o};
\node[gp node center] at (1.611,1.063) {o};
\node[gp node center] at (1.650,1.043) {o};
\node[gp node center] at (1.689,1.023) {o};
\node[gp node center] at (1.728,1.004) {o};
\node[gp node center] at (1.767,0.985) {o};
\node[gp node center] at (1.806,0.967) {o};
\node[gp node center] at (1.846,0.950) {o};
\node[gp node center] at (1.885,0.933) {o};
\node[gp node center] at (1.924,0.916) {o};
\node[gp node center] at (1.963,0.900) {o};
\node[gp node center] at (2.002,0.885) {o};
\node[gp node center] at (2.041,0.870) {o};
\node[gp node center] at (2.080,0.856) {o};
\node[gp node center] at (2.119,0.841) {o};
\node[gp node center] at (2.158,0.828) {o};
\node[gp node center] at (2.197,0.814) {o};
\node[gp node center] at (2.236,0.801) {o};
\node[gp node center] at (2.275,0.789) {o};
\node[gp node center] at (2.314,0.777) {o};
\node[gp node center] at (2.353,0.765) {o};
\node[gp node center] at (2.392,0.753) {o};
\node[gp node center] at (2.431,0.741) {o};
\node[gp node center] at (2.470,0.730) {o};
\node[gp node center] at (2.510,0.719) {o};
\node[gp node center] at (2.549,0.709) {o};
\node[gp node center] at (2.588,0.699) {o};
\node[gp node center] at (2.627,0.688) {o};
\node[gp node center] at (2.666,0.679) {o};
\node[gp node center] at (2.705,0.669) {o};
\node[gp node center] at (2.744,0.660) {o};
\node[gp node center] at (2.783,0.650) {o};
\node[gp node center] at (2.822,0.641) {o};
\node[gp node center] at (2.861,0.632) {o};
\node[gp node center] at (2.900,0.624) {o};
\node[gp node center] at (2.939,0.615) {o};
\node[gp node center] at (2.978,0.607) {o};
\node[gp node center] at (3.017,0.599) {o};
\node[gp node center] at (3.056,0.591) {o};
\node[gp node center] at (3.095,0.583) {o};
\node[gp node center] at (3.134,0.576) {o};
\node[gp node center] at (3.174,0.568) {o};
\node[gp node center] at (3.213,0.561) {o};
\node[gp node center] at (3.252,0.554) {o};
\node[gp node center] at (3.291,0.547) {o};
\node[gp node center] at (3.330,0.540) {o};
\node[gp node center] at (3.369,0.533) {o};
\node[gp node center] at (3.408,0.527) {o};
\node[gp node center] at (3.447,0.520) {o};
\node[gp node center] at (3.486,0.514) {o};
\node[gp node center] at (3.525,0.507) {o};
\node[gp node center] at (3.564,0.501) {o};
\node[gp node center] at (3.603,0.495) {o};
\node[gp node center] at (3.642,0.489) {o};
\node[gp node center] at (3.681,0.483) {o};
\node[gp node center] at (3.720,0.478) {o};
\node[gp node center] at (3.759,0.472) {o};
\node[gp node center] at (3.799,0.466) {o};
\node[gp node center] at (3.838,0.461) {o};
\node[gp node center] at (3.877,0.456) {o};
\node[gp node center] at (3.916,0.450) {o};
\node[gp node center] at (3.955,0.445) {o};
\node[gp node center] at (3.994,0.440) {o};
\node[gp node center] at (4.033,0.435) {o};
\node[gp node center] at (4.072,0.430) {o};
\node[gp node center] at (4.111,0.425) {o};
\node[gp node center] at (4.150,0.421) {o};
\node[gp node center] at (4.189,0.416) {o};
\node[gp node center] at (4.228,0.411) {o};
\node[gp node center] at (4.267,0.407) {o};
\node[gp node center] at (4.306,0.402) {o};
\node[gp node center] at (4.345,0.398) {o};
\node[gp node center] at (4.384,0.394) {o};
\node[gp node center] at (4.423,0.389) {o};
\node[gp node center] at (4.463,0.385) {o};
\node[gp node center] at (4.502,0.381) {o};
\node[gp node center] at (4.541,0.377) {o};
\node[gp node center] at (4.580,0.373) {o};
\node[gp node center] at (4.619,0.369) {o};
\node[gp node center] at (4.658,0.365) {o};
\node[gp node center] at (4.697,0.361) {o};
\node[gp node center] at (4.736,0.357) {o};
\node[gp node center] at (4.775,0.354) {o};
\node[gp node center] at (4.814,0.350) {o};
\node[gp node center] at (4.853,0.346) {o};
\node[gp node center] at (4.892,0.343) {o};
\node[gp node center] at (4.931,0.339) {o};
\node[gp node center] at (4.970,0.336) {o};
\node[gp node center] at (5.009,0.332) {o};
\node[gp node center] at (5.048,0.329) {o};
\node[gp node center] at (5.087,0.326) {o};
\node[gp node center] at (5.127,0.322) {o};
\node[gp node center] at (5.166,0.319) {o};
\node[gp node center] at (5.205,0.316) {o};
\node[gp node center] at (5.244,0.313) {o};
\node[gp node center] at (5.283,0.310) {o};
\node[gp node center] at (5.322,0.306) {o};
\node[gp node center] at (5.361,0.303) {o};
\node[gp node center] at (5.400,0.300) {o};
\node[gp node center] at (5.439,0.297) {o};
\node[gp node center] at (5.478,0.295) {o};
\node[gp node center] at (5.517,0.292) {o};
\node[gp node center] at (5.556,0.289) {o};
\node[gp node center] at (5.595,0.286) {o};
\node[gp node center] at (5.634,0.283) {o};
\node[gp node center] at (5.673,0.280) {o};
\node[gp node center] at (5.712,0.278) {o};
\node[gp node center] at (5.751,0.275) {o};
\node[gp node center] at (5.791,0.272) {o};
\node[gp node center] at (5.830,0.270) {o};
\node[gp node center] at (5.869,0.267) {o};
\node[gp node center] at (5.908,0.265) {o};
\node[gp node center] at (5.947,0.262) {o};
\node[gp node center] at (5.986,0.260) {o};
\node[gp node center] at (6.025,0.257) {o};
\node[gp node center] at (6.064,0.255) {o};
\node[gp node center] at (6.103,0.252) {o};
\node[gp node center] at (6.142,0.250) {o};
\node[gp node center] at (6.181,0.248) {o};
\node[gp node center] at (6.220,0.245) {o};
\node[gp node center] at (6.259,0.243) {o};
\node[gp node center] at (6.298,0.241) {o};
\node[gp node center] at (6.337,0.239) {o};
\node[gp node center] at (6.376,0.236) {o};
\node[gp node center] at (6.415,0.234) {o};
\node[gp node center] at (6.455,0.232) {o};
\node[gp node center] at (6.494,0.230) {o};
\node[gp node center] at (6.533,0.228) {o};
\node[gp node center] at (6.572,0.226) {o};
\node[gp node center] at (6.611,0.224) {o};
\node[gp node center] at (6.650,0.222) {o};
\node[gp node center] at (6.689,0.220) {o};
\node[gp node center] at (6.728,0.218) {o};
\node[gp node center] at (6.767,0.216) {o};
\node[gp node center] at (6.806,0.214) {o};
\node[gp node center] at (6.845,0.212) {o};
\node[gp node center] at (6.884,0.210) {o};
\node[gp node center] at (6.923,0.208) {o};
\node[gp node center] at (6.962,0.206) {o};
\node[gp node center] at (7.001,0.205) {o};
\node[gp node center] at (7.040,0.203) {o};
\node[gp node center] at (7.079,0.201) {o};
\node[gp node center] at (7.119,0.199) {o};
\node[gp node center] at (7.158,0.197) {o};
\node[gp node center] at (7.197,0.196) {o};
\node[gp node center] at (7.236,0.194) {o};
\node[gp node center] at (7.275,0.192) {o};
\node[gp node center] at (7.314,0.191) {o};
\node[gp node center] at (7.353,0.189) {o};
\node[gp node center] at (7.392,0.187) {o};
\node[gp node center] at (7.431,0.186) {o};
\node[gp node center] at (7.470,0.184) {o};
\node[gp node center] at (7.509,0.182) {o};
\node[gp node center] at (7.548,0.181) {o};
\node[gp node center] at (7.587,0.179) {o};
\node[gp node center] at (7.626,0.178) {o};
\node[gp node center] at (7.665,0.176) {o};
\node[gp node center] at (7.704,0.175) {o};
\node[gp node center] at (7.744,0.173) {o};
\node[gp node center] at (7.783,0.172) {o};
\node[gp node center] at (7.822,0.170) {o};
\node[gp node center] at (7.861,0.169) {o};
\node[gp node center] at (7.900,0.167) {o};
\node[gp node center] at (7.939,0.166) {o};
\node[gp node center] at (7.978,0.165) {o};
\node[gp node center] at (8.017,0.163) {o};
\node[gp node center] at (8.056,0.162) {o};
\node[gp node center] at (8.095,0.160) {o};
\node[gp node center] at (8.134,0.159) {o};
\node[gp node center] at (8.173,0.158) {o};
\node[gp node center] at (8.212,0.156) {o};
\node[gp node center] at (8.251,0.155) {o};
\node[gp node center] at (8.290,0.154) {o};
\node[gp node center] at (8.329,0.153) {o};
\node[gp node center] at (8.368,0.151) {o};
\node[gp node center] at (8.408,0.150) {o};
\node[gp node center] at (8.447,0.149) {o};
\node[gp node center] at (8.486,0.148) {o};
\node[gp node center] at (8.525,0.146) {o};
\node[gp node center] at (8.564,0.145) {o};
\node[gp node center] at (8.603,0.144) {o};
\node[gp node center] at (8.642,0.143) {o};
\node[gp node center] at (8.681,0.142) {o};
\node[gp node center] at (8.720,0.140) {o};
\node[gp node center] at (8.759,0.139) {o};
\node[gp node center] at (8.798,0.138) {o};
\node[gp node center] at (8.837,0.137) {o};
\node[gp node center] at (8.876,0.136) {o};
\node[gp node center] at (8.915,0.135) {o};
\node[gp node center] at (8.954,0.134) {o};
\node[gp node center] at (8.993,0.133) {o};
\node[gp node center] at (9.032,0.131) {o};
\node[gp node center] at (9.072,0.130) {o};
\node[gp node center] at (9.111,0.129) {o};
\node[gp node center] at (9.150,0.128) {o};
\node[gp node center] at (9.189,0.127) {o};
\node[gp node center] at (9.228,0.126) {o};
\node[gp node center] at (9.267,0.125) {o};
\node[gp node center] at (9.306,0.124) {o};
\node[gp node center] at (9.345,0.123) {o};
\node[gp node center] at (9.384,0.122) {o};
\node[gp node center] at (9.423,0.121) {o};
\node[gp node center] at (9.462,0.120) {o};
\node[gp node center] at (9.501,0.119) {o};
\node[gp node center] at (9.540,0.118) {o};
\node[gp node center] at (9.579,0.117) {o};
\node[gp node center] at (9.618,0.117) {o};
\node[gp node center] at (9.657,0.116) {o};
\node[gp node center] at (9.696,0.115) {o};
\node[gp node center] at (9.736,0.114) {o};
\node[gp node center] at (9.775,0.113) {o};
\node[gp node center] at (9.814,0.112) {o};
\node[gp node center] at (9.853,0.111) {o};
\node[gp node center] at (9.892,0.110) {o};
\node[gp node center] at (9.931,0.109) {o};
\node[gp node center] at (9.970,0.108) {o};
\node[gp node center] at (10.009,0.108) {o};
\node[gp node center] at (10.048,0.107) {o};
\node[gp node center] at (10.087,0.106) {o};
\node[gp node center] at (10.126,0.105) {o};
\node[gp node center] at (10.165,0.104) {o};
\node[gp node center] at (10.204,0.104) {o};
\node[gp node center] at (10.243,0.103) {o};
\node[gp node center] at (10.282,0.102) {o};
\node[gp node center] at (10.321,0.101) {o};
\node[gp node center] at (10.360,0.100) {o};
\node[gp node center] at (10.400,0.100) {o};
\node[gp node center] at (10.439,0.099) {o};
\node[gp node center] at (10.478,0.098) {o};
\node[gp node center] at (10.517,0.097) {o};
\node[gp node center] at (10.556,0.096) {o};
\node[gp node center] at (10.595,0.096) {o};
\node[gp node center] at (10.634,0.095) {o};
\node[gp node center] at (10.673,0.094) {o};
\node[gp node center] at (10.712,0.094) {o};
\node[gp node center] at (10.751,0.093) {o};
\node[gp node center] at (10.790,0.092) {o};
\node[gp node center] at (10.829,0.091) {o};
\node[gp node center] at (10.868,0.091) {o};
\node[gp node center] at (10.907,0.090) {o};
\node[gp node center] at (10.946,0.089) {o};
\node[gp node center] at (10.985,0.089) {o};
\node[gp node center] at (11.024,0.088) {o};
\node[gp node center] at (11.064,0.087) {o};
\node[gp node center] at (11.103,0.087) {o};
\node[gp node center] at (11.142,0.086) {o};
\node[gp node center] at (11.181,0.085) {o};
\node[gp node center] at (11.220,0.085) {o};
\node[gp node center] at (11.259,0.084) {o};
\node[gp node center] at (11.298,0.083) {o};
\node[gp node center] at (11.337,0.083) {o};
\node[gp node center] at (11.376,0.082) {o};
\node[gp node center] at (11.415,0.082) {o};
\node[gp node center] at (11.454,0.081) {o};
\node[gp node center] at (11.493,0.080) {o};
\node[gp node center] at (11.532,0.080) {o};
\node[gp node center] at (11.571,0.079) {o};
\node[gp node center] at (11.610,0.079) {o};
\node[gp node center] at (11.649,0.078) {o};
\node[gp node center] at (11.689,0.077) {o};
\node[gp node center] at (11.728,0.077) {o};
\node[gp node center] at (11.767,0.076) {o};
\node[gp node center] at (11.806,0.076) {o};
\node[gp node center] at (11.845,0.075) {o};
\node[gp node center] at (11.884,0.074) {o};
\node[gp node center] at (11.923,0.074) {o};
\node[gp node center] at (11.962,0.073) {o};
\node[gp node center] at (12.001,0.073) {o};
\node[gp node center] at (12.040,0.072) {o};
\node[gp node center] at (12.079,0.072) {o};
\node[gp node center] at (12.118,0.071) {o};
\node[gp node center] at (12.157,0.071) {o};
\node[gp node center] at (12.196,0.070) {o};
\node[gp node center] at (12.235,0.070) {o};
\node[gp node center] at (12.274,0.069) {o};
\node[gp node center] at (12.313,0.069) {o};
\node[gp node center] at (12.353,0.068) {o};
\node[gp node center] at (12.392,0.068) {o};
\node[gp node center] at (12.431,0.067) {o};
\node[gp node center] at (12.470,0.067) {o};
\node[gp node center] at (12.499,0.066) {o};
\draw[gp path] (0.000,8.749)--(0.000,0.000)--(12.499,0.000)--(12.499,8.749)--cycle;
\gpdefrectangularnode{gp plot 1}{\pgfpoint{0.000cm}{0.000cm}}{\pgfpoint{12.499cm}{8.749cm}}
\end{tikzpicture}

%% file: maxHeightValues.tex
\begin{tikzpicture}[gnuplot]
\path (0.000,0.000) rectangle (12.500,8.750);
\gpcolor{color=gp lt color border}
\gpsetlinetype{gp lt border}
\gpsetdashtype{gp dt solid}
\gpsetlinewidth{1.00}
\draw[gp path] (0.000,0.000)--(0.180,0.000);
\draw[gp path] (12.499,0.000)--(12.319,0.000);
\node[gp node right] at (-0.184,0.000) {0};
\draw[gp path] (0.000,0.000)--(0.180,0.000);
\draw[gp path] (12.499,0.000)--(12.319,0.000);
\draw[gp path] (0.000,0.972)--(0.180,0.972);
\draw[gp path] (12.499,0.972)--(12.319,0.972);
\node[gp node right] at (-0.184,0.972) {$1 \cdot 10^{-9}$};
\draw[gp path] (0.000,1.944)--(0.180,1.944);
\draw[gp path] (12.499,1.944)--(12.319,1.944);
\node[gp node right] at (-0.184,1.944) {$2 \cdot 10^{-9}$};
\draw[gp path] (0.000,2.916)--(0.180,2.916);
\draw[gp path] (12.499,2.916)--(12.319,2.916);
\node[gp node right] at (-0.184,2.916) {$3 \cdot 10^{-9}$};
\draw[gp path] (0.000,3.888)--(0.180,3.888);
\draw[gp path] (12.499,3.888)--(12.319,3.888);
\node[gp node right] at (-0.184,3.888) {$4 \cdot 10^{-9}$};
\draw[gp path] (0.000,4.861)--(0.180,4.861);
\draw[gp path] (12.499,4.861)--(12.319,4.861);
\node[gp node right] at (-0.184,4.861) {$5 \cdot 10^{-9}$};
\draw[gp path] (0.000,5.833)--(0.180,5.833);
\draw[gp path] (12.499,5.833)--(12.319,5.833);
\node[gp node right] at (-0.184,5.833) {$6 \cdot 10^{-9}$};
\draw[gp path] (0.000,6.805)--(0.180,6.805);
\draw[gp path] (12.499,6.805)--(12.319,6.805);
\node[gp node right] at (-0.184,6.805) {$7 \cdot 10^{-9}$};
\draw[gp path] (0.000,7.777)--(0.180,7.777);
\draw[gp path] (12.499,7.777)--(12.319,7.777);
\node[gp node right] at (-0.184,7.777) {$8 \cdot 10^{-9}$};
\draw[gp path] (0.000,8.749)--(0.180,8.749);
\draw[gp path] (12.499,8.749)--(12.319,8.749);
\node[gp node right] at (-0.184,8.749) {$9 \cdot 10^{-9}$};
\draw[gp path] (0.000,0.000)--(0.000,0.180);
\draw[gp path] (0.000,8.749)--(0.000,8.569);
\node[gp node center] at (0.000,-0.308) {$0$};
\draw[gp path] (2.083,0.000)--(2.083,0.180);
\draw[gp path] (2.083,8.749)--(2.083,8.569);
\node[gp node center] at (2.083,-0.308) {$20$};
\draw[gp path] (4.166,0.000)--(4.166,0.180);
\draw[gp path] (4.166,8.749)--(4.166,8.569);
\node[gp node center] at (4.166,-0.308) {$40$};
\draw[gp path] (6.250,0.000)--(6.250,0.180);
\draw[gp path] (6.250,8.749)--(6.250,8.569);
\node[gp node center] at (6.250,-0.308) {$60$};
\draw[gp path] (8.333,0.000)--(8.333,0.180);
\draw[gp path] (8.333,8.749)--(8.333,8.569);
\node[gp node center] at (8.333,-0.308) {$80$};
\draw[gp path] (10.416,0.000)--(10.416,0.180);
\draw[gp path] (10.416,8.749)--(10.416,8.569);
\node[gp node center] at (10.416,-0.308) {$100$};
\draw[gp path] (12.499,0.000)--(12.499,0.180);
\draw[gp path] (12.499,8.749)--(12.499,8.569);
\node[gp node center] at (12.499,-0.308) {$120$};
\draw[gp path] (0.000,8.749)--(0.000,0.000)--(12.499,0.000)--(12.499,8.749)--cycle;
\gpcolor{rgb color={0.000,0.000,1.000}}
\node[gp node right] at (1.667,-0.261) {$\pressure[*]_0$};
\node[gp node left] at (12.499,0.972) {$\criticalHeight$};
\gpsetdashtype{dash pattern=on 2.00*\gpdashlength off 5.00*\gpdashlength }
\draw[gp path](1.667,0.000)--(1.667,8.749);
\draw[gp path](0.000,0.972)--(12.499,0.972);
\gpcolor{color=gp lt color border}
\node[gp node center,rotate=-270] at (12.730,4.374) {maximal height in $\physUnit{m}$};
\node[gp node center] at (6.249,-0.769) {pressure in $\physUnit{Pa}$};
\node[gp node center] at (0.000,0.000) {o};
\node[gp node center] at (0.208,0.124) {o};
\node[gp node center] at (0.417,0.248) {o};
\node[gp node center] at (0.625,0.372) {o};
\node[gp node center] at (0.833,0.496) {o};
\node[gp node center] at (1.042,0.620) {o};
\node[gp node center] at (1.250,0.744) {o};
\node[gp node center] at (1.458,0.868) {o};
\node[gp node center] at (1.667,0.992) {o};
\node[gp node center] at (1.875,1.117) {o};
\node[gp node center] at (2.083,1.241) {o};
\node[gp node center] at (2.291,1.366) {o};
\node[gp node center] at (2.500,1.490) {o};
\node[gp node center] at (2.708,1.615) {o};
\node[gp node center] at (2.916,1.741) {o};
\node[gp node center] at (3.125,1.866) {o};
\node[gp node center] at (3.333,1.992) {o};
\node[gp node center] at (3.541,2.118) {o};
\node[gp node center] at (3.750,2.244) {o};
\node[gp node center] at (3.958,2.371) {o};
\node[gp node center] at (4.166,2.498) {o};
\node[gp node center] at (4.375,2.625) {o};
\node[gp node center] at (4.583,2.752) {o};
\node[gp node center] at (4.791,2.880) {o};
\node[gp node center] at (5.000,3.008) {o};
\node[gp node center] at (5.208,3.136) {o};
\node[gp node center] at (5.416,3.265) {o};
\node[gp node center] at (5.625,3.394) {o};
\node[gp node center] at (5.833,3.523) {o};
\node[gp node center] at (6.041,3.652) {o};
\node[gp node center] at (6.250,3.782) {o};
\node[gp node center] at (6.458,3.912) {o};
\node[gp node center] at (6.666,4.042) {o};
\node[gp node center] at (6.874,4.173) {o};
\node[gp node center] at (7.083,4.304) {o};
\node[gp node center] at (7.291,4.436) {o};
\node[gp node center] at (7.499,4.567) {o};
\node[gp node center] at (7.708,4.699) {o};
\node[gp node center] at (7.916,4.832) {o};
\node[gp node center] at (8.124,4.964) {o};
\node[gp node center] at (8.333,5.097) {o};
\node[gp node center] at (8.541,5.231) {o};
\node[gp node center] at (8.749,5.364) {o};
\node[gp node center] at (8.958,5.498) {o};
\node[gp node center] at (9.166,5.632) {o};
\node[gp node center] at (9.374,5.767) {o};
\node[gp node center] at (9.583,5.902) {o};
\node[gp node center] at (9.791,6.037) {o};
\node[gp node center] at (9.999,6.173) {o};
\node[gp node center] at (10.208,6.309) {o};
\node[gp node center] at (10.416,6.446) {o};
\node[gp node center] at (10.624,6.582) {o};
\node[gp node center] at (10.832,6.719) {o};
\node[gp node center] at (11.041,6.857) {o};
\node[gp node center] at (11.249,6.994) {o};
\node[gp node center] at (11.457,7.133) {o};
\node[gp node center] at (11.666,7.271) {o};
\node[gp node center] at (11.874,7.410) {o};
\node[gp node center] at (12.082,7.549) {o};
\node[gp node center] at (12.291,7.689) {o};
\node[gp node center] at (12.499,7.829) {o};
\node[gp node right] at (11.031,8.415) {f(x) = a x + b};
\gpcolor{rgb color={0.000,1.000,0.000}}
\gpsetdashtype{gp dt solid}
\draw[gp path] (11.215,8.415)--(12.131,8.415);
\draw[gp path] (0.000,0.000)--(0.126,0.075)--(0.253,0.150)--(0.379,0.226)--(0.505,0.301)%
  --(0.631,0.376)--(0.758,0.451)--(0.884,0.526)--(1.010,0.601)--(1.136,0.677)--(1.263,0.752)%
  --(1.389,0.827)--(1.515,0.902)--(1.641,0.977)--(1.768,1.053)--(1.894,1.128)--(2.020,1.203)%
  --(2.146,1.278)--(2.273,1.353)--(2.399,1.429)--(2.525,1.504)--(2.651,1.579)--(2.778,1.654)%
  --(2.904,1.729)--(3.030,1.804)--(3.156,1.880)--(3.283,1.955)--(3.409,2.030)--(3.535,2.105)%
  --(3.661,2.180)--(3.788,2.256)--(3.914,2.331)--(4.040,2.406)--(4.166,2.481)--(4.293,2.556)%
  --(4.419,2.632)--(4.545,2.707)--(4.671,2.782)--(4.798,2.857)--(4.924,2.932)--(5.050,3.007)%
  --(5.176,3.083)--(5.303,3.158)--(5.429,3.233)--(5.555,3.308)--(5.681,3.383)--(5.808,3.459)%
  --(5.934,3.534)--(6.060,3.609)--(6.186,3.684)--(6.313,3.759)--(6.439,3.835)--(6.565,3.910)%
  --(6.691,3.985)--(6.818,4.060)--(6.944,4.135)--(7.070,4.210)--(7.196,4.286)--(7.323,4.361)%
  --(7.449,4.436)--(7.575,4.511)--(7.701,4.586)--(7.828,4.662)--(7.954,4.737)--(8.080,4.812)%
  --(8.206,4.887)--(8.333,4.962)--(8.459,5.037)--(8.585,5.113)--(8.711,5.188)--(8.838,5.263)%
  --(8.964,5.338)--(9.090,5.413)--(9.216,5.489)--(9.343,5.564)--(9.469,5.639)--(9.595,5.714)%
  --(9.721,5.789)--(9.848,5.865)--(9.974,5.940)--(10.100,6.015)--(10.226,6.090)--(10.353,6.165)%
  --(10.479,6.240)--(10.605,6.316)--(10.731,6.391)--(10.858,6.466)--(10.984,6.541)--(11.110,6.616)%
  --(11.236,6.692)--(11.363,6.767)--(11.489,6.842)--(11.615,6.917)--(11.741,6.992)--(11.868,7.068)%
  --(11.994,7.143)--(12.120,7.218)--(12.246,7.293)--(12.373,7.368)--(12.499,7.443);
\gpcolor{color=gp lt color border}
\draw[gp path] (0.000,8.749)--(0.000,0.000)--(12.499,0.000)--(12.499,8.749)--cycle;
\gpdefrectangularnode{gp plot 1}{\pgfpoint{0.000cm}{0.000cm}}{\pgfpoint{12.499cm}{8.749cm}}
\draw[gp path] (1.875,4.812)--(2.055,4.812);
\draw[gp path] (6.874,4.812)--(6.694,4.812);
\node[gp node right] at (1.691,4.812) {$2 \cdot 10^{-9}$};
\draw[gp path] (1.875,5.312)--(2.055,5.312);
\draw[gp path] (6.874,5.312)--(6.694,5.312);
\node[gp node right] at (1.691,5.312) {$3 \cdot 10^{-9}$};
\draw[gp path] (1.875,5.812)--(2.055,5.812);
\draw[gp path] (6.874,5.812)--(6.694,5.812);
\node[gp node right] at (1.691,5.812) {$4 \cdot 10^{-9}$};
\draw[gp path] (1.875,6.312)--(2.055,6.312);
\draw[gp path] (6.874,6.312)--(6.694,6.312);
\node[gp node right] at (1.691,6.312) {$4 \cdot 10^{-9}$};
\draw[gp path] (1.875,6.812)--(2.055,6.812);
\draw[gp path] (6.874,6.812)--(6.694,6.812);
\node[gp node right] at (1.691,6.812) {$5 \cdot 10^{-9}$};
\draw[gp path] (1.875,7.312)--(2.055,7.312);
\draw[gp path] (6.874,7.312)--(6.694,7.312);
\node[gp node right] at (1.691,7.312) {$5 \cdot 10^{-9}$};
\draw[gp path] (1.875,7.812)--(2.055,7.812);
\draw[gp path] (6.874,7.812)--(6.694,7.812);
\node[gp node right] at (1.691,7.812) {$6 \cdot 10^{-9}$};
\draw[gp path] (1.875,8.312)--(2.055,8.312);
\draw[gp path] (6.874,8.312)--(6.694,8.312);
\node[gp node right] at (1.691,8.312) {$6 \cdot 10^{-9}$};
\draw[gp path] (1.875,4.812)--(1.875,4.992);
\draw[gp path] (1.875,8.312)--(1.875,8.132);
\node[gp node center] at (1.875,4.504) {$40$};
\draw[gp path] (2.875,4.812)--(2.875,4.992);
\draw[gp path] (2.875,8.312)--(2.875,8.132);
\node[gp node center] at (2.875,4.504) {$50$};
\draw[gp path] (3.875,4.812)--(3.875,4.992);
\draw[gp path] (3.875,8.312)--(3.875,8.132);
\node[gp node center] at (3.875,4.504) {$60$};
\draw[gp path] (4.874,4.812)--(4.874,4.992);
\draw[gp path] (4.874,8.312)--(4.874,8.132);
\node[gp node center] at (4.874,4.504) {$70$};
\draw[gp path] (5.874,4.812)--(5.874,4.992);
\draw[gp path] (5.874,8.312)--(5.874,8.132);
\node[gp node center] at (5.874,4.504) {$80$};
\draw[gp path] (6.874,4.812)--(6.874,4.992);
\draw[gp path] (6.874,8.312)--(6.874,8.132);
\node[gp node center] at (6.874,4.504) {$90$};
\draw[gp path] (1.875,8.312)--(1.875,4.812)--(6.874,4.812)--(6.874,8.312)--cycle;
\node[gp node center] at (1.875,4.881) {o};
\node[gp node center] at (2.075,5.012) {o};
\node[gp node center] at (2.275,5.143) {o};
\node[gp node center] at (2.475,5.274) {o};
\node[gp node center] at (2.675,5.406) {o};
\node[gp node center] at (2.875,5.538) {o};
\node[gp node center] at (3.075,5.670) {o};
\node[gp node center] at (3.275,5.803) {o};
\node[gp node center] at (3.475,5.936) {o};
\node[gp node center] at (3.675,6.069) {o};
\node[gp node center] at (3.875,6.202) {o};
\node[gp node center] at (4.075,6.336) {o};
\node[gp node center] at (4.275,6.470) {o};
\node[gp node center] at (4.474,6.605) {o};
\node[gp node center] at (4.674,6.740) {o};
\node[gp node center] at (4.874,6.875) {o};
\node[gp node center] at (5.074,7.010) {o};
\node[gp node center] at (5.274,7.146) {o};
\node[gp node center] at (5.474,7.282) {o};
\node[gp node center] at (5.674,7.419) {o};
\node[gp node center] at (5.874,7.555) {o};
\node[gp node center] at (6.074,7.693) {o};
\node[gp node center] at (6.274,7.830) {o};
\node[gp node center] at (6.474,7.968) {o};
\node[gp node center] at (6.674,8.106) {o};
\node[gp node center] at (6.874,8.245) {o};
\gpcolor{rgb color={0.000,1.000,0.000}}
\draw[gp path] (1.875,4.864)--(1.925,4.897)--(1.976,4.929)--(2.026,4.961)--(2.077,4.993)%
  --(2.127,5.025)--(2.178,5.058)--(2.228,5.090)--(2.279,5.122)--(2.329,5.154)--(2.380,5.187)%
  --(2.430,5.219)--(2.481,5.251)--(2.531,5.283)--(2.582,5.316)--(2.632,5.348)--(2.683,5.380)%
  --(2.733,5.412)--(2.784,5.444)--(2.834,5.477)--(2.885,5.509)--(2.935,5.541)--(2.986,5.573)%
  --(3.036,5.606)--(3.087,5.638)--(3.137,5.670)--(3.188,5.702)--(3.238,5.734)--(3.289,5.767)%
  --(3.339,5.799)--(3.390,5.831)--(3.440,5.863)--(3.491,5.896)--(3.541,5.928)--(3.592,5.960)%
  --(3.642,5.992)--(3.693,6.024)--(3.743,6.057)--(3.794,6.089)--(3.844,6.121)--(3.895,6.153)%
  --(3.945,6.186)--(3.996,6.218)--(4.046,6.250)--(4.097,6.282)--(4.147,6.315)--(4.198,6.347)%
  --(4.248,6.379)--(4.299,6.411)--(4.349,6.443)--(4.400,6.476)--(4.450,6.508)--(4.501,6.540)%
  --(4.551,6.572)--(4.602,6.605)--(4.652,6.637)--(4.703,6.669)--(4.753,6.701)--(4.804,6.733)%
  --(4.854,6.766)--(4.905,6.798)--(4.955,6.830)--(5.006,6.862)--(5.056,6.895)--(5.107,6.927)%
  --(5.157,6.959)--(5.208,6.991)--(5.258,7.024)--(5.309,7.056)--(5.359,7.088)--(5.410,7.120)%
  --(5.460,7.152)--(5.511,7.185)--(5.561,7.217)--(5.612,7.249)--(5.662,7.281)--(5.713,7.314)%
  --(5.763,7.346)--(5.814,7.378)--(5.864,7.410)--(5.915,7.442)--(5.965,7.475)--(6.016,7.507)%
  --(6.066,7.539)--(6.117,7.571)--(6.167,7.604)--(6.218,7.636)--(6.268,7.668)--(6.319,7.700)%
  --(6.369,7.732)--(6.420,7.765)--(6.470,7.797)--(6.521,7.829)--(6.571,7.861)--(6.622,7.894)%
  --(6.672,7.926)--(6.723,7.958)--(6.773,7.990)--(6.824,8.023)--(6.874,8.055);
\gpcolor{color=gp lt color border}
\draw[gp path] (1.875,8.312)--(1.875,4.812)--(6.874,4.812)--(6.874,8.312)--cycle;
\gpdefrectangularnode{gp plot 2}{\pgfpoint{1.875cm}{4.812cm}}{\pgfpoint{6.874cm}{8.312cm}}
\end{tikzpicture}

%% file: Appendix_Stokes_Dirichlet_to_Neumann.tex
\renewcommand{\domain}{\sndRevision[second]{\mathcal{D}} }
We define the Neumann-to-Dirichlet operator for the 
stationary Stokes problem 
\begin{equation}
	\label{equ:appendix:stationary Stokes}
  	\begin{split}
        \mu\innerProd[\bochnerLebesgueSet{2}{\domain}{\reals^{(3,3)}}]{
            \symmGrad{ \velocity }
        }{
            \symmGrad{ \testFuncVelocity }
        }
        -
        \innerProd[\lebesgueSet{2}{\domain}]{
            \pressure
        }{
            \diver{}{}{} \testFuncVelocity
        }
        &=
        \innerProd[\lebesgueSet{2}{\membrane}]{
            f
        }{
            \traceOp[\membrane]{\testFuncVelocity}
        }
        \\
        \innerProd[\lebesgueSet{2}{\domain}]{
            \diver{}{}{} \velocity
        }{ 
            q 
        } 
        &= 0
        \\
        \traceOp[\Gamma]{\velocity} &= 0
        \\
        \traceOp[\membrane]{ \velocity } &= g
    \end{split}
\end{equation}
with $ \velocity, \testFuncVelocity \in \sobolevHSet{1}{\domain} $,
$ \pressure, q \in \lebesgueSet{2}{\domain} $,
and
$ g \in \stokesDirDataSpace{1}{\membrane} $,
$ f \in \lebesgueTracesOfSolenoidals{2}{\membrane} $,
where 
\begin{equation*}
	\stokesDirDataSpace{1}{\membrane} = 
    \set{ 
        u \in \bochnerSobolevHSet{1}{\membrane}{\reals^3} 
    }{ 
        \integral{\membrane}{}{ u \normal[\membrane]{} }{x} = 0
    }
\end{equation*}
as function 
\begin{equation*}
	\funSig{\dirToNeumOp{}}{
      	\sobolevTracesOfSolenoidals{1}{\membrane}
    }{
      	\lebesgueTracesOfSolenoidals{2}{\membrane}
    }
\end{equation*}
mapping $ g $ to $ f $. 
\begin{remark}
    For the Stokes problem with mean-value-free pressure $ \pressure $ and no restriction
    on the Neumann data $ f $, the well-posedness of this operator 
    follows from \cite{Fabes+1988}, Theorem~4.15.
    But a unique solution of the above problem is then easily defined by
    setting $ \tilde{\pressure} = \pressure +
    \avgIntegral{\membrane}{}{f\cdot\normal[\membrane]{}}{\hausdorffM{2}} $.
\end{remark}
We further have the following properties
\begin{mylemma}
	\label{lemma:Dirichlet-to-Neumann operator Stokes:properties}
    \begin{itemize}
    	\item[(i)] $ \dirToNeumOp{} $ is self-adjoint and positive definite.
        \item[(ii)] There exists a positive definite operator $ \funSig{S}{
        \lebesgueTracesOfSolenoidals{2}{\membrane}}{\lebesgueTracesOfSolenoidals{2}{\membrane}} $
        such that $ S^2 = \dirToNeumOp{} $ on the domain of $ \dirToNeumOp{} $.
        \item[(iii)] $ \dirToNeumOp{} $ is continuous 
        from the $ \norm[\sobolevHSet{1}{\membrane}]{\cdot} $ topology of its domain
        to the $ \norm[\lebesgueSet{2}{\membrane}]{\cdot} $ topology of its range.
    \end{itemize}
\end{mylemma}
\begin{proof}
(i)
We may express the Dirichlet-to-Neumann operator as
\begin{align*}
  	\dirToNeumOp{} 
    = 
    \bigg\{ 
        (g,f) \in \left( \lebesgueTracesOfSolenoidals{2}{\membrane} \right)^2
        \bigg\vert\quad
  		&\exists \velocity \in \sobolevSolenoidals{\frac{3}{2}}{\membrane},
        \traceOp[\membrane]{\velocity} = g
        \colon
        \forall \testFuncVelocity \in \sobolevSolenoidals{\frac{3}{2}}{\membrane} \colon
        \\
        &\mu \innerProd[\lebesgueSet{2}{\domain}]{
          	\symmGrad{\velocity}
        }{
          	\symmGrad{\testFuncVelocity}
        }
        =
        \innerProd[\lebesgueSet{2}{\membrane}]{
          	f
        }{
          	\traceOp[\membrane]{\testFuncVelocity}
        }
    \bigg\}.
\end{align*}
By form methods (as used e.\,g. in \cite{Fujita+2001}, Chapter~7), the claim
follows from the coercivity, continuity, and symmetry of
$ (\velocity,\testFuncVelocity) \mapsto \innerProd[\lebesgueSet{2}{\domain}]{
\symmGrad{\velocity}}{\symmGrad{\testFuncVelocity}} $.

(ii) We refer to \cite{Sebestyen+2017}, Theorem~2.3.

(iii) This can be directly derived from \cite{Fabes+1988}, Theorem~4.15 
and the continuity of the trace operator.
\end{proof}


%% file: Appendix_Galerkin_approximation_height_equation.tex
\begin{mylemma}
	There is exactly one 
    $ \membraneHeight \in \bochnerLebesgueSet{2}{[0,\finTime]}{\sobolevHSetMVF{2}{\cortex}}
    \cap 
    \bochnerSobolevHSet{1}{[0,\finTime]}{\sobolevHSetMVF{1}{\cortex}} $
    satisfying
    \begin{equation}
      	\label{equ:appendix:existence and uniqueness of solutions of the height equation}
        \innerProd[ \lebesgueSet{2}{\cortex} ]{
            \timeDerivOperator\left(
                \pDiff{t}{}{}\membraneHeight \normal[\cortex]{}
            \right)
        }{
          	\testFuncMembraneHeight
            \normal[\cortex]{}
        }
        +
        \heightLinkerSysHeightBF{\membraneHeight}{\testFuncMembraneHeight}
        =
        \innerProd[ \lebesgueSet{2}{\cortex} ]{
          	f(t)
        }{
			\testFuncMembraneHeight
        }
    \end{equation}
    with
    \begin{equation*}
      	\heightLinkerSysHeightBF{\membraneHeight}{\testFuncMembraneHeight}
        =
        \membrStiffn 
        \innerProd[ \lebesgueSet{2}{\cortex} ]{
          	\laplacian{\cortex}{}{} \membraneHeight
        }{
          	\laplacian{\cortex}{}{} \testFuncMembraneHeight
        }
        + 
        \effectiveLengthParam
        \innerProd[ \bochnerLebesgueSet{2}{\cortex}{\reals^3} ]{
	        \grad{\cortex}{}{} \membraneHeight
        }{
          	\grad{\cortex}{}{} \testFuncMembraneHeight
        }
        +
        \anotherMembrParam
        \innerProd[ \lebesgueSet{2}{\cortex} ]{
          	\membraneHeight
        }{
          	\testFuncMembraneHeight
        }        
    \end{equation*}
    for almost every $ t \in [0,\finTime) $ and all $ \varphi \in \sobolevHSet{2}{\cortex} $,
    where
    $ f \in \bochnerDiffSet{}{[0,\finTime)}{\lebesgueSet{2}{\cortex}} $
    and $ \membraneHeight(0) = \membraneHeight_0 \in \lebesgueSet[0]{2}{\cortex} $.
\end{mylemma}
\begin{proof}
    We argue by a Petrov-Galerkin-type approximation:
    
    1)
    Let $ \left( \varphi_i \right)_{i\in\nats} $ be an orthonormal Schauder basis of
    $ \lebesgueSet{2}{\cortex} $ with eigenvalues $ ( \lambda_i )_{i\in\nats} $
    (sorted ascendingly)
    consisting of
    eigenfunctions of the Laplace-Beltrami operator $ \laplacian{\cortex}{}{} $
    (due to the divergence theorem and 
    $ \boundary{\cortex} = \emptyset $, the eigenfunctions $ \varphi_i $, $ i \geq 2 $, 
    are mean value free;
    for a spectral 
    theorem on Riemannian manifolds cf. \cite{Lablee2015}, Theorem~4.3.1).
    For $ m \in \nats $, we set $ \membraneHeight_m(t,x) = \sum_{i=2}^m \membraneHeight_i(t) \varphi_i(x) $
    and $ f_m(t,x) = \sum_{i=1}^m f_i(t) \varphi_i(x) $.
    Formally inserting into \eqref{equ:appendix:existence and uniqueness of solutions %
    of the height equation} and testing with $ \varphi_j $, $ j \in \{1,\dots,m \} $,
    we obtain
    the finite-dimensional system for the coefficient vectors
    $ \underline{\membraneHeight}_m = \left(\membraneHeight_i\right)_{i \in \{2,\dots,m\}} $
    and right hand side $ \underline{f}_m = \left( f_j \right)_{j\in\{1,\dots,m\}} $:
    \begin{equation}
      	\label{equ:appendix:existence and uniqueness of solutions of the height equation:%
        discrete system}
      	M \underline{\membraneHeight}_m' + \membrStiffn A^2 \underline{\membraneHeight}_m
        +
        \effectiveLengthParam A \underline{\membraneHeight}_m 
        +
        \anotherMembrParam
        \underline{\membraneHeight}_m
        =
        \underline{f}_m,
    \end{equation}
    where $ A = \text{diag}\left(\lambda_2,\dots,\lambda_m\right) $
    and 
    $ M = \left( 
        \innerProd[\lebesgueSet{2}{\cortex}]{
            \timeDerivOperator\left(
              	\varphi_i \normal[\cortex]{}
            \right)
        }{
			\varphi_j \normal[\cortex]{}
        }
    \right)_{j\in\{1,\dots,m\},i\in\{2,\dots,m\}} $. Due to the symmetry and positive definiteness of 
    $ \timeDerivOperator((\cdot)\normal[\cortex]{}) \cdot \normal[\cortex]{} $
    in $ \lebesgueSet{2}{\cortex} $ (see Lemma~\ref{lemma:Dirichlet-to-Neumann operator Stokes:%
    properties}), the columns of $ M $ are linearly independent
    ($ M$ without its last line would be invertible), and
    well-posedness of \eqref{equ:appendix:existence and uniqueness of solutions of the height %
    equation:discrete system} complemented by the initial condition 
    $ \underline{\membraneHeight}(0) =
    \underline{\membraneHeight}_{0,m} = \left( \innerProd[\lebesgueSet{2}{\cortex}]{
    \membraneHeight_0}{\varphi_j} \right)_{j\in\{2,\dots,m\}} $
    follows by multiplying with the Moore-Penrose pseudo left inverse
    $ (M^TM)^{-1}M^T $ and the Picard-Lindelöf theorem.

    2) Multiplying 
    \eqref{equ:appendix:existence and uniqueness of solutions of the height equation:%
    discrete system} by $ \membraneHeight_m = \left( 0, \membraneHeight_i \right)_{i\in\{2,\dots,m\}} $
    in the Euclidean scalar product, we obtain
    \begin{equation*}
      	\innerProd[\lebesgueSet{2}{\cortex}]{
          	\timeDerivOperator(\membraneHeight_m')
        }{
          	\membraneHeight_m
        }
        + \membrStiffn 
        \innerProd[\lebesgueSet{2}{\cortex}]{
          	\laplacian{\cortex}{}{} \membraneHeight_m
        }{
          	\laplacian{\cortex}{}{} \membraneHeight_m
        }
        +
        \effectiveLengthParam 
	    \innerProd[\bochnerLebesgueSet{2}{\cortex}{\reals^3}]{
          	\grad{\cortex}{}{}\membraneHeight_m 
        }{
          	\grad{\cortex}{}{}\membraneHeight_m 
        }
        +
        \anotherMembrParam
        \innerProd[\lebesgueSet{2}{\cortex}]{
          	\membraneHeight_m
        }{
          	\membraneHeight_m
        }
        =
        \innerProd[\lebesgueSet{2}{\cortex}]{
          	f
        }{
          	\membraneHeight_m
        }
    \end{equation*} 
    Leaving out $ \heightLinkerSysHeightBF{\membraneHeight_m}{\membraneHeight_m} $
    (non-negative term, cf. Remark~\ref{remark:Poincare inequality}),
    and applying the Cauchy-Schwartz inequality on the right, we arrive at
    \begin{equation*}
      	\innerProd[\lebesgueSet{2}{\cortex}]{
          	\timeDerivOperator(\membraneHeight_m')
        }{
          	\membraneHeight_m
        }
        \leq
        \frac{1}{2}
        \norm[\lebesgueSet{2}{\cortex}]{
          	f_m
        }^2
        +
        \frac{1}{2}
        \norm[\lebesgueSet{2}{\cortex}]{
          	\membraneHeight_m
        }^2
        \leq
        \frac{1}{2}
        \norm[\lebesgueSet{2}{\cortex}]{
          	f_m
        }^2
        +
        \frac{\theta^{-1}}{2}
        \norm[\lebesgueSet{2}{\cortex}]{
          	\rootTimeDerivOperator \membraneHeight_m
        }^2
    \end{equation*} 
    with $ \timeDerivOperator = \rootTimeDerivOperator^2 $.
    We recall the continuity 
    of $ \timeDerivOperator $ (see Lemma~\ref{lemma:Dirichlet-to-Neumann operator Stokes:%
    properties}) and observe
    $ \innerProd[\lebesgueSet{2}{\cortex}]{\timeDerivOperator(\membraneHeight_m')}{
    \membraneHeight_m} = \frac{1}{2}\diff{t}{}{}\norm[\lebesgueSet{2}{\cortex}]{
    \rootTimeDerivOperator\membraneHeight_m}^2 $.
    By applying the Gr\"{o}nwall inequality, we find the bound
    \begin{equation*}
        \theta\norm[\sobolevHSet{1}{\cortex}]{\membraneHeight_m}^2(t)
        \leq
      	\norm[\lebesgueSet{2}{\cortex}]{\rootTimeDerivOperator\membraneHeight_m}^2(t)
        \leq
        \norm[\lebesgueSet{2}{\cortex}]{
          	\rootTimeDerivOperator\membraneHeight_{0,m}
        }^2 e^{t\theta^{-1}} +
        \integral{0}{t}{
          	\norm[\lebesgueSet{2}{\cortex}]{f_m}^2(s)
            e^{(t-s)\theta^{-1}}
        }{s}.
    \end{equation*}
    Since $ \membraneHeight_{0,m} $ and $ f_m $ are bounded uniformly
    in $ \lebesgueSet{2}{\cortex} $ and $ \bochnerLebesgueSet{2}{[0,\finTime]}{
    \lebesgueSet{2}{\cortex}} $ w.\,r.\,t. $ m $, respectively,
    so is $ \membraneHeight_m $ in $ \bochnerLebesgueSet{\infty}{[0,\finTime]}{
    \sobolevHSet{1}{\cortex}} $.
    Not leaving out $ \heightLinkerSysHeightBF{\membraneHeight_m}{\membraneHeight_m} $,
    but integrating in time and using the previously achieved bound, 
    we may further bound 
    $ \norm[\bochnerLebesgueSet{2}{[0,\finTime]}{\lebesgueSet{2}{\cortex}}]{\laplacian{\cortex}{}{}
    \membraneHeight_m}^2 $ uniformly w.\,r.\,t. $ m $, eventually giving
    a bound on $ \membraneHeight_m $ in 
    $ \bochnerLebesgueSet{2}{[0,\finTime]}{\sobolevHSet{2}{\cortex}} $
    uniformly w.\,r.\,t. $ m $.
    Multiplying 
    \eqref{equ:appendix:existence and uniqueness of solutions of the height equation:%
    discrete system} by $ \membraneHeight_m' = (0,\membraneHeight'_i(t) )_{i\in\{2,\dots,m\}} $ 
    in the scalar product sense, we see
    \begin{equation*}
      	\theta \norm[\sobolevHSet{1}{\cortex}]{\membraneHeight'_m}^2
        \leq
        \innerProd[\lebesgueSet{2}{\cortex}]{
          	\timeDerivOperator(\membraneHeight_m')
        }{
	        \membraneHeight_m'
        }
        \leq
        \abs{ \heightLinkerSysHeightBF{\membraneHeight_m}{\membraneHeight_m'} }
        +
        \abs{ \innerProd[\lebesgueSet{2}{\cortex}]{f_m}{\membraneHeight'_m} }.
    \end{equation*}
    With Young's inequality, integration in time, and the previous bounds, we obtain
    a bound of $ \membraneHeight_m' $ in $ \bochnerLebesgueSet{2}{[0,\finTime]}{
    \sobolevHSet{1}{\cortex}} $ uniformly w.\,r.\,t. $ m $.
    All together there is a subsequence with indices $ m_k $ such that
    $ \membraneHeight_{m_k}' $ weakly converges in $ \sobolevHSet{1}{\cortex} $
    to $ \membraneHeight' \in \sobolevHSetMVF{1}{\cortex} $,
    $ \timeDerivOperator(\membraneHeight_{m_k}') $ converges weakly
    in $ \lebesgueSet{2}{\cortex} $ to $ L(\membraneHeight') $,
    $ \membraneHeight_{m_k} $ weakly converges in $ \sobolevHSet{2}{\cortex} $ to $ \membraneHeight 
    \in \sobolevHSetMVF{2}{\cortex} $
    and $ f_{m_k} $ weakly in $ \lebesgueSet{2}{\cortex} $ to $ f $ 
    for almost every $ t \in [0,\finTime) $.
    Uniqueness follows with the linearity of the equation.
\end{proof}
\begin{mycorollary}
	\label{corollary:appendix:existence and uniqueness constant mean value}
    (i)
	There is exactly one 
    $ \membraneHeight \in \bochnerLebesgueSet{2}{[0,\finTime]}{\sobolevHSet{2}{\cortex}}
    \cap 
    \bochnerSobolevHSet{1}{[0,\finTime]}{\sobolevHSetMVF{1}{\cortex}} $
    satisfying
    \begin{equation*}
        \innerProd[ \lebesgueSet{2}{\cortex} ]{
            \timeDerivOperator\left(
                \pDiff{t}{}{}\membraneHeight \normal[\cortex]{}
            \right)
        }{
          	\testFuncMembraneHeight
            \normal[\cortex]{}
        }
        +
        \heightLinkerSysHeightBF{\membraneHeight}{\testFuncMembraneHeight}
        =
        \innerProd[ \lebesgueSet{2}{\cortex} ]{
          	f(t)
        }{
			\testFuncMembraneHeight
        }
    \end{equation*}
    for almost every $ t \in [0,\finTime) $ and all $ \varphi \in \sobolevHSet{2}{\cortex} $,
    where
    $ f \in \bochnerDiffSet{}{[0,\finTime)}{\lebesgueSet{2}{\cortex}} $
    and $ \membraneHeight(0) = \membraneHeight_0 \in \lebesgueSet{2}{\cortex} $.
    
    (ii)
    The mean value $ \avgIntegral{\cortex}{}{\membraneHeight}{x} $ is constant in time.
\end{mycorollary}
\begin{proof}
	(i) Let $ \tilde{f} = f - \anotherMembrParam\avgIntegral{\cortex}{}{\membraneHeight_0}{x} $ 
    and
    $ \tilde{\membraneHeight}_0 = \membraneHeight_0 - 
    \avgIntegral{\cortex}{}{\membraneHeight_0}{x} $.
    Then consider the solution $ \tilde{\membraneHeight} $ of \eqref{equ:appendix:existence and uniqueness %
    of solutions of the height equation} with initial data $ \tilde{\membraneHeight}_0 $
    and right hand side $ \tilde{f} $.
    Set $ \membraneHeight = \tilde{\membraneHeight} + \avgIntegral{\cortex}{}{\membraneHeight_0}{x} $ and
    observe
    \begin{align*}
        \innerProd[ \lebesgueSet{2}{\cortex} ]{
            \timeDerivOperator\left(
                \pDiff{t}{}{}\membraneHeight \normal[\cortex]{}
            \right)
        }{
          	\testFuncMembraneHeight
            \normal[\cortex]{}
        }
        +
        \heightLinkerSysHeightBF{\membraneHeight}{\testFuncMembraneHeight}
        &=
        \innerProd[ \lebesgueSet{2}{\cortex} ]{
            \timeDerivOperator\left(
                \pDiff{t}{}{}\tilde{\membraneHeight} \normal[\cortex]{}
            \right)
        }{
          	\testFuncMembraneHeight
            \normal[\cortex]{}
        }        
        +
        \membrStiffn
        \innerProd[ \lebesgueSet{2}{\cortex} ]{
      		\laplacian{\cortex}{}{} \tilde{\membraneHeight}
        }{
      		\laplacian{\cortex}{}{} \testFuncMembraneHeight
        }
        \\
        &+
        \effectiveLengthParam
        \innerProd[ \lebesgueSet{2}{\cortex} ]{
      		\grad{\cortex}{}{} \tilde{\membraneHeight}
        }{
      		\grad{\cortex}{}{} \testFuncMembraneHeight
        }
        +
        \anotherMembrParam
        \innerProd[ \lebesgueSet{2}{\cortex} ]{
      		\tilde{\membraneHeight}
            +
      		\avgIntegral{\cortex}{}{\membraneHeight_0}{x}
        }{
      		\testFuncMembraneHeight
        }
        \\
        &=
        \anotherMembrParam
        \innerProd[\lebesgueSet{2}{\cortex}]{
          	\avgIntegral{\cortex}{}{\membraneHeight_0}{x}
        }{
          	\testFuncMembraneHeight
        }
        +
        \innerProd[\lebesgueSet{2}{\cortex}]{
          	\tilde{f}
        }{
          	\testFuncMembraneHeight
        }
        \\
        &=
        \innerProd[\lebesgueSet{2}{\cortex}]{
          	f
        }{
          	\testFuncMembraneHeight
        },
    \end{align*}
    so $ \membraneHeight $ is a solution as claimed.
    
    (ii) $ \avgIntegral{\cortex}{}{\membraneHeight_0}{x} $ is constant
    in time
    and this is by construction the mean value of $ \membraneHeight $.
\end{proof}


%% file: Appendix_Taylor_approximation_DtN_operator.tex
We want to show that the Dirichlet-to-Neumann operator
$ 
\funSig{\stokesDirichletToNeumannOp[t]{}}{
	\sobolevHSet{1}{\membrane(t)}
}{
  	\lebesgueSet{2}{\membrane(t)}
}
$
can be approximated by 
$ 
\funSig{\stokesDirichletToNeumannOp[0]{}}{
  	\sobolevHSet{1}{\eulerMembrane}
}{
  	\lebesgueSet{2}{\eulerMembrane}
}
$
in the sense that for $ \membraneHeight = \perturbParam \rescaled{ \membraneHeight } $
and $ \psi = \perturbParam \rescaled{ \psi } $,
it holds
\begin{equation*}
  	\integral{\membrane(t)}{}{
      	\stokesDirichletToNeumannOp[t]{
          	\liftedFun{
                \pDiff{t}{}{} \membraneHeight
                \normal[\eulerMembrane]{}
            }{\mathfrak{X}}
        }
        \cdot
        \liftedFun{
          	\psi
            \normal[\eulerMembrane]{}
        }{\mathfrak{X}}
    }{x}
    =
    \integral{\eulerMembrane}{}{
      	\stokesDirichletToNeumannOp[0]{
          	\pDiff{t}{}{} \membraneHeight
            \normal[\eulerMembrane]{}
        }
        \cdot
        \psi
        \normal[\eulerMembrane]{}
    }{x}
    +
    \landauSmallO[\perturbParam^3]
\end{equation*}
for $ \psi \in \sobolevHSet{1}{\eulerMembrane} $.

Take $ \velocity[1] $ and $ \velocity[2] $ as \revision[third]{parts of solutions} of stationary
Stokes problems being continuously differentiable such that
\begin{equation}
  	\label{equ:app:taylor approx of DtN:boundary on membrane}
  	\viscosity
  	\innerProd[
    	\bochnerLebesgueSet{2}{\generalDomain}{\reals^{(3,3)}}
    ]{
      	\symmGrad{\velocity[1]}
    }{
      	\symmGrad{\testFuncVelocity}
    }
    -
    \innerProd[
    	\lebesgueSet{2}{\generalDomain}
    ]{
      	\pressure[1]
    }{
      	\diver{}{}{} \testFuncVelocity
    }
    =
    \integral{\membrane(t)}{}{
      	\stokesDirichletToNeumannOp[t]{
          	\liftedFun{
          		\pDiff{t}{}{}\membraneHeight \normal[\eulerMembrane]{}
            }{\mathfrak{X}}
        }
        \cdot
        \traceOp[\membrane(t)]{ \testFuncVelocity }
    }{x}
\end{equation}
and
\begin{equation}
  	\label{equ:app:taylor approx of DtN:boundary on eulerMembrane}
  	\viscosity
  	\innerProd[
    	\bochnerLebesgueSet{2}{\generalDomain}{\reals^{(3,3)}}
    ]{
      	\symmGrad{\velocity[2]}
    }{
      	\symmGrad{\testFuncVelocity}
    }
    -
    \innerProd[
    	\lebesgueSet{2}{\generalDomain}
    ]{
      	\pressure[2]
    }{
      	\diver{}{}{} \testFuncVelocity
    }
    =
    \integral{\eulerMembrane}{}{
      	\stokesDirichletToNeumannOp[0]{
       		\pDiff{t}{}{}\membraneHeight \normal[\eulerMembrane]{}
        }
        \cdot
        \traceOp[\eulerMembrane]{ \testFuncVelocity }
    }{x}.
\end{equation}
Subtract \eqref{equ:app:taylor approx of DtN:boundary on membrane}
and \eqref{equ:app:taylor approx of DtN:boundary on eulerMembrane}
and choose $ \testFuncVelocity $ as \revision[third]{the velocity of a solution} of 
a Stokes problem on $ \generalDomain $
with Dirichlet boundary data $ \psi \normal[\eulerMembrane]{} $
on $ \eulerMembrane $, $ \liftedFun{ \psi\normal[\eulerMembrane]{} }{
\mathfrak{X}} $ on $ \membrane(t) $, and zero on the rest of the boundary,
which is possible under the assumption that 
$ \membrane(t) \cup \eulerMembrane $ is sufficiently regular.
\begin{align*}
  	\viscosity
  	\innerProd[
    	\bochnerLebesgueSet{2}{\generalDomain}{\reals^{(3,3)}}
    ]{
      	\symmGrad{\tilde{\velocity}}
    }{
      	\symmGrad{\testFuncVelocity}
    }
    &=
    \integral{\membrane(t)}{}{
      	\stokesDirichletToNeumannOp[t]{
          	\liftedFun{
          		\pDiff{t}{}{}\membraneHeight \normal[\eulerMembrane]{}
            }{\membrane(t)}
        }
        \cdot
        \liftedFun{ \psi\normal[\eulerMembrane]{} }{
        \mathfrak{X}}
    }{x}
    \\
    &-
    \integral{\eulerMembrane}{}{
      	\stokesDirichletToNeumannOp[0]{
       		\pDiff{t}{}{}\membraneHeight \normal[\eulerMembrane]{}
        }
        \cdot
        \psi \normal[\eulerMembrane]{}
    }{x}.
\end{align*}
According to \cite[p.\,311]{Cattabriga1961}, $ \symmGrad{\tilde{\velocity}} $ and
$ \symmGrad{\testFuncVelocity} $ may
be bounded by their Dirichlet boundary data. 
For $ x \in \membrane(t) $, $ x_0 \in \eulerMembrane $
such that $ x = x_0 + \pDiff{t}{}{} \membraneHeight(x_0) \normal[\eulerMembrane]{x_0} $, 
we have
\begin{equation}
  	\velocity[2](x) = \velocity[2](x_0) + \left( x - x_0 \right) \cdot 
    \pDiff[x_0]{x - x_0}{}{ \velocity[2] } + \landauSmallO[{ \norm[2]{x-x_0}^2 }].
\end{equation}
By choice, $ u^2(x) = u^1(x_0) $, so
\begin{equation}
  	\velocity[2](x) = \velocity[1](x_0) + \left( x - x_0 \right) \cdot 
    \pDiff[x_0]{x - x_0}{}{ \velocity[2] } + \landauSmallO[{ \norm[2]{x-x_0}^2 }]
\end{equation}
and therefore
\begin{equation}
  	\tilde{\velocity}(t,x) = \left( x - x_0 \right) \cdot 
    \pDiff[x_0]{x - x_0}{}{ \velocity[2] } + \landauSmallO[{ \norm[2]{x-x_0}^2 }].
\end{equation}
As $ \velocity[2](x_0) = \liftedFun{ 
\pDiff{t}{}{} \membraneHeight(t,x_0) }{ \mathfrak{X} }
= \perturbParam \liftedFun{ 
\pDiff{t}{}{} \rescaled{ \membraneHeight }(t,x_0) }{ \mathfrak{X} } $,
$ \velocity[2] = \landauSmallO[ \perturbParam ] $ in a sufficiently small neighbourhood
of $ x_0 $, and so $ \pDiff[x_0]{x-x_0}{}{ \velocity[2] } = 
\landauSmallO[ \perturbParam ] $. With $ \norm[2]{ x - x_0 } = 
\landauSmallO[ \perturbParam ] $, we have
$ \norm[2]{ \tilde{ \velocity }(t,x) } = \landauSmallO[ \perturbParam^2 ] $.
As the boundary data of $ \testFuncVelocity $ is of order $ \perturbParam $,
the claim follows.


%% file: Appendix_Differential_geometry.tex
This section contains basic differential geometric formulae
which are eventually used for calculating the shape derivatives in
\autoref{sec:PDE description of the height function}.
\begin{mytheorem}[{Transport theorem for surfaces, \cite[Theorem~32]{Barrett+2019}}]
    \label{theorem:surface transport theorem}
    Let $ \spaceParamDomain{\Gamma}{I} $ be a smooth evolving hypersurface
    in $ \reals^n $
    and
    $ \funSig{f}{\spaceParamDomain{\Gamma}{I}}{\reals} $ a function
    with existing material derivative. Then,
    \begin{equation*}
        \diff[r]{s}{}{
            \integral{\Gamma_s}{}{
                f\left(s,\cdot\right)
            }{\hausdorffM{n-1}}
        } =
        \integral{\Gamma_r}{}{
            \matDiff[r]{s}{f}
            +
            f^r
            \diver{\Gamma_s}{}{}
            \paramVel[r]
        }{\hausdorffM{n-1}}.
    \end{equation*}
\end{mytheorem}

\begin{mylemma}[{\cite[Lemma~37]{Barrett+2019}}]
    \label{lemma:mat deriv normal}
    Let $ \spaceParamDomain{\Gamma}{I} $ be a smooth evolving manifold in
    $ \reals^n $
    with normal fields $ \left( \normal[s]{} \right)_{s\in I} $.
    It holds
    \begin{equation*}
        \matDiff[r]{s}{\normal[s]{}} = 
        -\transposed{ \left(\grad{\Gamma_r}{}{} \paramVel[r] \right) }
        \normal[r]{}.
    \end{equation*}
\end{mylemma}
\hide{
\begin{proof}
    $ \matDiff[r]{s}{\normal[s]{}} $ is an
    element of the tangent space (material differentiate $ \abs{\normal[s]{}}^2 $).
    Hence, it is sufficient to show identity for the tangential coordiantes.
    To this purpose, take an orthonormal basis $ \tau^r_i
    = \pDiff{i}{}{}\localParam[r]{} \concat 
    \left( \localParam[r]{} \right)^{-1} $ 
    ($ \localParam[r]{} = \globalParam[r]{}\concat\localParam{} $)
    in a
    neighbourhood of an arbitrary point $ p \in \Gamma_s $.
    Since,
    \begin{equation*}
        \tau_i^s \cdot \normal[s]{} = 0,
    \end{equation*}
    we can shift the material derivative:
    \begin{align*}
        \matDiff[r]{s}{ \normal[s]{} } \cdot \tau_i^r &=
        - \normal[r]{} \cdot \matDiff[r]{s}{\tau_i^s} \\
        &\overset{(1)}{=}
        - \normal[r]{} \cdot
        \left(\grad{\Gamma_r}{}{} \paramVel[r] \right)
        \tau^r_i \\
        &=
        \transposed{ \left( \grad{\Gamma_r}{}{} \paramVel[r] \right) }
        \normal[r]{} \cdot \tau^r_i
    \end{align*}
    where (1) refers to
    \begin{equation}
        \label{equ:proof:mat deriv normal:c1}
        \begin{split}
            \matDiff[r]{s}{\tau_i^s} &= 
            \diff[r]{s}{}{ 
                \pDiff{i}{}{ 
                    \globalParam[s]{} 
                    \concat
                    \localParam{} 
                 } 
                 \concat 
                 \localParam[-1]{}
            }
            \concat 
            \left( \globalParam[r]{} \right)^{-1} \\
            &= \pDiff{i}{}{ 
                \paramVel[r] 
                \concat 
                \localParam[r]{}
            } 
            \concat
            \left( 
                \localParam{}
            \right)^{-1} 
            \concat
            \left( \globalParam[r]{} \right)^{-1}
            \\
            &\overset{(2)}{=}
            \grad{\Gamma_r}{}{} \paramVel[r]
            \tau^r_i
        \end{split}
    \end{equation}
    (2): tangential gradient formula for orthonormal basis
\end{proof}
}

Some simple algebraic observations:
\begin{mylemma}
    \label{lemma:ommit proj rules}
    Let $ \Gamma $ be an orientable differentiable real submanifold
    with normal field $ \normal{} $ 
    and $ \funSig{f}{\Gamma}{\reals^m} $ a function with existing tangential
    Jacobian.
    Further, $ \orthProj[\Gamma]{} = \identityMatrix - 
    \normal{}\tensorProd\normal{} $.
    It holds,
    \begin{equation}
        \label{equ:ommit proj}
        \grad{\Gamma}{}{} f \orthProj[\Gamma]{} = \grad{\Gamma}{}{}f
    \end{equation}
    \begin{equation}
        \label{equ:ommit proj transposed}
        \orthProj[\Gamma]{} \transposed{\left( \grad{\Gamma}{}{} f \right)}
        = \transposed{\left( \grad{\Gamma}{}{} f \right)}
    \end{equation}
    \begin{equation}
        \label{equ:ommit proj transposed frob}
        \transposed{\left( \grad{\Gamma}{}{} f \right)} 
        \orthProj[\Gamma]{}
        \frobScalarProd
        \grad{\Gamma}{}{} g
        =
        \transposed{\left( \grad{\Gamma}{}{} f \right)} 
        \frobScalarProd
        \grad{\Gamma}{}{} g
    \end{equation}
\end{mylemma}

An analogue of the Schwarz theorem for tangential gradients:
\begin{mylemma}[{\cite[Lemma~15]{Barrett+2019}}]
    \label{lemma:schwarz theorem for tangential gradients}
    Let $ \Gamma \subseteq \reals^n $ be a differentiable, orientable real submanifold
    with normal field $ \normal{} $ and
    $ \funSig{f}{\Gamma}{\reals} $ a function with
    existing tangential derivatives up to second order.
    \begin{equation*}
        \grad{\Gamma}{2}{f} - \transposed{\grad{\Gamma}{2}{f}} =
        \normal{}
        \tensorProd
        \left( 
            \grad{\Gamma}{}{}\normal{} 
            \grad{\Gamma}{}{}f
        \right) 
        -
        \left( 
            \grad{\Gamma}{}{}\normal{} 
            \grad{\Gamma}{}{} f
        \right)
        \tensorProd
        \normal{}
    \end{equation*}
\end{mylemma}
\hide{
\begin{proof}
    Choose $ \funSig{\bar{f}}{\reals^n}{\reals} $ to be the 
    extension of $ f $ which is constant in normal direction $ \normal{} $.

    It holds,
    \begin{align*}
        \tDiff{\Gamma}{ij}{}{}f 
        &\overset{\text{def}}{=}
          \tDiff{\Gamma}{i}{}{
            \pDiff{j}{}{}\bar{f}
            -
            \left( \grad{}{}{}\bar{f} \cdot \normal{} \right) 
            \normal{}_j
        }
        \\
        &\overset{\text{def}}{=}
        \pDiff{ij}{}{}\bar{f}
        -
        \pDiff{i}{}{
            \overline{
            \left( \grad{}{}{}\bar{f} \cdot \normal{} \right)
            \normal{}_j
            }
        }
        -
        \left(
        \grad{}{}{
            \pDiff{j}{}{}\bar{f}
            -
            \overline{
            \left( \grad{}{}{}\bar{f} \cdot \normal{} \right) 
            \normal{}_j
            }
        } \cdot \normal{}
        \right)
        \normal{}_i
        \\
        &\overset{(1)}{=}
        \pDiff{ij}{}{}\bar{f}
        -\grad{}{}{ \pDiff{j}{}{}\bar{f} } 
        \cdot 
        \normal{}
        \normal{}_i.
    \end{align*}
    (1): $ \grad{\Gamma}{}{}\bar{f}\cdot\normal{}\normal{}_j = 0 $
    on $ \Gamma $, so $ \tDiff{\Gamma}{i}{}{
    \grad{\Gamma}{}{}\bar{f}\cdot\normal{}\normal{}_j } = 0 $
    on $ \Gamma $ (extension in normal direction is constantly zero).

    We further compute,
    \begin{align*}
        0 
        &= 
        \tDiff{\Gamma}{j}{}{
            \grad{}{}{}\bar{f} \cdot \normal{}
        }
        \\
        &=
        \pDiff{j}{}{
            \grad{}{}{}\bar{f} \cdot \overline{\normal{}}
        } 
        -
        \grad{}{}{
            \grad{}{}{}\bar{f} \cdot \overline{\normal{}}
        }
        \cdot
        \normal{}
        \normal{}_j
        \\
        &\overset{(2)}{=}
        \grad{}{}{
            \pDiff{j}{}{}\bar{f}
        } 
        \cdot
        \overline{\normal{}}
        +
        \grad{}{}{}\bar{f}
        \cdot
        \pDiff{j}{}{\overline{\normal{}}}
        -
        \grad{}{}{
            \grad{}{}{}\bar{f} \cdot \overline{\normal{}}
        }
        \cdot
        \normal{}
        \normal{}_j.
    \end{align*}
    (2): Schwarz' theorem

    Substituting into the equation above leads to
    \begin{align*}
        \tDiff{\Gamma}{ij}{}{}f =
        \pDiff{ij}{}{}\bar{f}
        +
        \grad{}{}{}\bar{f}
        \cdot
        \pDiff{j}{}{\overline{\normal{}}}
        \normal{}_i
        -
        \grad{}{}{
            \grad{}{}{}\bar{f} \cdot \overline{\normal{}}
        }
        \cdot
        \normal{}
        \normal{}_j
        \normal{}_i.
    \end{align*}
    Due to symmetry, we then also have
    \begin{align*}
        \tDiff{\Gamma}{ji}{}{}f =
        \pDiff{ji}{}{}\bar{f}
        +
        \grad{}{}{}\bar{f}
        \cdot
        \pDiff{i}{}{\overline{\normal{}}}
        \normal{}_j
        -
        \grad{}{}{
            \grad{}{}{}\bar{f} \cdot \overline{\normal{}}
        }
        \cdot
        \normal{}
        \normal{}_i
        \normal{}_j.
    \end{align*}
    Consequently,
    \begin{align*}
        \tDiff{\Gamma}{ij}{}{}f - \tDiff{\Gamma}{ji}{}{} f
        =
        \grad{}{}{}\bar{f}
        \cdot
        \pDiff{j}{}{\overline{\normal{}}}
        \normal{}_i
        - 
        \grad{}{}{}\bar{f}
        \cdot
        \pDiff{i}{}{\overline{\normal{}}}
        \normal{}_j.
    \end{align*}
\end{proof}
}

This implies:
\begin{mycorollary}[{\cite[Lemma~16]{Barrett+2019}}]
    Let $ \Gamma $ be a differentiable orientable real submanifold
    with normal field $ \normal{} $ and mean curvature $ \meanCurv{} $.
    It holds,
    \label{lemma:laplacian of the normal}
    \begin{equation*}
        \laplacian{\Gamma}{}{} \normal{} 
        =
        \grad{\Gamma}{}{}\meanCurv{}
        -
        \normal{}
        \grad{\Gamma}{}{}\normal{}
        \frobScalarProd
        \grad{\Gamma}{}{}\normal{}.                   
    \end{equation*}
\end{mycorollary}

\subsection{Commutator rules}
\begin{mylemma}[Commutator rule for the tangential gradient,{\cite[Lemma~38]{Barrett+2019}}]
    \label{lemma:commutator rule for the tangential gradient}
    Let $ \funSig{\liftingParametr{}}{I\times\reals^n}{\reals^{n+k}} $,
    for an interval $ I \subseteq \reals $,
    be parametrisations of $ n $-dimensional manifolds
    $ \Gamma_s $ with an associated 
    material derivative $ \matDiff{s}{} $ and velocity fields
    $ \paramVel[r] = \pDiff[r]{s}{}{\liftingParametr{s,\cdot}}\concat
    \left( \liftingParametr{r,\cdot} \right)^{-1} $.
    Let $ \funSig{f}{\spaceParamDomain{\Gamma}{I}}{\reals} $ 
    be sufficiently regular. It holds
    \begin{equation*}
        \matDiff[r]{s}{
            \grad{\Gamma_s}{}{}f
        } = 
        \grad{\Gamma_r}{}{ 
            \matDiff[r]{s}{f}
        } 
        +
        \left( 
            \grad{\Gamma_r}{}{}\paramVel[r]
            - 
            2 \symmTangJacobian{\Gamma_r}{\paramVel[r]} 
        \right)
        \grad{\Gamma_r}{}{} f
    \end{equation*}
\end{mylemma}
\hide{
\begin{proof}
    We first compute the normal component of
    \begin{equation*}
        \left( \matDiff[r]{s}{ 
            \grad{\Gamma_s}{}{} f
        } - \grad{\Gamma_r}{}{ 
            \matDiff[r]{s}{ f }
        }
        \right) \cdot \normal[\Gamma_r]{} =
        \matDiff[r]{s}{ 
            \grad{\Gamma_s}{}{} f
        } \cdot \normal[\Gamma_r]{}.
    \end{equation*}
    In this order, we observe that
    \begin{equation*}
        \begin{split}
            0 = \matDiff[r]{s}{\grad{\Gamma_s}{}{}f \cdot 
            \normal[\Gamma_s]{} } &=
            \matDiff[r]{s}{ \grad{\Gamma_s}{}{}f } \cdot 
            \normal[\Gamma_r]{} + \grad{\Gamma_r}{}{}f \cdot
            \matDiff[r]{s}{ \normal[\Gamma_s]{} } \\
            &\overset{\autoref{lemma:mat deriv normal}}{=}
            \matDiff[r]{s}{\grad{\Gamma_s}{}{}f} \cdot 
            \normal[\Gamma_r]{}
            -\grad{\Gamma_r}{}{}f \cdot \left( \grad{\Gamma_r}{}{}\paramVel[r] \right)
            \normal[\Gamma_r]{} \\
            &= 
            \matDiff[r]{s}{\grad{\Gamma_s}{}{}f} \cdot 
            \normal[\Gamma_r]{}
            - \left( \transposed{\left(\grad{\Gamma_r}{}{\paramVel[r]}\right)}
            \grad{\Gamma_r}{}{}f \right) \cdot
            \normal[\Gamma_r]{} 
        \end{split}
    \end{equation*}
    Then, we analyse the tangential part. For sake of readability,
    we set $ \orthProj{} = \orthProj[\tangentSet{p}{\Gamma_s}]{} $.
    We substitute $f $ with its normal extension $ \bar{f} $
    in the tangential gradients and material derivatives; hence,
    we do not change values.
    \begin{equation*}
        \begin{split}
            \orthProj{
                \matDiff[r]{s}{
                    \grad{\Gamma_s}{}{} \bar{f}
                }
                - 
                \grad{\Gamma_r}{}{
                    \matDiff[r]{s}{ \bar{f} }
                }
            }
            &= 
            \orthProj{
                \matDiff[r]{s}{
                    \grad{}{}{} \bar{f}
                }                              	
            }
            -
            \grad{\Gamma_r}{}{
                \matDiff[r]{s}{ \bar{f} }
            } \\
            &=
            \orthProj{
                \matDiff[r]{s}{
                    \grad{}{}{} \bar{f}
                }                              	
            }
            -
            \grad{\Gamma_r}{}{
                \pDiff[r]{s}{}{\bar{f}} + \paramVel[r] \cdot \grad{}{}{} \bar{f}
            } \\
            &= 
            \orthProj{
                \pDiff[r]{s}{}{
                    \grad{}{}{} \bar{f}
                }             
                +
                \paramVel[r] \cdot \grad{}{2}{}\bar{f}
            }
            -
            \grad{\Gamma_r}{}{\left(
                \pDiff[s]{s}{}{\bar{f}} + V^{\Gamma_s} \cdot \grad{}{}{} \bar{f}
            \right)} \\                            
            &= \orthProj{
                \pDiff[r]{s}{}{
                    \grad{}{}{} \bar{f}
                }             
                +
                \grad{}{2}{}\bar{f} \paramVel[r]
                -
                \grad{}{}{
                    \pDiff[r]{s}{}{\bar{f}} + \paramVel[r] \cdot \grad{}{}{} \bar{f}
                }
            } \\
            &= - \orthProj{}\grad{}{}{}\paramVel[r] \grad{}{}{}\bar{f}
        \end{split}
    \end{equation*}
    Finally, we have
    \begin{equation*}
        \begin{split}
            \matDiff[r]{s}{\grad{\Gamma_s}{}{}f} 
            &= 
            \grad{\Gamma_r}{}{ \matDiff[r]{s}{ f } }
            - 
            \grad{\Gamma_r}{}{}\paramVel[r] \grad{\Gamma_r}{}{}f
            + 
            \left(
              	\transposed{\left( \grad{\Gamma_r}{}{}\paramVel[r] \right)}
            	\grad{\Gamma_r}{}{}f 
            \right) 
            \cdot 
            \normal[\Gamma_r]{} 
            \normal[\Gamma_r]{} 
            \\
            &=
            \grad{\Gamma_r}{}{ \matDiff[r]{s}{f} }
            - 
            \grad{\Gamma_r}{}{}\paramVel[r] \grad{\Gamma_r}{}{}f
            - 
            \left(
              	\transposed{\left( \grad{\Gamma_r}{}{} \paramVel[r] \right)}
                \grad{\Gamma_r}{}{}f 
            \right)
            + 
            \left(
              	\transposed{\left( \grad{\Gamma_r}{}{}\paramVel[r] \right)}
            	\grad{\Gamma_r}{}{}f 
            \right) 
            \\
            &+ 
            \left(
              	\transposed{\left( \grad{\Gamma_r}{}{}\paramVel[r] \right)}
            	\grad{\Gamma_r}{}{}f 
            \right) 
            \normal[\Gamma_r]{}
            \tensorProd
            \normal[\Gamma_r]{} \\                            
            &=
            \grad{\Gamma_r}{}{ \matDiff[r]{s}{ f } }
            - 
            \grad{\Gamma_r}{}{}\paramVel[r] 
            \grad{\Gamma_r}{}{}f
            + 
            \left(
              	\transposed{\left( \grad{\Gamma_r}{}{}\paramVel[r] \right)}
            	\grad{\Gamma_r}{}{}f 
            \right)
            - 
            \orthProj{
              	\transposed{ \left(\grad{\Gamma_r}{}{}\paramVel[r] \right)}
            	\grad{\Gamma_r}{}{}f 
            }
            \\
            &=
            \grad{\Gamma_r}{}{ \matDiff[r]{s}{f} }
            - 
            \orthProj{} \grad{\Gamma_r}{}{}\paramVel[r] \orthProj{}
            \grad{\Gamma_r}{}{}f
            + 
            \left(
              	\transposed{\left( \grad{\Gamma_r}{}{}\paramVel[r] \right)}
            	\grad{\Gamma_r}{}{}f 
            \right)
            - 
            \orthProj{} \transposed{\left( \grad{\Gamma_r}{}{}\paramVel[r] \right)} 
            \orthProj{} \grad{\Gamma_r}{}{}f 
            \\
            &= 
            \grad{\Gamma_r}{}{ \matDiff[r]{s}{ f } } + 
            \left(
            	\grad{\Gamma_r}{}{}\paramVel[r] 
                - 
                2\symmTangJacobian{\Gamma_r}{\paramVel[r]} 
            \right) 
            \grad{\Gamma_r}{}{}f.
        \end{split}
    \end{equation*}
\end{proof}
}

A similar rule exists for the tangential divergence:
\begin{mylemma}[Commutator rule for the tangential divergence,{\cite[Lemma~38]{Barrett+2019}}]
    \label{lemma:commutation rule divergence}
    Let $ \funSig{\liftingParametr{}}{I\times\reals^n}{\reals^{n+k}} $,
    for an interval $ I \subseteq \reals $,
    be parametrisations of $ n $-dimensional manifolds
    $ \Gamma_s $ with an associated 
    material derivative $ \matDiff{s}{} $ and velocity fields
    $ \paramVel[r] = \pDiff[r]{s}{}{\liftingParametr{s,\cdot}}\concat
    \left( \liftingParametr{r,\cdot} \right)^{-1} $.
    Let $ \funSig{f}{\spaceParamDomain{\Gamma}{I}}{\reals^n} $ 
    be sufficiently regular. It holds
    \begin{equation*}
        \matDiff[r]{s}{ \diver{\Gamma_s}{}{} f } 
         =
        \diver{\Gamma_r}{}{ \matDiff[r]{s}{f} }
        + \left( \grad{\Gamma_r}{}{}\paramVel[r] - 
        2 \symmTangJacobian{\Gamma_r}{\paramVel[r]} \right)
        \frobScalarProd \grad{\Gamma_r}{}{}f
    \end{equation*}
\end{mylemma}
\hide{
\begin{proof}
    First, we observe,
    \begin{equation*}
        \diver{\Gamma_r}{}{} f = \trace \grad{\Gamma_r}{}{} f
        = \trace\left( \orthProj{}{} \grad{\Gamma_r}{}{} f \right)
        = \trace\left( \transposed{\orthProj{}{}} \grad{\Gamma_r}{}{} f \right)
        = \orthProj{}{} \frobScalarProd \grad{\Gamma_r}{}{} f
    \end{equation*}
    and
    \begin{equation*}
        \orthProj[\Gamma_r]{} = 
        \grad{\Gamma_r}{}{} \identity[\Gamma_r].
    \end{equation*}
    Then,
    \begin{equation*}
        \begin{split}
            \matDiff[r]{s}{ \diver{\Gamma_s}{}{} f } 
            &=
            \matDiff[r]{s}{ 
              	\orthProj{} \frobScalarProd \grad{\Gamma_s}{}{} f 
            } 
            = 
            \matDiff[r]{s}{
            	\sum_i \left(\orthProj[\Gamma_s]{}\right)_i 
                \cdot 
            	\grad{\Gamma_s}{}{} f_i 
            } 
            \\
            &= 
            \sum_i 
            \left(
                \matDiff[r]{s}{
                    \grad{\Gamma_r}{}{} \left( \identity[\Gamma^s] \right)_i
                } 
                \cdot
                \grad{\Gamma_r}{}{} f_i
                +
                \left( \orthProj[\Gamma_r]{} \right)_i \cdot
                \matDiff[r]{s}{
                    \grad{\Gamma^s}{}{} f
                }
            \right) 
            \\
            &= 
            \sum_i 
            \bigg(
                \grad{\Gamma_r}{}{} \paramVel[r]_i
                \cdot
                \grad{\Gamma_r}{}{} f_i
                +
                \left(
                    \grad{\Gamma_r}{}{} \paramVel[r] 
                    - 
                    2 \symmTangJacobian{\Gamma_r}{\paramVel[r]}
                \right)
                \grad{\Gamma_r}{}{}\left(\identity[\Gamma_r]\right)_i
                \cdot
                \grad{\Gamma_r}{}{} f_i
                \\
                &+
                \left( \orthProj[\Gamma_r]{} \right)_i \cdot
                \matDiff[r]{s}{
                    \grad{\Gamma^s}{}{} f
                }
            \bigg) \\
            &= \sum_i \bigg(
                \grad{\Gamma_r}{}{} \paramVel[r]_i
                \cdot
                \grad{\Gamma_r}{}{} f_i
                +
                \left(
                \left(
                    \grad{\Gamma_r}{}{} \paramVel[r] 
                    - 
                    2 \symmTangJacobian{\Gamma_r}{\paramVel[r]}
                \right)
                \orthProj[\Gamma_r]{}
                \right)_i
                \cdot
                \grad{\Gamma_r}{}{} f_i
                \\
                &+
                \left( \orthProj[\Gamma_r]{} \right)_i \cdot
                \matDiff[r]{s}{
                    \grad{\Gamma^s}{}{} f
                }
            \bigg) \\
            &= 
            \left(
                \grad{\Gamma_r}{}{} \paramVel[r] 
                - 
                2 \symmTangJacobian{\Gamma_r}{\paramVel[r]}
            \right)
            \orthProj[\Gamma_r]{}
            \frobScalarProd
            \grad{\Gamma_r}{}{} f
            +
            \sum_i \bigg(
                \grad{\Gamma_r}{}{} \paramVel[r]_i
                \cdot
                \grad{\Gamma_r}{}{} f
                \\
                &+
                \left( \orthProj[\Gamma_r]{} \right)_i \cdot
                \matDiff[r]{s}{
                    \grad{\Gamma^s}{}{} f_i
                }
            \bigg) \\
            &= 
            \left(
                \grad{\Gamma_r}{}{}\paramVel[r] 
                - 
                2 \symmTangJacobian{\Gamma_r}{\paramVel[r]}
            \right)
            \orthProj[\Gamma_r]{}
            \frobScalarProd
            \grad{\Gamma_r}{}{} f
            +
            \sum_i \bigg(
                \grad{\Gamma_r}{}{} \paramVel[r]_i
                \cdot
                \grad{\Gamma_r}{}{} f 
                \\
                &+
                \left( \orthProj[\Gamma_r]{} \right)_i 
                \cdot
                \left(
                    \grad{\Gamma_r}{}{ \matDiff[r]{s}{f_i} }
                    +
                    \left( 
                        \grad{\Gamma_r}{}{}\paramVel[r] 
                        - 
                        2 \symmTangJacobian{\Gamma_r}{\paramVel[r]}
                    \right)
                    \grad{\Gamma_r}{}{}f_i
                \right)
            \bigg) \\
            &=
            \left(
                \grad{\Gamma_r}{}{}\paramVel[r] 
                - 
                2 \symmTangJacobian{\Gamma_r}{\paramVel[r]}
            \right)
            \orthProj[\Gamma_r]{}
            \frobScalarProd
            \grad{\Gamma_r}{}{} f
            +
            \diver{\Gamma_r}{}{ \matDiff[r]{s}{ f } }
            +
            \sum_i \bigg(
                \grad{\Gamma_r}{}{} \paramVel[r]_i
                \cdot
                \grad{\Gamma_r}{}{} f_i
                \\
                &+
                \left( \orthProj[\Gamma_r]{} \right)_i
                \cdot \left(
                    \left( 
                        \grad{\Gamma_r}{}{}\paramVel[r] 
                        - 
                        2 \symmTangJacobian{\Gamma_r}{\paramVel[r]}
                    \right)
                    \grad{\Gamma_r}{}{}f_i
                \right)
            \bigg) \\
            &=
            \left(
                \grad{\Gamma_r}{}{}\paramVel[r] 
                - 
                2 \symmTangJacobian{\Gamma_r}{\paramVel[r]}
            \right)
            \orthProj[\Gamma_r]{}
            \frobScalarProd
            \grad{\Gamma_r}{}{} f
            +
            \diver{\Gamma_r}{}{ \matDiff[r]{s}{ f } }
            +
            \grad{\Gamma_r}{}{} \paramVel[r]
            \frobScalarProd
            \grad{\Gamma_r}{}{} f \\
            &+ 
            \transposed{\left( 
                \grad{\Gamma_r}{}{}\paramVel[r] 
                - 
                2 \symmTangJacobian{\Gamma_r}{\paramVel[r]}
            \right)}                       
            \orthProj[\Gamma_r]{}
            \frobScalarProd
            \grad{\Gamma_r}{}{}f \\
            &\overset{(1)}{=}
            \diver{\Gamma_r}{}{ \matDiff[r]{s}{ f } }
            - 
            \orthProj[\Gamma_r]{} \grad{\Gamma_r}{}{}\paramVel[r] 
            \orthProj[\Gamma_r]{}
            \frobScalarProd
            \grad{\Gamma_r}{}{} f                            
            - 
            \orthProj[\Gamma_r]{} \transposed{\left(
                \grad{\Gamma_r}{}{}\paramVel[r]
            \right)} \orthProj[\Gamma_r]{}
            \frobScalarProd
            \grad{\Gamma_r}{}{} f 
            + 
            \grad{\Gamma_r}{}{}\paramVel[r] 
            \frobScalarProd
            \grad{\Gamma_r}{}{} f \\
            &= 
            \diver{\Gamma_r}{}{ \matDiff[r]{s}{ f } }
            + 
            \left(
                \grad{\Gamma_r}{}{}\paramVel[r]
                -
                2 \symmTangJacobian{\Gamma_r}{\paramVel[r]}
            \right)
            \frobScalarProd
            \grad{\Gamma_r}{}{} f
        \end{split}
    \end{equation*}
    with (1) being the observation that
    \begin{equation*}
        \begin{split}
            \left(
                \grad{\Gamma_r}{}{}\paramVel[r] 
                - 
                2 \symmTangJacobian{\Gamma_r}{\paramVel[r]}
            \right)
            \orthProj[\Gamma_r]{}
            \frobScalarProd
            \grad{\Gamma_r}{}{} f 
            &=
            \overset{(i)}{\overbrace{
                \grad{\Gamma_r}{}{}\paramVel[r] \orthProj[\Gamma_r]{} 
                \frobScalarProd
                \grad{\Gamma_r}{}{} f
            }}
            - 
            \orthProj[\Gamma_r]{} \grad{\Gamma_r}{}{}\paramVel[r] 
            \orthProj[\Gamma_r]{}
            \frobScalarProd
            \grad{\Gamma_r}{}{} f \\
            &-
            \underset{(iii)}{\underbrace{
                \orthProj[\Gamma_r]{} \transposed{\left(
                    \grad{\Gamma_r}{}{}\paramVel[r]
                \right)} \orthProj[\Gamma_r]{}
                \frobScalarProd
                \grad{\Gamma_r}{}{} f
            }}
        \end{split}
    \end{equation*}
    \begin{equation*}
        \begin{split}
            \transposed{
            \left(
                \grad{\Gamma_r}{}{}\paramVel[r] 
                - 
                2 \symmTangJacobian{\Gamma_r}{\paramVel[r]}
            \right)}
            \orthProj[\Gamma_r]{}
            \frobScalarProd
            \grad{\Gamma_r}{}{} f 
            &=
            \underset{(iv)}{\underbrace{
                \transposed{\left(
                    \grad{\Gamma_r}{}{}\paramVel[r] 
                \right)}
                \orthProj[\Gamma_r]{} 
                \frobScalarProd
                \grad{\Gamma_r}{}{} f
            }}
            - 
            \underset{(ii)}{\underbrace{
                \orthProj[\Gamma_r]{} \grad{\Gamma_r}{}{}\paramVel[r] 
                \orthProj[\Gamma_r]{}
                \frobScalarProd
                \grad{\Gamma_r}{}{} f 
            }}
            \\
            &- \orthProj[\Gamma_r]{} \transposed{\left(
                \grad{\Gamma_r}{}{}\paramVel[r]
            \right)} \orthProj[\Gamma_r]{}
            \frobScalarProd
            \grad{\Gamma_r}{}{} f                            
        \end{split}.
    \end{equation*}
    The terms $ (i) $ and $ (ii) $ as well as $ (iii) $ and $ (iv) $
    cancel out with the rules
    \eqref{equ:ommit proj transposed} and \eqref{equ:ommit proj %
    transposed frob} from \autoref{lemma:ommit proj rules}.
\end{proof}
}

We can now derive a commutator rule for the Laplace-Beltrami
operator:
\begin{mylemma}[Commutator rule for the Laplace-Beltrami operator]
    Let $ \spaceParamDomain{\Gamma}{I} $, $ I \subseteq \reals $,
    be a smooth evolving manifold and 
    $ \funSig{f}{\spaceParamDomain{\Gamma}{I}}{\reals} $ a function with
    tangential derivatives up to order two whose material derivative exists
    and let $ f $ be material differentiable itself.
    It holds,
    \begin{align*}
        \matDiff[r]{s}{
            \laplacian{\Gamma_s}{}{f}
        } 
        =
        \laplacian{\Gamma_r}{}{ 
            \matDiff[r]{s}{f}
        }
        &+
        \left( 
            \diver{\Gamma_r}{
            \transposed{\left(
                \grad{\Gamma_r}{}{}\paramVel[r]
                -
                2 \symmTangJacobian{\Gamma_r}{\paramVel[r]}
            \right)}
            }
        \right) 
        \cdot
        \grad{\Gamma_r}{}{} f
        \\
        &+
        2 \left( 
            \grad{\Gamma_r}{}{} \paramVel[r]
            -
            \transposed{\left(
                \grad{\Gamma_r}{}{} \paramVel[r]
            \right)}
            -
            \orthProj[\Gamma_r]{} 
            \grad{\Gamma_r}{}{} \paramVel[r]
        \right)
        \frobScalarProd
        \grad{\Gamma_r}{2}{} f
        \\
        &+
        \grad{\Gamma_r}{}{} \paramVel[r] 
        \frobScalarProd
        \normal{}^r
        \tensorProd
        \left(
            \weingartenMapping[r]{}
            \grad{\Gamma_r}{}{} f
        \right).
    \end{align*}
\end{mylemma}
\begin{proof}
    Commute two times:
    \begin{align*}
        \matDiff[r]{s}{
            \laplacian{\Gamma_s}{}{} f
        } =
        \matDiff[r]{s}{
            \diver{\Gamma_s}{} \grad{\Gamma_s}{}{} f
        }
        &\overset{(1)}{=}
        \diver{\Gamma_r}{ 
            \matDiff[r]{s}{
                \grad{\Gamma_s}{}{} f
             }
        }
        +
        \left(
            \grad{\Gamma_r}{}{} \paramVel[r]
            -
            2
            \symmTangJacobian{\Gamma_r}{\paramVel[r]}
        \right)
        \frobScalarProd
        \grad{\Gamma_r}{2}{} f
        \\
        &\overset{(2)}{=}
        \diver{\Gamma_r}{ 
            \grad{\Gamma_r}{}{ \matDiff[r]{s}{ f } }
            +
            \left(
                \grad{\Gamma_r}{}{} \paramVel[r]
                -
                2
                \symmTangJacobian{\Gamma_r}{\paramVel[r]}
            \right)
            \grad{\Gamma_r}{2}{} f
        }
        \\
        &+
        \left(
            \grad{\Gamma_r}{}{} \paramVel[r]
            -
            2
            \symmTangJacobian{\Gamma_r}{\paramVel[r]}
        \right)
        \frobScalarProd
        \grad{\Gamma_r}{2}{} f
        \\
        &\overset{(3)}{=}
        \laplacian{\Gamma_r}{}{ \matDiff[r]{s}{f} }
        +
        \left(
            \diver{\Gamma_r}{}
            \transposed{\left(
                \grad{\Gamma_r}{}{} \paramVel[r]
                -
                2
                \symmTangJacobian{\Gamma_r}{\paramVel[r]}
            \right)}
        \right) 
        \cdot
        \grad{\Gamma_r}{}{} f
        \\
        &+
        \left(
            \grad{\Gamma_r}{}{} \paramVel[r]
            -
            2
            \symmTangJacobian{\Gamma_r}{\paramVel[r]}                      	
        \right)
        \frobScalarProd
        \transposed{ \left( \grad{\Gamma_r}{2}{} f \right) }
        \\
        &+
        \left(
            \grad{\Gamma_r}{}{} \paramVel[r]
            -
            2
            \symmTangJacobian{\Gamma_r}{\paramVel[r]}
        \right)
        \frobScalarProd
        \grad{\Gamma_r}{2}{} f
        \\
        &\overset{(4)}{=}
        \laplacian{\Gamma_r}{}{ \matDiff[r]{s}{f} }
        +
        \left(
            \diver{\Gamma_r}{}
            \transposed{\left(
                \grad{\Gamma_r}{}{} \paramVel[r]
                -
                2
                \symmTangJacobian{\Gamma_r}{\paramVel[r]}
            \right)}
        \right) 
        \cdot
        \grad{\Gamma_r}{}{} f
        \\
        &+
        2
        \left(
            \grad{\Gamma_r}{}{} \paramVel[r]
            -
            2
            \symmTangJacobian{\Gamma_r}{\paramVel[r]}                      	
        \right)
        \frobScalarProd
        \left( \grad{\Gamma_r}{2}{} f \right)
        \\
        &-
        \left(
            \grad{\Gamma_r}{}{} \paramVel[r]
            -
            2
            \symmTangJacobian{\Gamma_r}{\paramVel[r]}                      	
        \right)
        \frobScalarProd
        \left(
            \weingartenMapping[r]{}
            \grad{\Gamma_r}{}{} f
        \right)
        \tensorProd
        \normal[r]{}
        \\
        &+
        \left(
            \grad{\Gamma_r}{}{} \paramVel[r]
            -
            2
            \symmTangJacobian{\Gamma_r}{\paramVel[r]}                      	
        \right)
        \frobScalarProd
        \normal[r]{} 
        \tensorProd 
        \left( 
            \weingartenMapping[r]{} \grad{\Gamma_r}{}{} f 
        \right)
        \\
        &\overset{(5)}{=}
        \laplacian{\Gamma_r}{}{ \matDiff[r]{s}{f} }
        +
        \left(
            \diver{\Gamma_r}{}
            \transposed{\left(
                \grad{\Gamma_r}{}{} \paramVel[r]
                -
                2
                \symmTangJacobian{\Gamma_r}{\paramVel[r]}
            \right)}
        \right) 
        \cdot
        \grad{\Gamma_r}{}{} f
        \\
        &+
        2
        \left(
            \grad{\Gamma_r}{}{} \paramVel[r]
            -
            \transposed{\left(
                \grad{\Gamma_r}{}{} \paramVel[r]
            \right)}
        \right)
        \frobScalarProd
        \left( \grad{\Gamma_r}{2}{} f \right)
        \\
        &-
        2 \left(
            \orthProj[\Gamma_r]{}
            \grad{\Gamma_r}{}{} \paramVel[r]
        \right)
        \frobScalarProd
        \grad{\Gamma_r}{2}{} f
        \\
        &+
        \left(
            \grad{\Gamma_r}{}{} \paramVel[r]
            -
            2
            \symmTangJacobian{\Gamma_r}{\paramVel[r]}                      	
        \right)
        \frobScalarProd
        \normal[r]{} 
        \tensorProd 
        \left( 
            \weingartenMapping[r]{} \grad{\Gamma_r}{}{} f 
        \right)                    
        \\
        &\overset{(6)}{=}
        \laplacian{\Gamma_r}{}{ \matDiff[r]{s}{f} }
        +
        \left(
            \diver{\Gamma_r}{}
            \transposed{\left(
                \grad{\Gamma_r}{}{} \paramVel[r]
                -
                2
                \symmTangJacobian{\Gamma_r}{\paramVel[r]}
            \right)}
        \right) 
        \cdot
        \grad{\Gamma_r}{}{} f
        \\
        &+
        2 \left( 
            \grad{\Gamma_r}{}{} \paramVel[r]
            -
            \transposed{\left(
                \grad{\Gamma_r}{}{} \paramVel[r]
            \right)}
            -
            \orthProj[\Gamma]{} 
            \grad{\Gamma_r}{}{} \paramVel[r]
        \right)
        \frobScalarProd
        \grad{\Gamma_r}{2}{} f
        \\
        &+\grad{\Gamma_r}{}{} \paramVel[r]
        \frobScalarProd
        \normal[r]{}
        \tensorProd
        \left(
            \weingartenMapping[r]{}
            \grad{\Gamma_r}{}{} f
        \right)
    \end{align*}
    (1): \autoref{lemma:commutation rule divergence}\\
    (2): \autoref{lemma:commutator rule for the tangential gradient}\\
    (3): $ \diver{}{ A v } = \left( \diver{}{}\transposed{A} \right) \cdot v + 
        A \frobScalarProd \transposed{\left( \grad{}{}{}v \right)} $\\
    (4): \autoref{lemma:schwarz theorem for tangential gradients} (Schwarz
    for tangential gradients)\\
    (5): 
    \begin{align*}
        \left(
            \grad{\Gamma_r}{}{} \paramVel[r]
            -
            2
            \symmTangJacobian{\Gamma_r}{\paramVel[r]}                      	
        \right)
        \frobScalarProd
        \grad{\Gamma_r}{2}{} f
        &=
        \left(
            \grad{\Gamma_r}{}{} \paramVel[r]
            - \orthProj[\Gamma_r]{} 
            \grad{\Gamma_r}{}{} \paramVel[r]
            -
            \transposed{\left( \grad{\Gamma_r}{}{} \paramVel[r] \right)}
            \orthProj[\Gamma_r]{}
        \right) 
        \frobScalarProd
        \grad{\Gamma_r}{2}{} f 
        \\
        &=
        \grad{\Gamma_r}{}{} \paramVel[r]
        \frobScalarProd
        \grad{\Gamma_r}{2}{} f
        - 
        \orthProj[\Gamma_r]{} 
        \grad{\Gamma_r}{}{} \paramVel[r]
        \frobScalarProd
        \grad{\Gamma_r}{2}{} f
        -
        \transposed{\left( \grad{\Gamma_r}{}{} \paramVel[r] \right) }
        \orthProj[\Gamma_r]{}
        \frobScalarProd
        \grad{\Gamma_r}{2}{} f
        \\
        &=
        \left(
            \grad{\Gamma_r}{}{} \paramVel[r]
            -
            \transposed{\left( \grad{\Gamma_r}{}{} \paramVel[r] \right) }
        \right)
        \frobScalarProd
        \grad{\Gamma_r}{2}{} f
        - 
        \orthProj[\Gamma_r]{} 
        \grad{\Gamma_r}{}{} \paramVel[r]
        \frobScalarProd
        \grad{\Gamma_r}{2}{} f
    \end{align*}
    and
    \begin{align*}
        \left(
            \grad{\Gamma_r}{}{} \paramVel[r]
            -
            2
            \symmTangJacobian{\Gamma_r}{\paramVel[r]}                      	
        \right)
        \frobScalarProd
        \left(
            \weingartenMapping[r]{} 
            \grad{\Gamma_r}{}{} f
        \right)
        \tensorProd
        \normal[r]{}
        &=
        \left(
            \grad{\Gamma_r}{}{} \paramVel[r]
            - \orthProj[\Gamma_r]{} 
            \grad{\Gamma_r}{}{} \paramVel[r]
            -
            \transposed{\left( \grad{\Gamma_r}{}{} \paramVel[r] \right)}
            \orthProj[\Gamma_r]{}
        \right) 
        \frobScalarProd
        \left(
            \weingartenMapping[r]{} 
            \grad{\Gamma_r}{}{} f
        \right)
        \tensorProd
        \normal[r]{}
        \\
        &=
        \trace\left(
            \transposed{\left( \grad{\Gamma_r}{}{} \paramVel[r] \right)}
            \left(
                \weingartenMapping[r]{} 
                \grad{\Gamma_r}{}{} f
            \right)
            \tensorProd
            \normal[r]{}
        \right)
        \\
        &+
        \trace\left(
            \transposed{\left( \grad{\Gamma_r}{}{} \paramVel[r] \right)}
            \orthProj[\Gamma_r]{}
            \left(
                \weingartenMapping[r]{} 
                \grad{\Gamma_r}{}{} f
            \right)
            \tensorProd
            \normal[r]{}
        \right)
        \\
        &+
        \trace\left(
            \orthProj[\Gamma_r]{}
            \grad{\Gamma_r}{}{} \paramVel[r]
            \left(
                \weingartenMapping[r]{} 
                \grad{\Gamma_r}{}{} f
            \right)
            \tensorProd
            \normal[r]{}
        \right)
        \\
        &= 0
    \end{align*}
    by using cyclic shifting.\\
    (6):
    \begin{align*}
        \left(
            \grad{\Gamma_r}{}{} \paramVel[r]
            -
            2
            \symmTangJacobian{\Gamma_r}{\paramVel[r]}                      	
        \right)
        \frobScalarProd
        \normal[r]{}
        \tensorProd
        \left(
            \weingartenMapping[r]{} 
            \grad{\Gamma_r}{}{} f
        \right)
        &=
        \left(
            \grad{\Gamma_r}{}{} \paramVel[r]
            - \orthProj[\Gamma_r]{} 
            \grad{\Gamma_r}{}{} \paramVel[r]
            -
            \transposed{\left( \grad{\Gamma_r}{}{} \paramVel[r] \right)}
            \orthProj[\Gamma_r]{}
        \right) 
        \frobScalarProd
        \normal[r]{}
        \tensorProd
        \left(
            \weingartenMapping[r]{} 
            \grad{\Gamma_r}{}{} f
        \right)
        \\
        &=
        \grad{\Gamma_r}{}{} \paramVel[r]
        \frobScalarProd
        \normal[r]{}
        \tensorProd
        \left(
            \weingartenMapping[r]{}
            \grad{\Gamma_r}{}{} f
        \right)
    \end{align*}
    by the same arguments as in (5).
\end{proof}


%% file: Appendix_Derivatives.tex
\subsection{Material derivative of the mean curvature}
    \begin{mylemma}[{\cite[Lemma~39]{Barrett+2019}}]
        \label{lemma:mat deriv mean curvature}
        Let $ \spaceParamDomain{\Gamma}{I} $ be an orientable,
        smooth evolving manifold. We then have
        \begin{equation*}
            \matDiff[0]{s}{ \meanCurv } =
            -\laplacian{\Gamma}{}{} \coeffNormalParamVel
            +
            \laplacian{\Gamma}{}{} \normal{} 
            \cdot 
            \tangParamVel
            - 
            \coeffNormalParamVel
            \abs{ \weingartenMapping{} }^2.
        \end{equation*}
    \end{mylemma}
\hide{
    \begin{proof}
        By applying the commutator rule \autoref{lemma:commutation rule divergence},
        we can write
        \begin{equation*}
            \matDiff[0]{s}{ 
              	\diver{\Gamma^s}{}{} \normal[\Gamma^s]{} 
            } 
            =
            \diver{\Gamma}{}{ 
              	\matDiff[0]{s}{ \normal[\Gamma^s]{} } 
            }
            + 
            \left( 
              	\grad{\Gamma}{}{}\paramVel -
                2 \symmTangJacobian{\Gamma}{\paramVel} 
            \right)
            \frobScalarProd
            \grad{\Gamma}{}{}\normal[\Gamma]{}.
        \end{equation*}
        Further,
        \begin{align*}
            \left( 
          		\grad{\Gamma}{}{}\paramVel -
            2 \symmTangJacobian{\Gamma}{\paramVel} \right)
            \frobScalarProd
            \grad{\Gamma}{}{}\normal[\Gamma]{}
            &=
            - \orthProj[\Gamma]{} \transposed{ \left(
                \grad{\Gamma}{}{}\paramVel
            \right)}
            \orthProj[\Gamma]{}
            \frobScalarProd
            \grad{\Gamma}{}{} \normal[\Gamma]{} \\
            &\overset{(1)}{=}
            -\transposed{ \left(
                \grad{\Gamma}{}{}\paramVel
            \right)}
            \frobScalarProd
            \grad{\Gamma}{}{} \normal[\Gamma]{},
        \end{align*}
        where (1) means application of \autoref{lemma:ommit proj rules}, \eqref{equ:ommit proj transposed},
        \eqref{equ:ommit proj transposed frob},
        since
        \begin{align*}
            \grad{\Gamma}{}{} \paramVel \orthProj[\Gamma]{}
            \frobScalarProd
            \grad{\Gamma}{}{} \normal[\Gamma]{}
            &= 
            \trace\left( \orthProj[\Gamma]{}
            \transposed{\left( \grad{\Gamma}{}{}\paramVel \right)}
            \grad{\Gamma}{}{} \normal[\Gamma]{}
            \right) \\
            &=
            \trace\left( 
            \transposed{\left( \grad{\Gamma}{}{}\paramVel \right)}
            \grad{\Gamma}{}{} \normal[\Gamma]{}
            \orthProj[\Gamma]{}
            \right) \\   
            &\overset{(2)}{=}
            \trace\left( 
            \transposed{\left( \grad{\Gamma}{}{}\paramVel \right)}
            \transposed{\left(
            \orthProj[\Gamma]{}
            \grad{\Gamma}{}{} \normal[\Gamma]{}
            \right)}
            \right) \\   
            &=
            \grad{\Gamma}{}{} \paramVel
            \frobScalarProd
            \grad{\Gamma}{}{} \normal[\Gamma]{}
        \end{align*}
        (2): symmetry of the Weingarten mapping.

        The velocity can be decomposed into normal and tangential parts:
        \begin{equation*}
            \paramVel = \coeffNormalParamVel \normal[\Gamma]{} + \tangParamVel
        \end{equation*}
        Therefore,
        \begin{equation*}
            \grad{\Gamma}{}{}\paramVel = \grad{\Gamma}{}{}\tangParamVel +
            \grad{\Gamma}{}{}\coeffNormalParamVel \tensorProd \normal[\Gamma]{}
            + \coeffNormalParamVel \grad{\Gamma}{}{} \normal[\Gamma]{}.
        \end{equation*}
        Inserting this allows for some simplifications
        \begin{align*}
            -\transposed{ \left(
                \grad{\Gamma}{}{}\paramVel
            \right)}
            \frobScalarProd
            \grad{\Gamma}{}{} \normal[\Gamma]{}
            &=
            \left(
            -\transposed{\left(
                \grad{\Gamma}{}{}\tangParamVel
            \right)}
            -
            \normal[\Gamma]{} \tensorProd \grad{\Gamma}{}{}\coeffNormalParamVel
            - \coeffNormalParamVel
            \transposed{\left(
                \grad{\Gamma}{}{} \normal[\Gamma]{}
            \right)}
            \right)
            \frobScalarProd
            \grad{\Gamma}{}{} \normal[\Gamma]{}
            \\
            &= 
            -\transposed{\left(
                \grad{\Gamma}{}{}\tangParamVel 
            \right)}
            \frobScalarProd
            \grad{\Gamma}{}{} \normal[\Gamma]{}
            -
            \coeffNormalParamVel \abs{\grad{\Gamma}{}{}\normal[\Gamma]{}}^2
        \end{align*}
        where we use that
        \begin{align*}
            \normal[\Gamma]{} \tensorProd \grad{\Gamma}{}{} \coeffNormalParamVel
            \frobScalarProd
            \grad{\Gamma}{}{} \normal[\Gamma]{}
            = 
            \trace\left(
                \grad{\Gamma}{}{} \coeffNormalParamVel \tensorProd \normal[\Gamma]{}
                \grad{\Gamma}{}{} \normal[\Gamma]{}
            \right)
            = 0.
        \end{align*}

        So far, we achieved
        \begin{align*}
            \matDiff[0]{s}{ \meanCurv } 
            &=
            \diver{\Gamma}{}{ \matDiff[0]{s}{ \normal[\Gamma_s]{} } }
            -\transposed{\left(
                \grad{\Gamma}{}{}\tangParamVel 
            \right)}
            \frobScalarProd
            \grad{\Gamma}{}{} \normal[\Gamma]{}
            - 
            \coeffNormalParamVel 
            \abs{\grad{\Gamma}{}{}
            \normal[\Gamma]{}}^2 
            \\
            &\overset{(3)}{=}
            -\diver{\Gamma}{ 
                \grad{\Gamma}{}{} \paramVel \normal[\Gamma]{}
            }                        
            -\transposed{\left(
                \grad{\Gamma}{}{} \tangParamVel
            \right)}
            \frobScalarProd
            \grad{\Gamma}{}{} \normal[\Gamma]{}
            - 
            \coeffNormalParamVel \abs{\grad{\Gamma}{}{}\normal[\Gamma]{}}^2 
        \end{align*}
        (3): \autoref{lemma:mat deriv normal}

        We further calculate:
        \begin{align*}
            -\diver{\Gamma}{ 
                \grad{\Gamma}{}{} \paramVel \normal[\Gamma]{}
            }                      
            &=
            -\diver{\Gamma}{ 
                \grad{\Gamma}{}{}\tangParamVel \normal[\Gamma]{}
                +
                \grad{\Gamma}{}{}\coeffNormalParamVel\tensorProd \normal[\Gamma]{}
                \normal[\Gamma]{}
                + 
                \coeffNormalParamVel 
                \grad{\Gamma}{}{}\normal[\Gamma]{}\normal[\Gamma]{}
            } 
            \\
            &=
            -\diver{\Gamma}{ 
                \grad{\Gamma}{}{}\tangParamVel \normal[\Gamma]{}
                +
                \grad{\Gamma}{}{}\coeffNormalParamVel
            }       
            \\
            &\overset{(4)}{=}
            -\laplacian{\Gamma}{}{}\coeffNormalParamVel
            +
            \diver{\Gamma}{
                \grad{\Gamma}{}{}
                \normal[\Gamma]{}
                \tangParamVel
            }
            \\
            &=
            -\laplacian{\Gamma}{}{}\coeffNormalParamVel
            +
            \grad{\Gamma}{}{}\normal[\Gamma]{} 
            \frobScalarProd
            \grad{\Gamma}{}{}\tangParamVel
            +
            \laplacian{\Gamma}{}{}\normal[\Gamma]{} 
            \cdot 
            \tangParamVel
        \end{align*}
        (4): $ \grad{\Gamma}{}{ \tangParamVel \cdot \normal[\Gamma]{} } = 0 $.
    \end{proof}
}
\subsection{Second order derivative of the integral  of the mean curvature}
    \begin{mylemma}
        \label{lemma:second derivative of mean curv int}
        Let $ \spaceParamDomain{\Gamma}{I} $, $ I \subseteq \reals $,
        be a smooth evolving hypersurface in $ n $ dimensions.
        It holds,
        \begin{align*}
            \diff[r]{s}{2}{
                \integral{\Gamma_s}{}{
                    \meanCurv[s]
                }{\hausdorffM{n-1}}
            }
            &= 
            \int_{\Gamma_r}
            - \matDiff[r]{s}{
                \coeffNormalParamVel[s]
            } \abs{ \weingartenMapping[r]{} }^2
            +
            2 \coeffNormalParamVel[r]
            \trace\left(
                \weingartenMapping[r]{}
                \grad{\Gamma_r}{}{
                    \transposed{\left(
                        \grad{\Gamma_r}{}{} \paramVel[r]
                    \right)}
                    \normal[r]{}
                }
                +
                \weingartenMapping[r]{}
                \transposed{\left(
                    \grad{\Gamma_r}{}{} \paramVel[r]
                \right)}
                \orthProj[\Gamma_r]{}
                \weingartenMapping[r]{}
            \right)
            \\
            &+ 
            \left(
                - \laplacian{\Gamma_r}{}{} \coeffNormalParamVel[r]
                +
                \grad{\Gamma_r}{}{}\meanCurv[r]
                \cdot
                \tangParamVel[r]
                -
                \coeffNormalParamVel[r]
                \abs{ \weingartenMapping[r]{} }^2
            \right)
            \diver{\Gamma_r}{}
            \normalParamVel[r]
            \\
            &+
            \meanCurv[r]
            \left(
                \diver{\Gamma_r}{
                    \matDiff[r]{s}{\normalParamVel[s]}
                }
                +
                \left(
                    \grad{\Gamma_r}{}{} \paramVel[r]
                    -
                    2
                    \symmTangJacobian{\Gamma_r}{\paramVel[r]}
                \right)
                \frobScalarProd
                \grad{\Gamma_r}{}{}\paramVel[r]
            \right)
            \\
            &+
            \left(
                - 
                \coeffNormalParamVel[r]
                \abs{\weingartenMapping[r]{}}^2
                +
                \meanCurv[r]
                \diver{\Gamma_r}{}
                \normalParamVel[r]
            \right)
            \diver{\Gamma_r}{}
            \paramVel[r]
            \text{d}\,\hausdorffM{n-1}.
        \end{align*}
    \end{mylemma}
    \begin{proof}
        \begin{equation*}
            \begin{split}
                \diff[r]{s}{2}{
                    \integral{\Gamma_s}{}{
                        \meanCurv[s]
                    }{\hausdorffM{n-1}}
                }
                &\overset{(1)}{=}
                \diff[r]{s}{}{
                    \integral{\Gamma_s}{}{
                        -\laplacian{\Gamma_s}{}{}\coeffNormalParamVel[s]
                        +
                        \laplacian{\Gamma_s}{}{}\normal[s]{}
                        \cdot 
                        \tangParamVel[s]
                        - 
                        \coeffNormalParamVel[s]
                        \abs{\weingartenMapping[s]{}}^2
                        +
                        \meanCurv[s]
                        \diver{\Gamma_s}{}
                        \paramVel[s]
                    }{\hausdorffM{n-1}}
                }
                \\
                &\overset{(2)}{=}
                \diff[r]{s}{}{
                    \integral{\Gamma_s}{}{
                        -\laplacian{\Gamma_s}{}{}\coeffNormalParamVel[s]
                        +
                        \left(
                            \grad{\Gamma_s}{}{}\meanCurv[s]
                            -
                            \normal[s]{}
                            \abs{ \weingartenMapping[s]{} }^2
                        \right)
                        \cdot 
                        \tangParamVel[s]
                        - 
                        \coeffNormalParamVel[s]
                        \abs{\weingartenMapping[s]{}}^2
                        +
                        \meanCurv[s]
                        \diver{\Gamma_s}{}
                        \paramVel[s]
                    }{\hausdorffM{n-1}}
                } 
                \\
                &\overset{(3)}{=}
                \diff[r]{s}{}{
                    \integral{\Gamma_s}{}{
                        -\laplacian{\Gamma_s}{}{}\coeffNormalParamVel[s]
                        - 
                        \coeffNormalParamVel[s]
                        \abs{\weingartenMapping[s]{}}^2
                        +
                        \meanCurv[s]
                        \diver{\Gamma_s}{}
                        \normalParamVel[s]
                    }{\hausdorffM{n-1}}
                } 
                \\
                &\overset{(4)}{=}
                \diff[r]{s}{}{
                    \integral{\Gamma_s}{}{
                        - 
                        \coeffNormalParamVel[s]
                        \abs{\weingartenMapping[s]{}}^2
                        +
                        \meanCurv[s]
                        \diver{\Gamma_s}{}
                        \normalParamVel[s]
                    }{\hausdorffM{n-1}}
                } 
                \\
                &=
                \int_{\Gamma_r}
                - \matDiff[r]{s}{
                    \coeffNormalParamVel[s]
                }
                \abs{ \weingartenMapping[r]{} }^2
                - 
                \coeffNormalParamVel[r]
                \matDiff[r]{s}{
                    \abs{ \weingartenMapping[s]{} }^2
                }
                \\
                &+ \matDiff[r]{s}{
                    \meanCurv[s]
                }
                \diver{\Gamma_r}{}
                \normalParamVel[r]
                +
                \meanCurv[r]
                \matDiff[r]{s}{
                    \diver{\Gamma_s}{}
                    \normalParamVel[s]                            
                }
                \\
                &+
                \left(
                    - 
                    \coeffNormalParamVel[r]
                    \abs{\weingartenMapping[r]{}}^2
                    +
                    \meanCurv[r]
                    \diver{\Gamma_r}{}
                    \normalParamVel[r]
                \right)
                \diver{\Gamma_r}{}
                \paramVel[r]
                \text{d}\,\hausdorffM{n-1}
                \\
                &\overset{(5)}{=}
                \int_{\Gamma_r}
                - \matDiff[r]{s}{
                    \coeffNormalParamVel[s]
                } \abs{ \weingartenMapping[r]{} }^2
                +
                2 \coeffNormalParamVel[r]
                \trace\left(
                    \weingartenMapping[r]{}
                    \grad{\Gamma_r}{}{
                        \transposed{\left(
                            \grad{\Gamma_r}{}{} \paramVel[r]
                        \right)}
                        \normal[r]{}
                    }
                    +
                    \weingartenMapping[r]{}
                    \transposed{\left(
                        \grad{\Gamma_r}{}{} \paramVel[r]
                    \right)}
                    \orthProj[\Gamma_r]{}
                    \weingartenMapping[r]{}
                \right)
                \\
                &+ \matDiff[r]{s}{
                    \meanCurv[s]
                }
                \diver{\Gamma_r}{}
                \normalParamVel[r]
                +
                \meanCurv[r]
                \matDiff[r]{s}{
                    \diver{\Gamma_s}{}
                    \normalParamVel[s]                            
                }
                \\
                &+
                \left(
                    - 
                    \coeffNormalParamVel[r]
                    \abs{\weingartenMapping[r]{}}^2
                    +
                    \meanCurv[r]
                    \diver{\Gamma_r}{}
                    \normalParamVel[r]
                \right)
                \diver{\Gamma_r}{}
                \paramVel[r]
                \text{d}\,\hausdorffM{n-1}
                \\
                &\overset{(6)}{=}
                \int_{\Gamma_r}
                - \matDiff[r]{s}{
                    \coeffNormalParamVel[s]
                } \abs{ \weingartenMapping[r]{} }^2
                +
                2 \coeffNormalParamVel[r]
                \trace\left(
                    \weingartenMapping[r]{}
                    \grad{\Gamma_r}{}{
                        \transposed{\left(
                            \grad{\Gamma_r}{}{} \paramVel[r]
                        \right)}
                        \normal[r]{}
                    }
                    +
                    \weingartenMapping[r]{}
                    \transposed{\left(
                        \grad{\Gamma_r}{}{} \paramVel[r]
                    \right)}
                    \orthProj[\Gamma_r]{}
                    \weingartenMapping[r]{}
                \right)
                \\
                &+ 
                \left(
                    - \laplacian{\Gamma_r}{}{} \coeffNormalParamVel[r]
                    +
                    \grad{\Gamma_r}{}{}\meanCurv[r]
                    \cdot
                    \tangParamVel[r]
                    -
                    \coeffNormalParamVel[r]
                    \abs{ \weingartenMapping[r]{} }^2
                \right)
                \diver{\Gamma_r}{}
                \normalParamVel[r]
                \\
                &+
                \meanCurv[r]
                \left(
                    \diver{\Gamma_r}{
                        \matDiff[r]{s}{\normalParamVel[s]}
                    }
                    +
                    \left(
                        \grad{\Gamma_r}{}{} \paramVel[r]
                        -
                        2
                        \symmTangJacobian{\Gamma_r}{\paramVel[r]}
                    \right)
                    \frobScalarProd
                    \grad{\Gamma_r}{}{}\paramVel[r]
                \right)
                \\
                &+
                \left(
                    - 
                    \coeffNormalParamVel[r]
                    \abs{\weingartenMapping[r]{}}^2
                    +
                    \meanCurv[r]
                    \diver{\Gamma_r}{}
                    \normalParamVel[r]
                \right)
                \diver{\Gamma_r}{}
                \paramVel[r]
                \text{d}\,\hausdorffM{n-1}
            \end{split}
        \end{equation*}
        (1): \autoref{lemma:mat deriv mean curvature}\\
        (2): \autoref{lemma:laplacian of the normal}\\
        (3): integration by parts\\
        (4): divergence theorem for manifolds without boundary \\
        (5): see \cite{Elliott+2017} \\
        (6): \autoref{lemma:mat deriv mean curvature}
        and commutator rule for tangential divergence (\autoref{lemma:%
        commutation rule divergence})
    \end{proof}
    \begin{mycorollary}
        \label{corollary:second derivative int mean curv normal vel}
        In case $ \paramVel[s] = h \normal{} \concat \globalParam[s]{}^{-1} $,
        we have
        \begin{align*}
            \diff[0]{s}{2}{
                \integral{\Gamma_s}{}{
                    \meanCurv[s]
                }{\hausdorffM{n-1}}
            } 
            =
            \int_{\Gamma}
            2 h
            \trace\left(
                \weingartenMapping{}
                \grad{\Gamma}{2}{} h
                +
                h
                \weingartenMapping{}^3
            \right)
            -
            h \meanCurv \laplacian{\Gamma}{}{} h
            - 3 h^2 \meanCurv \abs{ \weingartenMapping{} }^2
            +
            \meanCurv
            \grad{\Gamma}{}{} h 
            \cdot
            \grad{\Gamma}{}{} h
            +
            h^2 \meanCurv^3
            \;
            \text{d}\,\hausdorffM{n-1}
        \end{align*}
        We write $ \paramVel[0] = \paramVel $ etc.
    \end{mycorollary}
    \begin{proof}
        \begin{align*}
            \diff[0]{s}{2}{
                \integral{\Gamma_s}{}{
                    \meanCurv[s]
                }{\hausdorffM{n-1}}
            } 
            &\overset{(1)}{=}
            \int_{\Gamma}
            - \matDiff[0]{s}{
                \coeffNormalParamVel[s]
            } \abs{ \weingartenMapping{} }^2
            +
            2 \coeffNormalParamVel
            \trace\left(
                \weingartenMapping{}
                \grad{\Gamma}{}{
                    \transposed{\left(
                        \grad{\Gamma}{}{} \paramVel
                    \right)}
                    \normal{}
                }
                +
                \weingartenMapping{}
                \transposed{\left(
                    \grad{\Gamma}{}{} \paramVel
                \right)}
                \orthProj[\Gamma]{}
                \weingartenMapping{}
            \right)
            \\
            &+ 
            \left(
                - \laplacian{\Gamma}{}{} \coeffNormalParamVel
                +
                \grad{\Gamma}{}{}\meanCurv
                \cdot
                \tangParamVel
                -
                \coeffNormalParamVel
                \abs{ \weingartenMapping{} }^2
            \right)
            \diver{\Gamma}{}
            \normalParamVel
            \\
            &+
            \meanCurv
            \left(
                \diver{\Gamma}{
                    \matDiff[0]{s}{\normalParamVel[s]}
                }
                +
                \left(
                    \grad{\Gamma}{}{} \paramVel
                    -
                    2
                    \symmTangJacobian{\Gamma}{\paramVel}
                \right)
                \frobScalarProd
                \grad{\Gamma}{}{}\paramVel
            \right)
            \\
            &+
            \left(
                - 
                \coeffNormalParamVel
                \abs{\weingartenMapping{}}^2
                +
                \meanCurv
                \diver{\Gamma}{}
                \normalParamVel
            \right)
            \diver{\Gamma}{}
            \paramVel
            \text{d}\,\hausdorffM{n-1}
            \\
            &\overset{(2)}{=}
            \int_{\Gamma}
            2 \coeffNormalParamVel
            \trace\left(
                \weingartenMapping{}
                \grad{\Gamma}{}{
                    \transposed{\left(
                        \grad{\Gamma}{}{} \paramVel
                    \right)}
                    \normal{}
                }
                +
                \weingartenMapping{}
                \transposed{\left(
                    \grad{\Gamma}{}{} \paramVel
                \right)}
                \orthProj[\Gamma]{}
                \weingartenMapping{}
            \right)
            \\
            &-
            \left(
                \laplacian{\Gamma}{}{} \coeffNormalParamVel
                +
                \coeffNormalParamVel
                \abs{ \weingartenMapping{} }^2
            \right)
            \diver{\Gamma}{}
            \normalParamVel
            \\
            &+
            \meanCurv
            \left(
                \left(
                    \grad{\Gamma}{}{} \paramVel
                    -
                    2
                    \symmTangJacobian{\Gamma}{\paramVel}
                \right)
                \frobScalarProd
                \grad{\Gamma}{}{}\paramVel
            \right)
            \\
            &+
            \left(
                - 
                \coeffNormalParamVel
                \abs{\weingartenMapping{}}^2
                +
                \meanCurv
                \diver{\Gamma}{}
                \normalParamVel
            \right)
            \diver{\Gamma}{}
            \paramVel
            \text{d}\,\hausdorffM{n-1}
            \\
            &\overset{(3)}{=}
            \int_{\Gamma}
            2 h
            \trace\left(
                \weingartenMapping{}
                \grad{\Gamma}{}{
                    \transposed{\left(
                        \normal{} \tensorProd \grad{\Gamma}{}{} h
                        +
                        h \weingartenMapping{}
                    \right)}
                    \normal{}
                }
                +
                \weingartenMapping{}
                \transposed{\left(
                    \normal{} \tensorProd \grad{\Gamma}{}{} h
                    +
                    h \weingartenMapping{}                                
                \right)}
                \orthProj[\Gamma]{}
                \weingartenMapping{}
            \right)
            \\
            &-
            \left(
                \laplacian{\Gamma}{}{} h
                +
                h
                \abs{ \weingartenMapping{} }^2
            \right)
            h \meanCurv
            \\
            &+
            \meanCurv
            \left(
                \left(
                    \normal{} \tensorProd \grad{\Gamma}{}{} h
                    -
                    h \weingartenMapping{}
                \right)
                \frobScalarProd
                \left( 
                    \normal{} \tensorProd \grad{\Gamma}{}{} h 
                    + 
                    h \weingartenMapping{}
                \right)
            \right)
            \\
            &+
            \left(
                - 
                h
                \abs{\weingartenMapping{}}^2
                +
                h
                \meanCurv^2
            \right)
            h \meanCurv
            \;
            \text{d}\,\hausdorffM{n-1}
            \\
            &\overset{(4)}{=}
            \int_{\Gamma}
            2 h
            \trace\left(
                \weingartenMapping{}
                \grad{\Gamma}{2}{} h
                +
                h
                \weingartenMapping{}^3
            \right)
            \\
            &-
            h \meanCurv
            \left(
                \laplacian{\Gamma}{}{} h
                +
                h
                \abs{ \weingartenMapping{} }^2
            \right)
            \\
            &+
            \meanCurv
            \left(
                \grad{\Gamma}{}{} h 
                \cdot
                \grad{\Gamma}{}{} h
                -
                h^2 \abs{ \weingartenMapping{} }^2
            \right)
            \\
            &+
            h \meanCurv
            \left(
                - 
                h
                \abs{\weingartenMapping{}}^2
                +
                h
                \meanCurv^2
            \right)
            \;
            \text{d}\,\hausdorffM{n-1}
            \\
            &=
            \int_{\Gamma}
            2 h
            \trace\left(
                \weingartenMapping{}
                \grad{\Gamma}{2}{} h
                +
                h
                \weingartenMapping{}^3
            \right)
            -
            h \meanCurv \laplacian{\Gamma}{}{} h
            - 3 h^2 \meanCurv \abs{ \weingartenMapping{} }^2
            +
            \meanCurv
            \grad{\Gamma}{}{} h 
            \cdot
            \grad{\Gamma}{}{} h
            +
            h^2 \meanCurv^3
            \;
            \text{d}\,\hausdorffM{n-1}
        \end{align*}
        (1): \autoref{lemma:second derivative of mean curv int}\\
        (2): as $ \paramVel[s] \concat \globalParam[s]{} $ is 
        independent of $ s $ and (note $ \globalParam{} =
        \identity[\Gamma] $) 
        \begin{align*}
            \matDiff[0]{s}{\coeffNormalParamVel[s]}
            &=
            \pDiff[0]{s}{}{
                h \normal{} \cdot \normal[s]{} 
                \concat 
                \left( \identity[I], \globalParam[s]{} \right)
            }
            =
            h \normal{} \cdot \pDiff[0]{s}{}{\normal[s]{} 
            \concat 
            \left( \identity[I], \globalParam[s]{}} \right)
            \\
            &=
            h \normal{} \cdot
            \left(
                \grad{x}{}{ \bar{\normal{}} } 
                \pDiff[0]{s}{}{
                    \globalParam[s]{}
                }
            \right)
            \\
            &=
            h \normal{} \cdot
            \left(
                \grad{\Gamma}{}{ \normal{} }
                \pDiff[0]{s}{}{
                    \globalParam[s]{}
                }
            \right)
            \\
            &=
            h 
            \pDiff[0]{s}{}{
                \globalParam[s]{}
            } 
            \cdot
            \left(
                \transposed{ \left( \grad{\Gamma}{}{} \normal{} \right) }
                \normal{}
            \right)
            \\
            &= 0
        \end{align*}
        due to the symmetry of $ \grad{\Gamma}{}{}\normal{} $. \\
        (3): we have 
        \begin{align*}
            \grad{\Gamma}{}{} \paramVel = 
            \grad{\Gamma}{}{ h \normal{} } =
            \normal{} \tensorProd 
            \grad{\Gamma}{}{} h + h \weingartenMapping{}
        \end{align*}
        and
        \begin{align*}
            \diver{\Gamma}{} \paramVel = \diver{\Gamma}{ h \normal{} }
            = \grad{\Gamma}{}{}h \cdot \normal{} + h \diver{\Gamma}{} \normal{}
            =
            h \meanCurv,
        \end{align*}
        as well as
        \begin{align*}
            \orthProj[\Gamma]{} \grad{\Gamma}{}{}\paramVel =
            \left( \identityMatrix - \normal{} \tensorProd \normal{} \right)
            \left( \normal{} \tensorProd \grad{\Gamma}{}{} h + 
            h \grad{\Gamma}{}{} \normal{} \right) = 
            \normal{} \tensorProd \grad{\Gamma}{}{} + h \grad{\Gamma}{}{} \normal{}
            - \normal{} \tensorProd \grad{\Gamma}{}{} h
            - h \normal{} \tensorProd \left( \transposed{ \grad{\Gamma}{}{} \normal{} }
            \normal{} \right) = h \weingartenMapping{},
        \end{align*}
        which implies $ \transposed{ \left( \grad{\Gamma}{}{} \paramVel \right) }
        \orthProj[\Gamma]{} = h \weingartenMapping{} $,
        so 
        \begin{align*}
            2 \symmTangJacobian{\Gamma}{\paramVel} = 
            \orthProj[\Gamma]{} \grad{\Gamma}{}{} \paramVel +
            \transposed{\left( \grad{\Gamma}{}{} \paramVel \right)}
            \orthProj[\Gamma]{}
            = 2 h \weingartenMapping{}.
        \end{align*}
        (4): observe
        \begin{align*}
            \grad{\Gamma}{}{
                \transposed{\left(
                    \normal{} \tensorProd \grad{\Gamma}{}{} h
                    +
                    h \weingartenMapping{}
                \right)}
                \normal{}
            }
            =
            \grad{\Gamma}{}{
                \grad{\Gamma}{}{} h
                \tensorProd
                \normal{}
                \normal{}
            }
            = \grad{\Gamma}{2}{} h
        \end{align*}
        and
        \begin{align*}
            \transposed{\left(
                \normal{} \tensorProd \grad{\Gamma}{}{} h
                +
                h \weingartenMapping{}
            \right)} \orthProj[\Gamma]{}
            &=
            \grad{\Gamma}{}{} h
            \tensorProd
            \normal{}
            \orthProj[\Gamma]{}
            +
            h \weingartenMapping{}
            \\
            &=
            \grad{\Gamma}{}{} h
            \tensorProd
            \normal{}
            -
            \grad{\Gamma}{}{} h
            \tensorProd
            \normal{}
            \normal{}
            \tensorProd
            \normal{}
            +
            h \weingartenMapping{}
            \\
            &= h \weingartenMapping{}
         \end{align*}
        and
        \begin{align*}
            \left( 
                \normal{} \tensorProd \grad{\Gamma}{}{} h 
                - 
                h \weingartenMapping{} 
            \right)
            \frobScalarProd
            \left( 
                \normal{} \tensorProd \grad{\Gamma}{}{} h 
                +
                h \weingartenMapping{} 
            \right)
            &=
            \normal{} \tensorProd \grad{\Gamma}{}{} h                         
            \frobScalarProd
            \normal{} \tensorProd \grad{\Gamma}{}{} h 
            +
            \normal{} \tensorProd \grad{\Gamma}{}{} h                         
            \frobScalarProd
            h \weingartenMapping{}
            - 
            h \weingartenMapping{}
            \frobScalarProd
            \normal{} \tensorProd \grad{\Gamma}{}{} h
            - 
            h^2
            \abs{ \weingartenMapping{} }^2
            \\
            &=
            \trace\left(
                \grad{\Gamma}{}{} h \tensorProd \normal{}
                \normal{} \tensorProd \grad{\Gamma}{}{} h
            \right)
            - 
            h^2
            \abs{ \weingartenMapping{} }^2
            \\
            &=
            \grad{\Gamma}{}{} h 
            \cdot
            \grad{\Gamma}{}{} h
            -
            h^2
            \abs{ \weingartenMapping{} }^2
        \end{align*}
    \end{proof}
    \begin{mycorollary}
        \label{corollary:second derivative int mean curv sphere normal vel}
        In case $ \paramVel[s] = h \normal{} \concat \globalParam[s]{}^{-1} $,
        and $ \Gamma $ is a sphere with radius $ R $ in $ \reals^3 $,
        we have
        \begin{align*}
            \diff[0]{s}{2}{
                \integral{\Gamma_s}{}{
                    \meanCurv[s]
                }{\hausdorffM{n-1}}
            } 
            = 
            \frac{2}{R}
            \integral{\Gamma}{}{
                \grad{\Gamma}{}{} h \cdot \grad{\Gamma}{}{} h
            }{\hausdorffM{2}}
        \end{align*}
    \end{mycorollary}
    \begin{proof}
        \begin{align*}
            \diff[0]{s}{2}{
                \integral{\Gamma_s}{}{
                    \meanCurv[s]
                }{\hausdorffM{n-1}}
            } 
            &=
            \int_{\Gamma}
            2 h
            \trace\left(
                \weingartenMapping{}
                \grad{\Gamma}{2}{} h
                +
                h
                \weingartenMapping{}^3
            \right)
            -
            h \meanCurv \laplacian{\Gamma}{}{} h
            - 3 h^2 \meanCurv \abs{ \weingartenMapping{} }^2
            +
            \meanCurv
            \grad{\Gamma}{}{} h 
            \cdot
            \grad{\Gamma}{}{} h
            +
            h^2 \meanCurv^3
            \;
            \text{d}\,\hausdorffM{n-1}
            \\
            &\overset{(1)}{=}
            \int_{\Gamma}
            2 h
            \trace\left(
                \frac{1}{R}
                \orthProj[\Gamma]{}
                \grad{\Gamma}{2}{} h
                +
                h
                \frac{1}{R^3}
                \orthProj[\Gamma]{}
            \right)
            -
            h \frac{2}{R} \laplacian{\Gamma}{}{} h
            -
            h^2 \frac{6}{R^3} \trace\left( \orthProj[\Gamma]{} \right) 
            \\
            &+
            \frac{2}{R} \grad{\Gamma}{}{} h \cdot \grad{\Gamma}{}{} h
            +
            h^2 \frac{8}{R^3}
            \;
            \text{d}\,\hausdorffM{n-1}
            \\
            &\overset{(2)}{=}
            \frac{2}{R} 
            \integral{\Gamma}{}{
                h \trace\left( \grad{\Gamma}{2}{} h \right)
            }{\hausdorffM{2}}
            -
            \frac{2}{R}
            \integral{\Gamma}{}{
                h \laplacian{\Gamma}{}{} h
            }{\hausdorffM{2}}
            + 
            \left( \frac{4}{R^3} - \frac{12}{R^3} + \frac{8}{R^3} \right)
            \integral{\Gamma}{}{
                h^2
            }{\hausdorffM{2}}
            \\
            &+
            \frac{2}{R}
            \integral{\Gamma}{}{
                \grad{\Gamma}{}{} h \cdot \grad{\Gamma}{}{} h
            }{\hausdorffM{2}}
            \\
            &=
            \frac{2}{R}
            \integral{\Gamma}{}{
                \grad{\Gamma}{}{} h \cdot \grad{\Gamma}{}{} h
            }{\hausdorffM{2}}
        \end{align*}
        (1): $ \meanCurv = \frac{2}{R} $ and $ \weingartenMapping{} = \frac{1}{R}
        \orthProj[\Gamma]{} $.\\
        (2): $ \trace\left( \orthProj[\Gamma]{} \right) = 
        \trace \identityMatrix[3] - \trace \normal{} \tensorProd \normal{}
        = 3 - \sum_i \normal{}_i^2 = 3 - \abs{\normal{}}^2 = 2 $ and cyclic shifting
        of $ \orthProj[\Gamma]{} $ in the trace \\
    \end{proof}
\subsection{Second order derivative of the surface area}
    \begin{mylemma}
        \label{lemma:second derivative surface area}
        Let $ \spaceParamDomain{\Gamma}{I} $, $ I \subseteq \reals $,
        be a smooth orientable evolving hypersurface in $ \reals^n $, $ n \in \nats $.
        It holds,
        \begin{align*}
            \diff[0]{s}{2}{
                \integral{\Gamma_s}{}{
                    1
                }{\hausdorffM{2}}
            } = 
            \integral{\Gamma}{}{
                \diver{\Gamma}{
                    \matDiff[0]{r}{ \paramVel[r] }
                }
                +
                \left(
                    \grad{\Gamma}{}{}\paramVel
                    -
                    2\symmTangJacobian{\Gamma}{\paramVel}
                \right)
                \frobScalarProd
                \grad{\Gamma}{}{} \paramVel
                +
                \left( \diver{\Gamma}{} \paramVel \right)^2
            }{\hausdorffM{2}}
        \end{align*}
    \end{mylemma}
    \begin{proof}
        Apply the surface transport
        theorem (\autoref{theorem:surface transport theorem}) two times
        \begin{align*}
            \diff[0]{r}{}{
                \diff[r]{s}{}{
                    \integral{\Gamma_s}{}{
                        1
                    }{\hausdorffM{2}}
                }
            }
            =
            \diff[0]{r}{}{
                \integral{\Gamma_r}{}{
                    \diver{\Gamma_r}{} \paramVel[r]
                }{\hausdorffM{2}}.
            }
            =
            \integral{\Gamma}{}{
                \matDiff[0]{r}{
                    \diver{\Gamma_r}{} \paramVel[r]
                }
                +
                \left( \diver{\Gamma}{} \paramVel \right)^2
            }{\hausdorffM{2}}
        \end{align*}
        Then use the commutator rule for the tangential divergence
        (\autoref{lemma:commutation rule divergence}):
        \begin{align*}
            \integral{\Gamma}{}{
                \matDiff[0]{r}{
                    \diver{\Gamma_r}{} \paramVel[r]
                }
                +
                \left( \diver{\Gamma}{} \paramVel \right)^2
            }{\hausdorffM{2}}
            =
            \integral{\Gamma}{}{
                \diver{\Gamma}{
                    \matDiff[0]{r}{ \paramVel[r] }
                }
                +
                \left(
                    \grad{\Gamma}{}{}\paramVel
                    -
                    2\symmTangJacobian{\Gamma}{\paramVel}
                \right)
                \frobScalarProd
                \grad{\Gamma}{}{} \paramVel
                +
                \left( \diver{\Gamma}{} \paramVel \right)^2
            }{\hausdorffM{2}}
        \end{align*}
    \end{proof}
    \begin{mycorollary}
        \label{corollary:second derivative surface area normal vel}         	
        In case $ \globalParam[s]{} = x + s h \normal{} $, where $ x \in \Gamma $
        and $ \funSig{h}{\Gamma}{\reals^n} $, we get
        \begin{align*}
            \diff[0]{s}{2}{
                \integral{\Gamma_s}{}{1}{\hausdorffM{n-1}}
            }
            = 
            \integral{\Gamma}{}{
                \grad{\Gamma}{}{} h \cdot \grad{\Gamma}{}{} h
                -
                h^2 \abs{ \weingartenMapping{} }^2
                + h^2 \meanCurv^2
            }{\hausdorffM{n-1}}
        \end{align*}
    \end{mycorollary}
    \begin{proof}
        We use 
        \begin{align*}
            \diff[0]{s}{2}{
                \integral{\Gamma_s}{}{1}{\hausdorffM{n-1}}
            }
            &= 
            \integral{\Gamma}{}{ 
                \left(
                    \normal{} \tensorProd \grad{\Gamma}{}{} h
                    - h \weingartenMapping{}
                \right)
                \frobScalarProd
                \left(
                    \normal{} \tensorProd \grad{\Gamma}{}{} h
                    +
                    h \weingartenMapping{}
                \right)
                +
                h^2 \meanCurv^2
            }{\hausdorffM{n-1}}
            \\
            &=
            \integral{\Gamma}{}{
                \grad{\Gamma}{}{} h \cdot \grad{\Gamma}{}{} h
                -
                h^2 \abs{ \weingartenMapping{} }^2
                + h^2 \meanCurv^2
            }{\hausdorffM{n-1}}
        \end{align*}
    \end{proof}
\subsection{Derivative of the Willmore energy}
    \begin{mylemma}
        \label{lemma:first variation of the willmore energy}
        Let $ \spaceParamDomain{\Gamma}{I} $, $ I \subseteq \reals $, be a smooth 
        orientable
        evolving hypersurface in $ \reals^n $, $ n \in \nats $. It holds,
        \begin{equation*}
            \diff[0]{s}{}{
                \frac{1}{2}
                \integral{\Gamma_s}{}{
                    \meanCurv[s]^2
                }{\hausdorffM{2}}
            }
            =
            \integral{\Gamma}{}{
            -\coeffNormalParamVel 
            \laplacian{\Gamma}{}{} \meanCurv[s]
            - 
            \coeffNormalParamVel \abs{ \grad{\Gamma}{}{}\normal[\Gamma]{} }^2
            -
            \frac{1}{2}
            \coeffNormalParamVel \meanCurv^3
        }{\hausdorffM{2}}.
        \end{equation*}
    \end{mylemma}
    \begin{proof}
        \begin{align*}
            \diff[0]{s}{}{
                \frac{1}{2}
                \integral{\Gamma_s}{}{
                    \meanCurv[s]^2
                }{\hausdorffM{2}(x)}
            }
            &\overset{(1)}{=}
            \integral{\Gamma}{}{
                \meanCurv \matDiff[0]{s}{ \meanCurv[s] }
                +
                \frac{1}{2}
                \meanCurv^2
                \diver{\Gamma}{}{} \paramVel
            }{\hausdorffM{2}}
            \\
            &\overset{(2)}{=}
            \int_\Gamma
                \meanCurv
                \left(
                    -\laplacian{\Gamma}{}{}\coeffNormalParamVel
                    +
                    \laplacian{\Gamma}{}{}\normal[\Gamma]{}
                    \cdot
                    \tangParamVel
                    -
                    \coeffNormalParamVel \abs{ \grad{\Gamma}{}{}\normal[\Gamma]{} }^2
                \right) 
            \\
            &+ \frac{1}{2}\meanCurv^2 \diver{\Gamma}{}{} \paramVel
            \mathit{d}\,\hausdorffM{2}(x)
            \\
            &\overset{(2)}{=}
            \int_\Gamma
                \meanCurv
                \bigg(
                    -\laplacian{\Gamma}{}{} \coeffNormalParamVel
                    +
                    \left(
                        \grad{\Gamma}{}{} \meanCurv
                        -
                        \normal[\Gamma]{}
                        \grad{\Gamma}{}{}\normal[\Gamma]{}
                        \frobScalarProd
                        \grad{\Gamma}{}{}\normal[\Gamma]{}
                    \right)
                    \cdot
                    \tangParamVel
            \\
                    &-
                    \coeffNormalParamVel \abs{ \grad{\Gamma}{}{}\normal[\Gamma]{} }^2
                \bigg)
            \\
            &+ 
            \frac{1}{2}
            \meanCurv^2 
            \diver{\Gamma}{}{\left(
                \coeffNormalParamVel \normal[\Gamma]{}
                +
                \tangParamVel
            \right)}
            \mathit{d}\,\hausdorffM{2}(x)                            
            \\
            &\overset{(3)}{=}
            \integral{\Gamma}{}{
                -\coeffNormalParamVel 
                \laplacian{\Gamma}{}{} \meanCurv
                - 
                \coeffNormalParamVel \abs{ \grad{\Gamma}{}{}\normal[\Gamma]{} }^2
                +
                \frac{1}{2}
                \coeffNormalParamVel \meanCurv^3
            }{\hausdorffM{2}}.
        \end{align*}
        (1): \autoref{theorem:surface transport theorem}\\
        (2): see \cite{Elliott+2017} \\
        (3): \autoref{lemma:laplacian of the normal}\\
        (4): integration by parts
    \end{proof}
